\documentclass[10pt, reqno]{amsart}
\usepackage{mathrsfs}
\usepackage{amssymb,amsthm,amsmath}
\usepackage[numbers,sort&compress]{natbib}
\usepackage{amssymb,amsmath}
\usepackage{amsfonts}
\usepackage{mathrsfs}
\usepackage{latexsym}
\usepackage{amssymb}
\usepackage{amsthm}
\usepackage{color}
\usepackage{pdfsync}
\usepackage{indentfirst}
\usepackage{graphics}
\usepackage{subfigure}
\usepackage{epsfig}
\usepackage{float}
\usepackage{cases}
\usepackage{tikz}
\usetikzlibrary{arrows, positioning}
\date{today}
\usepackage[unicode,bookmarks=false]{hyperref}
\hypersetup{
colorlinks=true,
linkcolor=blue,
anchorcolor=g,
citecolor=red}
\usepackage{amsmath}
\usepackage{amsthm}

\usepackage{graphicx}

\usepackage{natbib}%reference

\renewcommand{\upsilon}{\rho}%将字母$\upsilon$替换成\rho

\let\oldsection\section
\renewcommand\section{\setcounter{equation}{0}\oldsection}

\theoremstyle{plain} % 斜体内容 + 粗体标签
\newtheorem{theorem}{Theorem}[section]
\newtheorem{lemma}[theorem]{Lemma}
\newtheorem{corollary}[theorem]{Corollary}
\newtheorem{proposition}[theorem]{Proposition}

\theoremstyle{definition} % 正体内容 + 粗体标签

\newtheorem{remark}[theorem]{Remark}

\def\ba{\begin{eqnarray}}
	\def\ea{\end{eqnarray}}

\newcommand{\beq}{\begin{equation}}
	\newcommand{\eeq}{\end{equation}}
\newcommand{\ben}{\begin{eqnarray}}
	\newcommand{\een}{\end{eqnarray}}
\newcommand{\beno}{\begin{eqnarray*}}
	\newcommand{\eeno}{\end{eqnarray*}}

\usepackage{color}

\title[{\scriptsize  Zero-viscosity limit of the chemotaxis-Navier-Stokes equations}]{ Zero-viscosity limit of the chemotaxis-Navier-Stokes equations with the Navier-slip boundary condition}

\author{Bolun Li}
\address[Bolun Li]{School of Mathematical Sciences, Dalian University of Technology, Dalian, 116024,  China}
\email{Libolun\_dut@163.com}
\author{ Fengqiang Shi}
\address[Fengqiang~Shi]{School of Mathematical Sciences, Dalian University of Technology, Dalian, 116024,  China}
\email{2192593228@mail.dlut.edu.cn}
\author{Wendong~Wang}
\address[Wendong~Wang]{School of Mathematical Sciences, Dalian University of Technology, Dalian, 116024,  China}
\email{wendong@dlut.edu.cn}

\date{\today}  % 可选，设置日期为今天，或者自定义日期
\allowdisplaybreaks
\begin{document}
	\begin{abstract}
 The interplay of chemotaxis and diffusion of nutrients or
signaling chemicals in bacterial suspensions can produce a
variety of structures with locally high concentrations of cells,
including phyllotactic patterns, filaments, and concentrations
in fabricated microstructures, which is described by the chemotaxis-Navier-Stokes flow by Tuval et al. in 2005. 
 Dombrowski et al. also observed that Bacterial flow in a sessile drop related to
those in the Boycott effect of sedimentation can carry bioconvective plumes, viewed
from below through the bottom of a petri dish, and the horizontal “turbulence”
white line near the top is the air-water-plastic contact line.  It’s interesting to verify these turbulent phenomena mathematically.
		For varying chemotactic and velocity viscosities, we derive the boundary layer equations of the chemotaxis-Navier-Stokes system rigorously in  a two-dimensional half-space under the Navier-slip boundary condition and obtain the vanishing viscosity limit  of the 2D chemotaxis-fluid coupled system in the anisotropic conormal Sobolev spaces. 
        
        % Using matched asymptotic expansions and anisotropic conormal Sobolev spaces, we construct an approximate solution and rigorously prove its $O(\varepsilon^{\frac 3 2})$ convergence rate in energy norm, extending boundary layer theory to biofluid systems.
        \end{abstract}
	\maketitle
	\noindent

% \subjclass[2020]{35Q83, 35Q84, 35Q30; 35B45, 35A01, 35B65}
    

      \textbf{ 2020 Mathematics Subject Classification.}
      % \subjclass[2020]
       {Primary 35Q92; Secondary 76N05, 35B44, 35Q30, 35Q35}
	\tableofcontents
	
	%\tableofcontents
	\parskip5pt
	\parindent=1.5em

%Introduction
\section{Introduction}
%%==================================================
%% chapter01.tex for DUT Thesis
%% version: 0.1
%% last update: Dec 25th, 2022
%%==================================================
%\subsection{Research background}
The chemotactic movement of microorganisms in a fluid represents a fundamental coupling between biology and fluid dynamics, where the directed motion of organisms (chemotaxis) interacts with the surrounding flow. To capture this interaction, Keller and Segel~\cite{KS1970,KS1971} established the classical Keller-Segel model in the 1970s (see, also, Patlak \cite{P1953}). The physical significance of this coupling was dramatically demonstrated by Dombrowski et al.~\cite{dombrowski2004} and Tuval et al.~\cite{Tuval2005}, who observed that aerobic bacteria aggregate near the water surface, forming a distinct boundary layer and triggering large-scale bioconvection and micro-turbulence. These experiments reveal that the presence of a fluid fundamentally alters bacterial spatial organization, making the study of boundary layers essential for understanding such systems. Although extensive mathematical results exist for chemotaxis-fluid equations~\cite{Winker2014,Winker2016}, the boundary layer analysis or the vanishing viscosity limit remains largely unexplored. Investigating this limit is not only mathematically challenging but also physically significant, as it reveals the mechanism by which a nearly inviscid fluid organizes microbial life at its boundary.

% \subsection{Research progress}
% \subsubsection{Prandtl boundary layer theory}
To address the vanishing viscosity limit in chemotaxis-fluid systems, we first draw insights from classical Navier-Stokes equations. Since Prandtl \cite{Prandtl1904}
introduced the boundary layer theory and derived the Prandtl equation by the 
 asymptotic expansion in 1904, 
a fundamental problem in mathematical fluid dynamics is to determine whether
in the inviscid limit $\nu\rightarrow 0$ the solutions of the Navier–Stokes equations converge
to those of the Euler equations. Let us recall some developments briefly.
 % The Prandtl boundary layer theory \cite{Prandtl1904} describes the flow behavior near a solid boundary in the limit of vanishing viscosity, where the analogous problem hinges on the Prandtl boundary layer expansion.

% A crucial prerequisite for justifying this expansion is the well-posedness of the underlying Prandtl equation. However, due to its inherent mathematical complexity, such well-posedness has only been obtained within specific functional frameworks.
To justify the Prandtl boundary layer expansion, one of the key steps is to
establish the well-posedness of the Prandtl equation.
Under the monotonicity assumption on the outflow velocity, Oleinik and Samokhin \cite{1999Samokhin} proved the local existence and uniqueness of classical solutions in the two-dimensional case. For flows with a favorable pressure gradient, Xin and Zhang \cite{2004XinZhang} established the global existence of weak solutions. More recently, the local well-posedness in Sobolev spaces was independently proved by Alexandre et al. \cite{AWXY2015} and Masmoudi and Wong \cite{MW2015} using direct energy methods. 
On the other hand, within the analytic setting, Sammartino and Caflisch \cite{SC1998} obtained the local existence and uniqueness of solutions for full analytic data (with extensions to tangential analytic data found in \cite{LCS2003,ZZ2016}). For the Gevrey data, the local well-posed results were obtained by Gerard-Varet-Masmoudi \cite{GM2015} and Li Yang \cite{LY2020}. Conversely, without the monotonicity condition, the system exhibits strong instability; notably, G\'erard-Varet and Dormy \cite{GD2010} demonstrated that the linearized Prandtl equation around non-monotonic shear flows is ill-posed in Sobolev spaces.  For more developments on this topic, we refer to \cite{ZZ2016,LWY2016,LMY2022} and the references therein.
% [119] Gerard-Varet, David; Masmoudi, Nader；Well-posedness for the Prandtl system without analyticity or monotonicity. Ann. Sci. Éc. Norm. Supér. (4) 48 (2015), no. 6, 1273-1325.
% [120] Li, Wei-Xi; Yang, Tong；Well-posedness in Gevrey function spaces for the Prandtl equations with non-degenerate critical points. J. Eur. Math. Soc. (JEMS) 22 (2020), no. 3, 717-775. 
% [121] Zhang, Ping; Zhang, Zhifei; Long time well-posedness of Prandtl system with small and analytic initial data. J. Funct. Anal. 270 (2016), no. 7, 2591–2615.
% [122] Liu, Cheng-Jie; Wang, Ya-Guang; Yang, Tong; On the ill-posedness of the Prandtl equations in three-dimensional space. Arch. Ration. Mech. Anal. 220 (2016), no. 1, 83–108.
% [123] Li, Wei-Xi; Masmoudi, Nader; Yang, Tong; Well-posedness in Gevrey function space for 3D Prandtl equations without structural assumption. Comm. Pure Appl. Math. 75 (2022), no. 8, 1755–1797.

Another key step is  the rigorous verification of the Prandtl boundary layer
expansion.
% Although the Prandtl equation is well-posed in the functional spaces mentioned
% above, there are few results on the rigorous verification of the Prandtl boundary layer
% expansion.
Even when the Prandtl equation is well-posed in the aforementioned functional spaces, rigorously verifying the Prandtl boundary layer expansion remains highly challenging, and only a few landmark results are available. Sammartino and Caflisch \cite{SC1998} first achieved this justification within the analytic framework, a result later re-proved by Wang, Wang, and Zhang \cite{WWZ2017} via a newly developed direct energy method. Moreover, the result was extended to the three-dimensional half-space $\mathbb{R}^3_+$ by Fei, Tao, and Zhang \cite{FTZ2018} utilizing a direct energy approach again. Moving beyond analytic spaces, Maekawa \cite{MY2014} successfully justified the boundary layer expansion under the specific condition that the initial vorticity is supported away from the boundary. We refer to \cite{KVW2022,KOS2025} and the  references therein for more recent progress.

%  Kukavica, Igor; Vicol, Vlad; Wang, Fei Remarks on the inviscid limit problem for the Navier-Stokes equations. Pure Appl. Funct. Anal. 7 (2022), no. 1, 283–306.

% Kukavica, Igor; Ożański, Wojciech; Sammartino, Marco The inviscid inflow-outflow problem via analyticity. Arch. Ration. Mech. Anal. 249 (2025), no. 3, Paper No. 27, 39 pp.

\subsection{Chemotaxis equations}

Parallel to the development of boundary layer theory in fluid dynamics, few significant progress has also been made in the study of pure chemotaxis equations. The classical Keller-Segel system \cite{KS1970,KS1971} has been extensively analyzed, particularly regarding the critical balance between chemotactic aggregation and diffusion that leads to either global existence or finite-time blow-up \cite{HP2009}. Hillen and Painter \cite{HP2009} provided a comprehensive review of global existence and blow-up phenomena in chemotaxis PDE models.

%\subsubsection{Chemotaxis equations without coupling fluids}
% In the presence of physical boundaries or system degeneracy, chemotaxis mechanisms can generate boundary layer effects. The work of Peng, Ruan and Zhu \cite{PRZ2012} from 2012 established the global well-posedness of large amplitude solutions and the convergence rates of the zero diffusion limit for a transformed Keller-Segel conservation law. A few years later, Wang, Xiang, and Yu \cite{WXY2016} explored the asymptotic dynamics of singular chemotaxis systems modeling tumor angiogenesis in 2016. Moving forward to 2018, Hou, Liu, Wang, and Wang \cite{HLWW2018} investigated the stability of boundary layers for a one-dimensional viscous hyperbolic system derived from chemotaxis, establishing the nonlinear stability of boundary layer profiles under small perturbations. The subsequent year saw an extension of this analysis by Hou and Wang \cite{houwang2019} to the half-plane setting for the Keller-Segel system with singular sensitivity, proving the convergence of boundary layers and deriving the sharp convergence rates as the chemical diffusion parameter tends to zero. Around the same time, Lee et al. \cite{leewang2020} investigated singularly perturbed boundary layer profiles of nonlocal semilinear chemotaxis problems in 2020. More relevant results can be found in \cite{CLW2021,chw2024,HW2024,PWZZ2018,PWZ2014,Hou2022ks} and the references therein.

\subsubsection{ The case of chemotaxis equations without fluid coupling}
By assuming the Neumann boundary condition, Meng-Xu-Wang \cite{MXW2020} considered the  boundary layer problem of a Keller-Segel model in a domain of two dimensional space with vanishing chemical diffusion coefficient:
\begin{equation}\label{eq:KS}
\left\{
    \begin{aligned} 
        &\partial_t n^{\varepsilon}  - \Delta n^{\varepsilon} = -\nabla \cdot (n^{\varepsilon} \nabla c^{\varepsilon}), \\
        &\partial_t c^{\varepsilon} - \varepsilon^2\Delta c^{\varepsilon} = -c^{\varepsilon}n^{\varepsilon}, \\
        &\partial_y c^{\varepsilon}=\partial_y n^{\varepsilon}=0, \quad {\rm on}~ y=0
    \end{aligned}
    \right.
\end{equation}
and they obtained the error equations $(n,c)=(n^{\varepsilon}-n^a, c^{\varepsilon}-c^a)$ satisfying
\ben\label{eq:MXW}
\|n\|_{L^\infty(\Omega\times (0,T))}\lesssim \varepsilon^{\frac32-\delta},\quad \|c\|_{L^\infty(\Omega\times (0,T))}\lesssim \varepsilon^{\frac12-\delta},
\een
where $\delta>0.$
% Meng, Linlin; Xu, Wen-Qing; Wang, Shu Boundary layer analysis for a 2-D Keller-Segel model. Open Math. 18 (2020), no. 1, 1895–1914.
Moreover, Hou-Wang \cite{houwang2019} established the convergence of boundary
layer solutions of the Keller-Segel model with singular sensitivity:
\begin{equation}\label{eq:KS}
\left\{
    \begin{aligned} 
        &
\partial_t n=\nabla\cdot(D \nabla n-\chi \frac{n}{c}\nabla c)\\
& \partial_t c=\varepsilon \Delta c-nc,
\end{aligned}
    \right.
\end{equation}
and by assuming the Dirichlet-Neumann boundary condiotion they justify that the boundary layer converges to
the outer layer (solution with $\varepsilon$= 0) plus the inner layer with the convergence rate of $(n,c)$ as $(\varepsilon^\frac12, \varepsilon^\frac14)$.
 % Hou, Qianqian; Wang, Zhian; Convergence of boundary layers for the Keller-Segel system with singular sensitivity in the half-plane. J. Math. Pures Appl. (9) 130 (2019), 251–287. 
Lee-Wang-Yang \cite{leewang2020} also considered  the stationary
problem of Keller-Segel model, which is reduced to a singularly perturbed nonlocal semi-linear elliptic
problem.
Carrillo-Hong-Wang \cite{chw2024}
 show that the model with physical boundary conditions proposed
by Tuval et al. in one dimension can generate a boundary layer solution as the oxygen diffusion rate $\varepsilon$ is small. Specifically, we show that the solution of the model with
$\varepsilon>0$ will converge to the solution with $\varepsilon=0$ (outer-layer solution) plus the boundary
layer profiles (inner-layer solution) with a sharp transition near the boundary as $\varepsilon\rightarrow 0$.
% Carrillo, José A.; Hong, Guangyi; Wang, Zhi-an Convergence of boundary layers of chemotaxis models with physical boundary conditions I: Degenerate initial data. SIAM J. Math. Anal. 56 (2024), no. 6, 7576–7643. 
 For more developments on general domains, we refer to  \cite{PRZ2012,PWZZ2018,PWZ2014,lwmy2025} and the references therein.
%  Lee-Wang-Moon-Yang \cite{lwmy2025} consider the stationary problem of the Keller-Segel system with physical boundary conditions describing the boundary-layer formation driven by chemotaxis. By using Fermi coordinates  the authors rigorously
% derive the asymptotic expansion of the boundary-layer profile and thickness in terms of the small diffusion rate with coefficients explicitly expressed by the domain’s geometric properties.

% Lee, Chiun-Chang; Moon, Sang-Hyuck; Wang, Zhi-An; Yang, Wen Geometry effects on the boundary-layer profiles of the Keller-Segel system. Trans. Amer. Math. Soc. 378 (2025), no. 12, 8871–8907.

% {Chemotaxis-fluid coupled systems}

% As stated in \cite{Tuval2005,dombrowski2004}, it is a more realistic scenario that chemotactic processes take place in a moving fluid. It's natural to consider the boundary layer effect of the chemotaxis-fluid coupled systems to explain the turbulent phenomenon of locally high concentrations of cells.
\subsubsection{The case of chemotaxis equations with fluid coupling} 
Only a few studies have explored boundary layer problems in chemotaxis-fluid systems. In 2024, Hou \cite{Hou2024} considered the vanishing diffusion limit issue for the chemotaxis–Navier–Stokes system in $\mathbb{R}^3$ with  the vanishing chemical diffusion rate.
% Hou, Qianqian Chemical diffusion limit of a chemotaxis-Navier-Stokes system. Appl. Anal. 103 (2024), no. 1, 198–210.
% considered the  chemotaxis–Navier–Stokes system \cite{Hou2022} analyzed the asymptotic behavior of the chemotaxis-Navier-Stokes system under Robin boundary conditions by constructing boundary layer correctors.
% In 2024, Wang et al. \cite{WWX2024} studied smooth solutions of the 3D chemotaxis-Stokes system with signal Dirichlet boundary conditions. 
Hou \cite{Hou2025} further demonstrated that the boundary layer effect in the chemotaxis-Navier-Stokes system as follows: 
\begin{equation}\label{eq:HOU}
    \left\{
\begin{aligned}
& n_t + \vec{u} \cdot \nabla n + \nabla \cdot (n \nabla c) = \Delta n, && (x,y,t) \in \mathbb{R}^2_+ \times (0,T), \\
& c_t + \vec{u} \cdot \nabla c + n c = \varepsilon \Delta c, && (x,y,t) \in \mathbb{R}^2_+ \times (0,T), \\
& \vec{u}_t + \vec{u} \cdot \nabla \vec{u} + \nabla p + (0,\lambda)n = \Delta \vec{u}, && (x,y,t) \in \mathbb{R}^2_+ \times (0,T), \\
& \nabla \cdot \vec{u} = 0, && (x,y,t) \in \mathbb{R}^2_+ \times (0,T), 
% & (m,c,\vec{u})(x,y,0) = (m_0,c_0,\vec{u}_0)(x,y), && (x,y) \in \mathbb{R}^2_+
\end{aligned}
\right.
\end{equation}
with the boundary conditions
\begin{equation}\label{eq:HOU-B}
\left\{
\begin{aligned}
& (\nabla n - n \nabla c) \cdot \vec{n} = 0, \quad \nabla c \cdot \vec{n} = 0, \quad \vec{u} = \mathbf{0} && \text{if } \varepsilon > 0, \\
& (\nabla n - n \nabla c) \cdot \vec{n} = 0, \quad \vec{u} = \mathbf{0} && \text{if } \varepsilon = 0,
\end{aligned}
\right.
\end{equation}
where the author obtained the convergence rate is
\ben\label{eq:HOU-RATE}
\left\| \left( n^\varepsilon - n^0, c^\varepsilon - c^0 \right) \right\|_{L^\infty([0,T]; L^\infty_{xy})} \leq C \varepsilon^{\frac{1}{2}},
\quad
\left\| \left( \vec{u}^\varepsilon - \vec{u}^0 \right) \right\|_{L^\infty([0,T]; L^\infty_{xy})} \leq C \varepsilon,
\een
which $(n^0,c^0,\vec{u}^0)$ is the solution of \eqref{eq:HOU} with $\varepsilon=0$.
% is inherently driven by the fluid, as it vanishes when the fluid component is removed.
%  some details???
Despite these advances, the vanishing viscosity limit for chemotaxis-fluid systems remains an open challenge. Biologically, chemotaxis is inseparable from the surrounding fluid dynamics; thus, the formation of the fluid boundary layer inevitably dictates the behavior of the chemotactic boundary layer. As both fluid and chemotactic viscosities vanish, the strong coupling of these boundary layers creates a highly singular region. A fundamental question naturally emerges: \\
{\it is it mathematically possible to establish a rigorous boundary layer expansion in this regime, thereby offering a mathematical answer to the complex phenomena reported in Dombrowski et. al.'s experiments in \cite{dombrowski2004}?}
\vspace{-10pt}
\subsection{Boundary layers in the  chemotaxis system}
% Motivated by the above observations, this paper aims to systematically investigate the vanishing viscosity limit and associated boundary layer behavior for chemotaxis-fluid equations on the two-dimensional half-plane $\mathbb{R}_+^2$. Specifically, we seek to address the following fundamental questions: When the fluid viscosity tends to zero, how do the fluid dynamics and chemotactic mechanisms jointly influence the formation of boundary layers? Is there an interactive effect between these two components? These questions remain mathematically unexplored. 
% Although the experimental work of Tuval et al.~\cite{Tuval2005} has demonstrated the physical existence of such phenomena, the underlying mathematical mechanisms urgently require rigorous investigation. The present work represents a first step toward filling this gap.
Motivated by these physical and mathematical challenges, this paper aims to systematically investigate the vanishing viscosity-diffusivity limits and the associated boundary layer expansions for the chemotaxis-fluid equations on the two-dimensional half-plane $\mathbb{R}_+^2$.
Consider the model coupling the chemotaxis equations for bacterial density $n^\varepsilon$ and oxygen concentration $c^\varepsilon$ with the incompressible Navier-Stokes equations for the velocity field $u^\varepsilon$, as shown below:
\vspace{-\abovedisplayskip}
\begin{equation}\label{eq:CNS}
\left\{
    \begin{aligned}
        &\partial_t n^{\varepsilon} + u^{\varepsilon} \cdot \nabla n^{\varepsilon} - \Delta n^{\varepsilon} = -\nabla \cdot (n^{\varepsilon} \nabla c^{\varepsilon}), \\
        &\partial_t c^{\varepsilon} + u^{\varepsilon} \cdot \nabla c^{\varepsilon} - \varepsilon^2\Delta c^{\varepsilon} = -c^{\varepsilon}n^{\varepsilon}, \\
        &\partial_t u^{\varepsilon} + u^{\varepsilon} \cdot \nabla u^{\varepsilon} - \varepsilon^2\Delta u^{\varepsilon} + \nabla p^{\varepsilon} = n^{\varepsilon} \nabla \phi, \\
        &\nabla \cdot u^{\varepsilon} = 0, \\
        &\partial_y c^{\varepsilon}=\partial_y n^{\varepsilon}=0, ~\partial_y u^{\varepsilon}_1=u^{\varepsilon}_2=0,\quad y=0,\\
        &(n^{\varepsilon},c^{\varepsilon},u^{\varepsilon})=(n_{in},c_{in},u_{in}),\quad t=0.
    \end{aligned}
    \right.
\end{equation}
Here, we assume that $\phi = y$ without loss of generality, then $\nabla \phi = (0,1)^\top = e_2$, and the system is supplemented with physically reasonable Neumann-slip boundary conditions~\cite{WL2024} and nonnegative initial data $n_{in},c_{in}\ge 0$.

Formally, as the viscosity coefficient $\varepsilon \to 0$, the solution of \eqref{eq:CNS} converges to the following inviscid limit system on $\mathbb{R}_+^2$:
\begin{equation}\label{eq:epsilon=0}
    \left\{
    \begin{aligned}
        &\partial_t n^0 + u^0 \cdot \nabla n^0 - \Delta n^0 = -\nabla \cdot (n^0 \nabla c^0), \\
        &\partial_t c^0 + u^0 \cdot \nabla c^0 = -c^0 n^0, \\
        &\partial_t u^0 + u^0 \cdot \nabla u^0 + \nabla p^0 = n^0 e_2, \\
        &\nabla \cdot u^0 = 0, \\
        &\partial_y n^0 = n^0 \partial_y c^0, \quad u_2^0 = 0, \quad y=0,\\
        &(u^0,c^0,n^0) = (n_{in},c_{in},u_{in}), \quad t=0.
    \end{aligned}
    \right.
\end{equation}
However, due to the mismatch of normal velocity boundary conditions between the viscous and inviscid systems, boundary layers may appear.

This research will follow the paradigm of matched asymptotic expansion method. First, we introduce boundary layer coordinates $z = y/\varepsilon$ and decompose the solution into an outer solution describing the flow far from the boundary and an internal correction term describing boundary layer behavior. Through asymptotic expansion of the original equations, we derive the equations satisfied by each order of outer solutions and internal correction terms, and construct approximate solutions $(n^a,u^a,c^a,p^a)$. The core of subsequent work will be to prove that the error between this approximate solution and the true solution of the original equation tends to zero as $\varepsilon \to 0$. To this end, we plan to introduce an anisotropic Sobolev space that can assign different weights to tangential ($x$) and normal ($y$) derivatives, thereby finely capturing the characteristics of drastic normal variations within the boundary layer\cite{Hou2022,WL2024}. By establishing energy inequalities for error variables and using Gronwall-type lemmas, we hope to obtain energy estimates, thereby strictly quantifying convergence rates. This energy estimation method draws on the techniques of Masmoudi and Rousset\cite{MR2012} in handling Navier slip boundary conditions, as well as the energy construction ideas of Wang, Wang and Zhang\cite{WWZ2017} in the analytic framework, and innovatively applies them to complex systems coupled with nonlinear chemotaxis terms. Ultimately, this research will for the first time rigorously verify from a mathematical perspective the boundary layer asymptotic structure of a class of chemotaxis-fluid equations under the vanishing viscosity limit, revealing how fluid viscosity and biological chemotaxis jointly act to cause the fine distribution of bacterial density and oxygen concentration near the boundary.

The rest of the paper is organized as follows. In Section~\ref{main results}, we introduce the functional framework, formulate the error equations, and state our main results. Section~\ref{key energy estimates} is devoted to stating the key \textit{a priori} estimates and providing the proof of the main theorem. Subsequently, in Section~\ref{approximate solution}, we detail the derivation of the asymptotic expansions and rigorously construct the approximate solutions. In Section~\ref{energy estimates}, we establish the necessary supporting lemmas and prove the \textit{a priori} estimates presented in Section~\ref{key energy estimates}. Finally, the well-posedness proofs for the inner and outer layer solutions are deferred to the Appendix.

% The structure of this paper is arranged as follows: First, in Section~\ref{main results}, we introduce the error equations, function spaces, and main results; second, in Section~\ref{key energy estimates}, we introduce the main a priori estimate results and the proof of the main theorem; subsequently, in Section~\ref{approximate solution}, we derive the asymptotic expansion in detail and obtain the construction of approximate solutions; finally, in Section~\ref{energy estimates}, we obtain the lemmas required for energy estimates and prove the a priori estimate results listed in Section~\ref{key energy estimates}. The proof of well-posedness of inner and outer solutions will be placed in the Appendix.

In this paper, to simplify notation, we let $h(\cdot)$ represent an undetermined function. Then the boundary operator is expressed as $\overline{h} (t,x) = h(t,x,0)$, and the shift operator is expressed as:
\begin{equation}
    \binom{h_1(\cdot,j)}{h_2(\cdot,j)}_p = \binom{h_1(\cdot,j)}{h_2(\cdot,j+1)},~\binom{h_1(\cdot,j)}{h_2(\cdot,j)}_{pp} = \binom{h_1(\cdot,j)}{h_2(\cdot,j+2)},
\end{equation}
and when $j < 0$, $h(\cdot,j) = 0$. Throughout this paper, we denote by $C$ a constant independent of $\varepsilon$ and by $C_\xi$ a coefficient dependent only on $\xi$, which may be different from line to line.

\section{Error equations, functional spaces and main results}\label{main results}

%The error equations
\subsection{The error equations}

Assume the approximate solution $(n^a,c^a,u^a,p^a)$ satisfies:
\begin{equation}
    \left\{
    \begin{aligned}\label{eq:approximate}
        &\partial_t n^a + u^a \cdot \nabla n^a - \nabla \cdot \left( \nabla n^a - n^a \nabla c^a \right) = -N, \\
        &\partial_t c^a + u^a \cdot \nabla c^a + c^a n^a - \varepsilon^2 \Delta c^a = -K, \\
        &\partial_t u^a + u^a \cdot \nabla u^a + \nabla p^a - \varepsilon^2 \Delta u^a - n^ae_2 = -U, \\
        &\nabla \cdot u^a = 0, \\
        &\partial_y n^a = \partial_y c^a = \partial_y u_1^a = 0, u_2^a = \varepsilon^2 f(t, x) , \quad  y=0,
    \end{aligned}
    \right.
\end{equation}
where $(N,C,U)$ are remainders which are small in some functional space.

We introduce the error between the solution and the approximate solution:
 \begin{equation*}
    n =n^\varepsilon - n^a, c = c^\varepsilon - c^a, u = u^\varepsilon - u^a, p = p^\varepsilon -p^a.
\end{equation*}
Then from \eqref{eq:CNS} and \eqref{eq:approximate}, we know that $(n,c,u,p)$ satisfies:
\begin{equation}
    \left\{
    \begin{aligned}\label{eq:error}
        &\partial_t n + \left( u^a \cdot \nabla n + u \cdot \nabla n^a + u \cdot \nabla n \right)  + \nabla \cdot \left( n \nabla c + n^a \nabla c + n \nabla c^a \right) \\
        &\quad - \Delta n= N, \\
        &\partial_t c + \left( u^a \cdot \nabla c + u \cdot \nabla c^a + u \cdot \nabla c \right) - \varepsilon^2 \Delta c + \left( c^a n + cn^a + cn \right) = K, \\
        &\partial_t u + \left( u^a \cdot \nabla u + u \cdot \nabla u^a + u \cdot \nabla u \right) + \nabla p - \varepsilon^2 \Delta u - ne_2 = U, \\
        &\nabla \cdot u = 0, \\
        &\partial_y n = \partial_y c = \partial_y u_1 = 0, ~u_2 = -\varepsilon^2 f, \quad  y=0.
    \end{aligned}
    \right.
\end{equation}

Taking the divergence operator on both sides of \eqref{eq:error}$_3$, we obtain the pressure equation:
\begin{equation}
    \left\{
    \begin{aligned}\label{eq:p}
        &-\Delta p = \nabla \cdot \left( (u^a + u) \cdot \nabla u \right) + \nabla \cdot \left( u \cdot \nabla u^a \right) - \partial_y n - \nabla \cdot U, \\
        &\partial_y p = U_2 - \partial_t u_2 - (u^a + u) \cdot \nabla u_2 - u \cdot \nabla u_2^a + \varepsilon^2 \Delta u_2 + n, , \quad  y=0.
    \end{aligned}
    \right.
\end{equation}
The vorticity $\omega, \omega^a$ is defined by:
\begin{equation}\label{def:omega}
    \omega = \partial_y u_1 - \partial_x u_2,~\omega^a = \partial_y u_1^a -\partial_x u_2^a.
\end{equation}
Then from \eqref{eq:error}$_3$ we can deduce that:
\begin{equation}
    \left\{
    \begin{aligned}\label{eq:omega}
        &\partial_t \omega + \left(u^a \cdot \nabla \omega + u \cdot \nabla \omega^a + u \cdot \nabla \omega \right) - \varepsilon^2 \Delta \omega + \partial_x n = -\nabla^\perp \cdot U, \\
        &\omega = \partial_y u_1 - \partial_x u_2 = \varepsilon^2 \partial_x f, \quad  y=0.
    \end{aligned}
    \right.
\end{equation}

%Differential operator and functional spaces
\subsection{Differential operator and functional spaces}\label{section_operator}

As in \cite{MR2012,MW2015,TWZ2020}, we introduce the conormal derivative:
\begin{equation}\label{def:derivative}
    \partial^{\alpha} = \partial^{\alpha_1 }_t \partial_x^{\alpha_2} \psi(y)^{\alpha_3} \partial_y^{\alpha_3},
\end{equation}
where $\alpha = (\alpha_1,\alpha_3,\alpha_3)$ is a multi-index and $\psi(y)$ is a smooth function defined by:
\ben\label{eq:psi-def}
\psi(y) = 
\begin{cases} 
\delta y, & \text{for } y \leq \frac 1 2, \\
\frac{\delta y}{1 + y}, & \text{for } y \geq 1,
\end{cases}
\een
where $\delta > 0$, to be decided later. 

Here we say that $\beta \leq \alpha$ in $\mathbb{N}^3$ if $\alpha =(\alpha_1 ,\alpha_3 ,\alpha_3), \beta =(\beta_1, \beta_2, \beta_3)$ satisfy $\beta_1 \leq \alpha_1, \beta_2 \leq \alpha_2, \beta_3 \leq \alpha_3$.
In addition, there is also the gradient operator:
\begin{equation*}
    \nabla = \binom{\partial_x}{\partial_y}, \quad \widehat{\nabla} = \binom{\partial_x}{\partial_z}.
\end{equation*}
where $z = \frac y \varepsilon$. 
Define the commutator of two operators $\mathcal{A}$ and $\mathcal{B}$ as $[\mathcal{A},\mathcal{B}] = \mathcal{A}\mathcal{B} - \mathcal{B}\mathcal{A}$. 
Then we have 
\begin{equation}\label{exchange oper}
    [\partial_y, \partial^{\alpha}] = \alpha_3 \psi' \partial^{\alpha-(0,0,1)} \partial_y.
\end{equation}
(If $\alpha_3 = 0$, we denote $[\partial_y, \partial^{\alpha}] = 0$.)

Denote $\left\| g \right\|_q = \| g \|_{L^q(\mathbb{R}^2_+)}$. For $l,m\in \mathbb{N}$, the conormal Sobolev spaces are defined as follows:

\begin{equation}\label{def:conormal sobolev}
    \begin{aligned}
        &Y^{l,m} \left(\mathbb{R}_+^2 \right) = \left\{ u \;\Bigg|\; \|u \|_{Y^{l,m}}^2 = \sum_{\alpha_1 \leq l,|\alpha|\leq m} \| \partial^{\alpha} u \|_2^2 < \infty \right\},\\
        &Y^{l,m}_\infty \left(\mathbb{R}_+^2 \right) = \left\{ u \;\Bigg|\; \|u \|_{Y_\infty^{l,m}}^2 = \sum_{\alpha_1 \leq l,|\alpha|\leq m} \| \partial^{\alpha} u \|_\infty^2 < \infty \right\},\\
        &Y^{l} \left(\mathbb{R}_+^2 \right) = \left\{ u \;\Bigg|\; \|u \|_{Y^{l}}^2 = \sum_{\tau \leq l} \| \partial^{\tau}_t u \|_2^2 < \infty \right\}. 
    \end{aligned}
\end{equation}
Similarly,  denote the general Sobolev spaces by
\begin{equation}\label{def:sobolev}
    \begin{aligned}
        &\overline{H}^{l,m} \left(\mathbb{R}_+^2 \right) = \left\{ u \;\Bigg|\; \| u \|_{\overline{H}^{l,m}}^2 = \sum_{\alpha_1 \leq l, |\alpha| \leq m} \| \partial^{\alpha_1}_t \partial_x^{\alpha_2} \partial_y^{\alpha_3} u\|_2^2 < \infty \right\}, \\
        &H^{m} \left(\mathbb{R}_+^2 \right) = \left\{ u \;\Bigg|\; \| u \|_{H^{m}}^2 = \sum_{|\alpha| \leq m} \| \partial_x^{\alpha_2} \partial_y^{\alpha_3} u\|_2^2 < \infty \right\}.
    \end{aligned}
\end{equation}

%Main results
\subsection{Main results}

 Let us first introduce a proposition on approximate solutions and remainders:
\begin{proposition}\label{prop:uniform bounds}
    Assuming $(n_{in},c_{in},u_{in}) \in H^{9}$ satisfy the compatibility conditions $(\mathcal{I}^{3})$ (see Section \ref{sec:loocal well-posed} for details), 
    and $(n^a,c^a,u^a)$ takes the form
        \begin{equation}\label{def:soluntion^a}
        \left\{
        \begin{aligned}
            n^a(t,x,y) &= n^{e,0}(t,x,y) + \varepsilon n^{b,1}(t,x,z) + \varepsilon^2 n^{b,2}(t,x,z), \\
            c^a(t,x,y) &= c^{e,0}(t,x,y) + \varepsilon c^{b,1}(t,x,z) + \varepsilon^2 c^{b,2}(t,x,z), \\
            u_1^a(t,x,y) &= u_1^{e,0}(t,x,y) + \varepsilon u_1^{b,1}(t,x,z), \\
            u_2^a(t,x,y) &= u_2^{e,0}(t,x,y) + \varepsilon^2 u_2^{b,2}(t,x,z), \\
            p^a(t,x,y) &= p^{e,0}(t,x,y) + \varepsilon^2 p^{b,2}(t,x,z),
        \end{aligned}
        \right.
    \end{equation}
then it follows that:
\begin{itemize}
\item
    \textbf{Uniform bounds of approximate solutions:}
    There exist $T_a > 0$ such that for any $t \in [0, T_a]$
    \begin{align}\label{approximate uniform bound}
        &\|n^a\|_{L_t^\infty Y^{1,3}_\infty} + \|c^a\|_{L_t^\infty Y^{1,4}_\infty} + \|u^a\|_{L_t^\infty Y^{1,4}_\infty} \leq C,  \nonumber\\
        &\|\partial_y n^a\|_{L_t^\infty Y^{1,2}_\infty} + \|\partial_y c^a\|_{L_t^\infty Y^{1,3}_\infty} + \|\partial_y u^a\|_{L_t^\infty Y^{1,3}_\infty} \leq C,  \\
        &\|\partial_y^2 c^a\|_{L_t^2 Y^{1,2}} + \|\partial_y^2 u^a\|_{L_t^2 Y^{1,2}} \leq C \varepsilon^{-\frac12}\nonumber.
    \end{align}
    
    \item 
    \textbf{Uniform bounds of the remainders:}
    For all $t \in [0, T_a]$, it holds that
    \begin{align}\label{remainder uniform bound}
        &\| N \|_{L_t^2 Y^{1,2}} 
        + \| K \|_{L_t^2 Y^1}^2 + \| \nabla K \|_{L_t^2 Y^{1,2}} 
        + \| U \|_{L_t^2 Y^1}^2 + \| \nabla^\perp\cdot U \|_{L_t^2 Y^{1,2}} \leq C\varepsilon^\frac{3}{2}.
    \end{align}
    % where $(j=1,2)$
    % \begin{align*}
    %     C_4(t)=&C\Bigl(\|n^{b,j}\|_{Z^{2,4}_1},\|\partial_z n^{b,j}\|_{Z^{1,4}_1},\|c^{b,j}\|_{Z^{1,5}_1}，\|\partial_z c^{b,j}\|_{Z^{1,4}_2},\|u_1^{b,1}\|_{Z^{2,6}_3}，\|\partial_z u_1^{b,1}\|_{Z^{1,4}_1},\\
    %     &\|u_1^{b,2}\|_{Z^{1,3}_0}，\|\partial_z u_1^{b,2}\|_{Z^{1,2}_0},\|n^{e,0}\|_{\overline{H}^{1,6}},\|c^{e,0}\|_{\overline{H}^{1,7}},\|u^{e,0}\|_{\overline{H}^{1,7}}\Bigr)\\
    %     +&\Big(\|n^{e,0}\|^2_{\overline{H}^{1,6}}+\|u_2^{e,0}\|^2_{\overline{H}^{1,6}} + \|c^{b,2}\|^2_{Z^{1,3}_0} + \|\partial_z c^{b,2}\|^2_{Z^{1,2}_0}  + \|n^{b,2}\|^2_{Z^{1,3}_0} + \|\partial_z n^{b,2}\|^2_{Z^{1,3}_0} \\
    %     &\quad + \|n^{b,1}\|^2_{Z^{1,3}_0} + \|\partial_z n^{b,1}\|^2_{Z^{1,3}_0}+ \|u_1^{b,1}\|^2_{Z^{1,4}_2} \Big) \\
    %     \times& \Big(\|\partial_z^2 c^{b,1}\|^2_{Z^{1,3}_2} + \|\partial_z^2 c^{b,2}\|^2_{Z^{1,3}_2} + \|\partial_z^2 n^{b,1}\|^2_{Z^{1,3}_0} + \|\partial_z^2 n^{b,2}\|^2_{Z^{1,3}_0} +\|\partial_z^2 u_1^{b,1}\|^2_{Z^{1,3}_2}\Big).
        % +&\Big(\|n^{e,0}\|^2_{\overline{H}^{1,m+4}}+\|u_2^{e,0}\|^2_{\overline{H}^{1,m+4}} + \|c^{b,2}\|^2_{Z^{1,m+1}_0} + \|\partial_z c^{b,2}\|^2_{Z^{1,m}_0}  + \|n^{b,2}\|^2_{Z^{1,m+1}_0} + \|\partial_z n^{b,2}\|^2_{Z^{1,m+1}_0} \\
        % &\quad + \|n^{b,1}\|^2_{Z^{1,m+1}_0} + \|\partial_z n^{b,1}\|^2_{Z^{1,m+1}_0}+ \|u_1^{b,1}\|^2_{Z^{1,m+2}_2} \Big) \\
        % \times& \Big(\|\partial_z^2 c^{b,1}\|^2_{Z^{1,m+1}_2} + \|\partial_z^2 c^{b,2}\|^2_{Z^{1,m+1}_2} + \|\partial_z^2 n^{b,1}\|^2_{Z^{1,m+1}_0} + \|\partial_z^2 n^{b,2}\|^2_{Z^{1,m+1}_0} +\|\partial_z^2 u_1^{b,1}\|^2_{Z^{1,m+1}_2}\Big).
    % \end{align*}
    \item 
    \textbf{Uniform bounds of $f$:}
    There exists $F(t,x)$ such that for $m \in \mathbb{N}$, $f(t,x) = \partial_x F$ with
    \begin{align}\label{f uniform bound}
        \| f e^{-y} \|_{L_t^\infty Y^{2,{5}}} + \| Fe^{-y} \|_{L_t^\infty Y^{2,2}} \leq C.
    \end{align}
    % where 
    % \begin{align*}
    %     C_5(t)=\|u_1^{b,1}\|_{Z^{2,6}_1}.
    % \end{align*}
    \end{itemize}
\end{proposition}

\begin{proof}
    The proof can be found in the Appendix (see Subsection \ref{proof:uniform bounds}).
\end{proof}
% The proof of \eqref{approximate uniform bound} and \eqref{remainder uniform bound} can be derived by \eqref{def:soluntion^a} and Proposition \ref{NKU} respectively, so we omit it. The proof of \eqref{f uniform bound}, it follows that

Our main result starts as follows.
\begin{theorem}\label{thm:main}
    Assume Proposition~\ref{prop:uniform bounds} holds. Let $\omega$ be defined as in \eqref{def:omega}. Then there exist $\varepsilon_0 > 0$, $T > 0$, and $C > 0$ independent of $\varepsilon$ such that for any $\varepsilon \in (0,\varepsilon_0)$, $t \in [0,T]$, the error equation \eqref{eq:error} admits a unique solution $(n,c,u)(t,\cdot) \in Y^{1,2}$, which satisfies
    \begin{align*}
        &\| n \|_{L^\infty_t Y^{1,2}} 
        + \| \nabla n \|_{L^2_t Y^{1,2}}
        \leq C \varepsilon^{\frac 3 2}, \\
        &\| c \|_{L^\infty_t Y^{1,2}} 
        + \| \nabla c \|_{L^\infty_t Y^{1,2}}
        + \| \nabla^2 c \|_{L^2_t Y^{1,2}}
        \leq C \varepsilon^{\frac 3 2}, \\
        &\| u \|_{L^\infty_t Y^{1,2}} 
        + \| \omega \|_{L^\infty_t Y^{1,2}}
        + \| \nabla \omega \|_{L^2_t Y^{1,2}}
        \leq C \varepsilon^{\frac 3 2}. 
    \end{align*}
    In particular, we have 
    \begin{equation*}
        \sup_{0 \leq t \leq T} \| (n, c, u)(t,\cdot)\|_{L^2 \cap L^\infty} \leq C\varepsilon^{\frac 3 2}.
    \end{equation*}
\end{theorem}

In Section 3, 
we construct the approximate solution $(n^a, c^a, u^a ,p^a)$ of \eqref{eq:CNS} by using the matched asymptotic expansion, which satisfies Proposition~\ref{prop:uniform bounds}. Thus, we have the  following conclusion.
\begin{corollary}
    Let Proposition~\ref{prop:uniform bounds} holds, then there exist $\varepsilon_0 > 0$, $T > 0$, and $C > 0$ independent of $\varepsilon$ such that for any $\varepsilon \in (0,\varepsilon_0)$, the chemotaxis-Navier-Stokes system \eqref{eq:CNS} admit a unique solution $(n^\varepsilon, c^\varepsilon, u^\varepsilon)$ on $[0,T]$, which satisfies
    \begin{align*}
        &\| n^\varepsilon - n^a \|_{L^\infty_t Y^{1,2}} 
        + \| \nabla (n^\varepsilon - n^a) \|_{L^2_t Y^{1,2}}
        \leq C \varepsilon^{\frac 3 2}, \\
        &\| c^\varepsilon - c^a \|_{L^\infty_t Y^{1,2}} 
        + \| \nabla (c^\varepsilon - c^a) \|_{L^\infty_t Y^{1,2}}
        + \| \nabla^2 (c^\varepsilon - c^a) \|_{L^2_t Y^{1,2}}
        \leq C \varepsilon^{\frac 3 2}, \\
        &\| u^\varepsilon - u^a \|_{L^\infty_t Y^{1,2}} 
        + \| \omega^\varepsilon - \omega^a \|_{L^\infty_t Y^{1,2}}
        + \| \nabla (\omega^\varepsilon - \omega^a) \|_{L^2_t Y^{1,2}}
        \leq C \varepsilon^{\frac 3 2}.
    \end{align*}
    In particular, we have 
    \begin{equation*}
        \sup_{0 \leq t \leq T} \| (n^\varepsilon, c^\varepsilon, u^\varepsilon)(t,\cdot) - (n^a, c^a, u^a)(t,\cdot) \|_{L^2 \cap L^\infty} \leq C\varepsilon^{\frac 3 2}.
    \end{equation*}
\end{corollary}

\begin{remark}[Difficulties in  the construction of boundary layer corrections ]\label{rmk:difficult}
 The above theorem implies the $L^\infty$ strong convergence from the system \eqref{eq:CNS} to a constructed solution, which seems to be the first result of vanishing limit for the  chemotaxis-Navier-Stokes system with chemotactic viscosity and fluid viscosity.
    One main difficulty lies in the construction of boundary layer corrections \(n^{b,j}\): a direct attempt to solve the equation for \(\partial_t n^{b,j}\) is difficult, as it contains the higher-order derivative \(\partial_z^2 n^{b,j+2}\). This creates a circular dependency where the equation for a low-order correction explicitly involves an unknown higher-order term.
    To overcome this obstacle, we instead exploit the equation of the lower-order term \(\partial_t n^{b,j-2}\). A key observation is that this equation induces a degenerate equation for \(n^{b,j}\): the time derivative of \(n^{b,j}\) no longer appears, and the equation reduces to a linear relation involving only \(\partial_z^2 n^{b,j}\) and known lower-order quantities. 
    % This reformulation avoids the explicit appearance of unknown higher-order terms, enabling a systematic order-by-order solution of the boundary layer corrections.
    In details, 
    % for the $n^{b,1}$ term, which is complex and coupled with the unknown function $n^{b,3}$ term:
    % \beno
    % &&\partial_t n^{b,1} + \sum_{k=0}^2 \left( (u^{E,1-k})_p \cdot \widehat{\nabla} n^{b,k} + (u^{b,1-k})_p \cdot \widehat{\nabla} n^{E,k} + (u^{b,1-k})_p \cdot \widehat{\nabla} n^{b,k} \right) \\
    %             &&\quad - \widehat{\nabla} \cdot \left( \widehat{\nabla} n^{b,1} - \sum_{k=0}^1 \left( n^{E,1-k} \widehat{\nabla} c^{b,k} + n^{b,1-k} \widehat{\nabla} c^{E,k} + n^{b,1-k} \widehat{\nabla} c^{b,k} \right) \right)_{pp} = 0,
    % \eeno
    % (see \eqref{eq:b1}$_1$).
    in order to estimate $n^{b,1}$, we identify the useful structure:
    % \beno
    % &&u_2^{E,0} \partial_z n^{b,0}  + u_2^{b,0} \partial_z n^{b,0}  \\
    %             && \quad-\partial_z \left( \partial_z n^{b,1} - \sum_{k=0}^1 \left( n^{E,1-k} \partial_z c^{b,k} + n^{b,1-k} \partial_z c^{E,k} + n^{b,1-k} \partial_z c^{b,k} \right) \right) = 0,
    % \eeno
    % (see \eqref{eq:b-1}$_1$). By using $u_2^{b,0}=u_2^{E,0}=p^{b,0}=0$ and integrating with respect to $z$ from $z$ to $\infty$, we can derive the following expression:
    \beno
    \partial_z n^{b,1} - \sum_{k=0}^1 \left( n^{E,1-k} \partial_z c^{b,k} + n^{b,1-k} \partial_z c^{E,k} + n^{b,1-k} \partial_z c^{b,k} \right) = 0,
    \eeno
    (see \eqref{eq:partial_z n^b1}). Notably, all terms in this equation are of order at most \(1\). As a result, \(n^{b,1}\) can be solved by combining the equation for \(c^{b,1}\) with other known lower-order terms. The only additional requirement is the well-posedness of the system of \( (n^{e,0}, c^{e,0}, u^{e,0}) \), whose boundary values can also be derived from this structure (see \eqref{bound:n^e0}).
    Similarly, for $n^{b,2}$ we use the following relation:
    \beno
     \partial_z n^{b,2} = \left( n^{E,1} \partial_z c^{b,1} + n^{b,1} \partial_z c^{E,1} + n^{b,1} \partial_z c^{b,1} \right) + n^{E,0} \partial_z c^{b,2}
    \eeno
    (see \eqref{eq:partial_z n^b2}).  Hence, the solving of $n^{b,2}$ can be obtained by using the equations of $c^{b,2}$ and $(n^{e,1}, c^{e,1}, u^{e,1})$, whose boundary condition is derived as in \eqref{bound:n^e1}. 
For the detailed constructing of  the approximate solution and the derivation process, we refer to Figure \ref{fig:1} as follows.

\begin{figure}[H]
\centering

% 你的流程图
\begin{tikzpicture}[
    box/.style={rectangle, draw, minimum width=5cm, minimum height=1cm, align=center},
    >=stealth
]
% 纵向间距缩小，横向x坐标不变
% 第一行
\node[box] (A1) at (0,0)  {$(u_2^{b,0}, p^{b,0})$};
\node[box] (A2) at (7,0)  {$\color{red}\partial_y n^{e,0} - n^{e,0}\partial_y c^{e,0}\big|_{y=0} = 0$};
% 第二行（缩短竖向距离）
\node[box] (B1) at (0,-1.2) {$(c^{b,0}, n^{b,0}, u_1^{b,0})$};
\node[box] (B2) at (7,-1.2) {$(n^{e,0}, c^{e,0}, u^{e,0})$};
% 第三行
\node[box] (C1) at (0,-2.4) {$(u_2^{b,1}, p^{b,1})$};
\node[box] (C2) at (7,-2.4) {$\color{red}\partial_y n^{e,1} - \sum\limits_{k=0}^1 n^{e,1-k}\partial_y c^{e,k}\big|_{y=0} = 0$};
% 第四行
\node[box] (D1) at (0,-3.6) {$(n^{b,1}, c^{b,1}, u_1^{b,1})$};
\node[box] (D2) at (7,-3.6) {$(n^{e,1}, c^{e,1}, u^{e,1})$};
% 第五行
\node[box] (E1) at (0,-4.8) {$(u_2^{b,2}, p^{b,2})$};
% 第六行
\node[box] (F1) at (0,-6.0){$(n^{b,2}, c^{b,2}, u_1^{b,2})$};

% 箭头
\draw[->] (A1.east) -- (A2.west);
\draw[->] (A2.south) -- (B2.north);
\draw[->] (B2.west) -- (B1.east);
\draw[->] (B1.south) -- (C1.north);
\draw[->] (C1.east) -- (C2.west);
\draw[->] (C2.south) -- (D2.north);
\draw[->] (D2.west) -- (D1.east);
\draw[->] (D1.south) -- (E1.north);
\draw[->] (E1.south) -- (F1.north);
\end{tikzpicture}

% 图名（自动编号 图1）
\caption{}
% 引用标签
\label{fig:1}

\end{figure}
\end{remark}

\begin{remark}[The sharp convergence rate]
A key observation is that although the approximate solution is expanded up to second-order boundary layer corrections, it yields a suboptimal convergence rate of $O(\varepsilon^{\frac 3 2})$, which is better that the rate in \eqref{eq:HOU-RATE} obtained in \cite{Hou2025}. This limitation is primarily driven by two analytical factors:
\begin{itemize}
    \item First, due to the \textit{degeneracy} of the equations governing $n^{b,j}$ discussed in Remark \ref{rmk:difficult}, truncating the expansion at $n^{b,2}$ inevitably relegates the term $\varepsilon \partial_t n^{b,1}$ to the remainder term $N$. As demonstrated in Proposition \ref{NKU}, this restricts the norm of $N$ to $O(\varepsilon^{\frac 3 2})$.
    
    \item Second, evaluating the normal derivatives $\partial_y K$ and $\partial_y U_1$ introduces an inherent loss of $\varepsilon$ scaling. Specifically, the boundary layer scaling $\partial_y = \varepsilon^{-1} \partial_z$ consumes one power of $\varepsilon$, which ultimately caps the overall convergence rate at $O(\varepsilon^{\frac 3 2})$.
\end{itemize}
Notably, if the approximate solution were only constructed up to the first-order correction $n^{b,1}$, the convergence rate would severely degrade to $O(\varepsilon^{\frac 1 2})$.
    % A key observation is that although we expand the solution up to second-order boundary layer corrections, the resulting approximate solution exhibits a convergence rate of $O(\varepsilon^{\frac 3 2})$. This non-optimal ratio is primarily determined by two key factors:

    % First, due to the \textit{degeneracy} of the equations for $n^{b,j}$ discussed in Remark 2.5, the construction of the approximate solution truncated at $n^{b,2}$ necessarily introduces the term $\varepsilon \partial_t n^{b,1}$ into the remainder term $N$. As demonstrated in Proposition \ref{NKU}, this leads to the norm of $N$ being of order $O(\varepsilon^{\frac 3 2})$.
    
    % Second, in the calculation of $\partial_y K$ and $\partial_y U_1$, we encounter a scaling issue. Recall that the normal derivative operator scales as $\partial_y = \varepsilon^{-1} \partial_z$. This inverse scaling causes the magnitude of these terms to decrease, thereby affecting the overall $O(\varepsilon^{\frac 3 2})$ convergence rate.

    % Notably, if the approximate solution is only constructed up to the first-order boundary layer correction $n^{b,1}$, the resulting convergence rate would be further reduced to $O(\varepsilon^{\frac 1 2})$.

\end{remark}

\begin{remark}[Difficulties in a priori estimates]
    We highlight three major technical difficulties encountered in our \textit{a priori} estimates and the strategies used to overcome them:
    \begin{itemize}
        \item \textbf{Treatment of \(n\) in the error equations:} Since the anisotropic conormal space \(Y^{0,m}\) does not embed into \(L_{x,y}^\infty\), controlling the \(L_{x,y}^\infty\)-norm of \(n\) requires at least one additional normal derivative, \(\partial_y\). However, directly estimating \(\|\partial_y n\|_{L^\infty_t L^2_{x,y}}\) would necessitate bounding third-order normal derivatives of boundary layer profiles (e.g., \(\partial_z^3 c^{b,1}\) and \(\partial_z^3 u^{b,1}_1\)), leading to severe boundary complications. To bypass this, we utilize Sobolev embedding in time, motivating the inclusion of time derivatives in our anisotropic conormal spaces \(Y^{l,m}\). Under this framework, the estimates are closed by simply controlling \(\|\partial_y n\|_{L_t^2 Y^{1,m}}\).
        
        \item \textbf{Outer layer boundary conditions:} The leading-order outer layer equation, derived via asymptotic expansions, satisfies a zero-flux boundary condition \eqref{eq:e0}. This is inherently weaker than the free-slip condition of the original viscous system. To address this, we rewrite the equation as 
        \begin{align*}
            \partial_y^2 n^{e,0} = \partial_t n^{e,0} + u^{e,0} \cdot \nabla n^{e,0} - \partial_x^2 n^{e,0} + \nabla \cdot (n^{e,0} \nabla c^{e,0}),    
        \end{align*}
        (see \eqref{eq:partial_y^2 n}), which allows us to recover the normal regularity of \(n^{e,0}\) via time and tangential derivatives.
        
        \item \textbf{Initial data regularity:} The higher regularity assumptions imposed on the initial data are specifically tailored to establish the convergence of the boundary layer expansions. These conditions could be relaxed if the goal were solely to prove the existence of solutions.
    \end{itemize}
\end{remark}

\begin{remark}[Origins of the boundary layer]
    Similar to the findings reported in \cite{Hou2025}, the boundary layer effect in the present system is induced by the presence of the fluid rather than the chemotactic mechanism alone. To verify this, we consider the chemotaxis-only subsystem by formally setting $u = \nabla p = \mathbf{0}$ in \eqref{eq:CNS}. From the chemical concentration equation with $\varepsilon=0$, we can deduce
    \begin{align*}
        \overline{\partial_y c^{0}}(t,x) = \overline{\partial_y c^0}(0,x) e^{-\int_0^t [\overline{n^0}(1+\overline{c^0})](\tau,x) d\tau} =0,
    \end{align*}
    thus by \eqref{eq:c^b1,n^b1}, \eqref{eq:c^b2,n^b2} we have $c^{b,1}= n^{b,1} =c^{b,2} =n^{b,2} \equiv 0$ through the uniqueness of solutions.
    In this case, the boundary layer correction terms for the chemotaxis variables vanish identically.

\end{remark}

\section{Key energy estimates and proof of main result}\label{key energy estimates}

%Key energy estimates
\subsection{Key energy estimates}
The key is to prove that the solution of error equation \eqref{eq:error} is uniformly bounded in the suitable functional spaces. Now we present some main propositions.

\begin{proposition}\label{prop:n}
    Let $m \geq 2 \in \mathbb{N}$. Then there exist $\delta  > 0$ and $\varepsilon_0 > 0$ such that for any $\varepsilon \in (0, \varepsilon_0)$, we have:
    \begin{align}
        \frac{1}{2}&\frac{d}{dt} \| n \|_{Y^{1,m}}^2 
        + (1 - C\delta - C\sigma) \| \nabla n \|_{Y^{1,m}}^2 
        - C\sigma\varepsilon^2 \| \nabla^2 c \|_{Y^{1,1}}^2 \nonumber\\
        &\leq \frac{C_\delta}{\sigma} \left( 1 + \| n^a \|_{Y^{1,m+1}_\infty}^{2} + \| u^a \|_{Y^{1,m+1}_\infty}^{2}+\| \nabla c^a \|_{Y^{1,m}_\infty}^{2} + \| f e^{-y} \|_{Y^{1,m+1}}^{2} \right)^2 \times \nonumber\\
        &\quad \left( \| n \|_{Y^{1,m}}^{2} + \| u \|_{Y^{1}}^{2} + \| \omega \|_{Y^{1,m}}^{2} + \| \nabla c \|_{Y^{1,m}}^{2} \right) 
        \\
        &\quad + \frac{C}{\sigma^3 \varepsilon^4} \left( \| u \|_{Y^{1}}^{2} + \| \omega \|_{Y^{1,m}}^{2} + \| n \|_{Y^{1,m}}^{2} + \| \nabla c \|_{Y^{1,m}}^{2} \right)^3 \nonumber\\
        &\quad + C\varepsilon^4 \| f e^{-y} \|_{Y^{1,m+1}}^{2} 
        \left( \| n^a \|_{Y^{1,m+1}_\infty}^{2} + \| \partial_y n^a \|_{Y^{1,m}}^{2} \right) 
        + C \| N \|_{Y^{1,m}}^{2}.\nonumber
    \end{align}
\end{proposition}

\begin{proposition}\label{prop:c}
    There exist $\delta  > 0$ and $\varepsilon_0 > 0$ such that for any $\varepsilon \in (0, \varepsilon_0)$, we have:
    \begin{align}
        \frac{1}{2}& \frac{d}{dt} \| c \|_{Y^{1}}^2 
        + \varepsilon^2 \| \nabla c \|_{Y^{1}}^2 
        - C\sigma \| \partial_y n \|_{Y^{1,2}}^2 \nonumber\\
        &\leq  C \left( 1 + \| c^a \|_{Y^{1}_{\infty}} 
        + \| \nabla c^a \|_{Y_{\infty}^{1}} 
        + \| n^a \|_{Y^{1}_{\infty}} 
        + \| u^a \|_{Y_{\infty}^{1}} 
        + \varepsilon^2 \| f e^{-y} \|_{Y^{1,3}} \right)\times  \\
        &\quad \left( \| n \|_{Y^{1}}^2 
        + \| c \|_{Y^{1}}^2 
        + \| \nabla c \|_{Y^{1}}^2 
        + \| u \|_{Y^{1}}^2+\| \omega \|_{Y^{1,2}}^2 \right)\nonumber \\
        &\quad+ \frac{C}{\sigma}  
        \left( \| n \|_{Y^{1,2}}^2 + \| c \|_{Y^{1}}^2  + \| \nabla c \|_{Y^{1}}^2 \right)^3 
        + C \| K \|_{Y^{1}}^2.\nonumber
    \end{align}
\end{proposition}

% Since $Y^{l,m}$ is not embedded in $L^\infty$, we need at least one derivative in normal variable $y$ in order to justify the boundary layer expansion. Thus we have:
\begin{proposition}\label{prop:nabla c}
    Let $m \geq 2 \in \mathbb{N}$. Then there exist $\delta  > 0$ and $\varepsilon_0 > 0$ such that for any $\varepsilon \in (0, \varepsilon_0)$, we have:
    \begin{equation}
        \begin{aligned}
            \frac{1}{2} &\frac{d}{dt} \| \nabla c \|_{Y^{1,m}}^2 
            + (1 - C\delta - C\sigma) \varepsilon^2 \| \nabla^2 c \|_{Y^{1,m}}^2 
            - C\sigma \| \nabla n \|_{Y^{1,m}}^2 \\
            &\leq \frac{C_\delta}{\sigma} \biggl( 1 +\| u^a \|_{Y^{1,m+1}_\infty} + \| \nabla u^a \|_{Y^{1,m}_\infty}+ \| c^a \|_{Y^{1,m}_\infty}
           + \| \nabla c^a \|_{Y^{1,m+1}_\infty} 
            + \| \nabla n^a \|_{Y^{1,m}_\infty}\\
            &+\| n^a \|_{Y^{1,m}_\infty} 
            + \| f e^{-y} \|_{Y^{1,m+1}}  \biggr)^2  \left( \| u \|_{Y^1}^2 
            + \| \nabla c \|_{Y^{1,m}}^2 
            + \| c \|_{Y^1}^2 
            + \| \omega \|_{Y^{1,m}}^2 
            + \| n \|_{Y^{1,m}}^2 \right)\\ 
            &+ \frac{C}{\sigma^3 \varepsilon^4} \left( \| c \|_{Y^1}^2 
            + \| \nabla c \|_{Y^{1,m}}^2 
            + \| u \|_{Y^1}^2 
            + \| \omega \|_{Y^{1,m}}^2 
            + \| n \|_{Y^{1,m}}^2 \right)^3 \\
            &+ C_\delta \varepsilon^4 \| f e^{-y} \|_{Y^{1,m+1}}^2 
            \left( 1 + \| \nabla c^a \|_{Y^{1,m+1}_\infty}^2 
            + \| \nabla^2 c^a \|_{Y^{1,m}}^2 \right) + C \| \nabla K \|_{Y^{1,m}}^2.
        \end{aligned}
    \end{equation}
\end{proposition}

\begin{proposition}\label{prop:u}
    There exist $\delta  > 0$ and $\varepsilon_0 > 0$ such that for any $\varepsilon \in (0, \varepsilon_0)$, we have:
    \begin{equation}
        \begin{aligned}
            \frac{1}{2} &\frac{d}{dt} \| u \|_{Y^1}^2 
            + \varepsilon^2 (1 - C\sigma) \| \nabla u \|_{Y^1}^2 \\
            &\leq \frac{C}{\sigma} \varepsilon^4 \left(\| F e^{-y} \|_{Y^{2,2}}^2+\| f e^{-y} \|_{Y^{2,3}}^2\right) 
           + C \Bigg( \| u \|_{Y^1}^2 + \| \omega \|_{Y^{1,2}}^2 
            + \varepsilon^4 \| f e^{-y} \|_{Y^{1,3}}^2 \Bigg)^{\!3} \\
            &+ C \left( 1 + \| u^a \|_{Y^{1}_\infty}^2+\| \nabla u^a \|_{Y^{1}_\infty}^2 \right) 
            \left( \| u \|_{Y^1}^2 + \| \omega \|_{Y^{1,2}}^2 + \| n \|_{Y^{1}}^2 +\varepsilon^4 \| f e^{-y} \|_{Y^{1,2}}^2\right) \\
            & + C \| U \|_{Y^{1}}^2.
        \end{aligned}
    \end{equation}
\end{proposition}

For $\partial_y u$, if we directly compute the norm of $\nabla u$ in the conormal Sobolev space, then higher-order terms will appear. Therefore, motivated by \cite{M2014,WYZ2017}, we use the vorticity equation to gain one normal derivative.
\begin{proposition}\label{prop:omega}
    Let $m \geq 2 \in \mathbb{N}$. Then there exist $\delta  > 0$ and $\varepsilon_0 > 0$ such that for any $\varepsilon \in (0, \varepsilon_0)$, we have:
    \begin{equation}
        \begin{aligned}
           & \frac{1}{2} \frac{d}{dt} \| \widehat{\omega} \|_{Y^{1,m}}^2 
            + (1 - C\delta - C\sigma) \varepsilon^2 \| \nabla \widehat{\omega} \|_{Y^{1,m}}^2 
            - C\sigma \| \nabla n \|_{Y^{1,m}}^2 \\
            &\leq \frac{C_\delta}{\sigma} \left( 1 + \|u^{a}\|_{Y^{1,m}}^{2}+\| u^a \|_{Y^{1,m+1}_\infty}^2+\| \omega^a \|_{Y^{1,m+1}_\infty}^2 
            + \| f e^{-y} \|_{Y^{1,m+1}}^2 \right) \left( \| u \|_{Y^1}^2 
            + \| \widehat{\omega} \|_{Y^{1,m}}^2 
            + \varepsilon^4 \| f e^{-y} \|_{Y^{1,m+1}}^2 \right) \\
            &+ \frac{C}{\sigma \varepsilon^4} \left( \| u \|_{Y^1}^2 
            + \| \widehat{\omega} \|_{Y^{1,m}}^2 +\varepsilon^4 \| f e^{-y} \|_{Y^{1,m}}^2\right)^3 + C \| \nabla^\perp \cdot U \|_{Y^{1,m}}^2 \\
            &+ C\varepsilon^4 \left( \| f e^{-y} \|_{Y^{1,m+3}}^2 
            + \| f e^{-y} \|_{Y^{2,m+2}}^2 \right)  \times \left( 1 + \| \omega^a \|_{Y^{1,m+1}_\infty}^2 
            + \| \partial_y \omega^a \|_{Y^{1,m}}^2 \right),
        \end{aligned}
    \end{equation}
    where $\widehat{\omega} = \omega - \varepsilon^2 \partial_x f e^{-y}$.
\end{proposition}

%Proof of Theorem
\subsection{Proof of Theorem \ref{thm:main}}

\begin{proof}[Proof of Theorem \ref{thm:main}]

We introduce the energy functional
\begin{align*}
    E(t) = \| n \|_{Y^{1,2}}^2 + \| c \|_{Y^{1}}^2 + \| \nabla c \|_{Y^{1,2}}^2 + \| u \|_{Y^{1}}^2 + \| \widehat{\omega} \|_{Y^{1,2}}^2 + \varepsilon^3,
\end{align*}
and
\begin{align*}
    D(t) = \| \nabla n \|_{Y^{1,2}}^2 + \varepsilon^2 \| \nabla^2 c \|_{Y^{1,2}}^2 + \varepsilon^2 \| \nabla\widehat{\omega} \|_{Y^{1,2}}^2.
\end{align*}

Combining the estimates of  Propositions~\ref{prop:n} to~\ref{prop:omega} together, letting $m = 2$ and $\sigma, \delta$ small enough, we have
\ben\label{eq:Et}
    \frac{d}{dt} E(t) + D(t) \leq C(t) E(t) + \frac{C'}{\varepsilon^4} E(t)^3,
\een
where 
\begin{align*}
    C(t) = & C \biggl( 1 + \left\|n^a\right\|_{Y_\infty^{1,3}} + \| \nabla n^a \|_{Y^{1,2}_\infty}
    + \left\|u^a\right\|_{Y_\infty^{1,4}} + \left\|\nabla u^a\right\|_{Y_\infty^{1,3}} + \left\|\omega^a\right\|_{Y_\infty^{1,3}}  + \left\|c^a\right\|_{Y_\infty^{1,4}} + \left\|\nabla c^a\right\|_{Y_\infty^{1,3}} \biggr)^6\\
    &+ C\varepsilon  \| fe^{-y}\|^2_{Y^{2,5}} \left( \left\|\partial_y^2 c^a \|_{Y^{1,2}}^2 + \|\partial_y^2 u^a \right\|_{Y^{1,2}}^2 
    % + \left\|\partial_y \omega^a\right\|_{Y^{1,2}}^2 
    \right) \\
    &+ C\varepsilon \| F e^{-y} \|_{Y^{2,2}}^2 + C\big( 1+\|fe^{-y} \|_{Y^{2,5}} \big)^6\\
    &+C\left(\| N \|_{Y^{1,2}}^{2} + \| K \|_{Y^{1}}^2 + \| \nabla K \|_{Y^{1,2}}^2 + \| U \|_{Y^{1}}^2 + \| \nabla^\perp \cdot U \|_{Y^{1,2}}^2\right),
\end{align*}
and we  used
\ben\label{eq:omega-f}
    \| \omega \|_{Y^{1,2}}^2\leq 2\| \widehat{\omega} \|_{Y^{1,2}}^2+2\varepsilon^4\| \partial_xfe^{-y} \|_{Y^{1,2}}^2,
\een
\begin{align*}
    \frac{C'}{\varepsilon^4}\left(2\varepsilon^4\| \partial_xfe^{-y} \|_{Y^{1,2}}^2 \right)^3\leq C\varepsilon^8 \|fe^{-y} \|_{Y^{1,3}}^6,
\end{align*}
and the definition of \eqref{def:conormal sobolev}:
\begin{align*}
   \left\|\cdot\right\|_{Y_\infty^{1}}\leq \left\|\cdot\right\|_{Y_\infty^{1,k}},\quad k\geq 1.
\end{align*}
Then by \eqref{eq:Et} we get
% and Proposition \ref{prop:uniform bounds}, the standard continuity arguments
% (see also \cite{gronwall}) 
% yield that:
\begin{align*}
    &\frac{\frac{d}{dt} E}{E^3} \leq \frac{C(t)}{E^2} + \frac{C'}{\varepsilon^4}, 
    \end{align*}
    and
    % &\Rightarrow - \frac{1}{2} \frac{d}{dt} (E^{-2}) \leq C(t)E^{-2} + \frac{C'}{\varepsilon^4}, \\
    \begin{align*}
    & E(t)^{-2} \geq \left( E(0)^{-2} - \frac{2C'}{\varepsilon^4} \int_0^t e^{2\int_0^\tau C(\xi) d\xi} d\tau \right) e^{-2\int_0^t C(\tau) d\tau}, \end{align*}
 which implies
 \begin{align*}
 E(t) \leq \sqrt{\frac{e^{2\int_0^t C(\tau) d\tau}}{E(0)^{-2} - \frac{2C'}{\varepsilon^4} \int_0^t e^{2\int_0^\tau C(\xi) d\xi} d\tau}}.
\end{align*}
Noting that $E(0) = \varepsilon^3$ 
%(derived from initial value of equations \eqref{eq:error}, \eqref{eq:omega} and \eqref{eq:u_1^b1}) 
and the  integrability of $C(t)$ in  Proposition \ref{prop:uniform bounds}, we deduce that
\begin{align*}
    E \leq \varepsilon^3 \frac{e^{ \int_0^t C(\tau) d \tau}}{\sqrt{1 - 2C' \varepsilon^2 \int_0^t e^{2\int_0^ \tau C(\xi) d\xi} d\tau}}.
\end{align*}
Thus, there exist constants $C,T_* > 0$, both independent of $\varepsilon$, such that
\begin{align*}
    \sup_{0 \leq t \leq T^*} E(t) \leq C\varepsilon^3 \quad \text{for all } \varepsilon \in (0,\varepsilon_0),
\end{align*}
which implies
\ben\label{eq:ncu-uniform}
    \| n \|_{L^\infty_tY^{1,2}}^2 + \| c \|_{L^\infty_tY^{1}}^2 + \| \nabla c \|_{L^\infty_tY^{1,2}}^2 + \| u \|_{L^\infty_tY^{1}}^2 + \| \widehat{\omega} \|_{L^\infty_tY^{1,2}}^2 \leq C\varepsilon^3, \quad  \varepsilon \in (0,\varepsilon_0), ~t \in [0,T_*].
\een
% Furthermore, we have
% \begin{align*}
%     D(t) \leq C(t) E(t) + \frac{C'}{\varepsilon^4} E(t)^3+\varepsilon^3 
% \end{align*}
% Using the integrability of $C(t)$ again,  for $t \in [0,T_*]$ there exists $C$ so that
% \begin{align*}
%     \| D(t) \|_{L_t^2} \leq C\varepsilon^3.
% \end{align*}
% Thus
% \begin{align*}
%     \| \nabla n \|_{L_t^2Y^{1,2}}^2 + \varepsilon^2 \| \nabla^2 c \|_{L_t^2Y^{1,2}}^2 + \varepsilon^2 \| \widehat{\omega} \|_{L_t^2Y^{1,2}}^2 \leq C\varepsilon^3 \quad \text{for all } \varepsilon \in (0,\varepsilon_0).
% \end{align*}
Using \eqref{eq:omega-f},  Lemma \ref{lem:omega} and Proposition \ref{prop:uniform bounds}, we get
\ben\label{eq:u-bounded}
\| {\omega} \|_{L^\infty_tY^{1,2}}^2\leq C,\quad \| \nabla u\|_{L^\infty_tY^{1,2}}^2\leq C.
\een
Finally, it follows from \eqref{eq:u-bounded}, \eqref{eq:ncu-uniform}
and Lemma~\ref{lem:embedding} that
\begin{align*}
    \| (n,c,u) \|_{L^\infty(0,T_*;L_{xy}^\infty)} \leq C\varepsilon^{\frac 3 2}.
\end{align*}
The proof is complete.
\end{proof}

%Construction of the approximate solution
\section{Construction of the approximate solution}\label{approximate solution}

In this section, we will construct the approximate solutions of the chemotaxis Navier-Stokes equations \eqref{eq:CNS} by using the matched asymptotic expansion method(see also \cite{TWZ2020,houwang2019}).

%Asymptotic expansion
\subsection{Asymptotic expansion}

As in \cite{TWZ2020,houwang2019}), it is natural to choose the following ansatz.
%We perform the following asymptotic expansion:
\begin{equation}\label{exp:asymptotic}
    \left\{
        \begin{aligned}
            &n^\varepsilon(t,x,y) = \sum_{j \geq 0} \varepsilon^j \left( n^{e,j} (t,x,y) + n^{b,j}(t,x,z) \right), \\
            &c^\varepsilon(t,x,y) = \sum_{j \geq 0} \varepsilon^j \left( c^{e,j} (t,x,y) + c^{b,j}(t,x,z) \right), \\
            &u^\varepsilon(t,x,y) = \sum_{j \geq 0} \varepsilon^j \left( u^{e,j} (t,x,y) + u^{b,j}(t,x,z) \right), \\
            &p^\varepsilon(t,x,y) = \sum_{j \geq 0} \varepsilon^j \left( p^{e,j} (t,x,y) + p^{b,j}(t,x,z) \right), \\
        \end{aligned}
    \right.
\end{equation}
where $z = \frac y \varepsilon$ and $n^{b,j},c^{b,j},u^{b,j},p^{b,j} $ with their derivatives are vanishing exponentially when $z\to \infty$.

Substituting \eqref{exp:asymptotic} into \eqref{eq:CNS}, we get
% \begin{equation}\label{ebquation}
% \left\{
\begin{align}\label{expanded equations}
    &\partial_t \sum_{j\geq 0} \varepsilon^j (n^{e,j} + n^{b,j}) + \sum_{j\geq 0} \varepsilon^j (u^{e,j} + u^{b,j}) \cdot \nabla \sum_{j\geq 0} \varepsilon^j (n^{e,j} + n^{b,j}) \nonumber \\
    &\quad - \Delta \sum_{j\geq 0} \varepsilon^j (n^{e,j} + n^{b,j}) + \nabla \cdot \left( \sum_{j\geq 0} \varepsilon^j (n^{e,j} + n^{b,j}) \nabla \sum_{j\geq 0} \varepsilon^j (c^{e,j} + c^{b,j}) \right) = 0, \nonumber\\
    &\partial_t \sum_{j\geq 0} \varepsilon^j (c^{e,j} + c^{b,j}) + \sum_{j\geq 0} \varepsilon^j (u^{e,j} + u^{b,j}) \cdot \nabla \sum_{j\geq 0} \varepsilon^j (c^{e,j} + c^{b,j}) \nonumber\\
    &\quad - \varepsilon^2 \Delta \sum_{j\geq 0} \varepsilon^j (c^{e,j} + c^{b,j}) + \sum_{j\geq 0} \varepsilon^j (c^{e,j} + c^{b,j}) \sum_{j\geq 0} \varepsilon^j (n^{e,j} + n^{b,j}) = 0,\\
    &\partial_t \sum_{j\geq 0} \varepsilon^j (u^{e,j} + u^{b,j}) + \sum_{j\geq 0} \varepsilon^j (u^{e,j} + u^{b,j}) \cdot \nabla \sum_{j\geq 0} \varepsilon^j (u^{e,j} + u^{b,j})  \nonumber\\
    &\quad+ \nabla \sum_{j\geq 0} \varepsilon^j (p^{e,j} + p^{b,j}) - \varepsilon^2 \Delta \sum_{j\geq 0} \varepsilon^j (u^{e,j} + u^{b,j}) - \sum_{j\geq 0} \varepsilon^j (n^{e,j} + n^{b,j}) e_2 = 0, \nonumber\\
    &\nabla \cdot \sum_{j\geq 0} \varepsilon^j (u^{e,j} + u^{b,j}) = 0. \nonumber
\end{align}
% \right.
% \end{equation}
Away from the boundary, by taking $z \to \infty$ we get
\begin{align}\label{outer layers}
    &\sum_{j\geq 0} \varepsilon^j \partial_t n^{e,j} + \sum_{j\geq 0} \varepsilon^j u^{e,j} \cdot \sum_{j\geq 0} \varepsilon^j \nabla n^{e,j}  \nonumber + \nabla \cdot \left(\sum_{j\geq 0} \varepsilon^j n^{e,j} \sum_{j\geq 0} \varepsilon^j \nabla c^{e,j} \right) \nonumber \\
    &\quad - \sum_{j\geq 0} \varepsilon^j \Delta n^{e,j}= 0, \nonumber\\
    &\sum_{j\geq 0} \varepsilon^j \partial_t c^{e,j} + \sum_{j\geq 0} \varepsilon^j u^{e,j} \cdot \sum_{j\geq 0} \varepsilon^j \nabla c^{e,j} + \sum_{j\geq 0} \varepsilon^j c^{e,j} \sum_{j\geq 0} \varepsilon^j n^{e,j} \nonumber\\
    &\quad  - \sum_{j\geq 0} \varepsilon^{j+2} \Delta c^{e,j}=0, \\
    &\sum_{j\geq 0} \varepsilon^j \partial_t u^{e,j} + \sum_{j\geq 0} \varepsilon^j u^{e,j} \cdot \sum_{j\geq 0} \varepsilon^j \nabla u^{e,j} + \sum_{j\geq 0} \varepsilon^j \nabla p^{e,j}  - \sum_{j\geq 0} \varepsilon^j  n^{e,j} e_2 \nonumber\\
    &\quad - \sum_{j\geq 0} \varepsilon^{j+2} \Delta u^{e,j} =0, \nonumber\\
    &\sum_{j\geq 0} \varepsilon^j \nabla \cdot u^{e,j} = 0, \nonumber
\end{align}
by noting that $(n^{b,j},c^{b,j},u^{b,j},p^{b,j}) \to 0$ when $z \to \infty$. Therefore, $(n^{e,j},c^{e,j},u^{e,j},p^{e,j})$ with $j \geq 0$ solves
\begin{equation}\label{eq:ej}
    \left\{
        \begin{aligned}
            &\partial_t n^{e,j} + \sum_{k=0}^j u^{e,j-k} \cdot \nabla n^{e,k} - \nabla \cdot \left( \nabla n^{e,j} - \sum_{k=0}^j n^{e,j-k} \nabla c^{e,k} \right) = 0, \\
            &\partial_t c^{e,j} + \sum_{k=0}^j u^{e,j-k} \cdot \nabla c^{e,k} - \Delta c^{e,j-2} = -\sum_{k=0}^j c^{e,j-k} n^{e,k}, \\
            &\partial_t u^{e,j} + \sum_{k=0}^j u^{e,j-k} \cdot \nabla u^{e,k} + \nabla p^{e,j} - \Delta u^{e,j-2} = \begin{pmatrix} 0 \\ n^{e,j} \end{pmatrix}, \\
            &\nabla \cdot u^{e,j} = 0.
        \end{aligned}
    \right.
\end{equation}

To analyze the behavior of the  equations near the boundary, it is necessary to unify the coordinates into $z$ by subtracting the outer layers  \eqref{outer layers} from the expanded equations \eqref{expanded equations}.  Assume that the boundary operator is expressed as $\overline{h} (t,x) = h(t,x,y)|_{y=0}$, and a Taylor expansion of $h (t,x,y)$ near $y=0$ is as follows:
\begin{align}\label{exp:taylor}
    h (t,x,y)= \sum_{l \geq 0} \varepsilon^l \frac{z^l}{l!} \overline{ \partial_y^l h} (t,x).
\end{align}
Moreover, define \begin{equation}\label{def:Ej}
    \left\{
    \begin{aligned}
        &n^{E,j} = \sum_{l = 0}^j \frac{z^l}{l!} \overline{ \partial_y^{l} n^{e,j-l}},\\
        &c^{E,j} = \sum_{l = 0}^j \frac{z^l}{l!} \overline{ \partial_y^{l} c^{e,j-l}},\\
        &u^{E,j} = \sum_{l = 0}^j \frac{z^l}{l!} \overline{ \partial_y^{l} u^{e,j-l}}.
    \end{aligned}
    \right.
\end{equation}
Then we have the following lemma.
\begin{lemma}\label{lem:layer formula} There holds
\begin{align*}
    &\sum_{j\geq 0} \varepsilon^j \partial_t n^{b,j} + \sum_{j\geq -1} \varepsilon^j \sum_{k=0}^{j+1} \left( (u^{E,j-k})_p \cdot \widehat{\nabla} n^{b,k} + (u^{b,j-k})_p \cdot \widehat{\nabla} n^{E,k} + (u^{b,j-k})_p \cdot \widehat{\nabla} n^{b,k} \right) \\
    &- \sum_{j\geq -2} \varepsilon^j \widehat{\nabla} \cdot \left( \widehat{\nabla} n^{b,j} - \sum_{k=0}^j \left( n^{E,j-k} \widehat{\nabla} c^{b,k} + n^{b,j-k} \widehat{\nabla} c^{E,k} + n^{b,j-k} \widehat{\nabla} c^{b,k} \right) \right)_{pp} = 0,\\
    &\sum_{j\geq 0} \varepsilon^j \partial_t c^{b,j} + \sum_{j\geq -1} \varepsilon^j \sum_{k=0}^{j+1} \left( (u^{E,j-k})_p \cdot \widehat{\nabla} c^{b,k} + (u^{b,j-k})_p \cdot \widehat{\nabla} c^{E,k} + (u^{b,j-k})_p \cdot \widehat{\nabla} c^{b,k} \right) \\
    &- \sum_{j\geq 0} \varepsilon^j \widehat{\nabla} \cdot (\widehat{\nabla} c^{b,j-2})_{pp} = - \sum_{j\geq 0} \varepsilon^j \sum_{k=0}^j \left( c^{E,j-k} n^{b,k} + c^{b,j-k} n^{E,k} + c^{b,j-k} n^{b,k} \right),\\
    &\sum_{j\geq 0} \varepsilon^j \partial_t u^{b,j} + \sum_{j\geq -1} \varepsilon^j \sum_{k=0}^{j+1} \left( (u^{E,j-k})_p \cdot \widehat{\nabla} u^{b,k} + (u^{b,j-k})_p \cdot \widehat{\nabla} u^{E,k} + (u^{b,j-k})_p \cdot \widehat{\nabla} u^{b,k} \right) \\
    &+ \sum_{j\geq -1} \varepsilon^j (\widehat{\nabla} p^{b,j})_p - \sum_{j\geq 0} \varepsilon^j \widehat{\nabla} \cdot (\widehat{\nabla} u^{b,j-2})_{pp} = \sum_{j\geq 0} \varepsilon^j  n^{b,j} e_2,\\
    &\sum_{j\geq -1} \varepsilon^j \widehat{\nabla} \cdot (u^{b,j})_p = 0,
\end{align*}
where for a vector-valued function $(h_1(\cdot,j), h_2(\cdot,j))$ we write \begin{equation}\label{def:()_p}
    \binom{h_1(\cdot,j)}{h_2(\cdot,j)}_p = \binom{h_1(\cdot,j)}{h_2(\cdot,j+1)},~\binom{h_1(\cdot,j)}{h_2(\cdot,j)}_{pp} = \binom{h_1(\cdot,j)}{h_2(\cdot,j+2)},
\end{equation}
and for a function let $h(\cdot,j) = 0$ when $j < 0$ (for example, $h(\cdot,j)=n^{b,j}, c^{b,j}, u_1^{e,j},  u_1^{b,j}$ etc).
\end{lemma}

\begin{proof}
For a function $h^{e,j}(t,x,y)$,  by \eqref{def:Ej} noting that 
\begin{align*}
    &\sum_{j \ge 0} \varepsilon^j h^{e,j} = \sum_{j \ge 0} \varepsilon^j\sum_{l \ge 0} \varepsilon^l \frac{z^l}{l!} \overline{\partial_y^l h^{e,j}} =\sum_{j \ge 0} \varepsilon^j \sum_{l=0}^j \frac{z^l}{l!} \overline{\partial_y^l h^{e,j-l}} = \sum_{j \ge 0} \varepsilon^j h^{E,j}, 
\end{align*}
and
% \begin{align*}
%     &\partial_y^k \sum_{j \ge 0} \varepsilon^j h^{e,j} 
%     =\sum_{j \ge 0} \varepsilon^j \sum_{l=0}^j \frac{z^l}{l!} \overline{\partial_y^{l+k} h^{e,j-l}} 
%     =\sum_{j \ge 0} \varepsilon^j \partial_z^{k}\sum_{l=-k}^j \frac{z^{l+k}}{(l+k)!} \overline{\partial_y^{l+k} h^{e,j-l}} 
%     = \sum_{j \ge 0} \varepsilon^j \partial_z^k h^{E,j+k}, 
% \end{align*}
then we have\\
(a)
\begin{align*}
    &\sum_{j\geq 0} \varepsilon^j (h^{e,j} + h^{b,j}) \sum_{j\geq 0} \varepsilon^j (g^{e,j} + g^{b,j})- \sum_{j\geq 0} \varepsilon^j h^{e,j}  \sum_{j\geq 0} \varepsilon^j g^{e,j} \\
    % &\qquad =\sum_{j\geq 0} \varepsilon^j \sum_{k=0}^j h^{e,j-k}g^{b,k}+\sum_{j\geq 0} \varepsilon^j \sum_{k=0}^j h^{b,j-k}g^{b,k}+\sum_{j\geq 0} \varepsilon^j \sum_{k=0}^j h^{b,j-k}g^{e,k}\\\
    % &\qquad =\sum_{j\geq 0} \varepsilon^j \sum_{k=0}^j \sum_{l=0}^j\varepsilon^l \frac{z^l}{l!} \overline{\partial_y^l h^{e,j-k}}g^{b,k}+\sum_{j\geq 0} \varepsilon^j \sum_{k=0}^j h^{b,j-k}g^{b,k}+\sum_{j\geq 0} \varepsilon^j \sum_{k=0}^j h^{b,j-k}g^{e,k}\\
    &\qquad =\sum_{j\geq 0} \varepsilon^j \sum_{k=0}^j \left( h^{E,j-k} g^{b,k} + h^{b,j-k} g^{E,k} + h^{b,j-k} g^{b,k} \right),
\end{align*}
(b)
\begin{align*}
    % (a)~~~&\sum_{j\geq 0} \varepsilon^j (h^{e,j} + h^{b,j}) 
    &\sum_{j\geq 0} \varepsilon^j (h^{e,j} + h^{b,j}) \partial_y \sum_{j\geq 0} \varepsilon^j (g^{e,j} + g^{b,j}) - \sum_{j\geq 0} \varepsilon^j h^{e,j}  \partial_y \sum_{j\geq 0} \varepsilon^j g^{e,j} \\
     &\qquad=\sum_{j\geq 0} \varepsilon^j h^{E,j}  \partial_z \sum_{j\geq 0} \varepsilon^{j-1} g^{b,j}+\sum_{j\geq 0} \varepsilon^j h^{b,j}  \partial_z \sum_{j\geq 0} \varepsilon^{j-1} g^{E,j} + \sum_{j\geq 0} \varepsilon^j h^{b,j} \partial_z \sum_{j\geq 0} \varepsilon^{j-1} g^{b,j} \\
    &\qquad =\sum_{j\geq -1} \varepsilon^{j} \sum_{k=0}^{j+1} \left( h^{E,{j+1}-k} \partial_z g^{b,k} + h^{b,{j+1}-k} \partial_z g^{E,k} + h^{b,{j+1}-k} \partial_zg^{b,k} \right),
    % (c)~~~&\partial_y \left( \sum_{j\geq 0} \varepsilon^j (h^{e,j} + h^{b,j}) \partial_y \sum_{j\geq 0} \varepsilon^j (g^{e,j} + g^{b,j}) \right)- \partial_y \left( \sum_{j\geq 0} \varepsilon^j h^{e,j}  \partial_y \sum_{j\geq 0} \varepsilon^j g^{e,j} \right) \\
    % &\qquad =\partial_y \left( \sum_{j\geq 0} \varepsilon^j h^{e,j}  \partial_y \sum_{j\geq 0} \varepsilon^j g^{b,j} \right)+\partial_y \left( \sum_{j\geq 0} \varepsilon^j h^{b,j}  \partial_y \sum_{j\geq 0} \varepsilon^j g^{e,j} \right)+\partial_y \left( \sum_{j\geq 0} \varepsilon^j h^{b,j}  \partial_y \sum_{j\geq 0} \varepsilon^j g^{b,j} \right)\\
    % &\qquad =\sum_{j\geq -2} \varepsilon^{j} \partial_z \sum_{k=0}^{j+2} \left( h^{E,{j+2}-k} \partial_z g^{b,k} + h^{b,{j+2}-k} \partial_z g^{E,k} + h^{b,{j+2}-k} \partial_zg^{b,k} \right),\\
    % (d)~~~&\partial_y^2 \sum_{j \ge 0} \varepsilon^j (h^{e,j}+h^{b,j}) - \partial_y^2 \sum_{j \ge 0} \varepsilon^j h^{e,j} = \sum_{j \ge -2} \varepsilon^j \partial_z^2 h^{b,j}.
\end{align*}
(c)
\begin{align*}
    &\partial_y \left( \sum_{j\geq 0} \varepsilon^j (h^{e,j} + h^{b,j}) \partial_y \sum_{j\geq 0} \varepsilon^j (g^{e,j} + g^{b,j}) \right)- \partial_y \left( \sum_{j\geq 0} \varepsilon^j h^{e,j}  \partial_y \sum_{j\geq 0} \varepsilon^j g^{e,j} \right) \\
    &\qquad =\partial_y \left( \sum_{j\geq 0} \varepsilon^j h^{E,j}  \partial_y \sum_{j\geq 0} \varepsilon^j g^{b,j} \right)+\partial_y \left( \sum_{j\geq 0} \varepsilon^j h^{b,j}  \partial_y \sum_{j\geq 0} \varepsilon^j g^{E,j} \right)+\partial_y \left( \sum_{j\geq 0} \varepsilon^j h^{b,j}  \partial_y \sum_{j\geq 0} \varepsilon^j g^{b,j} \right)\\
    &\qquad =\sum_{j\geq -2} \varepsilon^{j} \partial_z \sum_{k=0}^{j+2} \left( h^{E,{j+2}-k} \partial_z g^{b,k} + h^{b,{j+2}-k} \partial_z g^{E,k} + h^{b,{j+2}-k} \partial_zg^{b,k} \right),
    % (d)~~~&\partial_y^2 \sum_{j \ge 0} \varepsilon^j (h^{e,j}+h^{b,j}) - \partial_y^2 \sum_{j \ge 0} \varepsilon^j h^{e,j} = \sum_{j \ge -2} \varepsilon^j \partial_z^2 h^{b,j}.
\end{align*}
and 
(d)
\begin{align*}
    &\partial_y^2 \sum_{j \ge 0} \varepsilon^j (h^{e,j}+h^{b,j}) - \partial_y^2 \sum_{j \ge 0} \varepsilon^j h^{e,j} = \sum_{j \ge -2} \varepsilon^j \partial_z^2 h^{b,j}.
\end{align*}
Applying the above equalities, the proof of  Lemma \ref{lem:layer formula} is complete.
\end{proof}

It follows from Lemma \ref{lem:layer formula} immediately that
\begin{equation}\label{eq:bj}
    \left\{
        \begin{aligned}
            &\partial_t n^{b,j} + \sum_{k=0}^{j+1} \left( (u^{E,j-k})_p \cdot \widehat{\nabla} n^{b,k} + (u^{b,j-k})_p \cdot \widehat{\nabla} n^{E,k} + (u^{b,j-k})_p \cdot \widehat{\nabla} n^{b,k} \right) \\
            &\quad- \widehat{\nabla} \cdot \left( \widehat{\nabla} n^{b,j} - \sum_{k=0}^j \left( n^{E,j-k} \widehat{\nabla} c^{b,k} + n^{b,j-k} \widehat{\nabla} c^{E,k} + n^{b,j-k} \widehat{\nabla} c^{b,k} \right) \right)_{pp} = 0,\quad j\geq -2 \\
            &\partial_t c^{b,j} + \sum_{k=0}^{j+1} \left( (u^{E,j-k})_p \cdot \widehat{\nabla} c^{b,k} + (u^{b,j-k})_p \cdot \widehat{\nabla} c^{E,k} + (u^{b,j-k})_p \cdot \widehat{\nabla} c^{b,k} \right) \\
            &\quad- \widehat{\nabla} \cdot (\widehat{\nabla} c^{b,j-2})_{pp} = -\sum_{k=0}^j \left( c^{E,j-k} n^{b,k} + c^{b,j-k} n^{E,k} + c^{b,j-k} n^{b,k} \right),\quad  j\geq -1\\
            &\partial_t u^{b,j} + \sum_{k=0}^{j+1} \left( (u^{E,j-k})_p \cdot \widehat{\nabla} u^{b,k} + (u^{b,j-k})_p \cdot \widehat{\nabla} u^{E,k} + (u^{b,j-k})_p \cdot \widehat{\nabla} u^{b,k} \right) \\
            &\quad + (\widehat{\nabla} p^{b,j})_p- \widehat{\nabla} \cdot (\widehat{\nabla} u^{b,j-2})_{pp} =  n^{b,j} e_2, \quad j\geq -1\\
            &\widehat{\nabla} \cdot (u^{b,j})_p = 0, \quad j\geq -1.
        \end{aligned}
    \right.
\end{equation}
%where we define the outer mixing term: 
Moreover,
for the boundary conditions, plugging \eqref{exp:asymptotic} into \eqref{eq:CNS}$_{5}$, we get
\begin{equation}\label{bound}
    \left\{
        \begin{aligned}
            &\partial_y n^{e,j} + \partial_z n^{b,j+1} = 0\quad \text{on}~ y=z=0, \\
            &\partial_y c^{e,j} + \partial_z c^{b,j+1} = 0\quad \text{on}~ y=z=0, \\
            &\partial_y u_1^{e,j} + \partial_z u_1^{b,j+1} = 0\quad \text{on}~ y=z=0, \\
            &u_2^{e,j} + u_2^{b,j}=0\quad \text{on}~ y=z=0,
        \end{aligned}
    \right.
\end{equation}
where $j\geq -1.$
% and
% \begin{align}
%     (n^{e,0}, c^{e,0}, u^{e,0})|_{t=0}=(
% \end{align}

%Derivation of the matching terms
\subsection{Derivation of the matching terms}
We will follow a step-by-step procedure.

\textbf{Step I.  The order $O(\varepsilon^{-2})$ terms.} Taking $j=-2$ in the system \eqref{eq:bj}, and noting  that $\partial_z c^{E,0}=0$ due to \eqref{def:Ej} we have
%From the order $O(\varepsilon^{-2})$ terms o, we get:
\beno
        -\partial_z \left(\partial_z n^{b,0} - n^{E,0}\partial_z c^{b,0}  -n^{b,0}\partial_z c^{b,0} \right) = 0, 
    \eeno
which implies that
\begin{equation}\label{eq:partial_z n^b0}
    \begin{aligned}
        & \partial_z n^{b,0} - n^{E,0}\partial_z c^{b,0}  -n^{b,0}\partial_z c^{b,0}  = 0, 
    \end{aligned}
\end{equation}
by noting that the vanishing exponential decay of the above boundary-layer functions at $\infty.$
Using \eqref{eq:partial_z n^b0}, it follow that $\partial_z (n^{E,0} + n^{b,0} ) = (n^{E,0} + n^{b,0} ) \partial_z c^{b,0}$. Then by solving first-order linear ordinary differential equations,
we have
\begin{equation}\label{eq:n^b0}
    n^{b,0} = n^{E,0}\left( e^{c^{b,0}} - 1 \right),
\end{equation}
since $n^{b,0}, c^{b,0}$ decay fastly when $z\to \infty$.

\textbf{Step II. The order $O(\varepsilon^{-1})$ terms.} Taking $j=-1$ in the system \eqref{eq:bj}, and using \eqref{def:Ej} again we have
\begin{equation}\label{eq:b-1}
    \left\{
        \begin{aligned}
            &u_2^{E,0} \partial_z n^{b,0}  + u_2^{b,0} \partial_z n^{b,0}  \\
            &\quad -\partial_z \left( \partial_z n^{b,1} - \sum_{k=0}^1 \left( n^{E,1-k} \partial_z c^{b,k} + n^{b,1-k} \partial_z c^{E,k} + n^{b,1-k} \partial_z c^{b,k} \right) \right) = 0, \\
            &u_2^{E,0} \partial_z c^{b,0} + u_2^{b,0} \partial_z c^{b,0} = 0, \\
            &u_2^{E,0} \partial_z u^{b,0}  + u_2^{b,0} \partial_z u^{b,0} + \begin{pmatrix} 0 \\ \partial_z p^{b,0} \end{pmatrix} = 0, \\
            &\partial_z u_2^{b,0} = 0,
        \end{aligned}
    \right.
\end{equation}
which implies 
\begin{equation}\label{u_2^b0=0}
    u_2^{b,0} \equiv 0,
\end{equation}
since  $\lim_{z\to \infty} u_2^{b,0} = 0$. Using \eqref{bound}, we have
\begin{equation}\label{bound:u_2^e0}
    u_2^{E,0} =\overline{u_2^{e,0}}\equiv 0,
\end{equation}
and it follows from \eqref{eq:b-1}$_3$ that 
\begin{equation}\label{p^b0=0}
    p^{b,0} \equiv 0.
\end{equation}
Using \eqref{eq:b-1}$_1$, \eqref{u_2^b0=0} and \eqref{bound:u_2^e0} we have
\begin{equation*}
    \partial_z \left( \partial_z n^{b,1} - \sum_{k=0}^1 \left( n^{E,1-k} \partial_z c^{b,k} + n^{b,1-k} \partial_z c^{E,k} + n^{b,1-k} \partial_z c^{b,k} \right) \right) = 0,
\end{equation*}
which implies
\ben\label{eq:partial_z n^b1}
    \partial_z n^{b,1} - \sum_{k=0}^1 \left( n^{E,1-k} \partial_z c^{b,k} + n^{b,1-k} \partial_z c^{E,k} + n^{b,1-k} \partial_z c^{b,k} \right)= 0.
\een
Using \eqref{bound} with $j=-1$, we have
\ben\label{bound:ncu^b0}
        \partial_z n^{b,0} = \partial_z c^{b,0} = \partial_z u_1^{b,0} = 0 \quad \text{on}~ z=0.
    \een
    Letting $z=0$ in \eqref{eq:partial_z n^b1}, by \eqref{bound:ncu^b0} we have
    \begin{align*}
        \partial_z n^{b,1} = n^{E,0}\partial_z c^{b,1} + n^{b,0}(\overline{\partial_y c^{e,0}} + \partial_z c^{b,1}) \quad \text{on}~z=0,
    \end{align*}
which and  \eqref{bound} with $j=0$ yield that
\ben\label{bound:n^e0}
    \partial_y n^{e,0} -n^{e,0}\partial_y c^{e,0} = 0 \quad \text{on}~ y = 0.
\een

    % First, integrating $z$ from $z$ to $\infty$, we have
    % \begin{align*}
    %     \partial_z n^{b,1} - \sum_{k=0}^1 n^{E,1-k} \partial_z c^{b,k} + n^{b,1-k} \partial_z c^{E,k} + n^{b,1-k} \partial_z c^{b,k} = 0,
    % \end{align*}
    % provided the decay conditions on $n^{b,0}, \partial_z n^{b,1},z\partial_z  c^{b,0}$ and $\partial_zc^{b,1}$. 

\textbf{Step III.  The order $O(\varepsilon^{0})$ terms.} It follows from \eqref{eq:ej} with $j=0$, \eqref{bound:u_2^e0} and \eqref{bound:n^e0} that
\begin{equation}\label{eq:e0}
    \left\{
        \begin{aligned}
            &\partial_t n^{e,0} + u^{e,0} \cdot \nabla n^{e,0} - \nabla \cdot \left( \nabla n^{e,0} - n^{e,0} \nabla c^{e,0} \right) = 0, \\
            &\partial_t c^{e,0} + u^{e,0} \cdot \nabla c^{e,0} + c^{e,0} n^{e,0} = 0, \\
            &\partial_t u^{e,0} + u^{e,0} \cdot \nabla u^{e,0} + \nabla p^{e,0} = \begin{pmatrix} 0 \\ n^{e,0} \end{pmatrix}, \\
            &\nabla \cdot u^{e,0} = 0, \\
            &\partial_y n^{e,0} -n^{e,0}\partial_y c^{e,0} = 0, \quad u_2^{e,0} = 0 \quad \text{on}~ y = 0\\
            &(n^{e,0},c^{e,0},u^{e,0}) (0,x,y) = (n_{in},c_{in},u_{in}).
        \end{aligned}
    \right.
\end{equation}
By the uniqueness of the solution, we conclude that this coincides with the equation \eqref{eq:epsilon=0} when $\varepsilon = 0$.

For the boundary layer analysis, 
taking $j=0$ in the system \eqref{eq:bj} we get
% \begin{equation}\label{展开_0e}
%     \left\{
%         \begin{aligned}
%             &\partial_t n^{e,0} + u^{e,0} \cdot \nabla n^{e,0} - \nabla \cdot \left( \nabla n^{e,0} - n^{e,0} \nabla c^{e,0} \right) = 0, \\
%             &\partial_t c^{e,0} + u^{e,0} \cdot \nabla c^{e,0} + c^{e,0} n^{e,0} = 0, \\
%             &\partial_t u^{e,0} + u_1^{e,0} \cdot \partial_x u^{e,0} + u_2^{e,0} \cdot \partial_y u^{e,0} + \nabla p^{e,0} = \begin{pmatrix} 0 \\ n^{e,0} \end{pmatrix}, \\
%             &\nabla \cdot u^{e,0} = 0,
%         \end{aligned}
%     \right.
% \end{equation}
% and

% So, the $O(\varepsilon^0)$-order term of the outer solution satisfies the equations:
\begin{equation}\label{eq:b0}
    \left\{
        \begin{aligned}
            &\partial_t n^{b,0} + \sum_{k=0}^1 \left( (u^{E,-k})_p \cdot \widehat{\nabla} n^{b,k} + (u^{b,-k})_p \cdot \widehat{\nabla} n^{E,k} + (u^{b,-k})_p \cdot \widehat{\nabla} n^{b,k} \right) \\
            % &-\widehat{\nabla} \cdot \left( \widehat{\nabla} n^{b,0} - \sum_{k=0}^2 \left( n^{E,0-k} \widehat{\nabla} c^{b,k} + n^{b,0-k} \widehat{\nabla} c^{E,k} + n^{b,0-k} \widehat{\nabla} c^{b,k} \right) \right)_{pp} = 0,\\
            &\quad- \partial_x\left( \partial_x n^{b,0} - \left( n^{E,0} \partial_x c^{b,0} + n^{b,0} \partial_x c^{E,0} + n^{b,0} \partial_x c^{b,0} \right) \right) \\
            &\quad -\partial_z \left( \partial_z n^{b,2} - \sum_{k=0}^2 \left( n^{E,2-k} \partial_z c^{b,k} + n^{b,2-k} \partial_z c^{E,k} + n^{b,2-k} \partial_z c^{b,k} \right) \right) = 0, \\
            &\partial_t c^{b,0} + \sum_{k=0}^1 \left( (u^{E,-k})_p \cdot \widehat{\nabla} c^{b,k} + (u^{b,-k})_p \cdot \widehat{\nabla} c^{E,k} + (u^{b,-k})_p \cdot \widehat{\nabla} c^{b,k} \right) - \partial_z^2 c^{b,0} \\
            &\quad = - \left( c^{E,0} n^{b,0} + c^{b,0} n^{E,0} + c^{b,0} n^{b,0} \right), \\
            &\partial_t u^{b,0} + \sum_{k=0}^1 \left( (u^{E,-k})_p \cdot \widehat{\nabla} u^{b,k} + (u^{b,-k})_p \cdot \widehat{\nabla} u^{E,k} + (u^{b,-k})_p \cdot \widehat{\nabla} u^{b,k} \right) \\
            &\quad + (\widehat{\nabla} p^{b,0})_p - \partial_z^2 u^{b,0}
            = \begin{pmatrix} 0 \\ n^{b,0} \end{pmatrix}, \\
            &\widehat{\nabla} \cdot (u^{b,0})_p = 0,
        \end{aligned}
    \right.
\end{equation}
% with

% \begin{equation}
%     \left\{
%         \begin{aligned}
%             &\partial_y n^{e,0} + \partial_z n^{b,1} = 0 \quad \text{on}~ y=z=0, \\
%             &\partial_y c^{e,0} + \partial_z c^{b,1} = 0 \quad \text{on}~ y=z=0, \\
%             &\partial_y u_1^{e,0} + \partial_z u_1^{b,1} = 0 \quad \text{on}~ y=z=0, \\
%             &u_2^{e,0} + u_2^{b,0} = 0 \quad \text{on}~ y=z=0.
%         \end{aligned}
%     \right.
% \end{equation}

Combining \eqref{eq:b0}$_3$, \eqref{eq:b0}$_4$, \eqref{u_2^b0=0}, \eqref{p^b0=0} and \eqref{bound},  $(u_1^{b,0}, u_2^{b,1})$ satisfies the following equations:
\begin{equation}\label{eq:u_1^b0}
    \left\{
        \begin{aligned}
            &\partial_t u_1^{b,0} + \left( (u^{E,0})_p \cdot \widehat{\nabla} u_1^{b,0} + u_1^{b,0} \partial_x u_1^{E,0} + (u^{b,0})_p \cdot \widehat{\nabla} u_1^{b,0} \right) - \partial_z^2 u_1^{b,0} = 0,\\
            &\partial_x u_1^{b,0} + \partial_z u_2^{b,1} = 0,\\
            &\partial_z u_1^{b,0} = 0, \quad \text{on}~ z=0, \\
            &\lim_{z\to \infty} u_1^{b,0} = \lim_{z\to \infty} u_2^{b,1} =0,
        \end{aligned}
    \right.
\end{equation}
with the initial data $u_1^{b,0}(0,x,y)=0.$
By the uniqueness of the solution and the energy method with \eqref{bound}, we have
\begin{equation}\label{u_1^b0=u_2^b1=0}
    u_1^{b,0} = u_2^{b,1} = 0,
\end{equation}
which along with \eqref{u_2^b0=0}, \eqref{eq:n^b0} and \eqref{bound} implies that $c^{b,0}$ satisfies:
\begin{equation}\label{eq:c^b0}
    \left\{
        \begin{aligned}
            &\partial_t c^{b,0} + (u^{E,0})_p \cdot \widehat{\nabla} c^{b,0} - \partial_z^2 c^{b,0} = - \left( c^{E,0} n^{b,0} + c^{b,0} n^{E,0} + c^{b,0} n^{b,0} \right),\\
            &n^{b,0} = n^{E,0}\left( e^{c^{b,0}} - 1 \right),\\
            &\partial_z c^{b,0} = 0 \quad \text{on}~ z=0,\\
            &\lim_{z\to \infty} c^{b,0} = 0,
        \end{aligned}
    \right.
\end{equation}
with the initial data $c^{b,0}(0,x,y)=0.$ There exists a local non-negative solution of \eqref{eq:c^b0} by the standard contraction mapping principle in $L^\infty  L^2\cap L^\infty$. Indeed, 
 taking the $L^2$ inner product with $c^{b,0}$ on both sides of \eqref{eq:c^b0}$_1$, we have
    \begin{align*}
        &\frac{1}{2} \frac{d}{dt} \|c^{b,0}\|_2^2 + \|\partial_z c^{b,0}\|_2^2 \\
        &\quad \le - \int_R n^{E,0} c^{E,0} \int_0^\infty \left(e^{c^{b,0}} - 1\right) c^{b,0} \, dy dx - \int_R n^{E,0} \int_0^\infty e^{c^{b,0}} \left(c^{b,0}\right)^2 \, dy dx,
    \end{align*}
    due to \(n_{in},c_{in}\geq 0\) and \(n^{e,0},c^{e,0}\geq 0\) by \eqref{eq:e0}, 
    which implies 
%     Due to \(n^0,c^0>0\), we have \(n^{e,0},c^{e,0}>0\).
% Then we have
\begin{equation}\label{nc^b0=0}
    c^{b,0} = 0, n^{b,0} = 0.
\end{equation}
Using \eqref{nc^b0=0},  \eqref{u_2^b0=0}, \eqref{bound:u_2^e0}, \eqref{u_1^b0=u_2^b1=0} and \eqref{eq:b0} we get
\beno -\partial_z \left( \partial_z n^{b,2} - \sum_{k=0}^2 \left( n^{E,2-k} \partial_z c^{b,k} + n^{b,2-k} \partial_z c^{E,k} + n^{b,2-k} \partial_z c^{b,k} \right) \right) = 0\eeno
Integrating both sides of the equation with respect to $z$ from $z$ to $\infty$, we have
\ben\label{eq:partial_z n^b2}
    \partial_z n^{b,2} = \left( n^{E,1} \partial_z c^{b,1} + n^{b,1} \partial_z c^{E,1} + n^{b,1} \partial_z c^{b,1} \right) + n^{E,0} \partial_z c^{b,2}.
\een
Letting $z=0$ and using \eqref{bound} and \eqref{def:Ej}, we can deduce:
\ben\label{bound:n^e1}
    \partial_y n^{e,1} = n^{e,1} \partial_y c^{e,0} + n^{e,0} \partial_y c^{e,1} \quad \text{on}~ y=0.
\een

% \begin{remark}
%     See the appendix for the proof.
% \end{remark}
Using \eqref{nc^b0=0} again, recall 
  \eqref{eq:b0}$_3$ and we have $p^{b,1}$ satisfies:
\begin{equation*}
    \left\{
        \begin{aligned}
            &\partial_z p^{b,1} = 0,\\
            &\lim_{z\to \infty} p^{b,1} = 0,
        \end{aligned}
    \right.
\end{equation*}
which implies
\ben\label{p^b1=0}
    p^{b,1} = 0.
\een
Collecting this, \eqref{u_2^b0=0}, \eqref{p^b0=0}, \eqref{u_1^b0=u_2^b1=0},  \eqref{nc^b0=0}, and \eqref{p^b1=0}, we have 
\ben\label{eq:vanishing-layer}
u_2^{b,0}=u_1^{b,0}=u_2^{b,1}=p^{b,0}=p^{b,1}=n^{b,0}=
c^{b,0}=0.
\een

\textbf{Step IV. The order $O(\varepsilon^{1})$ terms.} 

% From the order $O(\varepsilon^{1})$ terms, we get:
% So, the $O(\varepsilon^1)$-order term of the outer solution satisfies the equations:
First, for the domain far from the boundary,  it follows from \eqref{eq:ej} and \eqref{bound:n^e1} that  
\begin{equation}\label{eq:e1}
    \left\{
        \begin{aligned}
            &\partial_t n^{e,1} + \sum_{k=0}^1 u^{e,1-k} \cdot \nabla n^{e,k} - \nabla \cdot \left( \nabla n^{e,1} - \sum_{k=0}^1 n^{e,1-k} \nabla c^{e,k} \right) = 0, \\
            &\partial_t c^{e,1} + \sum_{k=0}^1 u^{e,1-k} \cdot \nabla c^{e,k} + \sum_{k=0}^1 c^{e,1-k} n^{e,k} = 0, \\
            &\partial_t u^{e,1} + \sum_{k=0}^1 u^{e,k} \cdot \nabla u^{e,1-k} + \nabla p^{e,1} = \begin{pmatrix} 0 \\ n^{e,1} \end{pmatrix}, \\
            &\nabla \cdot u^{e,1} = 0, \\
            &\partial_y n^{e,1} - \sum_{k=0}^1 n^{e,1-k}\partial_y c^{e,k} = 0,~ \quad \text{on $y = 0$},\\
            &(n^{e,1},c^{e,1},u^{e,1})(0,x,y)=0.
        \end{aligned}
    \right.
\end{equation}
Moreover, by \eqref{bound} and \eqref{u_1^b0=u_2^b1=0} we have 
\ben\label{eq:u2e1}
u_2^{e,1} = 0,  \quad \text{on $y = 0$}. \een
Thus we have 
\begin{equation}\label{ncu^e1=0}
    n^{e,1},c^{e,1},u^{e,1} = 0.
\end{equation}
Next, for near the boundary by \eqref{eq:bj} we have
\vspace{-\abovedisplayskip}
\begin{equation}\label{eq:b1}
    \left\{
        \begin{aligned}
            &\partial_t n^{b,1} + \sum_{k=0}^2 \left( (u^{E,1-k})_p \cdot \widehat{\nabla} n^{b,k} + (u^{b,1-k})_p \cdot \widehat{\nabla} n^{E,k} + (u^{b,1-k})_p \cdot \widehat{\nabla} n^{b,k} \right) \\
            &\quad - \widehat{\nabla} \cdot \left( \widehat{\nabla} n^{b,1} - \sum_{k=0}^1 \left( n^{E,1-k} \widehat{\nabla} c^{b,k} + n^{b,1-k} \widehat{\nabla} c^{E,k} + n^{b,1-k} \widehat{\nabla} c^{b,k} \right) \right)_{pp} = 0, \\
            &\partial_t c^{b,1} + \sum_{k=0}^2 \left( (u^{E,1-k})_p \cdot \widehat{\nabla} c^{b,k} + (u^{b,1-k})_p \cdot \widehat{\nabla} c^{E,k} + (u^{b,1-k})_p \cdot \widehat{\nabla} c^{b,k} \right) - \partial_z^2 c^{b,1} \\
            &\quad = - \sum_{k=0}^1 \left( c^{E,1-k} n^{b,k} + c^{b,1-k} n^{E,k} + c^{b,1-k} n^{b,k} \right), \\
            &\partial_t u^{b,1} + \sum_{k=0}^2 \left( (u^{E,1-k})_p \cdot \widehat{\nabla} u^{b,k} + (u^{b,1-k})_p \cdot \widehat{\nabla} u^{E,k} + (u^{b,1-k})_p \cdot \widehat{\nabla} u^{b,k} \right) \\
            &\quad + (\widehat{\nabla} p^{b,1})_p - \partial_z^2 u^{b,1}= \begin{pmatrix} 0 \\ n^{b,1} \end{pmatrix}, \\
            &\widehat{\nabla} \cdot (u^{b,1})_p = 0,
        \end{aligned}
    \right.
\end{equation}
\vspace{-\belowdisplayskip}

Combining \eqref{eq:b1}, the boundary condition \eqref{bound}, \eqref{eq:vanishing-layer}, \eqref{eq:u2e1} and \eqref{ncu^e1=0}, we deduce that $u_1^{b,1}, u_2^{b,2}$ satisfy
\begin{equation}\label{eq:u_1^b1}
    \left\{
        \begin{aligned}
            &\partial_t u_1^{b,1} + \left( u^{E,0} \right)_p \cdot \widehat{\nabla} u_1^{b,1} + u_1^{b,1} \partial_x u_1^{E,0} - \partial_z^2 u_1^{b,1} = 0, \\
            &\partial_x u_1^{b,1} + \partial_z u_2^{b,2} = 0, \\
            &\partial_z u_1^{b,1} = - \overline{\partial_y u_1^{e,0}} \quad \text{on}~ z=0, \\
            &\lim_{z\to \infty} u_1^{b,1} = \lim_{z\to \infty} u_2^{b,2} = 0.
        \end{aligned}
    \right.
\end{equation}
Similarly, by \eqref{eq:partial_z n^b1} we have
 \beno
 \partial_z n^{b,1} - n^{E,0} \partial_z c^{b,1} = 0,\eeno
 which implies
  \beno
  n^{b,1} - n^{E,0}  c^{b,1} = 0
  \eeno
by applying the integration with respect to $z$. Using this,  $\eqref{eq:b1}_2$  
 and  \eqref{bound} we have
 % \beno
 % \partial_t c^{b,1} + u_1^{b,1} \partial_x c^{E,0} + (u^{E,0})_p \cdot \widehat{\nabla} c^{b,1} - \partial_z^2 c^{b,1} = - \left( c^{b,1} n^{E,0} + c^{E,0} n^{b,1} \right)\eeno
% $n^{b,1}, c^{b,1}$ satisfy:
\begin{equation}\label{eq:c^b1,n^b1}
    \left\{
        \begin{aligned}
            &\partial_t c^{b,1} + u_1^{b,1} \partial_x c^{E,0} + (u^{E,0})_p \cdot \widehat{\nabla} c^{b,1} - \partial_z^2 c^{b,1} = - \left( c^{b,1} n^{E,0} + c^{E,0} n^{b,1} \right), \\
            & n^{b,1} - n^{E,0}  c^{b,1} = 0, \\
            &\partial_z c^{b,1} = - \overline{\partial_y c^{e,0}} \quad \text{on}~ z=0, \\
            &\lim_{z\to \infty} c^{b,1} = \lim_{z\to \infty} n^{b,1} = 0.
        \end{aligned}
    \right.
\end{equation}
Moreover, by \eqref{eq:b1} $p^{b,2}$ satisfies
\begin{equation*}\label{}
    \left\{
        \begin{aligned}
            &\partial_z p^{b,2} = n^{b,1}, \\
            &\lim_{z\to\infty} p^{b,2} = 0,
        \end{aligned}
    \right.
\end{equation*}
which implies 
\begin{equation}\label{eq:pb2}
    p^{b,2} = -\int_z^\infty n^{b,1}(t,x,\zeta) d\zeta.
\end{equation}
% Then, using \eqref{eq:b1}$_1$ and \eqref{展开_2边界}, we can get the boundary condition of \eqref{eq:2e}:
% \begin{equation*}
%     \partial_y n^{e,2} - \sum_{k=0,2} n^{e,2-k}\partial_y c^{e,k} = -\int_0^\infty \Gamma(t,x,\zeta)d\zeta \quad \text{on}~ y=0,
% \end{equation*}
% where
% \begin{equation*}
%     \Gamma = \partial_t n^{b,1} + u_1^{b,1} \partial_x n^{E,0} + \left( u^{E,0} \right)_p \cdot \widehat{\nabla} n^{b,1} -\partial_x \left( \partial_x n^{b,1} - n^{b,1} \partial_x c^{E,0} - n^{E,0}\partial_x c^{b,1}\right) .
% \end{equation*}
% \begin{remark}
%     The proof is similar. Integrating both sides of \eqref{eq:b1}$_1$, we have
%     \begin{align*}
%         \partial_z n^{b,3}
%         &= \left(n^{E,0} \partial_z c^{b,3}\right)
%           + \left(n^{E,1} \partial_z c^{b,2} + n^{b,1} \partial_z c^{E,2} + n^{b,1} \partial_z c^{b,2}\right) \\
%           &\quad+ \left(n^{E,2} \partial_z c^{b,1} + n^{b,2} \partial_z c^{E,1} + n^{b,2} \partial_z c^{b,1}\right)
%           + \int_z^\infty \Gamma(\cdot,\zeta)\,d\zeta.
%     \end{align*}
%     Then by \eqref{展开_2边界}, the proof is completed.
% \end{remark}

\textbf{Step V. The order $O(\varepsilon^{2})$ terms.} 
%From the order $O(\varepsilon^{2})$ terms, we get:
% \begin{equation}\label{eq:2e}
%     \left\{
%         \begin{aligned}
%             &\partial_t n^{e,2} + \sum_{k=0}^2 u^{e,k} \cdot \nabla n^{e,2-k} - \nabla \cdot \left( \nabla n^{e,2} - \sum_{k=0}^2 n^{e,2-k} \nabla c^{e,k} \right) = 0, \\
%             &\partial_t c^{e,2} + \sum_{k=0}^2 u^{e,k} \cdot \nabla c^{e,2-k} - \Delta c^{e,0} = -\sum_{k=0}^2 c^{e,k} n^{e,2-k}, \\
%             &\partial_t u^{e,2} + \sum_{k=0}^2 u^{e,2-k} \cdot \nabla u^{e,k} + \nabla p^{e,2} - \Delta u^{e,0} = \begin{pmatrix} 0 \\ n^{e,2} \end{pmatrix}, \\
%             &\nabla \cdot u^{e,2} = 0,
%         \end{aligned}
%     \right.
% \end{equation}
% Then, for the Prandtl-type equation, we have

% \begin{equation}\label{展开_2边界}
%     \left\{
%         \begin{aligned}
%             &\partial_y n^{e,2} + \partial_z n^{b,3} = 0 \quad \text{on}~ y=z=0, \\
%             &\partial_y c^{e,2} + \partial_z c^{b,3} = 0 \quad \text{on}~ y=z=0, \\
%             &\partial_y u_1^{e,2} + \partial_z u_1^{b,3} = 0 \quad \text{on}~ y=z=0, \\
%             &u_2^{e,2} + u_2^{b,2} = 0 \quad \text{on}~ y=z=0.
%         \end{aligned}
%     \right.
% \end{equation}
% with
\begin{equation}\label{eq:b2}
    \left\{
        \begin{aligned}
            &\partial_t n^{b,2} + \sum_{k=0}^3 \left( (u^{E,2-k})_p \cdot \widehat{\nabla} n^{b,k} + (u^{b,2-k})_p \cdot \widehat{\nabla} n^{E,k} + (u^{b,2-k})_p \cdot \widehat{\nabla} n^{b,k} \right) \\
            &- \widehat{\nabla} \cdot \left( \widehat{\nabla} n^{b,2} - \sum_{k=0}^2 \left( n^{E,2-k} \widehat{\nabla} c^{b,k} + n^{b,2-k} \widehat{\nabla} c^{E,k} + n^{b,2-k} \widehat{\nabla} c^{b,k} \right) \right)_{pp} = 0, \\
            &\partial_t c^{b,2} + \sum_{k=0}^3 \left( (u^{E,2-k})_p \cdot \widehat{\nabla} c^{b,k} + (u^{b,2-k})_p \cdot \widehat{\nabla} c^{E,k} + (u^{b,2-k})_p \cdot \widehat{\nabla} c^{b,k} \right) \\
            &- \widehat{\nabla} \cdot (\widehat{\nabla} c^{b,0})_{pp} = - \sum_{k=0}^2 \left( c^{E,2-k} n^{b,k} + c^{b,2-k} n^{E,k} + c^{b,2-k} n^{b,k} \right), \\
            &\partial_t u^{b,2} + \sum_{k=0}^3 \left( (u^{E,2-k})_p \cdot \widehat{\nabla} u^{b,k} + (u^{b,2-k})_p \cdot \widehat{\nabla} u^{E,k} + (u^{b,2-k})_p \cdot \widehat{\nabla} u^{b,k} \right) \\
            & + (\widehat{\nabla} p^{b,2})_p - \widehat{\nabla} \cdot (\widehat{\nabla} u^{b,0})_{pp} = \begin{pmatrix} 0 \\ n^{b,2} \end{pmatrix}, \\
            &\widehat{\nabla} \cdot (u^{b,2})_p = 0,
        \end{aligned}
    \right.
\end{equation}
% So, the $O(\varepsilon^2)$-order term of the outer solution satisfies the equations:
% \begin{equation}\label{方程_2e}
%     \left\{
%     \begin{aligned}
%         &\partial_t n^{e,2} + \sum_{k=0,2} u^{e,k} \cdot \nabla n^{e,2-k} - \nabla \cdot \left( \nabla n^{e,2} - \sum_{k=0,2} n^{e,2-k} \nabla c^{e,k} \right) = 0, \\
%         &\partial_t c^{e,2} + \sum_{k=0,2} u^{e,k} \cdot \nabla c^{e,2-k} - \Delta c^{e,0} + \sum_{k=0,2} c^{e,k} n^{e,2-k} = 0, \\
%         &\partial_t u^{e,2} + \sum_{k=0,2} u^{e,2-k} \cdot \nabla u^{e,k} + \nabla p^{e,2} - \Delta u^{e,0} = \begin{pmatrix} 0 \\ n^{e,2} \end{pmatrix}, \\
%         &\nabla \cdot u^{e,2} = 0, \\
%         &\partial_y n^{e,2} - \sum_{k=0,2} n^{e,2-k}\partial_y c^{e,k} = -\int_0^\infty \Gamma(t,x,\zeta)d\zeta \quad \text{on}~ y=0, \\
%         &u_2^{e,2} = -\int_0^\infty \partial_x u_1^{b,1}(t,x,\zeta) d\zeta \quad \text{on}~ y=0.
%     \end{aligned}
%     \right.
% \end{equation}
thus
$u_1^{b,2}$ satisfy
\begin{equation}\label{eq:u_1^b2}
    \left\{
        \begin{aligned}
            &\partial_t u_1^{b,2} + \Big( (u^{E,1})_p \cdot \widehat{\nabla} u_1^{b,1} + (u^{b,1})_p \cdot \widehat{\nabla} u_1^{E,1} + (u^{b,1})_p \cdot \widehat{\nabla} u_1^{b,1} \Big) \\
            &+ (u^{E,0})_p \cdot \widehat{\nabla} u_1^{b,2} + u_1^{b,2} \partial_x u_1^{E,0} + \partial_x p^{b,2} - \partial_z^2 u_1^{b,2} = 0, \\
            &\partial_z u_1^{b,2} = 0, \quad z=0, \\
            &\lim_{z\to \infty} u_1^{b,2} = 0,
        \end{aligned}
    \right.
\end{equation}
where we used \eqref{ncu^e1=0}, \eqref{eq:vanishing-layer} and  \eqref{bound}.

Moreover, by \eqref{eq:partial_z n^b2}, $\eqref{eq:b2}_2$ and \eqref{bound} 
$n^{b,2},c^{b,2}$ satisfy
\begin{equation}\label{eq:c^b2,n^b2}
    \left\{
        \begin{aligned}
            &\partial_t c^{b,2} + (u^{E,0})_p \cdot \widehat{\nabla} c^{b,2} + u_1^{b,2} \partial_x c^{E,0} - \partial_z^2 c^{b,2}\\
            &+ \left( (u^{E,1})_p \cdot \widehat{\nabla} c^{b,1} + (u^{b,1})_p \cdot \widehat{\nabla} c^{E,1} + (u^{b,1})_p \cdot \widehat{\nabla} c^{b,1} \right) \\
            & + c^{b,2} n^{E,0} + \left( c^{E,1} n^{b,1} + c^{b,1} n^{E,1} + c^{b,1} n^{b,1} \right) + c^{E,0} n^{b,2} = 0, \\
            &\partial_z n^{b,2} = \left( n^{E,1} \partial_z c^{b,1} + n^{b,1} \partial_z c^{E,1} + n^{b,1} \partial_z c^{b,1} \right) + n^{E,0} \partial_z c^{b,2}, \\
            &\partial_z c^{b,2} = 0 \quad \text{on}~ z=0, \\
            &\lim_{z\to \infty} c^{b,2} = \lim_{z\to\infty} n^{b,2} = 0.
        \end{aligned}
    \right.
\end{equation}

% Finally, our derivation process is depicted in the following figure.
% \begin{center}
% \begin{tikzpicture}[
%     box/.style={
%         rectangle, draw, 
%         minimum width=5cm, 
%         minimum height=1cm, 
%         align=center
%     },
%     >=stealth, line width=0.5pt
% ]

% % 第一行
% \node[box] (A1) at (0,0)  {$(u_2^{b,0}, p^{b,0})$};
% \node[box] (A2) at (7,0)  {$\partial_y n^{e,0} - n^{e,0}\partial_y c^{e,0}\big|_{y=0} = 0$};

% % 第二行
% \node[box] (B1) at (0,-2) {$(c^{b,0}, n^{b,0}, u_1^{b,0})$};
% \node[box] (B2) at (7,-2) {$(n^{e,0}, c^{e,0}, u^{e,0})$};

% % 第三行
% \node[box] (C1) at (0,-4) {$(u_2^{b,1}, p^{b,1})$};
% \node[box] (C2) at (7,-4) {$\partial_y n^{e,1} - \sum\limits_{k=0}^1 n^{e,1-k}\partial_y c^{e,k}\big|_{y=0} = 0$};

% % 第四行
% \node[box] (D1) at (0,-6) {$(n^{b,1}, c^{b,1}, u_1^{b,1})$};
% \node[box] (D2) at (7,-6) {$(n^{e,1}, c^{e,1}, u^{e,1})$};

% % 第五行
% \node[box] (E1) at (0,-8) {$(u_2^{b,2}, p^{b,2})$};

% % 第六行
% \node[box] (F1) at (0,-10){$(n^{b,2}, c^{b,2}, u_1^{b,2})$};

% % 箭头（全部是直线！）
% \draw[->] (A1.east) -- (A2.west);
% \draw[->] (A2.south) -- (B2.north);
% \draw[->] (B2.west) -- (B1.east);
% \draw[->] (B1.south) -- (C1.north);
% \draw[->] (C1.east) -- (C2.west);
% \draw[->] (C2.south) -- (D2.north);
% \draw[->] (D2.west) -- (D1.east);
% \draw[->] (D1.south) -- (E1.north);
% \draw[->] (E1.south) -- (F1.north);

% \end{tikzpicture}

% {\large 多尺度耦合求解流程图}\\[5pt]
% \end{center}

%Combination of the solutions
\subsection{Combination of the solutions} 
Based on the definition of the approximate solution \eqref{def:soluntion^a}, we have the following lemma:
% Let us define the approximate solution:
% \begin{equation}\label{}
%     \left\{
%     \begin{aligned}
%         n^a &= n^{e,0} + \varepsilon n^{b,1} + \varepsilon^2 n^{b,2}, \\
%         c^a &= c^{e,0} + \varepsilon c^{b,1} + \varepsilon^2 c^{b,2}, \\
%         u_1^a &= u_1^{e,0} + \varepsilon u_1^{b,1}, ~u_2^a = u_2^{e,0} + \varepsilon^2 u_2^{b,2}, \\
%         p^a &= p^{e,0} + \varepsilon^2 p^{b,2}.
%     \end{aligned}
%     \right.
% \end{equation}
\begin{lemma}
    Let
    \begin{equation}\label{eq:f-u2b2}
        f(t, x) = \int_{0}^{\infty} \partial_x u_1^{b,1}(t, x, z) \, dz= u_2^{b,2}(t, x, 0), 
    \end{equation}
    then $(n^a,c^a,u^a,p^a)$ satisfies:
    \begin{equation}
        \left\{
        \begin{aligned}\label{}
            &\partial_t n^a + u^a \cdot \nabla n^a - \nabla \cdot \left( \nabla n^a - n^a \nabla c^a \right) = -N, \\
            &\partial_t c^a + u^a \cdot \nabla c^a + c^a n^a - \varepsilon^2 \Delta c^a = -K, \\
            &\partial_t u^a + u^a \cdot \nabla u^a + \nabla p^a - \varepsilon^2 \Delta u^a - \begin{pmatrix} 0 \\ n^a \end{pmatrix} = -U, \\
            &\nabla \cdot u^a = 0, \\
            &\partial_y n^a = \partial_y c^a = \partial_y u_1^a = 0, \quad u_2^a = \varepsilon^2 f(t, x) \quad \text{on}~ y=0,
        \end{aligned}
        \right.
    \end{equation}
    where $N,K,U$ is given by:
    \begin{equation}\label{eq:N}
        \begin{aligned}
            -N =& \varepsilon \partial_t n^{b,1} + \varepsilon^2 \partial_t n^{b,2} + \left( u_1^{e,0} \partial_x (\varepsilon n^{b,1} + \varepsilon^2 n^{b,2}) + \varepsilon u_1^{b,1} \partial_x n^{e,0} + \varepsilon u_1^{b,1} \partial_x (\varepsilon n^{b,1} + \varepsilon^2 n^{b,2}) \right) \\
            &+\left( (u_2^{e,0} - u_2^{E,0}) \partial_z (n^{b,1} + \varepsilon n^{b,2}) + \varepsilon^2 u_2^{b,2} \partial_y n^{e,0} + \varepsilon^2 u_2^{b,2} \partial_z (n^{b,1} + \varepsilon n^{b,2}) \right) \\
            &- \partial_x \left( \partial_x (\varepsilon n^{b,1} + \varepsilon^2 n^{b,2}) - n^{e,0} \partial_x (\varepsilon c^{b,1} + \varepsilon^2 c^{b,2}) - (\varepsilon n^{b,1} + \varepsilon^2 n^{b,2}) \partial_x c^{e,0} \right)\\
            &- (\varepsilon n^{b,1} + \varepsilon^2 n^{b,2}) \partial_x (\varepsilon c^{b,1} + \varepsilon^2 c^{b,2}) \\
            &- \partial_y \left( \left( y \overline{\partial_y n^{e,0}} + n^{E,0} - n^{e,0} \right) \partial_z c^{b,1} + \varepsilon n^{b,1} \left( \overline{\partial_y c^{e,0}} - \partial_y c^{e,0} \right) \right) \\
            &- \partial_y \left( \varepsilon (n^{E,0} - n^{e,0}) \partial_z c^{b,2} - \varepsilon^2 n^{b,2} \left( \partial_y c^{e,0} + \partial_z c^{b,1} \right) - \varepsilon (\varepsilon n^{b,1} + \varepsilon^2 n^{b,2}) \partial_z c^{b,2} \right),
        \end{aligned}
    \end{equation}
    and
    \begin{equation}\label{eq:K}
        \begin{aligned}
            -K =& \varepsilon (u_1^{e,0} - \overline{u_1^{e,0}}) \partial_x c^{b,1} + \varepsilon^2 (u_1^{e,0} - \overline{u_1^{e,0}}) \partial_x c^{b,2} + (u_2^{e,0} - y \overline{\partial_y u_2^{e,0}}) \partial_z c^{b,1} \\
            &+  \varepsilon(u_2^{e,0} - y \overline{\partial_y u_2^{e,0}}) \partial_z c^{b,2} - \varepsilon^2 \partial_y^2 c^{e,0} + \varepsilon (c^{e,0} - \overline{c^{e,0}}) n^{b,1} \\
            &+ \varepsilon^2 (c^{e,0} - c^{E,0}) n^{b,2} +  \varepsilon u_1^{b,1} \partial_x (c^{e,0} - \overline{c^{e,0}}) + \varepsilon^3 u_1^{b,1} \partial_x c^{b,2} \\
            &+ \varepsilon^2 u_2^{b,2} (\partial_y c^{e,0} - \overline{\partial_y c^{e,0}}) + \varepsilon^3 u_2^{b,2} \partial_z c^{b,2} - \varepsilon^2 \partial_x^2 c^a + \varepsilon^3 c^{b,1} n^{b,2} \\
            &+ \varepsilon^2 c^{b,2} (\varepsilon n^{b,1} + \varepsilon^2 n^{b,2})+\varepsilon c^{b,1}(n^{e,0}-n^{E,0})+\varepsilon^2 c^{b,2}(n^{e,0}-n^{E,0}) \\
            &- \varepsilon^2 \left( (u^{E,1})_p \cdot \widehat{\nabla} c^{b,1} + u_1^{b,1} \partial_x c^{E,1} + u_1^{b,2} \partial_x c^{E,0} + c^{E,1} n^{b,1} + c^{b,1} n^{E,1} \right),
        \end{aligned}
    \end{equation}
    and
    \begin{equation}\label{eq:U1}
        \begin{aligned}
            -U_1 =& \varepsilon (u_1^{e,0} - \overline{u_1^{e,0}}) \partial_x u_1^{b,1} + (u_2^{e,0} - y \overline{\partial_y u_2^{e,0}}) \partial_z u_1^{b,1} + \varepsilon^2 \partial_x p^{b,2} - \varepsilon^3 \partial_x^2 u_1^{b,1} - \varepsilon^2 \Delta u_1^{e,0}  \\
            &+ \varepsilon u_1^{b,1} \partial_x (u_1^{e,0} - \overline{u_1^{e,0}}) + \varepsilon^2 u_2^{b,2} \partial_y u_1^a+\varepsilon^2u^{b,1}_1\partial_x u^{b,1}_1,
        \end{aligned}
    \end{equation}
    and
    \begin{equation}\label{eq:U2}
        \begin{aligned}
            -U_2 =& \varepsilon^2 \partial_t u_2^{b,2} + \left( \varepsilon u_1^{b,1} \partial_x ( u_2^{e,0} - u_2^{E,0} ) + \varepsilon^2 u_1^{e,0} \partial_x u_2^{b,2} + \varepsilon^3 u_1^{b,1} \partial_x u_2^{b,2} \right) - \varepsilon^2 \Delta u_2^a \\
            &+ \left( \varepsilon^2 u_2^{b,2} \partial_y u_2^{e,0} + \varepsilon ( u_2^{e,0} -u_2^{E,0} ) \partial_z u_2^{b,2} + \varepsilon^3 u_2^{b,2} \partial_z u_2^{b,2} \right)- \varepsilon^2 n^{b,2}.
        \end{aligned}
    \end{equation}
\end{lemma}

\begin{proof}
   For the term $N$, by \eqref{eq:e0} we get
   \begin{align*}
        & \partial_t \left( n^{e,0} + \varepsilon n^{b,1} + \varepsilon^2 n^{b,2} \right) \\
        & \quad + \left\{ \left( u_1^{e,0} + \varepsilon u_1^{b,1} \right) \partial_x \left( n^{e,0} + \varepsilon n^{b,1} + \varepsilon^2 n^{b,2} \right) + \left( u_2^{e,0} + \varepsilon^2 u_2^{b,2} \right) \partial_y \left( n^{e,0} + \varepsilon n^{b,1} + \varepsilon^2 n^{b,2} \right) \right\} \\
        & \quad - \nabla \cdot \biggl\{ \nabla \left( n^{e,0} + \varepsilon n^{b,1} + \varepsilon^2 n^{b,2} \right) - \left( n^{e,0} + \varepsilon n^{b,1} + \varepsilon^2 n^{b,2} \right) \nabla \left( c^{e,0} + \varepsilon c^{b,1} + \varepsilon^2 c^{b,2} \right) \biggr\} \\
        &= \varepsilon \partial_t n^{b,1} + \varepsilon^2 \partial_t n^{b,2} \\
        & \quad + \left( u_1^{e,0} \partial_x \left( \varepsilon n^{b,1} + \varepsilon^2 n^{b,2} \right) + \varepsilon u_1^{b,1} \partial_x n^{e,0} + \varepsilon u_1^{b,1} \partial_x \left( \varepsilon n^{b,1} + \varepsilon^2 n^{b,2} \right) \right) \\
        & \quad + \left( u_2^{e,0} \partial_y \left( \varepsilon n^{b,1} + \varepsilon^2 n^{b,2} \right) + \varepsilon^2 u_2^{b,2} \partial_y n^{e,0} + \varepsilon^2 u_2^{b,2} \partial_y \left( \varepsilon n^{b,1} + \varepsilon^2 n^{b,2} \right) \right) \\
        & \quad - \nabla \cdot \biggl\{ \nabla \left( \varepsilon n^{b,1} + \varepsilon^2 n^{b,2} \right) - n^{e,0} \nabla \left( \varepsilon c^{b,1} + \varepsilon^2 c^{b,2} \right) - \left( \varepsilon n^{b,1} + \varepsilon^2 n^{b,2} \right) \nabla c^{e,0} \\
        &\quad - \left( \varepsilon n^{b,1} + \varepsilon^2 n^{b,2} \right) \nabla \left( \varepsilon c^{b,1} + \varepsilon^2 c^{b,2} \right) \biggr\}.
    \end{align*}
    Furthermore, by $\eqref{eq:c^b1,n^b1}_2$ and $\eqref{eq:c^b2,n^b2}_2$ we have
    \begin{align*}
        &-\partial_y \biggl\{ \partial_y \left( \varepsilon n^{b,1} + \varepsilon^2 n^{b,2} \right) - n^{e,0} \partial_y \left( \varepsilon c^{b,1} + \varepsilon^2 c^{b,2} \right) - \left( \varepsilon n^{b,1} + \varepsilon^2 n^{b,2} \right) \partial_y c^{e,0} \\
        &\quad - \left( \varepsilon n^{b,1} + \varepsilon^2 n^{b,2} \right) \partial_y \left( \varepsilon c^{b,1} + \varepsilon^2 c^{b,2} \right) \biggr\} \\
        &= - \partial_y \biggl\{ \left( n^{E,0} \partial_z c^{b,1} + \varepsilon \left( n^{E,1} \partial_z c^{b,1} + n^{b,1} \partial_z c^{E,1} + n^{b,1} \partial_z c^{b,1} + n^{E,0} \partial_z c^{b,2} \right) \right) \\
        &\quad- n^{e,0} \partial_z \left( c^{b,1} + \varepsilon c^{b,2} \right) - \left( \varepsilon n^{b,1} + \varepsilon^2 n^{b,2} \right) \partial_y c^{e,0} - \left( \varepsilon n^{b,1} + \varepsilon^2 n^{b,2} \right) \partial_z \left( c^{b,1} + \varepsilon c^{b,2} \right) \biggr\} \\
        &= -\partial_y \biggl\{ \left(y\overline{\partial_y n^{e,0}} + n^{E,0} - n^{e,0}\right) \partial_z c^{b,1} + \varepsilon n^{b,1} \left( \overline{\partial_y c^{e,0}} - \partial_y c^{e,0} \right) + \varepsilon \left( n^{E,0} - n^{e,0} \right) \partial_z c^{b,2}\\
        &\quad- \varepsilon^2 n^{b,2} \left( \partial_y c^{e,0} + \partial_z c^{b,1} \right) - \varepsilon \left( \varepsilon n^{b,1} + \varepsilon^2 n^{b,2} \right) \partial_z c^{b,2} \biggr\},
    \end{align*}
    where we used \eqref{ncu^e1=0}, \eqref{eq:u2e1} and \eqref{def:Ej} in the last equality. Then the proof of \eqref{eq:N} is complete due to \eqref{bound:u_2^e0}.

    For the term $K$, through $\eqref{eq:e0}_2$, $\eqref{eq:c^b1,n^b1}_1$ and  $\eqref{eq:c^b2,n^b2}_1$ , it shows that
    \begin{align*}
        &\partial_t c^{e,0} + u^{e,0} \cdot \nabla c^{e,0} + c^{e,0} n^{e,0} \\
        &+ \varepsilon \Bigl\{ \partial_t c^{b,1} + \left( u_1^{E,0} \partial_x c^{b,1} + z \overline{\partial_y u_2^{e,0}} \partial_z c^{b,1} \right) + u_1^{b,1} \partial_x c^{E,0} - \partial_z^2 c^{b,1} \\
        &+ \left( c^{E,0} n^{b,1} + c^{b,1} n^{E,0} \right) \Bigr\} + \varepsilon^2 \Bigl\{ \partial_t c^{b,2} + (u^{E,0})_p \cdot \widehat{\nabla} c^{b,2} \\
        &+ \left( (u^{E,1})_p \cdot \widehat{\nabla} c^{b,1} + (u^{b,1})_p \cdot \widehat{\nabla} c^{E,1} + (u^{b,1})_p \cdot \widehat{\nabla} c^{b,1} \right) + u_1^{b,2} \partial_x c^{E,0} - \partial_z^2 c^{b,2} \\
        &+ \left( c^{b,2} n^{E,0} + (c^{E,1} n^{b,1} + c^{b,1} n^{E,1} + c^{b,1} n^{b,1}) + c^{E,0} n^{b,2} \right) \Bigr\} = 0.
    \end{align*}
    Then, by adding the required terms such as $\varepsilon u_1^{e,0}\partial_x c^{b,1},u_2^{e,0}\partial_z c^{b,1}$, combining the corresponding terms defined in \eqref{def:soluntion^a} and moving the remainder to the right-hand side of the equation, we have
    \begin{align*}
        &\partial_t c^a + u_1^a \partial_x c^a + u_2^a \partial_y c^a + c^a n^a - \varepsilon^2 \Delta c^a \\
        &= \varepsilon\left(u_1^{e,0} - u_1^{E,0}\right)\partial_x c^{b,1} 
        + \varepsilon^2\left(u_1^{e,0} - u_1^{E,0}\right)\partial_x c^{b,2} 
        + \left(u_2^{e,0} - \varepsilon z\overline{\partial_y u_2^{e,0}}\right)\partial_z c^{b,1}\\
        &\quad + \varepsilon\left(u_2^{e,0} - \varepsilon z\overline{\partial_y u_2^{e,0}}\right)\partial_z c^{b,2} 
        - \varepsilon^2 \partial_y^2 c^{e,0} 
        + \varepsilon\left(c^{e,0} - c^{E,0}\right)n^{b,1}+ \varepsilon^2\left(c^{e,0} - c^{E,0}\right)n^{b,2}  \\
        &\quad
        + \varepsilon u_1^{b,1} \partial_x\left(c^{e,0} - c^{E,0}\right) 
        + \varepsilon^3 u_1^{b,1} \partial_x c^{b,2} + \varepsilon^2 u_2^{b,2}\left(\partial_y c^{e,0} - \overline{\partial_y c^{e,0}}\right) 
        + \varepsilon^3 u_2^{b,2} \partial_z c^{b,2}  \\
        &\quad - \varepsilon^2 \partial_x^2 c^a 
        + \varepsilon^3 c^{b,1} n^{b,2}+ \varepsilon^2 c^{b,2}\left(\varepsilon n^{b,1} + \varepsilon^2 n^{b,2}\right) 
        + \varepsilon c^{b,1}\left(n^{e,0} - n^{E,0}\right) 
        + \varepsilon^2 c^{b,2}\left(n^{e,0} - n^{E,0}\right) \\
        &\quad- \varepsilon^2 \left\{ \left((u^{E,1})_p \cdot \widehat{\nabla} c^{b,1} + u_1^{b,1} \partial_x c^{E,1}\right) 
        + u_1^{b,2} \partial_x c^{E,0} 
        + \left(c^{E,1} n^{b,1} + c^{b,1} n^{E,1}\right) \right\},
    \end{align*}
    which along with \eqref{eq:u2e1} complete the proof of \eqref{eq:K}.

    The processes of deriving the formula  $U_1$ and $U_2$ are similar by using \eqref{eq:e0}, \eqref{eq:u_1^b1} and \eqref{eq:pb2}, and we omitted it. 
\end{proof}

% \begin{proof}

% \eqref{eq:e0}
% \eqref{eq:c^b1,n^b1}

% \eqref{eq:c^b2,n^b2}

% \end{proof}

\section{Energy estimates}\label{energy estimates}

%Basic Analytical Tools
\subsection{Some important lemmas}

The following lemmas are frequently used in our proof.
\begin{lemma}[ \cite{KV2011}]\label{lem:separate}
    Let $\{x_\alpha\}_{\alpha \in \mathbb{N}^3}$ and $\{y_\beta \}_{\beta \in \mathbb{N}^3}$ be real numbers. Then we have
    \begin{align*}
        \sum_{|\alpha|=m} \sum_{|\beta|=j, \beta \leq \alpha} x_{\beta} y_{\alpha-\beta} = \left( \sum_{|\alpha|=j} x_{\alpha} \right) \left( \sum_{|\beta|=m-j} y_{\beta} \right).
    \end{align*}
\end{lemma}
% \begin{proof}
%     Proof can be found in \cite{KV2011}, thus we omit it.
% \end{proof}

\begin{lemma}\label{lem:embedding}
Let \(\left\| g \right\|_{Y^{1,m}}\leq C\), then the following inequality holds:
    \begin{equation}\label{eq:embedding}
        \sum_{\alpha_1 \leq 1, |\alpha| \leq m-1} \left\| \partial^\alpha g \right\|_\infty^2
        \leq C \sum_{\alpha_1 \leq 1,|\alpha| \leq m} \left\| \partial^\alpha g \right\|_2 \,
        \left\| \partial_y \partial^\alpha g \right\|_2.
    \end{equation}
\end{lemma}

\begin{proof} Note that 
   \begin{align*}
        \left\| g \right\|_{L^\infty} 
        \leq C\left\| \left\| g \right\|_{L_x^2}^{\frac12} \left\| \partial_x g \right\|_{L_x^2}^\frac12\right\|_{L_y^\infty} 
        \leq C  \left\| g\right\|_{L^2}^{\frac14}\left\| \partial_x g\right\|_{L^2}^{\frac14}\left\| \partial_y g\right\|_{L^2}^{\frac14}\left\| \partial_{xy}g\right\|_{L^2}^{\frac14}
        % &\leq C \left\|\left\| g \right\|_{L_x^2}\right\|_{L_y^2}^{\frac 1 2} \left\|\partial_y \left\| g \right\|_{L_x^2}\right\|_{L_y^2}^{\frac 1 2} + C \left\|\left\| \partial_x g \right\|_{L_x^2}\right\|_{L_y^2}^{\frac 1 2} \left\|\partial_y \left\| \partial_x g \right\|_{L_x^2}\right\|_{L_y^2}^{\frac 1 2},
    \end{align*}
    and the proof of \eqref{eq:embedding} is complete.

\end{proof}

\begin{lemma}\label{lem:psi}
    Assume that $\overline{g} = 0$ and $\psi$ is defined as in \eqref{eq:psi-def}, then for $1 < q \leq \infty$,
    \begin{equation}
        \left\| \frac{\partial^\alpha g}{\psi} \right\|_q
        \leq C_\delta \left( \left\| \partial_y \partial^\alpha g \right\|_{q} 
        + \left\| \partial^\alpha g \right\|_{q} \right).
    \end{equation}
\end{lemma}

\begin{proof}
    Denote
    \begin{equation*}
        \chi(y) = 
        \left\{
            \begin{aligned}
                &1, &0 \leq y \leq \frac 1 2 \\
                &0, &y \geq 1.
            \end{aligned}
        \right.
    \end{equation*}
    Then by Hardy inequality (see, for example,  \cite{T2004}) with $1<q<\infty$ and the Lipschitz continuity of $W^{1,\infty}$ when $q=\infty$,
   \begin{align*}
        \left\| \frac{\partial^\alpha g}{\psi} \right\|_{q} 
        &\leq \left\| \chi(y) \frac{y}{\psi} \right\|_{\infty} 
        \left\| \frac{\partial^\alpha g}{y} \right\|_{q} 
        +
        \left\| \frac{1 - \chi(y)}{\psi} \right\|_{\infty} 
        \left\| \partial^\alpha g \right\|_{q} \\
        &\leq \frac{C}{\delta} \left( \left\| \partial_y \partial^\alpha g \right\|_{q} 
        +
        \left\| \partial^\alpha g \right\|_{q} \right)
    \end{align*} 
\end{proof}

%The following lemma is an application of the lemma in 
For the conormal space $Y^{l,m}$,  we present the following Korn's type inequality.

%an alternative proof.
\begin{lemma}\label{lem:omega}
    Let $\nabla \cdot G = 0 $ and 
    $\omega^g = \partial_y g_1 - \partial_x g_2$, where
    $G = (g_1, g_2)$ and
    \(\widehat{g_2} = g_2 - \overline{g_2} e^{-y}\). For  \(1\leq m< \infty\), if $\delta$ is small enough, then there holds
    \begin{equation}
        \| \nabla G \|_{Y^{1,m}}
        \leq C_\delta \left( \| \omega^g \|_{Y^{1,m}}
        + \| \overline{g_2} e^{-y} \|_{Y^{1,m+1}} \right).
    \end{equation}
\end{lemma}

\begin{proof}
Due to \(\nabla \cdot G=0\), we have
    \begin{align*}
        \Delta g_2 &= \partial_x^2 g_2 + \partial_y^2 g_2 = \partial_x^2 g_2 - \partial_y \partial_x g_1 = -\partial_x \omega^g,
    \end{align*}
    and it follows that
    \begin{align*}
        -\sum_{\alpha_1 \leq 1, |\alpha| \leq m} \langle \partial^\alpha \Delta g_2, \partial^\alpha \widehat{g_2} \rangle 
        &= \sum_{\alpha_1 \leq 1, |\alpha| \leq m} \langle \partial^\alpha \partial_x \omega^g, \partial^\alpha \widehat{g_2}\rangle.
    \end{align*}
Applying the absolute value inequality to the right-hand side yields
    \begin{align}\label{eq:laplace g2}
         -\sum_{\alpha_1 \leq 1, |\alpha| \leq m} \langle \partial^\alpha \Delta g_2, \partial^\alpha \widehat{g_2} \rangle 
        \leq \sum_{\alpha_1 \leq 1, |\alpha| \leq m} |\langle \partial^\alpha \partial_x \omega^g, \partial^\alpha \widehat{g_2} \rangle|.
    \end{align}
    By integration by parts, we have
    \begin{align*}
        &-\sum_{\alpha_1 \leq 1, |\alpha| \leq m} \langle \partial^\alpha \Delta g_2, \partial^\alpha \widehat{g_2} \rangle \\
        % =& -\sum_{\alpha_1 \leq 1, |\alpha| \leq m} \langle \partial^\alpha \partial_x^2 g_2, \partial^\alpha \widehat{g_2} \rangle 
        % -\sum_{\alpha_1 \leq 1, |\alpha| \leq m} \langle \partial^\alpha \partial_y^2 g_2, \partial^\alpha \widehat{g_2} \rangle\\
        =& \sum_{\alpha_1 \leq 1, |\alpha| \leq m} \langle \partial^\alpha \partial_x g_2, \partial^\alpha \partial_x \widehat{g_2} \rangle+ \sum_{\alpha_1 \leq 1, |\alpha| \leq m} \langle \partial^\alpha \partial_y g_2, \partial^\alpha \partial_y \widehat{g_2} \rangle \\
        &- \sum_{\alpha_1 \leq 1, |\alpha| \leq m} \langle [\partial^\alpha, \partial_y] \partial_y g_2, \partial^\alpha \widehat{g_2} \rangle  + \sum_{\alpha_1 \leq 1, |\alpha| \leq m} \langle \partial^\alpha \partial_y g_2, [\partial_y, \partial^\alpha] \widehat{g_2} \rangle, 
      \end{align*}
      which is greater than 
    % replaced \(\widehat{g_2}\) by \(g_2 - \overline{g_2} e^{-y}\), 
    \begin{align*}
        \geq &\sum_{\alpha_1 \leq 1, |\alpha| \leq m} \| \partial^\alpha \partial_x g_2 \|_2^2 
        - \sum_{\alpha_1 \leq 1, |\alpha| \leq m} |\langle \partial^\alpha \partial_x (\overline{g_2} e^{-y}), \partial^\alpha \partial_x g_2 \rangle|\\
        &+ \sum_{\alpha_1 \leq 1, |\alpha| \leq m} \|\partial^\alpha \partial_y g_2\|_{2}^2 - \sum_{\alpha_1 \leq 1, |\alpha| \leq m} |\langle \partial^\alpha \partial_y g_2, \partial^\alpha \partial_y (\overline{g_2} e^{-y}) \rangle| \\
        &- C\delta \sum_{\alpha_1 \leq 1, |\alpha| \leq m} |\langle \partial^\alpha \partial_y g_2, \partial^{\alpha-(0,0,1)} \partial_y g_2 \rangle| \\
        &- C\delta \sum_{\alpha_1 \leq 1, |\alpha| \leq m} |\langle \partial^\alpha \partial_y g_2, \partial^{\alpha-(0,0,1)} \partial_y (\overline{g_2} e^{-y}) \rangle|.
    \end{align*}
   Moreover, by H\"older's inequality, it shows
    \begin{align*}
        &-\sum_{\alpha_1 \leq 1, |\alpha| \leq m} \langle \partial^\alpha \Delta g_2, \partial^\alpha \widehat{g_2} \rangle\\
        &\geq (1 - C\sigma-C\delta) \sum_{\alpha_1 \leq 1, |\alpha| \leq m} \|\partial^\alpha \nabla g_2\|_{2}^2 - C\left( \delta + \frac{1}{\sigma} \right) \sum_{\alpha_1 \leq 1, |\alpha| \leq m} \|\partial^\alpha \nabla( \overline{g_2} e^{-y} )\|_{2}^2\\
        &\quad - C\delta \left( \sum_{\alpha_1 \leq 1, |\alpha| \leq m} \|\partial^\alpha \partial_y g_2\|_{2}^2 
        + \sum_{\alpha_1 \leq 1, |\alpha| \leq m} \|\partial^{\alpha-(0,0,1)} \partial_y g_2\|_{2}^2 \right) \\
        &\geq  (1 - C\sigma-C\delta) \sum_{\alpha_1 \leq 1, |\alpha| \leq m} \|\partial^\alpha \nabla g_2\|_{2}^2 - C\left( \delta + \frac{1}{\sigma} \right) \sum_{\alpha_1 \leq 1, |\alpha| \leq m} \|\partial^\alpha \nabla( \overline{g_2} e^{-y} )\|_{2}^2.
    \end{align*}

   For the term on  the right hand of \eqref{eq:laplace g2}, it shows
    \begin{align*}
        &\sum_{\alpha_1 \leq 1, |\alpha| \leq m} |\langle \partial^\alpha \partial_x \omega^g, \partial^\alpha \widehat{g_2} \rangle| 
        = \sum_{\alpha_1 \leq 1, |\alpha| \leq m} |\langle \partial^\alpha \omega^g, \partial^\alpha \partial_x \widehat{g_2} \rangle| \\
        &\leq \frac{C}{\sigma} \sum_{\alpha_1 \leq 1, |\alpha| \leq m} \|\partial^\alpha \omega^g\|_{2}^2 
           + C\sigma \sum_{\alpha_1 \leq 1, |\alpha| \leq m} \|\partial^\alpha \partial_x g_2\|_{2}^2 + C\sigma \sum_{\alpha_1 \leq 1, |\alpha| \leq m} \|\partial^\alpha (\partial_x \overline{g_2} e^{-y} )\|_{2}^2.
    \end{align*}
    Summing up these estimates, we have
    \begin{align*}
        &(1 - C\sigma-C\delta) \sum_{\alpha_1 \leq 1, |\alpha| \leq m} \|\partial^\alpha \nabla g_2\|_{2}^2 \\
        &\leq \frac{C}{\sigma} \sum_{\alpha_1 \leq 1, |\alpha| \leq m} \|\partial^\alpha \omega^g\|_{2}^2 + C\left( \delta + \frac{1}{\sigma} \right) \sum_{\alpha_1 \leq 1, |\alpha| \leq m+1} \|\partial^\alpha ( \overline{g_2} e^{-y} )\|_{2}^2.
    \end{align*}
    Letting $\sigma, \delta$ small enough, we have
    \begin{align}\label{eq:nabla g2}
        \sum_{\alpha_1 \leq 1, |\alpha| \leq m} \|\partial^\alpha \nabla g_2\|_{2}^2 
        &\leq C \sum_{\alpha_1 \leq 1, |\alpha| \leq m} \|\partial^\alpha \omega^g\|_{2}^2 
           + C \sum_{\alpha_1 \leq 1, |\alpha| \leq m+1} \|\partial^\alpha ( \overline{g_2} e^{-y} )\|_{2}^2.
    \end{align}
    For $g_1$, as \(\nabla \cdot G=0\), we have
    \begin{align}\label{eq:nabla g1}
        \sum_{\alpha_1 \leq 1, |\alpha| \leq m} \|\partial^\alpha \nabla g_1\|_{2}^2 
        &= \sum_{\alpha_1 \leq 1, |\alpha| \leq m} \|\partial^\alpha \partial_x g_1\|_{2}^2 
           + \sum_{\alpha_1 \leq 1, |\alpha| \leq m} \|\partial^\alpha \partial_y g_1\|_{2}^2\nonumber \\
        &\leq \sum_{\alpha_1 \leq 1, |\alpha| \leq m} \|\partial^\alpha \partial_y g_2\|_{2}^2 
           + \sum_{\alpha_1 \leq 1, |\alpha| \leq m} \|\partial^\alpha (\partial_y g_1 - \partial_x g_2 + \partial_x g_2)\|_{2}^2 \nonumber\\
        &\leq \sum_{\alpha_1 \leq 1, |\alpha| \leq m} \|\partial^\alpha \nabla g_2\|_{2}^2 
           + \sum_{\alpha_1 \leq 1, |\alpha| \leq m} \|\partial^\alpha \omega^g\|_{2}^2. 
        % &\leq C \sum_{\alpha_1 \leq 1, |\alpha| \leq m} \|\partial^\alpha \omega^g\|_{2}^2 
        % + C \sum_{\alpha_1 \leq 1, |\alpha| \leq m+1} \|\partial^\alpha ( \overline{g_2} e^{-y} )\|_{2}^2.
    \end{align}
    Combining \eqref{eq:nabla g2} and \eqref{eq:nabla g1}, this proof is complete.
\end{proof}

%Nonlinear estimates in Conormal Soblev spaces.
\subsection{Nonlinear estimates in conormal Sobolev spaces.}
In this section, we  deal with the product terms or nonlinear terms. For simplicity, $\langle \cdot , \cdot \rangle$  means the inner
product in $L^2_{xy}(\mathbb{R}^2_+)$.

First,  for $H\cdot \nabla G$ we have the following estimates. 

\begin{lemma}\label{lem:H cdot nabla G}
  Let  $u^a$ and $u$  be as in \eqref{eq:approximate} and \eqref{eq:error}. Then there holds 
    
    \textup{(a)}
    \begin{equation}
        \begin{aligned}
            \sum_{\alpha_1 \leq 1, |\alpha| \leq m}& \bigl| \bigl\langle \partial^\alpha (u^a \cdot \nabla G), \partial^\alpha G \bigr\rangle \bigr|\\
            &\leq \frac{C_\delta}{\sigma} \left( 1 + \| u^a \|_{Y^{1,m+1}_\infty} + \| f e^{-y} \|_{Y^{1,m}_\infty} \right)^2 \| G \|_{Y^{1,m}}^2 + C\sigma \varepsilon^4 \| \partial_y G \|_{Y^{1,m}}^2,
        \end{aligned}
    \end{equation}
    
    \textup{(b)}
    \begin{equation}
        \begin{aligned}
            \sum_{\alpha_1 \leq 1, |\alpha| \leq m} 
            &\bigl| \bigl\langle \partial^{\alpha} (u \cdot \nabla G^a), \partial^{\alpha} Q \bigr\rangle \bigr|\\
            &\leq C \left( \left( \| G^a \|_{Y^{1,m+1}_\infty}^2 + 1 \right) \| Q \|_{Y^{1,m}}^2 + \| u \|_{Y^1}^2 + \| \omega \|_{Y^{1,m}}^2 \right) \\
            &\quad + C \varepsilon^4 \| f e^{-y} \|_{Y^{1,m+1}}^2 
            \left( \| G^a \|_{Y^{1,m+1}_\infty}^2 + \| \partial_y G^a \|_{Y^{1,m}}^2 \right),
        \end{aligned}
    \end{equation}
    and
    \textup{(c)}
    \begin{equation}
        \begin{aligned}
            &\sum_{\alpha_1 \leq 1, |\alpha| \leq m} 
            \bigl| \bigl\langle \partial^{\alpha} (u \cdot \nabla G), \partial^{\alpha} G \bigr\rangle \bigr|\\
            &\leq \frac{C}{\sigma} \left( \| u \|_{Y^1}^2 + \| \omega \|_{Y^{1,m}}^2 + \| G \|_{Y^{1,m}}^2 \right)  + \frac{C}{\sigma \varepsilon^4} \left( \| u \|_{Y^1}^2 + \| \omega \|_{Y^{1,m}}^2 + \| G \|_{Y^{1,m}}^2 \right)^3 \\
            &\quad + \frac{C}{\sigma} \varepsilon^2 \| f e^{-y} \|_{Y^{1,m+1}}^2 \| G \|_{Y^{1,m}}^2 + C\sigma \varepsilon^2 \| \nabla G \|_{Y^{1,m}}^2.
        \end{aligned}
    \end{equation}
\end{lemma}
\begin{proof}
    \textup{(a)} The term \(\sum_{\alpha_1 \leq 1, |\alpha| \leq m} 
            \bigl| \bigl\langle \partial^\alpha (u^a \cdot \nabla G), \partial^\alpha G \bigr\rangle \bigr|\) is divided into two terms:
    \begin{equation*}
        \begin{aligned}
            &\sum_{\alpha_1 \leq 1, |\alpha| \leq m} 
            \bigl| \bigl\langle \partial^\alpha (u^a \cdot \nabla G), \partial^\alpha G \bigr\rangle \bigr| \\
            &\leq \sum_{\alpha_1 \leq 1, |\alpha| \leq m} 
            \bigl| \bigl\langle u^a \cdot \partial^\alpha \nabla G, \partial^\alpha G \bigr\rangle \bigr| + C \sum_{\alpha_1 \leq 1, |\alpha| \leq m} 
            \sum_{\beta \leq \alpha, |\beta| \leq m-1} 
            \bigl| \bigl\langle \partial^{\alpha-\beta} u^a \cdot \partial^\beta \nabla G, \partial^\alpha G \bigr\rangle \bigr|\\
            &=:a_1+a_2.
        \end{aligned}
    \end{equation*}
    For \(a_1\), by integration by parts and \eqref{exchange oper},
    \begin{align*}
        a_1&
        \leq \sum_{\alpha_1 \leq 1, |\alpha| \leq m} 
        \bigl| \bigl\langle u^a \cdot \nabla \partial^\alpha G, \partial^\alpha G \bigr\rangle \bigr| + \sum_{\alpha_1 \leq 1, |\alpha| \leq m} 
        \bigl| \bigl\langle u_2^a [\partial^\alpha, \partial_y] G, \partial^\alpha G \bigr\rangle \bigr| \\
        &\leq C \varepsilon^2 \sum_{\alpha_1 \leq 1, |\alpha| \leq m} 
        \Bigl| \int_{\partial \mathbb{R}^2_+} f (\overline{\partial^\alpha G})^2 \, dx \Bigr| + C \delta \sum_{\alpha_1 \leq 1, |\alpha| \leq m} 
        \bigl| \bigl\langle u_2^a \partial^{\alpha-(0,0,1)} \partial_y G, \partial^\alpha G \bigr\rangle \bigr|  
    \end{align*}
    and by H\"older's inequality and Gagliardo–Nirenberg interpolation inequality, this becomes
    \begin{align*}
        &\leq C \varepsilon^2 \| f \|_{L_x^\infty} 
        \sum_{\alpha_1 \leq 1, |\alpha| \leq m} 
        \| \partial^\alpha G \|_2 \, \| \partial_y \partial^\alpha G \|_2 + C \delta \Bigl\| \frac{\widehat{u_2^a}}{\psi} \Bigr\|_\infty 
        \sum_{\alpha_1 \leq 1, |\alpha| \leq m} 
        \| \partial^\alpha G \|_2^2 \\
        &\quad + C \delta \varepsilon^2 \| f e^{-y} \|_\infty 
        \sum_{\alpha_1 \leq 1, |\alpha| \leq m} 
        \| \partial^{\alpha-(0,0,1)} \partial_y G \|_2 \, \| \partial^\alpha G \|_2 \\
        &\leq \frac{C}{\sigma} \left( 1 
        + \Bigl\| \frac{\widehat{u_2^a}}{\psi} \Bigr\|_\infty 
        + \| f e^{-y} \|_\infty^2 \right) 
        \sum_{\alpha_1 \leq 1, |\alpha| \leq m} 
        \| \partial^\alpha G \|_2^2  + C \sigma \varepsilon^4 
        \sum_{\alpha_1 \leq 1, |\alpha| \leq m} 
        \| \partial^\alpha \partial_y G \|_2^2,
    \end{align*}
    where  \(\widehat{u^a_2}=u_2^a-\varepsilon^2fe^{-y}\).

    For \(a_2\), by H\"older's inequality and Lemma \ref{lem:separate}, it shows
    \begin{align*}
        a_2
        \leq& C \Biggl( \sum_{\alpha_1 \leq 1, |\alpha| \leq m} 
        \| \partial^\alpha u_1^a \|_\infty^2 \Biggr)^{\frac 1 2} 
        \sum_{\alpha_1 \leq 1, |\alpha| \leq m} 
        \| \partial^\alpha G \|_2^2 \\
        &+ C \Biggl( \sum_{\alpha_1 \leq 1, |\alpha| \leq m} 
        \Bigl\| \frac{\partial^\alpha\widehat{ u_2^a}}{\psi} \Bigr\|_\infty^2 \Biggr)^{\frac 1 2} 
        \sum_{\alpha_1 \leq 1, |\alpha| \leq m} 
        \| \partial^\alpha G \|_2^2 \\
        &+ C \varepsilon^2 \Biggl( \sum_{\alpha_1 \leq 1, |\alpha| \leq m} 
        \| \partial^\alpha (f e^{-y}) \|_\infty^2 \Biggr)^{\frac 1 2}  \Biggl( \sum_{\alpha_1 \leq 1, |\alpha| \leq m} 
        \| \partial^\alpha G \|_2^2 \Biggr)^{\frac 1 2} \Biggl( \sum_{\alpha_1 \leq 1, |\alpha| \leq m} 
        \| \partial^\alpha \partial_y G \|_2^2 \Biggr)^{\frac 1 2} \\
        \leq& C \left( 1 + \frac{1}{\sigma} \right) 
        \Biggl[ \Biggl( \sum_{\alpha_1 \leq 1, |\alpha| \leq m} 
        \| \partial^\alpha u_1^a \|_\infty^2 \Biggr)^{\frac 1 2} 
        + \Biggl( \sum_{\alpha_1 \leq 1, |\alpha| \leq m} 
        \Bigl\| \frac{\partial^\alpha\widehat{ u_2^a}}{\psi} \Bigr\|_\infty^2 \Biggr)^{\frac 1 2} \\
        &\quad + \sum_{\alpha_1 \leq 1, |\alpha| \leq m} 
        \| \partial^\alpha (f e^{-y}) \|_\infty^2 \Biggr] 
        \sum_{\alpha_1 \leq 1, |\alpha| \leq m} 
        \| \partial^\alpha G \|_2^2  + C \sigma \varepsilon^4 
        \sum_{\alpha_1 \leq 1, |\alpha| \leq m} 
        \| \partial^\alpha \partial_y G \|_2^2.
    \end{align*}
    Combining the above estimates and by Lemma \ref{lem:psi}, the  proof (a) is complete.
    % Where we have used\\
    % (1) Divergence condition:
    % \begin{align*}
    %     \nabla \cdot u^a = 0.
    % \end{align*}
    % (2) Gauss theorem:
    % \begin{align*}
    %     \Bigl| \int_{\mathbb{\partial R_p^2}} (\partial^\alpha G)^2 \, dx \Bigr| = \Bigl| \int_{\mathbb{R}} \partial_y (\partial^\alpha G)^2 \, dxdy \Bigr|.
    % \end{align*}
    % (3) Summation formula(see ):
    % \begin{align*}
    %     &\sum_{\alpha_1 \leq 1, |\alpha| \leq m} \sum_{\beta \leq \alpha, i \leq |\beta| \leq j} H_{\alpha-\beta} G_\beta Q_\alpha \\
    %     &\leq \left( \sum_{\alpha_1 \leq 1, |\alpha| \leq m} \left( \sum_{\beta \leq \alpha, i \leq |\beta| \leq j} H_{\alpha-\beta} G_\beta \right) ^2 \right)^{\frac 1 2} 
    %     \left( \sum_{\alpha_1 \leq 1, |\alpha| \leq m} Q_\alpha^2 \right)^{\frac 1 2} \\
    %     &\leq \left( \sum_{\alpha_1 \leq 1, |\alpha| \leq m} \sum_{\beta \leq \alpha, i \leq |\beta| \leq j} H_{\alpha-\beta}^2 G_\beta^2 \right)^{\frac 1 2} 
    %     \left( \sum_{\alpha_1 \leq 1, |\alpha| \leq m} Q_\alpha^2 \right)^{\frac 1 2} \\
    %     &\leq \left( \sum_{\alpha_1 \leq 1, |\alpha| \leq i} H_{\alpha}^2 \right)^{\frac 1 2} 
    %     \left( \sum_{\alpha_1 \leq 1, |\alpha| \leq j} G_\alpha^2 \right)^{\frac 1 2} 
    %     \left( \sum_{\alpha_1 \leq 1, |\alpha| \leq m} Q_\alpha^2 \right)^{\frac 1 2}.
    % \end{align*}
    
    \textup{(b)} For \(\sum_{\alpha_1 \leq 1, |\alpha| \leq m} 
        \bigl| \bigl\langle \partial^\alpha (u \cdot \nabla G^a), \partial^\alpha Q \bigr\rangle \bigr|\), by H\"older's inequality and \(\widehat{u_2}=u_2-\varepsilon^2fe^{-y}\), we have
    % 第一部分
     \begin{align*}
       &\sum_{\alpha_1 \leq 1, |\alpha| \leq m} 
        \bigl| \bigl\langle \partial^\alpha (u \cdot \nabla G^a), \partial^\alpha Q \bigr\rangle \bigr|\\
        &\leq C \sum_{\alpha_1 \leq 1, |\alpha| \leq m} 
        \sum_{\beta \leq \alpha} 
        \bigl| \bigl\langle \partial^{\alpha-\beta} u \cdot \partial^\beta \nabla G^a, \partial^\alpha Q \bigr\rangle \bigr|\\
        &\leq C \Biggl( \sum_{\alpha_1 \leq 1, |\alpha| \leq m} 
        \sum_{\beta \leq \alpha} 
        \| \partial^{\alpha-\beta} u_1 \|_2^2 \,
        \| \partial^\beta \partial_x G^a \|_\infty^2 \Biggr)^{\!\frac 1 2} 
        \Biggl( \sum_{\alpha_1 \leq 1, |\alpha| \leq m} 
        \| \partial^\alpha Q \|_2^2 \Biggr)^{\!\frac 1 2} \\
        &\quad + C \Biggl( \sum_{\alpha_1 \leq 1, |\alpha| \leq m} 
        \sum_{\beta \leq \alpha} 
        \Bigl\| \frac{\partial^{\alpha-\beta} \widehat{u_2}}{\psi} \Bigr\|_2^2 \,
        \| \psi \partial^\beta \partial_y G^a \|_\infty^2 \Biggr)^{\!\frac 1 2} 
        \Biggl( \sum_{\alpha_1 \leq 1, |\alpha| \leq m} 
        \| \partial^\alpha Q \|_2^2 \Biggr)^{\!\frac 1 2} \\
        &\quad + C\varepsilon^2 \Biggl( \sum_{\alpha_1 \leq 1, |\alpha| \leq m} 
        \sum_{\beta \leq \alpha} 
        \| \partial^{\alpha-\beta} (f e^{-y}) \|_\infty^2 \,
        \| \partial^\beta \partial_y G^a \|_2^2 \Biggr)^{\!\frac 1 2} 
        \Biggl( \sum_{\alpha_1 \leq 1, |\alpha| \leq m} 
        \| \partial^\alpha Q \|_2^2 \Biggr)^{\!\frac 1 2},
    \end{align*}
    % 第二部分
    and by Lemma \ref{lem:separate}, this becomes
    \begin{align*}
        &\leq C_\delta \Biggl( \sum_{\alpha_1 \leq 1, |\alpha| \leq m} 
        \| \partial^\alpha u_1 \|_2^2 
        + \sum_{\alpha_1 \le 1,|\alpha|\le m} \left\| \frac{\partial^\alpha \widehat{u_2} }\psi \right\|_2^2 
         \Biggr)^{\frac12}
        \Biggl( \sum_{\alpha_1 \leq 1, |\alpha| \leq m+1} 
        \| \partial^\alpha G^a \|_\infty^2 \Biggr)^{\frac 1 2} 
        \Biggl( \sum_{\alpha_1 \leq 1, |\alpha| \leq m} 
        \| \partial^\alpha Q \|_2^2 \Biggr)^{\frac 1 2} \\
        &\quad + C \varepsilon^2 \Biggl( \sum_{\alpha_1 \leq 1, |\alpha| \leq m} 
        \| \partial^\alpha (f e^{-y}) \|_\infty^2 \Biggr)^{\frac 1 2} 
        \Biggl( \sum_{\alpha_1 \leq 1, |\alpha| \leq m} 
        \| \partial^\alpha \partial_y G^a \|_2^2 \Biggr)^{\!\frac 1 2}  \Biggl( \sum_{\alpha_1 \leq 1, |\alpha| \leq m} 
        \| \partial^{\alpha} Q \|_2^2 \Biggr)^{\frac 1 2}.
    \end{align*}
    Through Lemma \ref{lem:psi}, Lemma \ref{lem:omega} and the definition \eqref{def:derivative} we know that
    \begin{align*}
        &\sum_{\alpha_1 \leq 1, |\alpha| \leq m} 
        \| \partial^\alpha u_1 \|_2^2 
        + \sum_{\alpha_1 \le 1,|\alpha|\le m} \left\| \frac{\partial^\alpha \widehat{u_2} }\psi \right\|_2^2 \\
        &\le C\left( \| u_1 \|_{Y^{1,m}}^2 + \| u_2 \|_{Y^{1,m}}^2 + \| \partial_y u_2 \|_{Y^{1,m}}^2 + \| fe^{-y} \|_{Y^{1,m}}^2 \right)\\
        &\le C\left( \| u \|_{Y^1}^2 + \| \nabla u \|_{Y^{1,m}}^2 + \| fe^{-y} \|_{Y^{1,m}}^2 \right) \\
        &\le C\left( \| u_1 \|_{Y^{1,m}}^2 + \| u_2 \|_{Y^{1,m}}^2 + \| \partial_y u_2 \|_{Y^{1,m}}^2 + \| fe^{-y} \|_{Y^{1,m}}^2 \right)\\
        &\le C\left( \| u \|_{Y^1}^2 + \| \omega \|_{Y^{1,m}}^2 + \| fe^{-y} \|_{Y^{1,m+1}}^2 \right).
    \end{align*}
    Then by using H\"older's inequality we have
    \begin{align*}
        &\sum_{\alpha_1 \leq 1, |\alpha| \leq m} 
        \bigl| \bigl\langle \partial^\alpha (u \cdot \nabla G^a), \partial^\alpha Q \bigr\rangle \bigr| \\
        &\leq C \left( \| u \|_{Y^1} 
        + \|\omega \|_{Y^{1,m}} + \varepsilon^2  \| f e^{-y} \|_{Y^{1,m+1}}^2 \right)
        \| \partial^\alpha G^a \|_{Y^{1,m+1}_\infty}
        \| Q \|_{Y^{1,m}} \\
        &\quad + C \varepsilon^2 
        \|  f e^{-y} \|_{Y^{1,m}_\infty}
        \| \partial_y G^a \|_{Y^{1,m}} \| Q \|_{Y^{1,m}} \\
        &\leq C \left( \left( \| G^a \|_{Y^{1,m+1}_\infty}^2 + 1 \right) \| Q \|_{Y^{1,m}}^2 + \| u \|_{Y^1}^2 + \| \omega \|_{Y^{1,m}}^2 \right) \\
        &\quad + C \varepsilon^4 \| f e^{-y} \|_{Y^{1,m+1}}^2 
        \left( \| G^a \|_{Y^{1,m+1}_\infty}^2 + \| \partial_y G^a \|_{Y^{1,m}}^2 \right).
    \end{align*}

    \textup{(c)} For \(\sum_{\alpha_1 \leq 1, |\alpha| \leq m} \left| \langle \partial^{\alpha}(u \cdot \nabla G), \partial^{\alpha} G \rangle \right|\), we divide it into two situations:
    \begin{align*}
        &\sum_{\alpha_1 \leq 1, |\alpha| \leq m} \left| \langle \partial^{\alpha}(u \cdot \nabla G), \partial^{\alpha} G \rangle \right|\\
        &\leq C \sum_{\alpha_1 \leq 1, |\alpha| \leq m} \left| \langle \partial^{\alpha} u \cdot \nabla G, \partial^{\alpha} G \rangle \right|
        + C \sum_{\alpha_1 \leq 1, |\alpha| \leq m} \sum_{\substack{\beta \leq \alpha, 1 \leq |\beta|}} \left| \langle \partial^{\alpha-\beta} u \cdot \partial^{\beta} \nabla G, \partial^{\alpha} G \rangle \right| \\
        &=: c_1 + c_2,
    \end{align*}
     For \(c_1\), by H\"older's inequality, Lemma \ref{lem:embedding} and \(\widehat{u_2}=u_2-\varepsilon^2fe^{-y}\), we have
    \begin{align*}
        c_1 &\leq C \left( \sum_{\alpha_1 \leq 1, |\alpha| \leq m} \|\partial^\alpha u_1\|_{2}^2 \right)^{\frac{1}{2}} 
        \left( \left\| \partial_x G \right\|_{L^\infty}^2 \right)^{\frac{1}{2}}
        \left( \sum_{\alpha_1 \leq 1, |\alpha| \leq m} \|\partial^\alpha G\|_{2}^2 \right)^{\frac{1}{2}} \\
        &\quad + C \left( \sum_{\alpha_1 \leq 1, |\alpha| \leq m} \left\| \frac{\partial^\alpha \widehat{u}_2}{\psi} \right\|_{2}^2 \right)^{\frac{1}{2}}
        \left( \|\psi \partial_y G\|_{L^\infty}^2 \right)^{\frac{1}{2}}
        \left( \sum_{\alpha_1 \leq 1, |\alpha| \leq m} \|\partial^\alpha G\|_{2}^2 \right)^{\frac{1}{2}} \\
        &\quad + C \varepsilon^2 \left( \sum_{\alpha_1 \leq 1, |\alpha| \leq m} \|\partial^\alpha (f e^{-y})\|_{L^\infty}^2 \right)^{\frac{1}{2}}
        \left( \|\nabla G\|_{2}^2 \right)^{\frac{1}{2}}
        \left( \sum_{\alpha_1 \leq 1, |\alpha| \leq m} \|\partial^\alpha G\|_{2}^2 \right)^{\frac{1}{2}}\\
        &\leq C \left( \sum_{\alpha_1 \leq 1, |\alpha| \leq m} \|\partial^\alpha u_1\|_{2}^2 
        + \sum_{\alpha_1 \leq 1, |\alpha| \leq m} \left\| \frac{\partial^\alpha \widehat{u}_2}{\psi} \right\|_{2}^2 \right)^{\frac{1}{2}}\times \\
        &\quad \left( \sum_{\alpha_1 \leq 1, |\alpha| \leq 2} \|\partial^\alpha G\|_{2}^2 \right)^{\frac{1}{4}}
        \left( \sum_{\alpha_1 \leq 1, |\alpha| \leq 2} \|\partial^\alpha \partial_y G\|_{2}^2 \right)^{\frac{1}{4}}
        \left( \sum_{\alpha_1 \leq 1, |\alpha| \leq m} \|\partial^\alpha G\|_{2}^2 \right)^{\frac{1}{2}} \\
        &\quad + C \varepsilon^2 \left( \sum_{\alpha_1 \leq 1, |\alpha| \leq m} \|\partial^\alpha (f e^{-y})\|_{L^\infty}^2 \right)^{\frac{1}{2}}
        \left( \|\nabla G\|_{2}^2 \right)^{\frac{1}{2}}
        \left( \sum_{\alpha_1 \leq 1, |\alpha| \leq m} \|\partial^\alpha G\|_{2}^2 \right)^{\frac{1}{2}},
        \end{align*}
       and by Lemma \ref{lem:psi} and Lemma \ref{lem:omega}, this becomes
        \begin{align*}
        &\leq C \left( \sum_{\alpha_1 \leq 1, |\alpha| \leq m} \|\partial^\alpha u\|_{2}^2 
        + \sum_{\alpha_1 \leq 1, |\alpha| \leq m} \|\partial^{\alpha} \omega\|_{2}^2 
        + \varepsilon^4 \sum_{\alpha_1\leq1,|\alpha| \leq m+1} \|\partial^{\alpha} (f e^{-y})\|_{2}^2 \right)^{\frac{1}{2}} \\
        &\times \left( \sum_{\alpha_1 \leq 1, |\alpha| \leq 2} \|\partial^\alpha G\|_{2}^2 \right)^{\frac{1}{4}}
        \left( \sum_{\alpha_1 \leq 1, |\alpha| \leq 2} \|\partial^\alpha \partial_y G\|_{2}^2 \right)^{\frac{1}{4}}
        \left( \sum_{\alpha_1 \leq 1, |\alpha| \leq m} \|\partial^\alpha G\|_{2}^2 \right)^{\frac{1}{2}} \\
        &+ C \varepsilon^2 \left( \sum_{\alpha_1 \leq 1, |\alpha| \leq m} \|\partial^\alpha (f e^{-y})\|_{L^\infty}^2 \right)^{\frac{1}{2}}
    \left( \|\nabla G\|_{2}^2 \right)^{\frac{1}{2}}
    \left( \sum_{\alpha_1 \leq 1, |\alpha| \leq m} \|\partial^\alpha G\|_{2}^2 \right)^{\frac{1}{2}}\\
        &\leq C \left( \|u\|_{Y^{1}} + \|\omega\|_{Y^{1,m}} \right) 
        \left\| G \right\|_{Y^{1,m}}^{\frac 1 2} \left\|\partial_y G \right\|_{Y^{1,m}}^{\frac 1 2} \left\| G \right\|_{Y^{1,m}} \\
        &\quad + C \varepsilon^2 \|fe^{-y}\|_{Y^{1,m+1}} 
        \left\| G \right\|_{Y^{1,m}}^{\frac 1 2} \left\|\partial_y G \right\|_{Y^{1,m}}^{\frac 1 2} \left\| G \right\|_{Y^{1,m}} \\
        &\quad + C \varepsilon^2 \|fe^{-y}\|_{Y^{1,m}_{\infty}} 
        \|\nabla G\|_{Y^{1,m}} \left\| G \right\|_{Y^{1,m}}.
    \end{align*}
    For \(c_2\), using the same methods as H\"older's inequality, Lemma \ref{lem:psi} and Lemma \ref{lem:omega}, we have
\begin{align*}
        c_2
        &\leq C \left( \sum_{\alpha_1 \leq 1, |\alpha| \leq m-1} \|\partial^{\alpha} u\|_{L^\infty}^2 \right)^{\frac{1}{2}}
        \left( \sum_{\alpha_1 \leq 1, |\alpha| \leq m} \|\partial^{\alpha} \nabla G\|_{2}^2 \right)^{\frac{1}{2}}
        \left( \sum_{\alpha_1 \leq 1, |\alpha| \leq m} \|\partial^{\alpha} G\|_{2}^2 \right)^{\frac{1}{2}}\\
        &\leq C \left( \sum_{\alpha_1 \leq 1, |\alpha| \leq m} \|\partial^{\alpha} u\|_{2}^2 \right)^{\frac{1}{4}}
        \left( \sum_{\alpha_1 \leq 1, |\alpha| \leq m} \|\partial^{\alpha} \partial_y u\|_{2}^2 \right)^{\frac{1}{4}}\\
        &\quad \times \left( \sum_{\alpha_1 \leq 1, |\alpha| \leq m} \|\partial^{\alpha} \nabla G\|_{2}^2 \right)^{\frac{1}{2}}
        \left( \sum_{\alpha_1 \leq 1, |\alpha| \leq m} \|\partial^{\alpha} G\|_{2}^2 \right)^{\frac{1}{2}}\\
        &\leq C \left( \|u\|_{Y^{1}} + \|\omega\|_{Y^{1,m}} \right) \|\nabla G\|_{Y^{1,m}} \left\| G \right\|_{Y^{1,m}}
        + C \varepsilon^2 \|fe^{-y}\|_{Y^{1,m+1}} \|\nabla G\|_{Y^{1,m}} \left\| G \right\|_{Y^{1,m}}.
    \end{align*}
    In summary, by H\"older's inequality, we have
    \begin{align*}
        &\sum_{\alpha_1 \leq 1, |\alpha| \leq m} \left| \langle \partial^{\alpha}(u \cdot \nabla G), \partial^{\alpha} G \rangle \right| \\
     &\leq C \left( \|u\|_{Y^{1}}^2 + \|\omega\|_{Y^{1,m}}^2 \right)  + \frac{C}{\sigma \varepsilon^2} \left( \left\| G \right\|_{Y^{1,m}}^2 \right)^3\\
     &\quad + C\sigma \varepsilon^2 \|\nabla G\|_{Y^{1,m}}^2+\frac{C}{\sigma} \varepsilon^2 \left\| G \right\|_{Y^{1,m}}^2+C\varepsilon^2 \|fe^{-y}\|_{Y^{1,m+1}}^2 \left\| G \right\|_{Y^{1,m}}^2\\
     &\quad +C \left\| G \right\|_{Y^{1,m}}^2 + \frac{C}{\sigma^2 \varepsilon^4} \left( \|u\|_{Y^{1}}^2 + \|\omega\|_{Y^{1,m}}^2 \right)^2 \left\| G \right\|_{Y^{1,m}}^2+\frac{C}{\sigma} \varepsilon^2 \|fe^{-y}\|_{Y^{1,m}_{\infty}}^2 \left\| G \right\|_{Y^{1,m}}^2\\
     &\leq \frac{C}{\sigma} \left( \|u\|_{Y^1}^2 + \|\omega\|_{Y^{1,m}}^2 + \left\| G \right\|_{Y^{1,m}}^2 \right) + \frac{C}{\sigma \varepsilon^4} \left( \|u\|_{Y^1}^2 + \|\omega\|_{Y^{1,m}}^2 + \left\| G \right\|_{Y^{1,m}}^2 \right)^3 \\
        &\quad + \frac{C}{\sigma} \varepsilon^2 \|fe^{-y}\|_{Y^{1,m+1}}^2 \left\| G \right\|_{Y^{1,m}}^2 + C\sigma \varepsilon^2 \|\nabla G\|_{Y^{1,m}}^2.
     \end{align*}
    % \begin{align*}
    %     &\leq C \left\| g \right\|_{Y^{1,m}}^2 + \frac{C}{\sigma^2 \varepsilon^4} \left( \|u\|_{Y^{1}}^2 + \|\omega\|_{Y^{1,m}}^2 \right)^2 \left\| g \right\|_{Y^{1,m}}^2 \\
    %     &\quad + C\sigma \varepsilon^2 \|\nabla G\|_{Y^{1,m}}^2 + C \left( \|u\|_{Y^{1}}^2 + \|\omega\|_{Y^{1,m}}^2 \right) \\
    %     &\quad + \frac{C}{\sigma \varepsilon^2} \left( \left\| g \right\|_{Y^{1,m}}^2 \right)^3 + \frac{C}{\sigma} \varepsilon^2 \left\| g \right\|_{Y^{1,m}}^2 + C\varepsilon^2 \|fe^{-y}\|_{Y^{1,m+1}}^2 \left\| g \right\|_{Y^{1,m}}^2\\
    %     &\quad  + \frac{C}{\sigma} \varepsilon^2 \|fe^{-y}\|_{Y^{1,m}_{\infty}}^2 \left\| g \right\|_{Y^{1,m}}^2+ C\sigma \varepsilon^2 \|\nabla G\|_{Y^{1,m}}^2
    %     \end{align*}
\end{proof}

\begin{lemma}\label{lem:H nabla G}
    For $H\nabla G$, we have
    
    \textup{(a)}
    \begin{equation}
        \sum_{\alpha_1 \leq 1, |\alpha| \leq m} 
        \bigl| \bigl\langle \partial^\alpha (H^a \nabla G), \partial^\alpha \nabla Q \bigr\rangle \bigr|
        \leq C \| H^a \|_{Y^{1,m}_\infty} 
        \| \nabla G \|_{Y^{1,m}} 
        \| \nabla Q \|_{Y^{1,m}},
    \end{equation}
    
    \textup{(b)}
    \begin{equation}
        \sum_{\alpha_1 \leq 1, |\alpha| \leq m} 
        \bigl| \bigl\langle \partial^\alpha (H \nabla G^a), \partial^\alpha \nabla Q \bigr\rangle \bigr|
        \leq C \| H \|_{Y^{1,m}} 
        \| \nabla G^a \|_{Y^{1,m}_\infty} 
        \| \nabla Q \|_{Y^{1,m}},
    \end{equation}
    
    \textup{(c)}
    \begin{equation}
        \begin{aligned}
            \sum_{\alpha_1 \leq 1, |\alpha| \leq m} 
            \bigl| \bigl\langle \partial^{\alpha} (H \nabla G), \partial^{\alpha} \nabla Q \bigr\rangle \bigr| &\leq C \| H \|_{Y^{1,m}}^{\frac 1 2} 
            \| \nabla H \|_{Y^{1,m}}^{\frac 1 2} 
            \| \nabla G \|_{Y^{1,m}} 
            \| \nabla Q \|_{Y^{1,m}} \\
            &\quad + \| H \|_{Y^{1,m}} 
            \| \nabla G \|_{Y^{1,1}}^{\frac 1 2} 
            \| \nabla^2 G \|_{Y^{1,1}}^{\frac 1 2} 
            \| \nabla Q \|_{Y^{1,m}},
        \end{aligned}
    \end{equation}
and
    \textup{(d)} While $\partial_y Q|_{y=0}=0$, it shows that
    \begin{equation}
        \begin{aligned}
            \sum_{\alpha_1 \leq 1, |\alpha| \leq m} 
            \bigl| \bigl\langle \partial^{\alpha} (H \nabla G), \partial^{\alpha} \nabla Q \bigr\rangle \bigr| &\leq C \| H \|_{Y^{1,m}}^{\frac 1 2} 
            \| \nabla H \|_{Y^{1,m}}^{\frac 1 2} 
            \| \nabla G \|_{Y^{1,m}} 
            \| \nabla Q \|_{Y^{1,m}} \\
            &+ C_\delta \| H \|_{Y^{1,m}} 
            \| G \|_{Y^{1,2}}^{\frac 1 2} 
            \| \nabla G \|_{Y^{1,2}}^{\frac 1 2} \left( \| \nabla Q \|_{Y^{1}} 
            + \| \nabla^2 Q \|_{Y^{1,m}} \right).
        \end{aligned}
    \end{equation}
\end{lemma}
\begin{proof}
   (a) By H\"older's inequality, direct computations derive that 
   %The proofs for (a) and (b) are derived through , as follows:
    \begin{align*}
        &\sum_{\alpha_1 \leq 1, |\alpha| \leq m} \left| \langle \partial^\alpha (H^a \nabla G), \partial^\alpha \nabla Q \rangle \right| \\
        &\leq C \sum_{\alpha_1 \leq 1, |\alpha| \leq m} \sum_{\beta \leq \alpha} \left| \langle \partial^{\alpha-\beta} H^a \partial^\beta \nabla G, \partial^\alpha \nabla Q \rangle \right| \\
        &\leq C \left( \sum_{\alpha_1 \leq 1, |\alpha| \leq m} \|\partial^\alpha H^a\|_{L^\infty}^2 \right)^{\frac{1}{2}}
        \left( \sum_{\alpha_1 \leq 1, |\alpha| \leq m} \|\partial^\alpha \nabla G\|_{2}^2 \right)^{\frac{1}{2}}
        \left( \sum_{\alpha_1 \leq 1, |\alpha| \leq m} \|\partial^\alpha \nabla Q\|_{2}^2 \right)^{\frac{1}{2}},
    \end{align*}
and  (b) is similar.
    % \begin{align*}
    %     \sum_{\alpha_1 \leq 1, |\alpha| \leq m} &\left| \langle \partial^\alpha (H \nabla G^a), \partial^\alpha \nabla Q \rangle \right| \\
    %     &\leq C \sum_{\alpha_1 \leq 1, |\alpha| \leq m} \sum_{\beta \leq \alpha} \left| \langle \partial^{\alpha-\beta} H \partial^\beta \nabla G^a, \partial^\alpha \nabla Q \rangle \right| \\
    %     &\leq C \left( \sum_{\alpha_1 \leq 1, |\alpha| \leq m} \|\partial^\alpha H\|_{2}^2 \right)^{\frac{1}{2}}
    %     \left( \sum_{\alpha_1 \leq 1, |\alpha| \leq m} \|\partial^\alpha \nabla G^a\|_{L^\infty}^2 \right)^{\frac{1}{2}}
    %     \left( \sum_{\alpha_1 \leq 1, |\alpha| \leq m} \|\partial^\alpha \nabla Q\|_{2}^2 \right)^{\frac{1}{2}}.
    % \end{align*}
    
    (c) The term \(\sum_{\alpha_1 \leq 1, |\alpha| \leq m} \left| \langle \partial^\alpha (H \nabla G), \partial^\alpha \nabla Q \rangle \right|\)
can be divided into  two parts, one case is at least one derivative falls on $\nabla G$ (denoted $c_1$) and the other is  all derivatives fall on $H$ (denoted $c_2$):
\begin{align*}
    &\sum_{\alpha_1 \leq 1, |\alpha| \leq m} \left| \langle \partial^\alpha (H \nabla G), \partial^\alpha \nabla Q \rangle \right| \\
    &\leq C \sum_{\alpha_1 \leq 1, |\alpha| \leq m} \sum_{\substack{\beta \leq \alpha \\ 1 \leq |\beta|}} \left| \langle \partial^{\alpha-\beta} H \,\partial^\beta \nabla G, \partial^\alpha \nabla Q \rangle \right|
          + C \sum_{\alpha_1 \leq 1, |\alpha| \leq m} \left| \langle \partial^\alpha H \,\nabla G, \partial^\alpha \nabla Q \rangle \right| \\
    &=: c_1 + c_2.
\end{align*}
By Lemma \ref{lem:embedding} we have
\begin{align*}
c_1
&\leq C \left( \sum_{\alpha_1 \leq 1, |\alpha| \leq m-1} \|\partial^\alpha H\|_{L^\infty}^2 \right)^{\!\! \frac{1}{2}}
      \left( \sum_{\alpha_1 \leq 1, |\alpha| \leq m} \|\partial^\alpha \nabla G\|_{2}^2 \right)^{\!\! \frac{1}{2}}
      \left( \sum_{\alpha_1 \leq 1, |\alpha| \leq m} \|\partial^\alpha \nabla Q\|_{2}^2 \right)^{\!\! \frac{1}{2}} \\[4pt]
&\leq C \left( \sum_{\alpha_1 \leq 1, |\alpha| \leq m} \|\partial^\alpha H\|_{2}^2 \right)^{\!\! \frac{1}{4}}
      \left( \sum_{\alpha_1 \leq 1, |\alpha| \leq m} \|\partial^\alpha \partial_y H\|_{2}^2 \right)^{\!\! \frac{1}{4}} \\
&\quad \times \left( \sum_{\alpha_1 \leq 1, |\alpha| \leq m} \|\partial^\alpha \nabla G\|_{2}^2 \right)^{\!\! \frac{1}{2}}
      \left( \sum_{\alpha_1 \leq 1, |\alpha| \leq m} \|\partial^\alpha \nabla Q\|_{2}^2 \right)^{\!\! \frac{1}{2}}
\end{align*}
and
\begin{align*}
c_2
&\leq C \left( \sum_{\alpha_1 \leq 1, |\alpha| \leq m} \|\partial^\alpha H\|_{2}^2 \|\nabla G\|_{L^\infty}^2 \right)^{\!\! \frac{1}{2}}
      \left( \sum_{\alpha_1 \leq 1, |\alpha| \leq m} \|\partial^\alpha \nabla Q\|_{2}^2 \right)^{\!\! \frac{1}{2}} \\[4pt]
&\leq C \left( \sum_{\alpha_1 \leq 1, |\alpha| \leq m} \|\partial^\alpha H\|_{2}^2 \right)^{\!\! \frac{1}{2}}
      \left( \sum_{\tau \leq 1, |\alpha| \leq 1} \|\partial^\alpha \nabla G\|_{2}^2 \right)^{\!\! \frac{1}{4}} \\
&\quad \times \left( \sum_{\tau \leq 1, |\alpha| \leq 1} \|\partial^\alpha \nabla^2 G\|_{2}^2 \right)^{\!\! \frac{1}{4}}
      \left( \sum_{\alpha_1 \leq 1, |\alpha| \leq m} \|\partial^\alpha \nabla Q\|_{2}^2 \right)^{\!\! \frac{1}{2}}.
\end{align*}
Combining \(c_1\) and \(c_2\),  (c) is proved.\\
    (d) Note that 
    \begin{align*}
        &\sum_{\alpha_1 \leq 1, |\alpha| \leq m} \left| \langle \partial^\alpha (H \nabla G), \partial^\alpha \nabla Q \rangle \right| \\
        &\quad \leq C \sum_{\alpha_1 \leq 1, |\alpha| \leq m} \sum_{\substack{\beta \leq \alpha, 1 \leq |\beta|}} \left| \langle \partial^{\alpha-\beta} H \partial^\beta \nabla G, \partial^\alpha \nabla Q \rangle \right|  + C \sum_{\alpha_1 \leq 1, |\alpha| \leq m} \left| \langle \partial^\alpha H \nabla G, \partial^\alpha \nabla Q \rangle \right| \\
        &\quad =: d_1+d_2.
    \end{align*}
    The estimate of \(d_1\) is similar to \(c_1\) in (c), thus we only need to consider \(d_2\). By H\"older's inequality and Lemma \ref{lem:embedding}, it shows
    \begin{align*}
        d_2
        &\leq C \sum_{\alpha_1 \leq 1, |\alpha| \leq m} \left| \langle \partial^\alpha H \partial_x G, \partial^\alpha \partial_x Q \rangle \right|
           + C \sum_{\alpha_1 \leq 1, |\alpha| \leq m} \left| \langle \partial^\alpha H \partial_y G, \partial^\alpha \partial_y Q \rangle \right| \\
        &\leq C \left( \sum_{\alpha_1 \leq 1, |\alpha| \leq m} \|\partial^\alpha H\|_2^2 \left\| \partial_x G \right\|_{L^\infty}^2 \right)^{\frac{1}{2}}
           \left( \sum_{\alpha_1 \leq 1, |\alpha| \leq m} \|\partial^\alpha \partial_x Q\|_2^2 \right)^{\frac{1}{2}} \\
        &\quad + C \left( \sum_{\alpha_1 \leq 1, |\alpha| \leq m} \|\partial^{\alpha} H\|_2^2 \|\psi \partial_y G\|_{L^\infty}^2 \right)^{\frac{1}{2}}
           \left( \sum_{\alpha_1 \leq 1, |\alpha| \leq m} \left\| \frac{\partial^\alpha \partial_y Q}{\psi} \right\|_2^2 \right)^{\frac{1}{2}},
    \end{align*}
   and  by Lemma \ref{lem:embedding} and Lemma \ref{lem:psi}, this becomes
    \begin{align*}
        &\leq C \left( \sum_{\alpha_1 \leq 1, |\alpha| \leq m} \|\partial^\alpha H\|_2^2 \right)^{\frac{1}{2}}
           \left\| \partial_x G \right\|_{L^\infty}
           \left( \sum_{\alpha_1 \leq 1, |\alpha| \leq m} \|\partial^\alpha \partial_x Q\|_2^2 \right)^{\frac{1}{2}} \\
        &\quad + C_\delta \left( \sum_{\alpha_1 \leq 1, |\alpha| \leq m} \|\partial^\alpha H\|_2^2 \right)^{\frac{1}{2}}
           \|\psi \partial_y G\|_{L^\infty}
           \left( \sum_{\tau \leq 1}\|\partial^\tau_t\partial_y Q\|_2^2 + \sum_{\alpha_1 \leq 1, |\alpha| \leq m} \|\partial^\alpha \nabla^2 Q\|_2^2 \right)^{\frac{1}{2}}\\
        &\leq C_\delta \left( \sum_{\alpha_1 \leq 1, |\alpha| \leq m} \|\partial^\alpha H\|_2^2 \right)^{\frac{1}{2}}
           \left( \sum_{\tau \leq 1, |\alpha| \leq 2} \|\partial^\alpha G\|_2^2 \right)^{\frac{1}{4}}
           \left( \sum_{\tau \leq 1, |\alpha| \leq 2} \|\partial^\alpha \partial_y G\|_2^2 \right)^{\frac{1}{4}} \\
        &\quad \times \left( \sum_{\tau \leq 1}\|\partial^\tau_t \nabla Q\|_2^2 + \sum_{\alpha_1 \leq 1, |\alpha| \leq m} \|\partial^\alpha \nabla^2 Q\|_2^2 \right)^{\frac{1}{2}}.
    \end{align*}
    Combining \(d_1\) and \(d_2\), the proof of (d) is complete.
\end{proof}

\begin{lemma}\label{lem:nabla H cdot nabla G}
    For $\nabla H\cdot \nabla G$, we have
    
    \textup{(a)}
    \begin{equation}
        \sum_{\alpha_1 \leq 1, |\alpha| \leq m} 
        \bigl| \bigl\langle \partial^\alpha ((\nabla H^a \cdot \nabla) G), \partial^\alpha \nabla Q \bigr\rangle \bigr|
        \leq C \| \nabla H^a \|_{Y^{1,m}_\infty} 
        \| \nabla G \|_{Y^{1,m}} 
        \| \nabla Q \|_{Y^{1,m}}.
    \end{equation}
    
    \textup{(b)}
    \begin{equation}
        \sum_{\alpha_1 \leq 1, |\alpha| \leq m} 
        \bigl| \bigl\langle \partial^\alpha ((\nabla H \cdot \nabla) G^a), \partial^\alpha \nabla Q \bigr\rangle \bigr|
        \leq C \| \nabla G^a \|_{Y^{1,m}_\infty} 
        \| \nabla H \|_{Y^{1,m}} 
        \| \nabla Q \|_{Y^{1,m}}.
    \end{equation}
    
    \textup{(c)} While $\partial_y Q|_{y=0}=0$, it shows that
    \begin{equation}
        \begin{aligned}
            &\sum_{\alpha_1 \leq 1, |\alpha| \leq m} 
            \bigl| \bigl\langle \partial^\alpha ((\nabla H \cdot \nabla) G), \partial^\alpha \nabla Q \bigr\rangle \bigr| \\
            &\leq C \| \nabla H \|_{Y^{1,m}} 
            \| \nabla G \|_{Y^{1,m}}^{\frac 1 2} 
            \| \nabla^2 G \|_{Y^{1,m}}^{\frac 1 2} 
            \| \nabla Q \|_{Y^{1,m}} \\
            &\quad + C_\delta \| H \|_{Y^{1,2}}^{\frac 1 2} 
            \| \nabla H \|_{Y^{1,2}}^{\frac 1 2} 
            \| \nabla G \|_{Y^{1,m}} \left( \| \nabla Q \|_{Y^{1}} 
            + \| \nabla^2 Q \|_{Y^{1,m}} \right).
        \end{aligned}
    \end{equation}
\end{lemma}
\begin{proof}
The proof of (a) and (b) is similar to the cases of Lemma \ref{lem:H nabla G}.

For term (c), we have
% split it into two parts, separating the cases where at least one derivative falls on $\nabla H$ (denoted $e_1$) and where all derivatives fall on $G$ (denoted $e_2$):
    \begin{align*}
        \sum_{\alpha_1 \leq 1, |\alpha| \leq m} \left| \langle \partial^\alpha (\nabla H \cdot \nabla) G, \partial^\alpha \nabla Q \rangle \right| 
        & \leq C \sum_{\alpha_1 \leq 1, |\alpha| \leq m} \sum_{\substack{\beta \leq \alpha, |\beta| \leq m-1}} \left| \langle \partial^{\alpha-\beta} \nabla H \cdot \partial^\beta \nabla G, \partial^\alpha \nabla Q \rangle \right| \\
        &\quad + \sum_{\alpha_1 \leq 1, |\alpha| \leq m} \left| \langle \nabla H \cdot \partial^\alpha \nabla G, \partial^\alpha \nabla Q \rangle \right| \\
        &=: e_1+e_2.
    \end{align*}
    For \(e_1\), by H\"older inequality and Lemma \ref{lem:embedding}, it shows
    \begin{align*}
        e_1
        &\leq C \left( \sum_{\alpha_1 \leq 1, |\alpha| \leq m} \|\partial^\alpha \nabla H\|_2^2 \right)^{\frac{1}{2}}
        \left( \sum_{\tau \leq 1, |\alpha| \leq m-1} \|\partial^\alpha \nabla G\|_{L^\infty}^2 \right)^{\frac{1}{2}}
        \left( \sum_{\alpha_1 \leq 1, |\alpha| \leq m} \|\partial^\alpha \nabla Q\|_2^2 \right)^{\frac{1}{2}} \\
        &\leq C \left( \sum_{\alpha_1 \leq 1, |\alpha| \leq m} \|\partial^\alpha \nabla H\|_2^2 \right)^{\frac{1}{2}}
        \left( \sum_{\alpha_1 \leq 1, |\alpha| \leq m} \|\partial^\alpha \nabla G\|_2^2 \right)^{\frac{1}{4}}
        \left( \sum_{\alpha_1 \leq 1, |\alpha| \leq m} \|\partial^\alpha \nabla^2 G\|_2^2 \right)^{\frac{1}{4}}\times\\
        &\quad \left( \sum_{\alpha_1 \leq 1, |\alpha| \leq m} \|\partial^\alpha \nabla Q\|_2^2 \right)^{\frac{1}{2}}.
    \end{align*}
The estimate for \(e_2\) follows a similar strategy, but requires Lemma \ref{lem:psi} to handle \(\partial_y Q\):
    \begin{align*}
    e_2
        &\leq C \sum_{\alpha_1 \leq 1, |\alpha| \leq m} \|\partial_x H\|_{L^\infty} \|\partial^\alpha \nabla G\|_2 \|\partial^\alpha \partial_x Q\|_2 \\
        &\quad + C \sum_{\alpha_1 \leq 1, |\alpha| \leq m} \|\psi \partial_y H\|_{L^\infty} \|\partial^\alpha \nabla G\|_2 \left\| \frac{\partial^\alpha \partial_y Q}{\psi} \right\|_2\\
        &\leq C_\delta \left( \sum_{\tau \leq 1, |\alpha| \leq 2} \|\partial^\alpha H\|_2^2 \right)^{\frac{1}{4}}
           \left( \sum_{\tau \leq 1, |\alpha| \leq 2} \|\partial^\alpha \partial_y H\|_2^2 \right)^{\frac{1}{4}}
           \left( \sum_{\alpha_1 \leq 1, |\alpha| \leq m} \|\partial^\alpha \nabla G\|_2^2 \right)^{\frac{1}{2}}\times \\
        &\quad  \left( \sum_{\alpha_1 \leq 1, |\alpha| \leq m} \|\partial^\alpha \nabla Q\|_2^2 
           + \sum_{\alpha_1 \leq 1, |\alpha| \leq m} \|\partial^\alpha \nabla^2 Q\|_2^2 \right)^{\frac{1}{2}}.
    \end{align*}
    Combining \(e_1\) and \(e_2\), this proof can be proven.
\end{proof}

%\subsection{The proof of conormal Sobolev-type estimates}
\subsection{The proof of conormal Sobolev-type estimates}\label{section_能量估计}
Next, we perform the energy estimate, which essentially involves taking $\partial^\alpha$ on both sides of the equation, forming the $L^2$-inner product with $\partial^\alpha G$ (where $G$ corresponds to the principal term of the equation), and then summing over $\alpha$. 
% Since this step is analogous, we omit it here.

\begin{proof}[Proof of Proposition~\ref{prop:n}.] First, apply \(\partial^\alpha\) to both sides of the first equation of \eqref{eq:error} and take the inner product with \(\partial^\alpha n\) , this gives
\begin{align*}
    &\frac{1}{2} \frac{d}{dt} \sum_{\alpha_1 \leq 1, |\alpha| \leq m} \|\partial^\alpha n\|_2^2 
       - \sum_{\alpha_1 \leq 1, |\alpha| \leq m} \langle \partial^\alpha \Delta n, \partial^\alpha n \rangle \\
    &\leq \left| \sum_{\alpha_1 \leq 1, |\alpha| \leq m} \langle \partial^\alpha (u^a \cdot \nabla n + u \cdot \nabla n^a + u \cdot \nabla n), \partial^\alpha n \rangle \right| \\
    &\quad + \left| \sum_{\alpha_1 \leq 1, |\alpha| \leq m} \langle \partial^\alpha (\nabla \cdot (n^a \nabla c + n \nabla c^a + n \nabla c)), \partial^\alpha n \rangle \right| \\
    &\quad + \left| \sum_{\alpha_1 \leq 1, |\alpha| \leq m} \langle \partial^\alpha N, \partial^\alpha n \rangle \right| \\
    &=: A_1 + A_2 + A_3.
\end{align*}
Then  using parts by integration, we get
\begin{align*}
    &- \sum_{\alpha_1 \leq 1, |\alpha| \leq m} \langle \partial^\alpha \Delta n, \partial^\alpha n \rangle \\
    &= - \sum_{\alpha_1 \leq 1, |\alpha| \leq m} \langle \partial_x \partial^\alpha \partial_x n, \partial^\alpha n \rangle 
       - \sum_{\alpha_1 \leq 1, |\alpha| \leq m} \langle \partial_y \partial^\alpha \partial_y n, \partial^\alpha n \rangle \\
    &\qquad - \sum_{\alpha_1 \leq 1, |\alpha| \leq m} \langle [\partial^\alpha, \partial_y] \partial_y n, \partial^\alpha n \rangle \\
    &= \sum_{\alpha_1 \leq 1, |\alpha| \leq m} \langle \partial^\alpha \partial_x n, \partial^\alpha \partial_x n \rangle 
       + \sum_{\alpha_1 \leq 1, |\alpha| \leq m} \langle \partial^\alpha \partial_y n, \partial^\alpha \partial_y n \rangle \\
    &\qquad -
    \sum_{\alpha_1 \leq 1, |\alpha| \leq m} \langle \partial^\alpha \partial_y n, [\partial_y, \partial^\alpha ] n \rangle - \sum_{\alpha_1 \leq 1, |\alpha| \leq m} \langle [\partial^\alpha, \partial_y] \partial_y n, \partial^\alpha n \rangle. 
\end{align*}
For the third term or the last term,  it is sufficient to consider \(\alpha_3 \ne 0\) (other cases vanish), 
% let the truncation function \(\psi\) in \(\partial^\alpha n\) exchange to \([\partial^\alpha, \partial_y]\), this becomes
and we have 
\begin{equation}\label{ineq:Delta n}
    \begin{aligned}
   - \sum_{\alpha_1 \leq 1, |\alpha| \leq m} \langle \partial^\alpha \Delta n, \partial^\alpha n \rangle &\geq \sum_{\alpha_1 \leq 1, |\alpha| \leq m} \|\partial^\alpha \nabla n\|_2^2 
       - C\delta \sum_{\alpha_1 \leq 1, |\alpha| \leq m} |\langle \partial^\alpha \partial_y n, \partial^{\alpha-(0,0,1)} \partial_y n \rangle| \\
    &\geq (1 - C\delta) \sum_{\alpha_1 \leq 1, |\alpha| \leq m} \|\partial^\alpha \nabla n\|_2^2.
\end{aligned} 
\end{equation}
For \(A_1\), we divide it into three parts:
    \begin{align*}
    A_1
    &\leq \sum_{\alpha_1 \leq 1, |\alpha| \leq m} |\langle \partial^\alpha (u^a \cdot \nabla n), \partial^\alpha n \rangle| + \sum_{\alpha_1 \leq 1, |\alpha| \leq m} |\langle \partial^\alpha (u \cdot \nabla n^a), \partial^\alpha n \rangle| \\
    &\quad + \sum_{\alpha_1 \leq 1, |\alpha| \leq m} |\langle \partial^\alpha (u \cdot \nabla n), \partial^\alpha n \rangle| \\
    &=: A_{11} + A_{12} + A_{13}.
\end{align*}
For \(A_{11}\), by Lemma~\ref{lem:H cdot nabla G} (a), it shows
\begin{align*}
    A_{11} &\leq \frac{C_\delta}{\sigma} \left( 1 + \|u^a\|_{Y^{1,m+1}_{\infty}} + \|fe^{-y}\|_{Y^{1,m}_{\infty}}^2 \right)^2 \|n\|_{Y^{1,m}}^2 + C\sigma \varepsilon^4 \|\partial_y n\|_{Y^{1,m}}^2.
\end{align*}
For \(A_{12}\), by Lemma~\ref{lem:H cdot nabla G} (b), it follows that
\begin{align*}
    A_{12} &\leq C \left( \left( \|n^a\|_{Y^{1,m+1}_{\infty}}^2 + 1 \right) \|n\|_{Y^{1,m}}^2 
        + \|u\|_{Y^{1}}^2 + \|\omega\|_{Y^{1,m}}^2 \right) \\
    &\quad + C \varepsilon^4 \|fe^{-y}\|_{Y^{1,m+1}}^2 
        \left( \|n^a\|_{Y^{1,m+1}_{\infty}}^2 + \|\partial_y n^a\|_{Y^{1,m}}^2 \right).
\end{align*}
For \(A_{13}\), by Lemma~\ref{lem:H cdot nabla G} (c), we have
\begin{align*}
    A_{13} &\leq \frac{C}{\sigma} \left( \|u\|_{Y^{1}}^2 + \|\omega\|_{Y^{1,m}}^2 + \|n\|_{Y^{1,m}}^2 \right) + \frac{C}{\sigma \varepsilon^4} \left( \|u\|_{Y^{1}}^2 + \|\omega\|_{Y^{1,m}}^2 + \|n\|_{Y^{1,m}}^2 \right)^3 \\
    &\quad + \frac{C}{\sigma} \varepsilon^2 \|fe^{-y}\|_{Y^{1,m+1}}^2 \|n\|_{Y^{1,m}}^2 + C\sigma \varepsilon^2 \|\nabla n\|_{Y^{1,m}}^2.
\end{align*}
Summing up, we get
\begin{equation}\label{ineq:A1}
    \begin{aligned}
    A_1 \leq &\frac{C_\delta}{\sigma} \left( 1 +\|n^a\|_{Y^{1,m+1}_{\infty}}+ \|u^a\|_{Y^{1,m+1}_{\infty}} + \|fe^{-y}\|_{Y^{1,m}_{\infty}}^2 \right)^2 \|n\|_{Y^{1,m}}^2+ C\sigma \varepsilon^2 \|\nabla n\|_{Y^{1,m}}^2\\
    &\quad + C\varepsilon^4 \|fe^{-y}\|_{Y^{1,m+1}}^2 
        \left(\|n\|_{Y^{1,m}}^2+ \|n^a\|_{Y^{1,m+1}_{\infty}}^2 + \|\partial_y n^a\|_{Y^{1,m}}^2 \right)\\
        &\quad \frac{C}{\sigma} \left( \|u\|_{Y^{1}}^2 + \|\omega\|_{Y^{1,m}}^2 + \|n\|_{Y^{1,m}}^2 \right) + \frac{C}{\sigma \varepsilon^4} \left( \|u\|_{Y^{1}}^2 + \|\omega\|_{Y^{1,m}}^2 + \|n\|_{Y^{1,m}}^2 \right)^3\\
        \leq &\frac{C_\delta}{\sigma} \left( 1 +\|n^a\|_{Y^{1,m+1}_{\infty}}+ \|u^a\|_{Y^{1,m+1}_{\infty}} + \|fe^{-y}\|_{Y^{1,m+1}}^2 \right)^2 \left( \|u\|_{Y^{1}}^2 + \|\omega\|_{Y^{1,m}}^2 + \|n\|_{Y^{1,m}}^2 \right)\\
    &\quad + C\sigma \varepsilon^2 \|\nabla n\|_{Y^{1,m}}^2+ C \varepsilon^4 \|fe^{-y}\|_{Y^{1,m+1}}^2 
        \left( \|n^a\|_{Y^{1,m+1}_{\infty}}^2 + \|\partial_y n^a\|_{Y^{1,m}}^2 \right)\\
        &\quad + \frac{C}{\sigma \varepsilon^4} \left( \|u\|_{Y^{1}}^2 + \|\omega\|_{Y^{1,m}}^2 + \|n\|_{Y^{1,m}}^2 \right)^3.
\end{aligned}
\end{equation}
For \(A_2\), we have the following estimate:
\begin{align*}
    A_2
    &\leq \sum_{\alpha_1 \leq 1, |\alpha| \leq m} |\langle \partial^\alpha (n^a \nabla c + n \nabla c^a + n \nabla c), \partial^\alpha \nabla n \rangle| \\
    &\quad + \sum_{\alpha_1 \leq 1, |\alpha| \leq m} |\langle [\partial^\alpha, \partial_y] (n^a \partial_y c + n \partial_y c^a + n \partial_y c), \partial^\alpha n \rangle| \\
    &\quad + \sum_{\alpha_1 \leq 1, |\alpha| \leq m} |\langle \partial^\alpha (n^a \partial_y c + n \partial_y c^a + n \partial_y c), [\partial_y, \partial^\alpha] n \rangle|,
\end{align*}
by \eqref{exchange oper} and H\"older inequality, which becomes
\begin{align*}
    &\leq \sum_{\alpha_1 \leq 1, |\alpha| \leq m} |\langle \partial^\alpha (n^a \nabla c + n \nabla c^a + n \nabla c), \partial^\alpha \nabla n \rangle| \\
    &\quad + C\delta \sum_{\alpha_1 \leq 1, |\alpha| \leq m} |\langle \partial^\alpha (n^a \partial_y c + n \partial_y c^a + n \partial_y c), \partial^{\alpha-(0,0,1)} \partial_y n \rangle|\\
    &\le C \left( \sum_{\alpha_1 \leq 1, |\alpha| \leq m} | \langle \partial^\alpha (n^a \nabla c), \partial^\alpha \nabla n \rangle |
        + \sum_{\alpha_1 \leq 1, |\alpha| \leq m} | \langle \partial^\alpha (n^a \partial_y c), \partial^{\alpha - (0,0,1)} \partial_y n \rangle | \right) \\
    &\quad+ C \left( \sum_{\alpha_1 \leq 1, |\alpha| \leq m} | \langle \partial^\alpha (n \nabla c^a), \partial^\alpha \nabla n \rangle |
        + \sum_{\alpha_1 \leq 1, |\alpha| \leq m} | \langle \partial^\alpha (n \partial_y c^a), \partial^{\alpha - (0,0,1)} \partial_y n \rangle | \right) \\
    &\quad+ C \left( \sum_{\alpha_1 \leq 1, |\alpha| \leq m} | \langle \partial^\alpha (n \nabla c), \partial^\alpha \nabla n \rangle |
        + \sum_{\alpha_1 \leq 1, |\alpha| \leq m} | \langle \partial^\alpha (n \partial_y c), \partial^{\alpha - (0,0,1)} \partial_y n \rangle | \right)\\
    &=: A_{21}+A_{22}+A_{23}.
\end{align*}
For \(A_{21}\), by Lemma~\ref{lem:H nabla G} (a), we have
% A_{21} 部分
\begin{align*}
    A_{21} &\le C \|n^{a}\|_{Y_{\infty}^{1,m}} \|\nabla c\|_{Y^{1,m}} \|\nabla n\|_{Y^{1,m}} \\
    &\le \frac{C}{\sigma} \|n^{a}\|_{Y_{\infty}^{1,m}}^{2} \|\nabla c\|_{Y^{1,m}}^{2} + C\sigma \|\nabla n\|_{Y^{1,m}}^{2}.
\end{align*}
For \(A_{22}\), by Lemma~\ref{lem:H nabla G} (b), it shows 
% A_{22} 部分
\begin{align*}
    A_{22} &\le C \|n\|_{Y^{1,m}} \|\nabla c^{a}\|_{Y_{\infty}^{1,m}} \|\nabla n\|_{Y^{1,m}} \\
    &\le \frac{C}{\sigma} \|\nabla c^{a}\|_{Y_{\infty}^{1,m}}^{2} \|n\|_{Y^{1,m}}^{2} + C\sigma \|\nabla n\|_{Y^{1,m}}^{2}.
\end{align*}
For \(A_{23}\), by Lemma~\ref{lem:H nabla G} (c), this becomes
% A_{23} 部分
\begin{align*}
    A_{23} &\le \|n\|_{Y^{1,m}}^{\frac 1 2} \|\nabla n\|_{Y^{1,m}}^{\frac 1 2} \|\nabla c\|_{Y^{1,m}} \|\nabla n\|_{Y^{1,m}} + \|n\|_{Y^{1,m}} \|\nabla c\|_{Y^{1,1}}^{\frac 1 2} \|\nabla^{2} c\|_{Y^{1,1}}^{\frac 1 2} \|\nabla n\|_{Y^{1,m}} \\
    &\le C\sigma \|\nabla n\|_{Y^{1,m}}^{2} + \frac{C}{\sigma} \|n\|_{Y^{1,m}} \|\nabla n\|_{Y^{1,m}} \|\nabla c\|_{Y^{1,m}}^{2}  + \frac{C}{\sigma} \|n\|_{Y^{1,m}}^{2} \|\nabla c\|_{Y^{1,1}} \|\nabla^{2} c\|_{Y^{1,1}} \\
    &\le C\sigma \|\nabla n\|_{Y^{1,m}}^{2} + \frac{C}{\sigma^{3}} \|n\|_{Y^{1,m}}^{2} \left(\|\nabla c\|_{Y^{1,m}}^{2}\right)^{2} + C\sigma \varepsilon^{2} \|\nabla^{2} c\|_{Y^{1,1}}^{2} + \frac{C}{\sigma^{3}\varepsilon^{2}} \left(\|n\|_{Y^{1,m}}^{2}\right)^{2} \|\nabla c\|_{Y^{1,1}}^{2}.
\end{align*}
Summing up, we get
\begin{equation}\label{ineq:A2}
    \begin{aligned}
        A_{2}\leq& \frac{C}{\sigma} \|n^{a}\|_{Y_{\infty}^{1,m}}^{2} \|\nabla c\|_{Y^{1,m}}^{2} + C\sigma \|\nabla n\|_{Y^{1,m}}^{2}+\frac{C}{\sigma} \|\nabla c^{a}\|_{Y_{\infty}^{1,m}}^{2} \|n\|_{Y^{1,m}}^{2}\\
        &\quad + \frac{C}{\sigma^{3}} \|n\|_{Y^{1,m}}^{2} \left(\|\nabla c\|_{Y^{1,m}}^{2}\right)^{2} + C\sigma \varepsilon^{2} \|\nabla^{2} c\|_{Y^{1,1}}^{2} \\
        &\quad+ \frac{C}{\sigma^{3}\varepsilon^{2}} \left(\|n\|_{Y^{1,m}}^{2}\right)^{2} \|\nabla c\|_{Y^{1,1}}^{2}\\
        \leq& \frac{C}{\sigma}\left( 1+\|n^{a}\|_{Y_{\infty}^{1,m}}^{2} +\|\nabla c^{a}\|_{Y_{\infty}^{1,m}}^{2}\right)^2\left(\|\nabla c\|_{Y^{1,m}}^{2} +\|n\|_{Y^{1,m}}^{2}\right) \\
        &\quad+ \frac{C}{\sigma^{3}\varepsilon^{4}} \left(\|n\|_{Y^{1,m}}^{2}+\|\nabla c\|_{Y^{1,m}}^{2}\right)^{3} + C\sigma \|\nabla n\|_{Y^{1,m}}^{2}+ C\sigma \varepsilon^{2} \|\nabla^{2} c\|_{Y^{1,1}}^{2} .
    \end{aligned}
\end{equation}
For $A_3$, using the Young inequality, we have
\begin{equation}\label{ineq:A3}
    \begin{aligned}
    \biggl| \sum_{|\alpha| \leq m,\; \tau \leq 1} \langle \partial^{\alpha} N, \; \partial^{\alpha} n \rangle \biggr|
    &\le C \sum_{|\alpha| \leq m,\; \tau \leq 1} \|\partial^{\alpha} N\|_{2}^{2}
       + C \sum_{|\alpha| \leq m,\; \tau \leq 1} \|\partial^{\alpha} n\|_{2}^{2} \\
    &\le C \|n\|_{Y^{1,m}}^{2} + C \|N\|_{Y^{1,m}}^{2}.
\end{aligned}
\end{equation}
Summing up, one can prove Proposition \ref{prop:n} by combining \eqref{ineq:Delta n}, \eqref{ineq:A1}, \eqref{ineq:A2} and \eqref{ineq:A3}.
\end{proof}

\begin{proof}
[Proof of Proposition~\ref{prop:c}.] Applying \(\partial^{\tau}_t\) to both sides of the second equation of \eqref{eq:error} and taking the inner product with \(\partial^{\tau}_t c\), this gives
% 导出的微分不等式部分
\begin{align*}
\frac{1}{2} \frac{d}{dt} \sum_{\tau \leq 1} \|\partial^{\tau}_t c\|_{2}^{2}
- \varepsilon^{2} \sum_{\tau \leq 1} \langle \partial^{\tau}_t \Delta c, \; \partial^{\tau}_t c \rangle 
&\le \sum_{\tau \leq 1} \bigl| \langle \partial^{\tau}_t (u^{a} \cdot \nabla c + u \cdot \nabla c + u \cdot \nabla c^{a}), \; \partial^{\tau}_t c \rangle \bigr| \\
&\quad + \sum_{\tau \leq 1} \bigl| \langle \partial^{\tau}_t (c^{a} n + c n^{a} + c n), \; \partial^{\tau}_t c \rangle \bigr| \\
&\quad + \sum_{\tau \leq 1} \bigl| \langle \partial^{\tau}_t K, \; \partial^{\tau}_t c \rangle \bigr|\\
&=:B_1+B_2+B_3.
\end{align*}
On the left, by integration by parts, we have
\begin{align}\label{ineq:c}
-\varepsilon^{2} \sum_{\tau \leq 1} \langle \partial^{\tau}_t \Delta c, \; \partial^{\tau}_t c \rangle
= \varepsilon^{2} \sum_{\tau \leq 1} \|\partial^{\tau}_t \nabla c\|_{2}^{2}.
\end{align}
On the right, for \(B_1\), we divide it into three parts:
% 1. 对流项分解
\begin{align*}
    &\sum_{\tau \le 1} \bigl| \langle \partial^{\tau}_t (u^{a} \cdot \nabla c + u \cdot \nabla c + u \cdot \nabla c^{a}), \; \partial^{\tau}_t c \rangle \bigr| \\
    &\le  \sum_{\tau \le 1} \bigl| \langle \partial^{\tau}_t (u^{a} \cdot \nabla c), \; \partial^{\tau}_t c \rangle \bigr|
    + C \sum_{\tau \le 1} \bigl| \langle \partial^{\tau}_t (u \cdot \nabla c^{a}), \; \partial^{\tau}_t c \rangle \bigr| + C \sum_{\tau \le 1} \bigl| \langle \partial^{\tau}_t (u \cdot \nabla c), \; \partial^{\tau}_t c \rangle \bigr|\\
     &=:B_{11}+B_{12}+B_{13}.
\end{align*}
For the first and second terms, by applying the H\"older inequality, we obtain
% 2. 三个子项的逐项估计
\begin{align*}
    & C \sum_{\tau \le 1} \bigl| \langle \partial^{\tau}_t (u^{a} \cdot \nabla c), \; \partial^{\tau}_t c \rangle \bigr|  \le C \biggl( \sum_{\tau \le 1} \|\partial^{\tau}_t u^{a}\|_{\infty}^{2} \biggr)^{\!\frac 1 2}
    \biggl( \sum_{\tau \le 1} \|\partial^{\tau}_t \nabla c\|_{2}^{2} \biggr)^{\!\frac 1 2}
    \biggl( \sum_{\tau \le 1} \|\partial^{\tau}_t c\|_{2}^{2} \biggr)^{\!\frac 1 2}, \\[4pt]
    & C \sum_{\tau \le 1} \bigl| \langle \partial^{\tau}_t (u \cdot \nabla c^{a}), \; \partial^{\tau}_t c \rangle \bigr|  \le C \biggl( \sum_{\tau \le 1} \|\partial^{\tau}_t u\|_{2}^{2} \biggr)^{\!\frac 1 2}
    \biggl( \sum_{\tau \le 1} \|\partial^{\tau}_t \nabla c^{a}\|_{\infty}^{2} \biggr)^{\!\frac 1 2}
    \biggl( \sum_{\tau \le 1} \|\partial^{\tau}_t c\|_{2}^{2} \biggr)^{\!\frac 1 2}.
\end{align*}
For the remaining term, a direct application of Lemma \ref{lem:embedding} and H\"older inequality yields
\begin{align*}
    &  \sum_{\tau \le 1} \bigl| \langle \partial^{\tau}_t (u \cdot \nabla c), \; \partial^{\tau}_t c \rangle \bigr| \\
    & \le \biggl( \sum_{\tau \le 1} \|\partial^{\tau}_t u\|_{\infty}^{2} \biggr)^{\!\frac 1 2}
        \biggl( \sum_{\tau \le 1} \|\partial^{\tau}_t \nabla c\|_{2}^{2} \biggr)^{\!\frac 1 2}
        \biggl( \sum_{\tau \le 1} \|\partial^{\tau}_t c\|_{2}^{2} \biggr)^{\!\frac 1 2} \\
&  \le \biggl( \sum_{|\alpha| \le 2,\; \alpha_1 \le 1} \|\partial^{\alpha} u\|_{2} \, \|\partial_{y}\partial^{\alpha}  u\|_{2} \biggr)^{\!\frac 1 2}
        \biggl( \sum_{\tau \le 1} \|\partial^{\tau}_t \nabla c\|_{2}^{2} \biggr)^{\!\frac 1 2}
        \biggl( \sum_{\tau \le 1} \|\partial^{\tau}_t c\|_{2}^{2} \biggr)^{\!\frac 1 2} \\
&  \le \biggl( \sum_{\tau \le 1} \|\partial^{\tau}_t u\|_{2}^{2}
        + \sum_{|\alpha| \le 2,\; \alpha_1 \le 1} \|\partial^{\alpha} \nabla u\|_{2}^{2} \biggr)^{\!\frac 1 2}
        \biggl( \sum_{\tau \le 1} \|\partial^{\tau}_t \nabla c\|_{2}^{2} \biggr)^{\!\frac 1 2}
        \biggl( \sum_{\tau \le 1} \|\partial^{\tau}_t c\|_{2}^{2} \biggr)^{\!\frac 1 2}.
\end{align*}
Summing up, we get
\begin{equation}\label{ineq:B1}
    \begin{aligned}
        B_1\leq &C\biggl(\|u^{a}\|_{Y_{\infty}^{1}}\|\nabla c\|_{Y^{1}}\|c\|_{Y^{1}} 
    +\|u\|_{Y^{1}}\|\nabla c^a\|_{Y_{\infty}^{1}}\|c\|_{Y^{1}}+\left(\|u\|_{Y^{1}}+\|\nabla u\|_{Y^{1,2}}\right)\|\nabla c\|_{Y^{1}}\|c\|_{Y^{1}}\biggr)\\
    \leq &C\left(1+\|u^{a}\|_{Y_{\infty}^{1}}+\|\nabla c^a\|_{Y_{\infty}^{1}}\right)\left(\|\nabla c\|_{Y^{1}}^2 
    +\|u\|_{Y^{1}}^2+\|c\|_{Y^{1}}^2+\|\omega\|_{Y^{1,2}}^2\right)+\frac{C}{\sigma}\left( \|\nabla c\|_{Y^{1}}^2+\|c\|_{Y^{1}}^2\right)^3.
    \end{aligned}
\end{equation}
For \(B_2\), using the same methods as in \(B_{13}\), it follows that
% 3. 源项估计
\begin{align}\label{ineq:B2}
    & \sum_{\tau \le 1} \bigl| \langle \partial^{\tau}_t (c^{a} n + c n^{a} + c n), \; \partial^{\tau}_t c \rangle \bigr| \nonumber\\
& \quad \le C \biggl( \sum_{\tau \le 1} \|\partial^{\tau}_t c^{a}\|_{\infty}^{2} \biggr)^{\!\frac 1 2}
        \biggl( \sum_{\tau \le 1} \|\partial^{\tau}_t n\|_{2}^{2} \biggr)^{\!\frac 1 2}
        \biggl( \sum_{\tau \le 1} \|\partial^{\tau}_t c\|_{2}^{2} \biggr)^{\!\frac 1 2}\nonumber \\
& \qquad + C \biggl( \sum_{\tau \le 1} \|\partial^{\tau}_t n^{a}\|_{\infty}^{2} \biggr)^{\!\frac 1 2}
        \sum_{\tau \le 1} \|\partial^{\tau}_t c\|_{2}^{2}
        + C \biggl( \sum_{\tau \le 1} \|\partial^{\tau}_t n\|_{\infty}^{2} \biggr)^{\!\frac 1 2}
        \sum_{\tau \le 1} \|\partial^{\tau}_t c\|_{2}^{2} \nonumber\\
& \quad \le C \left( \|c^{a}\|_{Y_{\infty}^{1}} + \|n^{a}\|_{Y_{\infty}^{1}} \right)
        \left( \|n\|_{Y^{1}}^{2} + \|c\|_{Y^1}^{2} \right) + C \|n\|_{Y^{1,2}}^{\frac 1 2} \|\partial_{y} n\|_{Y^{1,2}}^{\frac 1 2} \|c\|_{Y^1}^{2}\nonumber\\
        &\quad \leq C \left(1+ \|c^{a}\|_{Y_{\infty}^{1}} + \|n^{a}\|_{Y_{\infty}^{1}} \right)
        \left( \|n\|_{Y^{1}}^{2} + \|c\|_{Y^1}^{2} \right) + \frac{C}{\sigma } \left(\|n\|_{Y^{1,2}}^2+\|c\|_{Y^1}^{2}\right)^3+C\sigma \|\partial_y n\|_{Y^{1,2}}^2.
\end{align}
For \(B_3\), by Young inequality, it shows
% 4. 外力项估计
\begin{align}\label{ineq:B3}
\sum_{\tau \le 1} \bigl| \langle \partial^{\tau}_t K, \; \partial^{\tau}_t c \rangle \bigr|
&\le C \sum_{\tau \le 1} \|\partial^{\tau}_t K\|_{2}^{2}
   + C \sum_{\tau \le 1} \|\partial^{\tau}_t c\|_{2}^{2} \nonumber\\
&\le C \|K\|_{Y^{1}}^{2} + C \|c\|_{Y^1}^{2}.
\end{align}
Summing up, this proposition can be proved by \eqref{ineq:c}, \eqref{ineq:B1}, \eqref{ineq:B2} and \eqref{ineq:B3}.
\end{proof}

\begin{proof}
[Proof of Proposition~\ref{prop:nabla c}.] Apply \(\partial^\alpha \nabla\) to both sides of the second equation in \eqref{eq:error} and take the inner product with \(\partial^\alpha \nabla c\), We have:
% 2. 导出的微分不等式
\begin{align*}
& \frac{1}{2} \frac{d}{dt} \sum_{\alpha_1 \leq 1, |\alpha| \leq m} \|\partial^{\alpha} \nabla c\|_{2}^{2}
   - \varepsilon^{2} \sum_{\alpha_1 \leq 1, |\alpha| \leq m} \langle \partial^{\alpha} \nabla \Delta c, \; \partial^{\alpha} \nabla c \rangle \\
& \quad \le \Biggl| \sum_{\alpha_1 \leq 1, |\alpha| \leq m} \langle \partial^{\alpha} \left( (\nabla u^{a} \cdot \nabla) c + (\nabla u \cdot \nabla) c^{a} + (\nabla u \cdot \nabla) c \right), \; \partial^{\alpha} \nabla c \rangle \Biggr| \\
& \qquad + \Biggl| \sum_{\alpha_1 \leq 1, |\alpha| \leq m} \langle \partial^{\alpha} \left( (u^{a} \cdot \nabla) \nabla c + (u \cdot \nabla) \nabla c^{a} + (u \cdot \nabla) \nabla c \right), \; \partial^{\alpha} \nabla c \rangle \Biggr|\\
    & \qquad + \Biggl| \sum_{\alpha_1 \leq 1, |\alpha| \leq m} \langle \partial^{\alpha} (\nabla c^{a} n + \nabla c n^{a} + \nabla c n), \; \partial^{\alpha} \nabla c \rangle \Biggr| \\
    & \qquad + \Biggl| \sum_{\alpha_1 \leq 1, |\alpha| \leq m} \langle \partial^{\alpha} (c^{a} \nabla n + c \nabla n^{a} + c \nabla n), \; \partial^{\alpha} \nabla c \rangle \Biggr| + \Biggl| \sum_{\alpha_1 \leq 1, |\alpha| \leq m} \langle \partial^{\alpha} \nabla K, \; \partial^{\alpha} \nabla c \rangle \Biggr| \\
    & =:\bar{B}_{1} + \bar{B}_{2} + \bar{B}_{3} + \bar{B}_{4} + \bar{B}_{5}.
\end{align*}
On the left, by integration by parts, it shows
% 1. 分解为 x 与 y 方向
\begin{align*}
    &- \varepsilon^{2} \sum_{\alpha_1 \leq 1, |\alpha| \leq m} \langle \partial^{\alpha} \nabla \Delta c, \; \partial^{\alpha} \nabla c \rangle \\
    % &= -\varepsilon^{2} \sum_{\alpha_1 \leq 1, |\alpha| \leq m} \langle \partial_{x} \partial^{\alpha} \nabla \partial_{x} c, \; \partial^{\alpha} \nabla c \rangle  - \varepsilon^{2} \sum_{\alpha_1 \leq 1, |\alpha| \leq m} \langle \partial_{y} \partial^{\alpha} \nabla \partial_{y} c, \; \partial^{\alpha} \nabla c \rangle \\
    % &\quad - \varepsilon^{2} \sum_{\alpha_1 \leq 1, |\alpha| \leq m} \langle [\partial^{\alpha}, \partial_{y}] \nabla \partial_{y} c, \; \partial^{\alpha} \nabla c \rangle\\
    &= \varepsilon^{2} \sum_{\alpha_1 \leq 1, |\alpha| \leq m} \langle \partial^{\alpha} \nabla \partial_{x} c, \; \partial^{\alpha} \nabla \partial_{x} c \rangle + \varepsilon^{2} \sum_{\alpha_1 \leq 1, |\alpha| \leq m} \langle \partial^{\alpha} \nabla \partial_{y} c, \; \partial^{\alpha} \nabla \partial_{y} c \rangle \\
    &\quad - \varepsilon^{2} \sum_{\alpha_1 \leq 1, |\alpha| \leq m} \langle \partial^{\alpha} \nabla \partial_{y} c, \; [\partial^{\alpha}, \partial_{y}] \nabla c \rangle - \varepsilon^{2} \sum_{\alpha_1 \leq 1, |\alpha| \leq m} \langle [\partial^{\alpha}, \partial_{y}] \nabla \partial_{y} c, \; \partial^{\alpha} \nabla c \rangle,
\end{align*}
which is greater than
% for the third term, by , and for the last term, if \(\alpha_3 \ne 0\), let the truncation function \(\psi\) in \(\partial^\alpha \nabla c\) exchange to \([\partial^\alpha, \partial_y]\), this becomes
% 3. 下界估计
\begin{align}\label{ineq:nablac}
&\ge \varepsilon^{2} \sum_{\alpha_1 \leq 1, |\alpha| \leq m} \|\partial^{\alpha} \nabla^{2} c\|_{2}^{2}
   - C \delta \varepsilon^{2} \sum_{\alpha_1 \leq 1, |\alpha| \leq m}
     \bigl| \langle \partial^{\alpha} \nabla \partial_{y} c, \; \partial^{\alpha - (0,0,1)} \partial_{y} \nabla c \rangle \bigr| \nonumber\\
&\ge (1 - C \delta) \, \varepsilon^{2} \sum_{\alpha_1 \leq 1, |\alpha| \leq m} \|\partial^{\alpha} \nabla^{2} c\|_{2}^{2},
\end{align}
due to \eqref{exchange oper}.

{\bf Estimate of \(\bar{B}_1\).}  We divide it into three parts:
% 1. B1 的分解
\begin{align*}
    \bar{B}_1
    & \le C \sum_{\alpha_1 \leq 1, |\alpha| \leq m} \bigl| \langle \partial^{\alpha} ( (\nabla u^{a} \cdot \nabla) c), \; \partial^{\alpha} \nabla c \rangle \bigr| + C \sum_{\alpha_1 \leq 1, |\alpha| \leq m} \bigl| \langle \partial^{\alpha} ( (\nabla u \cdot \nabla) c^{a}), \; \partial^{\alpha} \nabla c \rangle \bigr| \\
    & \qquad + C \sum_{\alpha_1 \leq 1, |\alpha| \leq m} \bigl| \langle \partial^{\alpha} ( (\nabla u \cdot \nabla) c), \; \partial^{\alpha} \nabla c \rangle \bigr| \\
    &\quad =: \bar{B}_{11}+\bar{B}_{12}+\bar{B}_{13}.
\end{align*}
For \(\bar{B}_{11}\), by Lemma~\ref{lem:nabla H cdot nabla G} (a), it follows that
% 2. \bar{B}_{11} 估计（引理 3.1）
\begin{align*}
\bar{B}_{11} &\le C \|\nabla u^{a}\|_{Y_{\infty}^{1,m}} \|\nabla c\|_{Y^{1,m}}^{2}.
\end{align*}
For \(\bar{B}_{12}\), by Lemma~\ref{lem:nabla H cdot nabla G} (b) and Lemma \ref{lem:omega}, we have
% 3. \bar{B}_{12} 估计（引理 3.2）
\begin{align*}
    \bar{B}_{12} &\le C \|\nabla c^{a}\|_{Y_{\infty}^{1,m}} \|\nabla u\|_{Y^{1,m}} \|\nabla c\|_{Y^{1,m}} \\
    &\le C \|\nabla c^{a}\|_{Y_{\infty}^{1,m}} \left( \|\omega\|_{Y^{1,m}} + \varepsilon^{2} \|f e^{-y}\|_{Y^{1,m+1}} \right) \|\nabla c\|_{Y^{1,m}}.
\end{align*}
For \(\bar{B}_{13}\), by Lemma~\ref{lem:nabla H cdot nabla G} (c), it shows
% 4. B_{13} 的初步估计（引理 3.3）
\begin{align*}
\bar{B}_{13} &\le C \|\nabla u\|_{Y^{1,m}} \|\nabla c\|_{Y^{1,m}}^{\frac 1 2} \|\nabla^{2} c\|_{Y^{1,m}}^{\frac 1 2} \|\nabla c\|_{Y^{1,m}} \\
       &\quad + C_\delta \|u\|_{Y^{1,2}}^{\frac 1 2} \|\nabla u\|_{Y^{1,2}}^{\frac 1 2} \|\nabla c\|_{Y^{1,m}} \left( \|\nabla c\|_{Y^{1}} + \|\nabla^{2} c\|_{Y^{1,m}} \right),
\end{align*}
and by Lemma \ref{lem:omega} and Young inequality, this becomes
% 5. B_{13} 的进一步化简与最终合并
\begin{align*}
    & \le C \left(  \|\omega\|_{Y^{1,m}}^{2} + \varepsilon^{4} \|f e^{-y}\|_{Y^{1,m+1}}^{2} \right) + C \|\nabla c\|_{Y^{1,m}} \|\nabla^{2} c\|_{Y^{1,m}} \|\nabla c\|_{Y^{1,m}}^{2} \\
    & \quad + C_\delta \left( \|u\|_{Y^{1}}^{2} + \|\omega\|_{Y^{1,m}}^{2} + \varepsilon^{4} \|f e^{-y}\|_{Y^{1,m+1}}^{2} \right) \|\nabla c\|_{Y^{1,m}}^{2} + C \|\nabla c\|_{Y^{1}}^{2} \\
    & \quad + \frac{C_\delta}{\sigma \varepsilon^{2}} \left( \|u\|_{Y^{1}}^{2} + \|\omega\|_{Y^{1,m}}^{2} + \varepsilon^{4} \|f e^{-y}\|_{Y^{1,m+1}}^{2} \right) \|\nabla c\|_{Y^{1,m}}^{2}  + C \sigma \varepsilon^{2} \|\nabla^{2} c\|_{Y^{1,m}}^{2} \\
    & \leq C_\delta \left( \|\nabla c\|_{Y^{1,1}}^{2} + \|u\|_{Y^{1}}^{2} + \|\omega\|_{Y^{1,m}}^{2} + \varepsilon^{4} \|f e^{-y}\|_{Y^{1,m+1}}^{2} \right)  + \frac{C}{\sigma \varepsilon^{2}} \left( \|\nabla c\|_{Y^{1,m}}^{2} \right)^{3} \\
    & \quad + \frac{C_\delta}{\sigma \varepsilon^{2}} \left( \|u\|_{Y^{1}}^{2} + \|\omega\|_{Y^{1,m}}^{2} + \varepsilon^{4} \|f e^{-y}\|_{Y^{1,m+1}}^{2} \right) \|\nabla c\|_{Y^{1,m}}^{2} + C \sigma \varepsilon^{2} \|\nabla^{2} c\|_{Y^{1,m}}^{2}.
\end{align*}
In summary, we get
\begin{equation}\label{ineq:barB1}
    \begin{aligned}
        \bar{B}_1\leq &C \|\nabla u^{a}\|_{Y_{\infty}^{1,m}} \|\nabla c\|_{Y^{1,m}}^{2}+C \|\nabla c^{a}\|_{Y_{\infty}^{1,m}} \left( \|\omega\|_{Y^{1,m}} + \varepsilon^{2} \|f e^{-y}\|_{Y^{1,m+1}} \right) \|\nabla c\|_{Y^{1,m}}\\
        &\quad+C_\delta \left( \|\nabla c\|_{Y^{1,1}}^{2} + \|u\|_{Y^{1}}^{2} + \|\omega\|_{Y^{1,m}}^{2} + \varepsilon^{4} \|f e^{-y}\|_{Y^{1,m+1}}^{2} \right)  + \frac{C}{\sigma \varepsilon^{2}} \left( \|\nabla c\|_{Y^{1,m}}^{2} \right)^{3} \\
    & \quad + \frac{C_\delta}{\sigma \varepsilon^{2}} \left( \|u\|_{Y^{1}}^{2} + \|\omega\|_{Y^{1,m}}^{2} + \varepsilon^{4} \|f e^{-y}\|_{Y^{1,m+1}}^{2} \right) \|\nabla c\|_{Y^{1,m}}^{2} + C \sigma \varepsilon^{2} \|\nabla^{2} c\|_{Y^{1,m}}^{2}\\
   \leq & \frac{C_\delta}{\sigma} \bigr(1+\|\nabla u^{a}\|_{Y_{\infty}^{1,m}}+\|\nabla c^{a}\|_{Y_{\infty}^{1,m}}+\|f e^{-y}\|_{Y^{1,m+1}}\bigl)^2 \bigr(\|\nabla c\|_{Y^{1,m}}^{2}+\|u\|_{Y^{1}}^{2}+\|\omega
   \|_{Y^{1,m}}^{2} \bigl) \\
   &\quad+C_\delta \varepsilon^4\|f e^{-y}\|_{Y^{1,m+1}}^2 \left(1+\|\nabla c^{a}\|_{Y_{\infty}^{1,m}} \right)+ C \sigma \varepsilon^{2} \|\nabla^{2} c\|_{Y^{1,m}}^{2}\\
        &\quad + \frac{C_\delta}{\sigma^2 \varepsilon^{4}} \left( \|\nabla c\|_{Y^{1,m}}^{2}+\|u\|_{Y^{1}}^{2} + \|\omega\|_{Y^{1,m}}^{2} \right)^{3}.
    \end{aligned}
\end{equation}
{\bf Estimate of \(\bar{B}_2\).} Similar with \(\bar{B}_1\), it is divided into three terms:
% 1. B2 的分解
\begin{align*}
    \bar{B}_2
    & \le C \sum_{\alpha_1 \leq 1, |\alpha| \leq m} \bigl| \langle \partial^{\alpha} ( (u^{a} \cdot \nabla) \nabla c), \; \partial^{\alpha} \nabla c \rangle \bigr| + C \sum_{\alpha_1 \leq 1, |\alpha| \leq m} \bigl| \langle \partial^{\alpha} ( (u \cdot \nabla) \nabla c^{a}), \; \partial^{\alpha} \nabla c \rangle \bigr| \\
    & \qquad + C \sum_{\alpha_1 \leq 1, |\alpha| \leq m} \bigl| \langle \partial^{\alpha} ( (u \cdot \nabla) \nabla c), \; \partial^{\alpha} \nabla c \rangle \bigr| \\
    &=: \bar{B}_{21}+\bar{B}_{22}+\bar{B}_{23}.
\end{align*}
For \(\bar{B}_{21}\), by Lemma~\ref{lem:H cdot nabla G} (a), we can derive
% 2. B_{21} 估计（引理 1.1）
\begin{align*}
    \bar{B}_{21} &\le \frac{C_\delta}{\sigma} \left( 1 + \|u^{a}\|_{Y_{\infty}^{1,m+1}} + \|f e^{-y}\|_{Y_{\infty}^{1,m}} \right)^{2} \|\nabla c\|_{Y^{1,m}}^{2} + C \sigma \varepsilon^{4} \|\partial_{y} \nabla c\|_{Y^{1,m}}^{2}.
\end{align*}
For \(\bar{B}_{22}\), by Lemma~\ref{lem:H cdot nabla G} (b), we have
% 3. B_{22} 估计（引理 1.2）
\begin{align*}
\bar{B}_{22} &\le C \left( \left( \|\nabla c^{a}\|_{Y_{\infty}^{1,m+1}}^{2} + 1 \right) \|\nabla c\|_{Y^{1,m}}^{2}
          + \|u\|_{Y^1}^{2} + \|\omega\|_{Y^{1,m}}^{2} \right) \\
       &\quad + C \varepsilon^{4} \|f e^{-y}\|_{Y^{1,m+1}}^{2}
          \left( \|\nabla c^{a}\|_{Y_{\infty}^{1,m+1}}^{2} + \|\partial_{y} \nabla c^{a}\|_{Y^{1,m}}^{2} \right).
\end{align*}
For \(\bar{B}_{23}\), by Lemma~\ref{lem:H cdot nabla G} (c), it follows that
% 4. B_{23} 估计（引理 1.3）
\begin{align*}
    \bar{B}_{23} &\le \frac{C}{\sigma} \left( \|u\|_{Y^1}^{2} + \|\omega\|_{Y^{1,m}}^{2} + \|\nabla c\|_{Y^{1,m}}^{2} \right) + \frac{C}{\sigma \varepsilon^{4}} \left( \|u\|_{Y^1}^{2} + \|\omega\|_{Y^{1,m}}^{2} + \|\nabla c\|_{Y^{1,m}}^{2} \right)^{3} \\
    &\quad + \frac{C}{\sigma} \varepsilon^{2} \|f e^{-y}\|_{Y^{1,m+1}}^{2} \|\nabla c\|_{Y^{1,m}}^{2} + C \sigma \varepsilon^{2} \|\nabla^{2} c\|_{Y^{1,m}}^{2}.
\end{align*}
Summing up, we get
\begin{equation}\label{ineq:barB2}
    \begin{aligned}
        \bar{B}_2\leq& \frac{C_\delta}{\sigma} \left( 1 + \|\nabla c^{a}\|_{Y_{\infty}^{1,m+1}}+\|u^{a}\|_{Y_{\infty}^{1,m+1}} + \|f e^{-y}\|_{Y_{\infty}^{1,m}} \right)^{2} \|\nabla c\|_{Y^{1,m}}^{2}+ C \sigma \varepsilon^{2} \|\nabla^{2} c\|_{Y^{1,m}}^{2}\\
          &+ \frac{C}{\sigma} \left( \|u\|_{Y^1}^{2} + \|\omega\|_{Y^{1,m}}^{2} + \|\nabla c\|_{Y^{1,m}}^{2} \right) + \frac{C}{\sigma \varepsilon^{4}} \left( \|u\|_{Y^1}^{2} + \|\omega\|_{Y^{1,m}}^{2} + \|\nabla c\|_{Y^{1,m}}^{2} \right)^{3} \\
    &+ \frac{C}{\sigma} \varepsilon^{2} \|f e^{-y}\|_{Y^{1,m+1}}^{2} \|\nabla c\|_{Y^{1,m}}^{2} + C \varepsilon^{4} \|f e^{-y}\|_{Y^{1,m+1}}^{2}
          \left( \|\nabla c^{a}\|_{Y_{\infty}^{1,m+1}}^{2} + \|\partial_{y} \nabla c^{a}\|_{Y^{1,m}}^{2} \right)\\
          \leq& \frac{C_\delta}{\sigma} \left( 1 + \|\nabla c^{a}\|_{Y_{\infty}^{1,m+1}}+\|u^{a}\|_{Y_{\infty}^{1,m+1}} + \|f e^{-y}\|_{Y^{1,m+1}} \right)^{2} \left( \|u\|_{Y^1}^{2} + \|\omega\|_{Y^{1,m}}^{2} + \|\nabla c\|_{Y^{1,m}}^{2} \right)\\
          & + \frac{C}{\sigma \varepsilon^{4}} \left( \|u\|_{Y^1}^{2} + \|\omega\|_{Y^{1,m}}^{2} + \|\nabla c\|_{Y^{1,m}}^{2} \right)^{3} + C \sigma \varepsilon^{2} \|\nabla^{2} c\|_{Y^{1,m}}^{2}\\
          &+ C \varepsilon^{4} \|f e^{-y}\|_{Y^{1,m+1}}^{2}
          \left( \|\nabla c^{a}\|_{Y_{\infty}^{1,m+1}}^{2} + \|\partial_{y} \nabla c^{a}\|_{Y^{1,m}}^{2} \right).
    \end{aligned}
\end{equation}
{\bf Estimate of \(\bar{B}_3\).} It splits into three terms:
% 1. B3 的分解
\begin{align*}
    \bar{B}_3
    & \le C \sum_{\alpha_1 \leq 1, |\alpha| \leq m} \bigl| \langle \partial^{\alpha} ( \nabla c^{a} n), \; \partial^{\alpha} \nabla c \rangle \bigr| + C \sum_{\alpha_1 \leq 1, |\alpha| \leq m} \bigl| \langle \partial^{\alpha} ( \nabla c n^{a}), \; \partial^{\alpha} \nabla c \rangle \bigr| \\
    & \qquad + C \sum_{\alpha_1 \leq 1, |\alpha| \leq m} \bigl| \langle \partial^{\alpha} ( \nabla c n), \; \partial^{\alpha} \nabla c \rangle \bigr| \\
    &=: \bar{B}_{31}+\bar{B}_{32}+\bar{B}_{33}.
\end{align*}
For \(\bar{B}_{31}\), by Lemma~\ref{lem:H nabla G} (b), we obtain
% 2. \bar{B}_{31} 估计（引理 2.2）
\begin{align*}
    \bar{B}_{31} &\le C \|n\|_{Y^{1,m}} \|\nabla c^{a}\|_{Y_{\infty}^{1,m}} \|\nabla c\|_{Y^{1,m}}.
\end{align*}
For \(\bar{B}_{32}\), by Lemma~\ref{lem:H nabla G} (a), this yields
% 3. \bar{B}_{32} 估计（引理 2.1）
\begin{align*}
    \bar{B}_{32} &\le C \|n^{a}\|_{Y_{\infty}^{1,m}} \|\nabla c\|_{Y^{1,m}}^2.
\end{align*}
For \(\bar{B}_{33}\), by Lemma~\ref{lem:H nabla G} (d), it shows
% 4. \bar{B}_{33} 估计（引理 2.4）
\begin{align*}
    \bar{B}_{33} &\le C \|n\|_{Y^{1,m}}^{\frac 1 2} \|\nabla n\|_{Y^{1,m}}^{\frac 1 2} \|\nabla c\|_{Y^{1,m}}^2 \\
    &\quad + C_\delta \|n\|_{Y^{1,m}} \|c\|_{Y^{1,2}}^{\frac 1 2} \|\nabla c\|_{Y^{1,2}}^{\frac 1 2}
          \left( \|\nabla c\|_{Y^{1}} + \|\nabla^{2} c\|_{Y^{1,m}} \right).
\end{align*}
In summary, we get
\begin{equation}\label{ineq:barB3}
    \begin{aligned}
        \bar{B}_3\leq& C \|n\|_{Y^{1,m}} \|\nabla c^{a}\|_{Y_{\infty}^{1,m}} \|\nabla c\|_{Y^{1,m}}+C \|n^{a}\|_{Y_{\infty}^{1,m}} \|\nabla c\|_{Y^{1,m}}^2+C \|n\|_{Y^{1,m}}^{\frac 1 2} \|\nabla n\|_{Y^{1,m}}^{\frac 1 2} \|\nabla c\|_{Y^{1,m}}^2 \\
    &\quad + C_\delta \|n\|_{Y^{1,m}} \|c\|_{Y^{1,2}}^{\frac 1 2} \|\nabla c\|_{Y^{1,2}}^{\frac 1 2}
          \left( \|\nabla c\|_{Y^{1}} + \|\nabla^{2} c\|_{Y^{1,m}} \right)\\
    \leq& \frac{C_\delta}{\sigma}\left(1+\|\nabla c^{a}\|_{Y_{\infty}^{1,m}}+\|n^{a}\|_{Y_{\infty}^{1,m}}\right) \left(\|n\|_{Y^{1,m}}^2+ \|c\|_{Y^{1,m}}^2+  \|\nabla c\|_{Y^{1,m}}^2\right)\\
    &\quad +\frac{C_\delta}{\sigma} \left(\|n\|_{Y^{1,m}}^2+\|\nabla c\|_{Y^{1,m}}^2\right)^3 +C\sigma\|\nabla n\|_{Y^{1,m}}^2+C\sigma \varepsilon^2\|\nabla^2 c\|_{Y^{1,m}}^2.
    \end{aligned}
\end{equation}
{\bf Estimate of \(\bar{B}_4\).} We can separate it into three parts:
% 1. \bar{B}_4 的分解
\begin{align*}
    \bar{B}_4
    & \le C \sum_{\alpha_1 \leq 1, |\alpha| \leq m} \bigl| \langle \partial^{\alpha} (c^{a} \nabla n), \; \partial^{\alpha} \nabla c \rangle \bigr| + C \sum_{\alpha_1 \leq 1, |\alpha| \leq m} \bigl| \langle \partial^{\alpha} (c \nabla n^{a}), \; \partial^{\alpha} \nabla c \rangle \bigr| \\
    & \qquad + C \sum_{\alpha_1 \leq 1, |\alpha| \leq m} \bigl| \langle \partial^{\alpha} (c \nabla n), \; \partial^{\alpha} \nabla c \rangle \bigr| \\
    &=: \bar{B}_{41}+\bar{B}_{42}+\bar{B}_{43}.
\end{align*}
For \(\bar{B}_{41}\), by Lemma~\ref{lem:H nabla G} (a), we have
% 2. \bar{B}_{41} 估计（引理 2.2）
\begin{align*}
    \bar{B}_{41} &\le C \|c^{a}\|_{Y_{\infty}^{1,m}} \|\nabla n\|_{Y^{1,m}} \|\nabla c\|_{Y^{1,m}} \\
    &\le \frac{C}{\sigma} \|c^{a}\|_{Y_{\infty}^{1,m}}^{2} \|\nabla c\|_{Y^{1,m}}^{2} + C \sigma \|\nabla n\|_{Y^{1,m}}^{2} .
\end{align*}
For \(\bar{B}_{42}\), by Lemma~\ref{lem:H nabla G} (b), this becomes
% 3. \bar{B}_{42} 估计（引理 2.1）
\begin{align*}
\bar{B}_{42} &\le C \|c\|_{Y^{1,m}} \|\nabla n^{a}\|_{Y_{\infty}^{1,m}} \|\nabla c\|_{Y^{1,m}}.
\end{align*}
For \(\bar{B}_{43}\), by Lemma~\ref{lem:H nabla G} (d), it follows that
% 4. \bar{B}_{43} 估计（引理 2.4）及后续化简
\begin{align*}
\bar{B}_{43} &\le C \|c\|_{Y^{1,m}}^{\frac 1 2} \|\nabla c\|_{Y^{1,m}}^{\frac 1 2} \|\nabla n\|_{Y^{1,m}} \|\nabla c\|_{Y^{1,m}} \\
       &\quad + C_\delta \|c\|_{Y^{1,m}} \|n\|_{Y^{1,2}}^{\frac 1 2} \|\nabla n\|_{Y^{1,2}}^{\frac 1 2}
          \left( \|\nabla c\|_{Y^{1,1}} + \|\nabla^{2} c\|_{Y^{1,m}} \right), 
\end{align*}
which is controlled by
\begin{align*}
       % &\le \frac{C}{\sigma} (\|c\|_{Y^{1,m}}^{2} +\|\nabla c\|_{Y^{1,m}}^{2}) \|\nabla c\|_{Y^{1,m}}^{2}
       %    + C \sigma \|\nabla n\|_{Y^{1,m}}^{2} + C_\delta \|c\|_{Y^{1,m}}^{2} \|n\|_{Y^{1,2}} \|\nabla n\|_{Y^{1,2}}\\
       % &\quad + C_\delta \|\nabla c\|_{Y^{1,1}}^{2} + \frac{C_\delta}{\sigma \varepsilon^{2}} \|c\|_{Y^{1,m}}^{2} \|n\|_{Y^{1,2}} \|\nabla n\|_{Y^{1,2}}
       %    + C \sigma \varepsilon^{2} \|\nabla^{2} c\|_{Y^{1,m}}^{2} \\[4pt]
       &\leq \frac{C}{\sigma} \left( \|c\|_{Y^{1,m}}^{2} + \|\nabla c\|_{Y^{1,m}}^{2} \right) \|\nabla c\|_{Y^{1,m}}^{2}
          + C \sigma \|\nabla n\|_{Y^{1,m}}^{2} + \frac{C_\delta}{\sigma} \left( \|c\|_{Y^{1,m}}^{2} \right)^{2} \|n\|_{Y^{1,2}}^{2}\\
       &\quad + C \sigma \|\nabla n\|_{Y^{1,2}}^{2} + C_\delta \|\nabla c\|_{Y^{1,1}}^{2} + \frac{C_\delta}{\sigma^{3} \varepsilon^{4}} \|c\|_{Y^{1,m}}^{4} \|n\|_{Y^{1,2}}^{2}
          + C \sigma \|\nabla n\|_{Y^{1,2}}^{2}
          + C \sigma \varepsilon^{2} \|\nabla^{2} c\|_{Y^{1,m}}^{2}.
\end{align*}
In summary, we get
\begin{equation}\label{ineq:barB4}
    \begin{aligned}
        \bar{B}_4\leq& \frac{C_\delta}{\sigma} \left(1+\|c^{a}\|_{Y_{\infty}^{1,m}}^{2} \right)\|\nabla c\|_{Y^{1,m}}^{2} + C \sigma \|\nabla n\|_{Y^{1,m}}^{2}+C \|c\|_{Y^{1,m}} \|\nabla n^{a}\|_{Y_{\infty}^{1,m}} \|\nabla c\|_{Y^{1,m}}\\
        &+ \frac{C}{\sigma} \left( \|c\|_{Y^{1,m}}^{2} + \|\nabla c\|_{Y^{1,m}}^{2} \right) \|\nabla c\|_{Y^{1,m}}^{2}
          + \frac{C_\delta}{\sigma^{3} \varepsilon^{4}} \|c\|_{Y^{1,m}}^{4} \|n\|_{Y^{1,2}}^{2}
          + C \sigma \varepsilon^{2} \|\nabla^{2} c\|_{Y^{1,m}}^{2}\\
    \leq& \frac{C_\delta}{\sigma} \left(1+\|c^{a}\|_{Y_{\infty}^{1,m}}^{2} +\|\nabla n^{a}\|_{Y_{\infty}^{1,m}}^2\right)\left(\|\nabla c\|_{Y^{1,m}}^{2}+\|c\|_{Y^{1,m}}^{2}\right) + C \sigma \|\nabla n\|_{Y^{1,m}}^{2}\\
        &+ \frac{C_\delta}{\sigma^{3} \varepsilon^{4}}\left( \|c\|_{Y^{1,m}}^{2} +\|n\|_{Y^{1,2}}^{2}+\|\nabla c\|_{Y^{1,m}}^{2}\right)^3
          + C \sigma \varepsilon^{2} \|\nabla^{2} c\|_{Y^{1,m}}^{2}.
    \end{aligned}
\end{equation}
{\bf Estimate of \(\bar{B}_5\).} By Young inequality, this yields
% 5. \bar{B}5 估计
\begin{align}\label{ineq:barB5}
    \bar{B}_5
    &\le C \sum_{\tau \le 1,\; |\alpha| \le m} \|\partial^{\alpha} \nabla K\|_{2}^{2}
              + C \sum_{\tau \le 1,\; |\alpha| \le m} \|\partial^{\alpha} \nabla c\|_{2}^{2} \nonumber\\
    & \le C \|\nabla c\|_{Y^{1,m}}^{2} + C \|\nabla K\|_{Y^{1,m}}^{2}.
\end{align}
Summing up, one can prove this proposition by \eqref{ineq:nablac} to \eqref{ineq:barB5}.\end{proof}

\begin{proof}
[Proof of Proposition \ref{prop:u}.] Applying \(\partial^\tau_t\) to both sides of the third equation in \eqref{eq:error} and taking the inner product with \(\partial^\tau_t u\), this yields
% 导出的微分不等式
\begin{align*}
    & \frac{1}{2} \frac{d}{dt} \sum_{\tau \le 1} \|\partial^{\tau}_t u\|_{2}^{2}
    - \varepsilon^{2} \sum_{\tau \le 1} \langle \partial^{\tau}_t \Delta u, \; \partial^{\tau}_t u \rangle \\
    & \quad \le \Biggl| \sum_{\tau \le 1} \langle \partial^{\tau}_t \left( u^{a} \cdot \nabla u + u \cdot \nabla u^{a} + u \cdot \nabla u \right), \; \partial^{\tau}_t u \rangle \Biggr| + \Biggl| \sum_{\tau \le 1} \langle \partial^{\tau}_t \nabla p, \; \partial^{\tau}_t u \rangle \Biggr| \\
    & \qquad + \Biggl| \sum_{\tau \le 1} \langle \partial^{\tau}_t n, \; \partial^{\tau}_t u_{2} \rangle \Biggr| + \Biggl| \sum_{\tau \le 1} \langle \partial^{\tau}_t U, \; \partial^{\tau}_t u \rangle \Biggr|=:I_1'+\cdots+I_4'
\end{align*}
On the left hand side, by integration by parts we have
% Laplacian 项的分部积分
\begin{align*}
&- \varepsilon^{2} \sum_{\tau \le 1} \langle \partial^{\tau}_t \Delta u, \; \partial^{\tau}_t u \rangle \\
&= -\varepsilon^{2} \Biggl( \sum_{\tau \le 1} \langle \partial^{\tau}_t \partial_{x}^{2} u, \; \partial^{\tau}_t u \rangle
   + \sum_{\tau \le 1} \langle \partial^{\tau}_t \partial_{y}^{2} u, \; \partial^{\tau}_t u \rangle \Biggr) \\
&= \varepsilon^{2} \Biggl( \sum_{\tau \le 1} \langle \partial^{\tau}_t \partial_{x} u, \; \partial^{\tau}_t \partial_{x} u \rangle
   + \sum_{\tau \le 1} \langle \partial^{\tau}_t \partial_{y} u, \; \partial^{\tau}_t \partial_{y} u \rangle - \sum_{\tau \le 1} \int_{\partial \mathbb{R}_{+}^{2}} \partial^{\tau}_t \partial_{y} u_{2} \; \partial^{\tau}_t u_{2}dx \Biggr),
\end{align*}
and by Gauss theorem, H\"older inequality and Young inequality, this becomes
\begin{align}\label{ineq:u}
&= \varepsilon^{2} \sum_{\tau \le 1} \|\partial^{\tau}_t \nabla u\|_{2}^{2}
   + \varepsilon^{4} \sum_{\tau \le 1} \int_{\partial \mathbb{R}_{+}^{2}} \partial^{\tau}_t \partial_{y} u_{2} \; \partial^{\tau}_t f dx\nonumber\\
   &\ge \varepsilon^{2} \sum_{\tau \le 1} \|\partial^{\tau}_t \nabla u\|_{2}^{2}
   - \varepsilon^{4} \sum_{\tau \le 1} \biggl| \int_{\mathbb{R}_{+}^{2}} \partial^{\tau}_t f \;
      \partial_{y} \left( \partial^{\tau}_t \partial_{x} u_{1} e^{-y} \right) \biggr| \nonumber\\
&\ge \varepsilon^{2} (1 - C \sigma) \sum_{\tau \le 1} \|\partial^{\tau}_t \nabla u\|_{2}^{2}
   - \frac{C}{\sigma} \varepsilon^{4} \Biggl( \sum_{\tau \le 1} \|\partial^{\tau}_t \partial_{x} f e^{-y}\|_{2}^{2}
   + \sum_{\tau \le 1} \|\partial^{\tau}_t f e^{-y}\|_{2}^{2} \Biggr).
\end{align}
On the right hand side,  we decompose the convective term into the following three parts:
% 1. 对流项分解
\beno
    I_1'&\leq & \Biggl| \sum_{\tau \le 1} \langle \partial^{\tau}_t \left( u^{a} \cdot \nabla u + u \cdot \nabla u^{a} + u \cdot \nabla u \right), \; \partial^{\tau}_t u \rangle \Biggr| \\
    & \le& C \sum_{\tau \le 1} \bigl| \langle \partial^{\tau}_t (u^{a} \cdot \nabla u), \; \partial^{\tau}_t u \rangle \bigr| + C \sum_{\tau \le 1} \bigl| \langle \partial^{\tau}_t (u \cdot \nabla u^{a}), \; \partial^{\tau}_t u \rangle \bigr|  + C \sum_{\tau \le 1} \bigl| \langle \partial^{\tau}_t (u \cdot \nabla u), \; \partial^{\tau}_t u \rangle \bigr|,
\eeno
 which are estimated, respectively. For the first term \(u^{a} \cdot \nabla u\), by H\"older inequality, we derive
% 2. 第一项 (u^a ⋅ ∇ u) 估计
\begin{align*}
    \sum_{\tau \le 1} \bigl| \langle \partial^{\tau}_t (u^{a} \cdot \nabla u), \; \partial^{\tau}_t u \rangle \bigr|
    &\le C \Biggl( \sum_{\tau \le 1} \|\partial^{\tau}_t u^{a}\|_{\infty}^{2} \Biggr)^{\frac{1}{2}} \Biggl( \sum_{\tau \le 1} \|\partial^{\tau}_t \nabla u\|_{2}^{2} \Biggr)^{\frac 1 2}
     \Biggl( \sum_{\tau \le 1} \|\partial^{\tau}_t u\|_{2}^{2} \Biggr)^{\frac 1 2}.
\end{align*}
For the second term \(u \cdot \nabla u^{a}\), we have
% 3. 第二项 (u ⋅ ∇ u^a) 估计
\begin{align*}
\sum_{\tau \le 1} \bigl| \langle \partial^{\tau}_t (u \cdot \nabla u^{a}), \; \partial^{\tau}_t u \rangle \bigr|
&\le C \Biggl( \sum_{\tau \le 1} \|\partial^{\tau}_t \nabla u^{a}\|_{\infty}^{2} \Biggr)^{\!\frac 1 2}
      \sum_{\tau \le 1} \|\partial^{\tau}_t u\|_{2}^{2} .
\end{align*}
For the last term \(u \cdot \nabla u\), we perform a more refined estimate:
% 4. 第三项 (u ⋅ ∇ u) 估计
\begin{align*}
    &\sum_{\tau \le 1} \bigl| \langle \partial^{\tau}_t (u \cdot \nabla u), \; \partial^{\tau}_t u \rangle \bigr| \le C \Biggl( \sum_{\tau \le 1} \|\partial^{\tau}_t u\|_{\infty}^{2} \Biggr)^{\frac 1 2}
    \Biggl( \sum_{\tau \le 1} \|\partial^{\tau}_t \nabla u\|_{2}^{2} \Biggr)^{\frac 1 2}
    \Biggl( \sum_{\tau \le 1} \|\partial^{\tau}_t u\|_{2}^{2} \Biggr)^{\frac 1 2} ,
\end{align*}
and by Lemma \ref{lem:embedding} and Lemma \ref{lem:omega}, this becomes
\begin{align*}
    &\le C \Biggl( \sum_{\alpha_1 \le 1,\; |\alpha| \le 2} \|\partial^{\alpha} u\|_{2}^{2} \Biggr)^{\frac 1 4}
    \Biggl( \sum_{\alpha_1 \le 1,\; |\alpha| \le 2} \|\partial^{\alpha} \partial_{y} u\|_{2}^{2} \Biggr)^{\frac 1 4} \Biggl( \sum_{\tau \le 1} \|\partial^{\tau}_t \nabla u\|_{2}^{2} \Biggr)^{\frac 1 2}
    \Biggl( \sum_{\tau \le 1} \|\partial^{\tau}_t u\|_{2}^{2} \Biggr) ^\frac12\\
    &\le C \Biggl( \sum_{\tau \le 1} \|\partial^{\tau}_t u\|_{2}^{2}
    + \sum_{\alpha_1 \le 1,\; |\alpha| \le 2} \|\partial^{\alpha} \omega\|_{2}^{2}
    + \varepsilon^{4} \sum_{\alpha_1 \le 1,\; |\alpha| \le 3} \|\partial^{\alpha} (f e^{-y})\|_{2}^{2} \Biggr)
    \Biggl( \sum_{\tau \le 1} \|\partial^{\tau}_t u\|_{2}^{2} \Biggr)^{\frac 1 2}.
\end{align*}
In summary, we get
\begin{equation}\label{ineq:unabla u}
    \begin{aligned}
       % &\Biggl| \sum_{\tau \le 1} \langle \partial^{\tau}_t \left( u^{a} \cdot \nabla u + u \cdot \nabla u^{a} + u \cdot \nabla u \right), \; \partial^{\tau}_t u \rangle \Biggr|\\
      I_1' &\leq C\|u^{a}\|_{Y_{\infty}^{1}}\|\nabla u\|_{Y^{1}}\| u\|_{Y^{1}}+C\|\nabla u^a\|_{Y^{1}_\infty}\|u\|_{Y^{1}}^2\\
       &\qquad+C\left( \|u\|_{Y^{1}}^2+\|\omega\|_{Y^{1,2}}^2+\varepsilon^4 \|fe^{-y}\|_{Y^{1,3}}^2\right)\|u\|_{Y^{1}}\\
       &\quad\leq C\left(1+\|u^{a}\|_{Y_{\infty}^{1}}+\|\nabla u^a\|_{Y^{1}_\infty}\right)\left(\|\omega\|_{Y^{1,2}}^2+\| u\|_{Y^{1}}^2\right)+\varepsilon^4 \|u^{a}\|_{Y_{\infty}^{1}}\|fe^{-y}\|_{Y^{1,2}}^2\\
       &\qquad+C\left( \|u\|_{Y^{1}}^2+\|\omega\|_{Y^{1,2}}^2+\varepsilon^4 \|fe^{-y}\|_{Y^{1,3}}^2\right)^3.
    \end{aligned}
\end{equation}
Finally, we estimate the pressure term, the coupling term and the external force term, leading to the following inequalities:
% 5. 压力、耦合与外力项估计
\begin{align}\label{ineq:p1}
    I_2'\leq  \Biggl| \sum_{\tau \le 1} \langle \partial^{\tau}_t \nabla p, \; \partial^{\tau}_t u \rangle \Biggr|
\le C \|\nabla p\|_{Y^1}^{2} + C \|u\|_{Y^1}^{2},
\end{align}
\begin{align}\label{ineq:p2}
I_3'\leq \Biggl| \sum_{\tau \le 1} \langle \partial^{\tau}_t n, \; \partial^{\tau}_t u_{2} \rangle \Biggr|
&\le C \|n\|_{Y^{1}}^{2} + C \|u_{2}\|_{Y^1}^{2},
\end{align}
\begin{align}\label{ineq:p3}
I_4'\leq  \Biggl| \sum_{\tau \le 1} \langle \partial^{\tau}_t U, \; \partial^{\tau}_t u \rangle \Biggr|
&\le C \|U\|_{Y^{1}}^{2} + C \|u\|_{Y^1}^{2}.
\end{align}
% Then for $\|\nabla p\|_{Y^1}^{2}$, we have  we apply 
Next we estimate the pressure. Multiplying 
\(\partial^\tau_t\) to both sides of equation \eqref{eq:p} and taking the inner product with \(\partial^\tau_t p\), this yields
\begin{align*}
    - \sum_{\tau \le 1} \langle \partial^{\tau}_t \Delta p, \; \partial^{\tau}_t p \rangle 
    &= \sum_{\tau \le 1} \langle \partial^{\tau}_t \nabla \cdot \left( u^{a} \cdot \nabla u + u \cdot \nabla u^{a} + u \cdot \nabla u \right), \; \partial^{\tau}_t p \rangle \\
    &\quad - \sum_{\tau \le 1} \langle \partial^{\tau}_t \partial_{y} n, \; \partial^{\tau}_t p \rangle - \sum_{\tau \le 1} \langle \partial^{\tau}_t \nabla \cdot U, \; \partial^{\tau}_t p \rangle.
\end{align*}
For the left term, by integration by parts, it shows
\begin{align}\label{ineq:nablap}
    - \sum_{\tau \le 1} \langle \partial^{\tau}_t \Delta p, \; \partial^{\tau}_t p \rangle &= \sum_{\tau \le 1} \langle \partial^{\tau}_t \partial_x p, \partial^{\tau}_t \partial_x p \rangle
       + \sum_{\tau \le 1} \langle \partial^{\tau}_t \partial_y p,  \partial^{\tau}_t \partial_y p \rangle - \sum_{\tau \le 1} \int_{\mathbb{R}} \overline{\partial^{\tau}_t \partial_y p  \partial^{\tau}_t p} \, dx \nonumber\\
    &= \sum_{\tau \le 1} \|\partial^{\tau}_t \nabla p\|_2^2
       - \sum_{\tau \le 1} \int_{\mathbb{R}}  \overline{\partial^{\tau}_t \partial_y p \; \partial^{\tau}_t p} dx.
\end{align}
For the right one of  \(\nabla \cdot \left( u^{a} \cdot \nabla u + u \cdot \nabla u^{a} + u \cdot \nabla u \right)\), by integration by parts, we have
% 1. 分部积分（散度项转化为对流项）
\begin{align*}
& \sum_{\tau \le 1} \langle \partial^{\tau}_t \nabla \cdot \left( u^{a} \cdot \nabla u + u \cdot \nabla u^{a} + u \cdot \nabla u \right), \; \partial^{\tau}_t p \rangle \\
& = -\sum_{\tau \le 1} \langle \partial^{\tau}_t \left( u^{a} \cdot \nabla u + u \cdot \nabla u^{a} + u \cdot \nabla u \right), \; \partial^{\tau}_t \nabla p \rangle \\
& \quad + \sum_{\tau \le 1} \int_{\mathbb{R}}\overline{\partial^{\tau}_t (u^{a} \cdot \nabla u_{2} + u \cdot \nabla u_{2}^{a} + u \cdot \nabla u_{2}) \; \partial^{\tau}_t p}  \, dx,
\end{align*}
which is bounded by
% 2. 三项分别估计
\begin{align*}
& \le C \Biggl( \sum_{\tau \le 1} \|\partial^{\tau}_t u^{a}\|_{\infty}^{2} \Biggr)^{\!\frac 1 2}
     \Biggl( \sum_{\tau \le 1} \|\partial^{\tau}_t \nabla u\|_{2}^{2} \Biggr)^{\!\frac 1 2}
     \Biggl( \sum_{\tau \le 1} \|\partial^{\tau}_t \nabla p\|_{2}^{2} \Biggr)^{\!\frac 1 2} \\
& \quad + C \Biggl( \sum_{\tau \le 1} \|\partial^{\tau}_t u\|_{2}^{2} \Biggr)^{\!\frac 1 2}
     \Biggl( \sum_{\tau \le 1} \|\partial^{\tau}_t \nabla u^{a}\|_{\infty}^{2} \Biggr)^{\!\frac 1 2}
     \Biggl( \sum_{\tau \le 1} \|\partial^{\tau}_t \nabla p\|_{2}^{2} \Biggr)^{\!\frac 1 2} \\
& \quad + C \Biggl( \sum_{\tau \le 1} \|\partial^{\tau}_t u\|_{\infty}^{2} \Biggr)^{\!\frac 1 2}
     \Biggl( \sum_{\tau \le 1} \|\partial^{\tau}_t \nabla u\|_{2}^{2} \Biggr)^{\!\frac 1 2}
     \Biggl( \sum_{\tau \le 1} \|\partial^{\tau}_t \nabla p\|_{2}^{2} \Biggr)^{\!\frac 1 2} \\
& \quad + \sum_{\tau \le 1} \int_{\mathbb{R}} \overline{\partial^{\tau}_t (u^{a} \cdot \nabla u_{2} + u \cdot \nabla u_{2}^{a} + u \cdot \nabla u_{2}) \; \partial^{\tau}_t p} dx,
\end{align*}
and by Lemma \ref{lem:embedding} and Lemma \ref{lem:omega}, this becomes
\begin{align*}
& \le \frac{C}{\mu} \left( \|u^{a}\|_{Y_{\infty}^{1}}^{2} + \|\nabla u^{a}\|_{Y_{\infty}^{1}}^{2} \right)
     \left( \|u\|_{Y^{1}}^{2} + \|\omega\|_{Y^{1,1}}^{2} + \varepsilon^{4} \|f e^{-y}\|_{Y^{1,2}}^{2} \right) \\
& \quad + C \mu \|\nabla p\|_{Y^{1}}^{2}
   + \frac{C}{\mu} \|u\|_{Y^{1,2}} \|\nabla u\|_{Y^{1,2}}^{3}
   + C \mu \|\nabla p\|_{Y^{1}}^{2} \\
& \quad + \sum_{\tau \le 1} \int_{\mathbb{R}} \overline{\partial^{\tau}_t (u^{a} \cdot \nabla u_{2} + u \cdot \nabla u_{2}^{a} + u \cdot \nabla u_{2}) \partial^{\tau}_t p}  \, dx
\end{align*}
\begin{align}\label{ineq:nablap2}
& \le \frac{C}{\mu} \left(1+ \|u^{a}\|_{Y_{\infty}^{1}}^{2} + \|\nabla u^{a}\|_{Y_{\infty}^{1}}^{2} \right)
     \left( \|u\|_{Y^{1}}^{2} + \|\omega\|_{Y^{1,1}}^{2} + \varepsilon^{4} \|f e^{-y}\|_{Y^{1,2}}^{2} \right) \nonumber\\
& \quad + \frac{C}{\mu} \left( \|u\|_{Y^{1}}^{2}+ \|\omega\|_{Y^{1,2}}^{2} + \varepsilon^{4} \|f e^{-y}\|_{Y^{1,3}}^{2} \right)^{3}
   + C \mu \|\nabla p\|_{Y^{1}}^{2} \\
& \quad + \sum_{\tau \le 1} \int_{\mathbb{R}} \overline{\partial^{\tau}_t (u^{a} \cdot \nabla u_{2} + u \cdot \nabla u_{2}^{a} + u \cdot \nabla u_{2}) \partial^{\tau}_t p}  \, dx.\nonumber
\end{align}
For \(\partial_{y} n\), it follows that
% 1. 关于 -∂_y n 项的估计
\begin{align}\label{ineq:partialy n}
- \sum_{\tau \le 1} \langle \partial^{\tau}_t \partial_{y} n, \; \partial^{\tau}_t p \rangle
&= \sum_{\tau \le 1} \langle \partial^{\tau}_t n, \; \partial^{\tau}_t \partial_{y} p \rangle
   - \sum_{\tau \le 1} \int_{\mathbb{R}} \overline{\partial^{\tau}_t n \; \partial^{\tau}_t p} \, dx \nonumber\\
&\le \frac{C}{\mu} \|n\|_{Y^{1}}^{2} + C \mu \|\partial_{y} p\|_{Y^{1}}^{2}
   - \sum_{\tau \le 1} \int_{\mathbb{R}} \overline{\partial^{\tau}_t n \; \partial^{\tau}_t p} \, dx.
\end{align}
For \(\nabla \cdot U\), by integration by parts, this yields
% 2. 关于 -∇·U 项的估计
\begin{align}\label{ineq:nablaU}
    - \sum_{\tau \le 1} \langle \partial^{\tau}_t \nabla \cdot U, \; \partial^{\tau}_t p \rangle
    &= \sum_{\tau \le 1} \langle \partial^{\tau}_t U, \; \partial^{\tau}_t \nabla p \rangle
    - \sum_{\tau \le 1} \int_{\mathbb{R}} \overline{\partial^{\tau}_t U_{2} \; \partial^{\tau}_t p}  \, dx \nonumber\\
    &\le \frac{C}{\mu} \|U\|_{Y^{1}}^{2} + C \mu \|\nabla p\|_{Y^{1}}^{2}
    - \sum_{\tau \le 1} \int_{\mathbb{R}} \overline{\partial^{\tau}_t U_{2} \; \partial^{\tau}_t p} \, dx.
\end{align}
Finally, by \eqref{ineq:nablap}, \eqref{ineq:nablap2}, \eqref{ineq:partialy n} and \eqref{ineq:nablaU}, we have
% 3. 组合后的总估计
\begin{align*}
& \sum_{\tau \le 1} \|\partial^{\tau}_t \nabla p\|_{2}^{2}
   - \sum_{\tau \le 1} \int_{\mathbb{R}} \overline{\partial^{\tau}_t \partial_{y} p \; \partial^{\tau}_t p} \, dx \nonumber\\
& \quad \le \frac{C}{\mu} \left( 1 + \|u^{a}\|_{Y_{\infty}^{1}}^{2} + \|\nabla u^{a}\|_{Y_{\infty}^{1}}^{2} \right)
        \left( \|u\|_{Y^{1}}^{2} + \|\omega\|_{Y^{1,1}}^{2} + \varepsilon^{4} \|f e^{-y}\|_{Y^{1,2}}^{2} \right) \nonumber\\
& \qquad + \frac{C}{\mu} \left( \|u\|_{Y^1}^{2} + \|\omega\|_{Y^{1,2}}^{2} + \varepsilon^{4} \|f e^{-y}\|_{Y^{1,3}}^{2} \right)^{3}
        + C \mu \|\nabla p\|_{Y^1}^{2} \nonumber\\
& \qquad + \int_{\mathbb{R}} \overline{\partial^{\tau}_t  \left(u^{a} \cdot \nabla u_{2} + u \cdot \nabla u_{2}^{a} + u \cdot \nabla u_{2}  \;\right) \partial^{\tau}_t p} \, dx \\
& \qquad + \frac{C}{\mu} \|n\|_{Y^{1}}^{2} + C \mu \|\partial_{y} p\|_{Y^1}^{2}
   - \sum_{\tau \le 1} \int_{\mathbb{R}} \overline{\partial^{\tau}_t n \; \partial^{\tau}_t p} \, dx \nonumber\\
& \qquad + \frac{C}{\mu} \|U\|_{Y^{1}}^{2} + C \mu \|\nabla p\|_{Y^{1}}^{2}
   - \sum_{\tau \le 1} \int_{\mathbb{R}} \overline{\partial^{\tau}_t U_{2} \; \partial^{\tau}_t p}   \, dx. \nonumber
\end{align*}

For the second term in equation \eqref{eq:error} involving \(u\), we simplify it as follows:
% 1. 主体估计（含边界项）
\begin{equation}\label{p-equation}
   \begin{aligned}
(1 - C\mu) \|\nabla p\|_{Y^{1}}^{2}
    &\le \frac{C}{\mu} \left( 1 + \|u^{a}\|_{Y_{\infty}^{1}}^{2} + \|\nabla u^{a}\|_{Y_{\infty}^{1}}^{2} \right) \\
    &\quad \times \left( \|u\|_{Y^{1}}^{2} + \|\omega\|_{Y^{1,1}}^{2} + \|n\|_{Y^{1}}^{2} + \varepsilon^{4} \|f e^{-y}\|_{Y^{1,2}}^{2} \right) \\
    &\quad + \frac{C}{\mu} \left( \|u\|_{Y^1}^{2} + \|\omega\|_{Y^{1,2}}^{2} + \varepsilon^{4} \|f e^{-y}\|_{Y^{1,3}}^{2} \right)^{3} \\
    &\quad + \sum_{\tau \le 1} \int_{\mathbb{R}} \overline{\partial^{\tau}_t \left(\partial_t u_2 - \varepsilon^{2} \Delta u_2 \;\right) \partial^{\tau}_t p} \, dx
   + \frac{C}{\mu} \|U\|_{Y^{1}}^{2}.
\end{aligned} 
\end{equation}
By \(\nabla \cdot u=0\) and the boundary conditions ( \(\overline{\partial_y u_1}=0, \overline{\partial_y^2u_2}=-\overline{\partial_{xy}u_1} =0 \)), there holds:
% 2. 边界项的处理（利用跳跃条件）
\begin{align*}
\sum_{\tau \le 1} \int_{\mathbb{R}} \overline{\partial^{\tau}_t \left(\partial_t u_2 - \varepsilon^{2} \Delta u_2 \;\right) \partial^{\tau}_t p} \, dx
&= \sum_{\tau \le 1} \int_{\mathbb{R}} \partial^{\tau}_t (-\varepsilon^{2} \partial_t f + \varepsilon^{4} \partial_x^{2} f) \; \overline{\partial^{\tau}_t p} \, dx \\
&= \sum_{\tau \le 1}\int_{\mathbb{R}_{+}^{2}} \partial^{\tau}_t (-\varepsilon^{2} \partial_t f + \varepsilon^{4} \partial_x^{2} f) \; 
    \partial_y \left( \partial^{\tau}_t p \; e^{-y} \right) \, dx \, dy,
\end{align*}
and applying H\"older inequality, we know that
% 3. 边界项的进一步估计
\begin{equation}\label{p-equation1}
    \begin{aligned}
    &= \sum_{\tau \le 1} \left(\int_{\mathbb{R}_{+}^{2}} \partial^{\tau}_t (-\varepsilon^{2} \partial_t f + \varepsilon^{4} \partial_x^{2} f) \; 
        \partial_y \partial^{\tau}_t p \; e^{-y} \, dx \, dy + \int_{\mathbb{R}_{+}^{2}} \partial^{\tau}_t (\varepsilon^{2} \partial_t f - \varepsilon^{4} \partial_x^{2} f) \; 
        \partial^{\tau}_t p \; e^{-y} \, dx \, dy\right) \\
    &= -\sum_{\tau \le 1}\left(\int_{\mathbb{R}_{+}^{2}} \partial^{\tau}_t (\varepsilon^{2} \partial_t f - \varepsilon^{4} \partial_x^{2} f) \; 
        \partial_y \partial^{\tau}_t p \; e^{-y} \, dx \, dy  +\int_{\mathbb{R}_{+}^{2}} \partial^{\tau}_t (\varepsilon^{2} \partial_t F - \varepsilon^{4} \partial_x f) \; 
        \partial_x \partial^{\tau}_t p \; e^{-y} \, dx \, dy\right) \\
    &\le C \varepsilon^{2} \left( \|f e^{-y}\|_{Y^{2,2}} + \varepsilon^{2} \|f e^{-y}\|_{Y^{1,3}} \right) \|\partial_y  p\|_{Y^1} + C \varepsilon^{2} \left( \|F e^{-y}\|_{Y^{2,2}} + \varepsilon^{2} \|f e^{-y}\|_{Y^{1,2}} \right) \|\partial_x p\|_{Y^1},
\end{aligned}
\end{equation}
 where the definition of \(F\) is in \eqref{f uniform bound}.
Combining \eqref{p-equation} and \eqref{p-equation1}, we have
% 4. 最终简化形式（取 μ 适当小）
\begin{align}\label{ineq:nablap3}
\frac{1}{2}(1-C\mu) \|\nabla p\|_{Y^1}^{2}
&\le C \left( 1 + \|u^{a}\|_{Y_{\infty}^{1}}^{2}+\|\nabla u^{a}\|_{Y_{\infty}^{1}}^{2} \right)
     \left( \|u\|_{Y^1}^{2} + \|\omega\|_{Y^{1,1}}^{2} + \|n\|_{Y^{1}}^{2} + \varepsilon^{4} \|f e^{-y}\|_{Y^{1,2}}^{2} \right) \nonumber\\
&\quad + C \left( \|u\|_{Y^1}^{2} + \|\omega\|_{Y^{1,2}}^{2} + \varepsilon^{4} \|f e^{-y}\|_{Y^{1,3}}^{2} \right)^{3} \\
&\quad + C \varepsilon^{4} \left( \|F e^{-y}\|_{Y^{2,2}}^{2} + \|f e^{-y}\|_{Y^{2,2}}^{2} + \varepsilon^{2} \|f e^{-y}\|_{Y^{1,3}}^{2} \right)
   + C \|U\|_{Y^{1}}^{2}.\nonumber
\end{align}
In summary, we can prove this proposition by \eqref{ineq:u} to \eqref{ineq:p3} and \eqref{ineq:nablap3}.\end{proof}

\begin{proof}
[Proof of Proposition~\ref{prop:omega}.] First, we make the boundary corrections on the \(\omega\)'s equation \eqref{eq:omega} by
\begin{align*}
    \widehat{\omega} = \omega - \varepsilon^2 \partial_x f e^{-y},
\end{align*}
then we have
% 4. 能量不等式（移项并取绝对值）
\begin{align*}
    &\frac{1}{2} \frac{d}{dt} \sum_{\alpha_1 \leq 1, |\alpha| \leq m} \|\partial^{\alpha} \widehat{\omega}\|_{2}^{2}
    - \varepsilon^{2} \sum_{\alpha_1 \leq 1, |\alpha| \leq m} \langle \partial^{\alpha} \Delta \widehat{\omega}, \; \partial^{\alpha} \widehat{\omega} \rangle \\
    & \quad \le \sum_{\alpha_1 \leq 1, |\alpha| \leq m} \bigl| \langle \partial^{\alpha} (u^{a} \cdot \nabla \widehat{\omega} + u \cdot \nabla \omega^{a} + u \cdot \nabla \widehat{\omega}), \; \partial^{\alpha} \widehat{\omega} \rangle \bigr|  + \sum_{\alpha_1 \leq 1, |\alpha| \leq m} \bigl| \langle \partial^{\alpha} \partial_x n, \; \partial^{\alpha} \widehat{\omega} \rangle \bigr| \\
    & \qquad + \varepsilon^{2} \sum_{\alpha_1 \leq 1, |\alpha| \leq m} \bigl| \langle \partial^{\alpha} \left( \partial_t + (u^{a} + u) \cdot \nabla - \varepsilon^{2} \Delta \right)(\partial_x f e^{-y}), \; \partial^{\alpha} \widehat{\omega} \rangle \bigr|  + \sum_{\alpha_1 \leq 1, |\alpha| \leq m} \bigl| \langle \partial^{\alpha} \nabla^\perp \cdot U, \; \partial^{\alpha} \widehat{\omega} \rangle \bigr| \\
    & \quad =: \bar{C}_{1} + \bar{C}_{2} + \bar{C}_{3} + \bar{C}_{4}.
\end{align*}
For the left, by integration by parts, we get
% 2. Laplacian 项的分部积分与交换子处理
\begin{align*}
    &- \varepsilon^{2} \sum_{\alpha_1 \leq 1, |\alpha| \leq m} \langle \partial^{\alpha} \Delta \widehat{\omega}, \; \partial^{\alpha} \widehat{\omega} \rangle \\
    % &= -\varepsilon^{2} \sum_{\alpha_1 \leq 1, |\alpha| \leq m} \langle \partial_x \partial^{\alpha} \partial_x \widehat{\omega}, \; \partial^{\alpha} \widehat{\omega} \rangle - \varepsilon^{2} \sum_{\alpha_1 \leq 1, |\alpha| \leq m} \langle \partial_y \partial^{\alpha} \partial_y \widehat{\omega}, \; \partial^{\alpha} \widehat{\omega} \rangle \\
    % &\qquad - \varepsilon^{2} \sum_{\alpha_1 \leq 1, |\alpha| \leq m} \langle [\partial^{\alpha}, \partial_y] \partial_y \widehat{\omega}, \; \partial^{\alpha} \widehat{\omega} \rangle \\[4pt]
    &= \varepsilon^{2} \sum_{\alpha_1 \leq 1, |\alpha| \leq m} \langle \partial^{\alpha} \partial_x \widehat{\omega}, \; \partial^{\alpha} \partial_x \widehat{\omega} \rangle + \varepsilon^{2} \sum_{\alpha_1 \leq 1, |\alpha| \leq m} \langle \partial^{\alpha} \partial_y \widehat{\omega}, \; \partial^{\alpha} \partial_y \widehat{\omega} \rangle \\
    &\quad + \varepsilon^{2} \sum_{\alpha_1 \leq 1, |\alpha| \leq m} \langle \partial^{\alpha} \partial_y \widehat{\omega}, \; [\partial_y, \partial^{\alpha}] \widehat{\omega} \rangle - \varepsilon^{2} \sum_{\alpha_1 \leq 1, |\alpha| \leq m} \langle [\partial^{\alpha}, \partial_y] \partial_y \widehat{\omega}, \; \partial^{\alpha} \widehat{\omega} \rangle ,
\end{align*}
% for the third term, by \eqref{exchange oper}, and for the last term, if \(\alpha_3 \ne 0\), let the truncation function \(\psi\) in \(\partial^\alpha \widehat{\omega}\) exchange to \([\partial^\alpha, \partial_y]\), this becomes
which equals with
\begin{align}\label{ineq:omega}
    &= \varepsilon^{2} \sum_{\alpha_1 \leq 1, |\alpha| \leq m} \|\partial^{\alpha} \nabla \widehat{\omega}\|_{2}^{2}
   - C \delta \varepsilon^{2} \sum_{\alpha_1 \leq 1, |\alpha| \leq m} \langle \partial^{\alpha} \partial_y \widehat{\omega}, \; \partial^{\alpha - (0,0,1)} \partial_y \widehat{\omega} \rangle \nonumber\\
    &\ge \varepsilon^{2} \sum_{\alpha_1 \leq 1, |\alpha| \leq m} \|\partial^{\alpha} \nabla \widehat{\omega}\|_{2}^{2} - C \delta \varepsilon^{2} \Biggl( \sum_{\alpha_1 \leq 1, |\alpha| \leq m} \|\partial^{\alpha} \partial_y \widehat{\omega}\|_{2}^{2}
          + \sum_{\alpha_1 \leq 1, |\alpha| \leq m} \|\partial^{\alpha - (0,0,1)} \partial_y \widehat{\omega}\|_{2}^{2} \Biggr) \nonumber\\[4pt]
    &\ge (1 - C \delta) \varepsilon^{2} \sum_{\alpha_1 \leq 1, |\alpha| \leq m} \|\partial^{\alpha} \nabla \widehat{\omega}\|_{2}^{2}.
\end{align}
{\bf Estimate of \(\bar{C}_1\).} We have
% 1. 对流项分解
\begin{align*}
   \bar{C}_1
    &\le \sum_{\alpha_1 \leq 1, |\alpha| \leq m} \bigl| \langle \partial^{\alpha} (u^{a} \cdot \nabla \widehat{\omega}), \; \partial^{\alpha} \widehat{\omega} \rangle \bigr| + \sum_{\alpha_1 \leq 1, |\alpha| \leq m} \bigl| \langle \partial^{\alpha} (u \cdot \nabla \omega^{a}), \; \partial^{\alpha} \widehat{\omega} \rangle \bigr| \\
    & \qquad + \sum_{\alpha_1 \leq 1, |\alpha| \leq m} \bigl| \langle \partial^{\alpha} (u \cdot \nabla \widehat{\omega}), \; \partial^{\alpha} \widehat{\omega} \rangle \bigr|.
\end{align*}
For the first term \(u^{a} \cdot \nabla \widehat{\omega}\) , by Lemma~\ref{lem:H cdot nabla G} (a), it follows that
% 2. C_{11} 估计（引理 1.1）
\begin{align*}
\sum_{\alpha_1 \leq 1, |\alpha| \leq m} \bigl| \langle \partial^{\alpha} (u^{a} \cdot \nabla \widehat{\omega}), \; \partial^{\alpha} \widehat{\omega} \rangle \bigr|
&\le \frac{C_\delta}{\sigma} \left( 1 + \|u^{a}\|_{Y_{\infty}^{1,m+1}} + \|f e^{-y}\|_{Y_{\infty}^{1,m}} \right)^{2} \|\widehat{\omega}\|_{Y^{1,m}}^{2}  + C \sigma \varepsilon^{4} \|\partial_{y} \widehat{\omega}\|_{Y^{1,m}}^{2}.
\end{align*}
For the second term \(u \cdot \nabla \omega^{a}\), by Lemma~\ref{lem:H cdot nabla G} (b), this yields
% 3. C_{12} 估计（引理 1.2）
\begin{align*}
\sum_{\alpha_1 \leq 1, |\alpha| \leq m} \bigl| \langle \partial^{\alpha} (u \cdot \nabla \omega^{a}), \; \partial^{\alpha} \widehat{\omega} \rangle \bigr|
&\le C \left( \left( \|\omega^{a}\|_{Y_{\infty}^{1,m+1}}^{2} + 1 \right) \|\widehat{\omega}\|_{Y^{1,m}}^{2}
      + \|u\|_{Y^1}^{2} + \|\omega\|_{Y^{1,m}}^{2} \right) \\
&\quad + C \varepsilon^{4} \|f e^{-y}\|_{Y^{1,m+1}}^{2}
      \left( \|\omega^{a}\|_{Y_{\infty}^{1,m+1}}^{2} + \|\partial_{y} \omega^{a}\|_{Y^{1,m}}^{2} \right)\\
      &\le C \left( \left( \|\omega^{a}\|_{Y_{\infty}^{1,m+1}}^{2} + 1 \right) \|\widehat{\omega}\|_{Y^{1,m}}^{2}
      + \|u\|_{Y^1}^{2}  \right) \\
&\quad + C \varepsilon^{4} \|f e^{-y}\|_{Y^{1,m+1}}^{2}
      \left(1+ \|\omega^{a}\|_{Y_{\infty}^{1,m+1}}^{2} + \|\partial_{y} \omega^{a}\|_{Y^{1,m}}^{2} \right).
\end{align*}
For the last term \(u \cdot \nabla \widehat{\omega}\), by Lemma~\ref{lem:H cdot nabla G} (c), we derive
% 4. C_{13} 估计（引理 1.3）
\begin{align*}
\sum_{\alpha_1 \leq 1, |\alpha| \leq m} \bigl| \langle \partial^{\alpha} (u \cdot \nabla \widehat{\omega}), \; \partial^{\alpha} \widehat{\omega} \rangle \bigr|
&\le \frac{C}{\sigma} \left( \|u\|_{Y^1}^{2} + \|\omega\|_{Y^{1,m}}^{2} + \|\widehat{\omega}\|_{Y^{1,m}}^{2} \right) \\
&\quad + \frac{C}{\sigma \varepsilon^{4}} \left( \|u\|_{Y^1}^{2} + \|\omega\|_{Y^{1,m}}^{2} + \|\widehat{\omega}\|_{Y^{1,m}}^{2} \right)^{3} \\
&\quad + \frac{C}{\sigma} \varepsilon^{2} \|f e^{-y}\|_{Y^{1,m+1}}^{2} \|\widehat{\omega}\|_{Y^{1,m}}^{2}
   + C \sigma \varepsilon^{2} \|\nabla \widehat{\omega}\|_{Y^{1,m}}^{2}\\
   &\le \frac{C}{\sigma} \left( \|u\|_{Y^1}^{2} + \varepsilon^4\|fe^{-y}\|_{Y^{1,m+1}}^{2} + \|\widehat{\omega}\|_{Y^{1,m}}^{2} \right) \\
&\quad + \frac{C}{\sigma \varepsilon^{4}} \left( \|u\|_{Y^1}^{2} + \varepsilon^4\|fe^{-y}\|_{Y^{1,m+1}}^{2} + \|\widehat{\omega}\|_{Y^{1,m}}^{2} \right)^{3} \\
&\quad + \frac{C}{\sigma} \varepsilon^{2} \|f e^{-y}\|_{Y^{1,m+1}}^{2} \|\widehat{\omega}\|_{Y^{1,m}}^{2}
   + C \sigma \varepsilon^{2} \|\nabla \widehat{\omega}\|_{Y^{1,m}}^{2}.
\end{align*}
Summing up, we get
\begin{equation}\label{ineq:barC1}
    \begin{aligned}
        \bar{C}_1\leq& \frac{C_\delta}{\sigma} \left( 1 + \|u^{a}\|_{Y_{\infty}^{1,m+1}} + \|f e^{-y}\|_{Y^{1,m+1}}+\|\omega^{a}\|_{Y_{\infty}^{1,m+1}} \right)^{2} \left(\|u\|_{Y^1}^{2}+\|\widehat{\omega}\|_{Y^{1,m}}^{2}\right) \\
        &\quad + C \sigma \varepsilon^{2} \|\nabla \widehat{\omega}\|_{Y^{1,m}}^{2} + \frac{C}{\sigma \varepsilon^{4}} \left( \|u\|_{Y^1}^{2} + \varepsilon^4\|fe^{-y}\|_{Y^{1,m+1}}^{2} + \|\widehat{\omega}\|_{Y^{1,m}}^{2} \right)^{3}\\
&\quad + \frac{C}{\sigma} \varepsilon^{4} \|f e^{-y}\|_{Y^{1,m+1}}^{2}
      \left(1+ \|\omega^{a}\|_{Y_{\infty}^{1,m+1}}^{2} + \|\partial_{y} \omega^{a}\|_{Y^{1,m}}^{2} \right).
    \end{aligned}
\end{equation}
{\bf Estimate of \(\bar{C}_2\).}, by H\"older inequality, we obtain
% 5. C_2 估计（∂_x n 项）
\begin{align}\label{ineq:barC2}
\bar{C}_{2} \le C \sigma \|\nabla n\|_{Y^{1,m}}^{2} + \frac{C}{\sigma} \|\widehat{\omega}\|_{Y^{1,m}}^{2}.
\end{align}
{\bf Estimate of \(\bar{C}_3\).} We now estimate each term in \(\bar{C}_3\) separately. First, \(\bar{C}_3\) is decomposed  into three parts:
% 1. C_3 的分解
\begin{align*}
    \bar{C}_3
    & \le \varepsilon^{2} \sum_{\alpha_1 \leq 1, |\alpha| \leq m} \bigl| \langle \partial^{\alpha} \partial_t (\partial_x f e^{-y}), \; \partial^{\alpha} \widehat{\omega} \rangle \bigr| + \varepsilon^{2} \sum_{\alpha_1 \leq 1, |\alpha| \leq m} \bigl| \langle \partial^{\alpha} ((u^{a} + u) \cdot \nabla) \partial_x f e^{-y}, \; \partial^{\alpha} \widehat{\omega} \rangle \bigr| \\
    & \quad + \varepsilon^{2} \sum_{\alpha_1 \leq 1, |\alpha| \leq m} \bigl| \langle \partial^{\alpha} \varepsilon^{2} \Delta (\partial_x f e^{-y}), \; \partial^{\alpha} \widehat{\omega} \rangle \bigr| \\
    &=: \bar{C}_{31}+\bar{C}_{32}+\bar{C}_{33}.
\end{align*}
Next for \(\bar{C}_{31}\), the estimate is
% 2.1 第一项估计（时间导数项）
\begin{align*}
    \bar{C}_{31}
    \le \varepsilon^{4} \|f e^{-y}\|_{Y^{2,m+2}}^{2} + \|\widehat{\omega}\|_{Y^{1,m}}^{2}.
\end{align*}
Then for \(\bar{C}_{32}\), applying H\"older inequality and Lemma \ref{lem:omega} gives
% 2.2 第二项估计（对流项）
\begin{align*}
    \bar{C}_{32}
    &\le C \varepsilon^{2} \| u^a + u \|_{Y^{1,m}} \|f e^{-y}\|_{Y_{\infty}^{1,m+2}} \|\widehat{\omega}\|_{Y^{1,m}}\\
    &\le C \varepsilon^{2} \left( \|u^{a}\|_{Y^{1,m}} + \|u\|_{Y^{1}} + \|\omega\|_{Y^{1,m-1}} 
    + \|f e^{-y}\|_{Y^{1,m}} \right)\|f e^{-y}\|_{Y_{\infty}^{1,m+2}}
    \|\widehat{\omega}\|_{Y^{1,m}} \\
    &\le C \varepsilon^{4} \|f e^{-y}\|_{Y_{\infty}^{1,m+2}}^{2}
       + C \left( \|u\|_{Y^{1}}^{2} + \|\omega\|_{Y^{1,m-1}}^{2}+\|u^{a}\|_{Y^{1,m}}^{2} + \|f e^{-y}\|_{Y^{1,m}}^{2} \right) \|\widehat{\omega}\|_{Y^{1,m}}^{2}\\
       &\le C \varepsilon^{4} \|f e^{-y}\|_{Y^{1,m+3}}^{2}
       + C \left( \|u\|_{Y^{1}}^{2} + \|\widehat{\omega}\|_{Y^{1,m-1}}^{2}+\|u^{a}\|_{Y^{1,m}}^{2} + \|f e^{-y}\|_{Y^{1,m}}^{2} \right) \|\widehat{\omega}\|_{Y^{1,m}}^{2}.
\end{align*}
Finally for \(\bar{C}_{33}\),
% 2.3 第三项估计（扩散项）
\begin{align*}
    \bar{C}_{33} \le C \varepsilon^{8} \|f e^{-y}\|_{Y^{1,m+3}}^{2} + C \|\widehat{\omega}\|_{Y^{1,m}}^{2}.
\end{align*}
In summary, we get
\begin{equation}\label{ineq:barC3}
    \begin{aligned}
        \bar{C}_3&\leq \varepsilon^{4} \|f e^{-y}\|_{Y^{2,m+2}}^{2} +C \varepsilon^{4} \|f e^{-y}\|_{Y^{1,m+3}}^{2}\\
        &\quad+ C \left( 1+\|u\|_{Y^{1}}^{2} + \|\widehat{\omega}\|_{Y^{1,m-1}}^{2}+\|u^{a}\|_{Y^{1,m}}^{2} + \|f e^{-y}\|_{Y^{1,m}}^{2} \right) \|\widehat{\omega}\|_{Y^{1,m}}^{2}\\
        &\leq \varepsilon^{4} \|f e^{-y}\|_{Y^{2,m+2}}^{2} +C \varepsilon^{4} \|f e^{-y}\|_{Y^{1,m+3}}^{2}\\
        &\quad+ C\sigma  \left( 1+\|u^{a}\|_{Y^{1,m}}^{2} + \|f e^{-y}\|_{Y^{1,m}}^{2} \right)\left( \|\widehat{\omega}\|_{Y^{1,m}}^{2}+\|u\|_{Y^{1}}^{2} \right)+\frac{C\sigma}{\varepsilon^4}\left( \|u\|_{Y^{1}}^{2} + \|\widehat{\omega}\|_{Y^{1,m}}^{2}\right)^3.
    \end{aligned}
\end{equation}
{\bf Estimate of \(\bar{C}_4\).} We have
% 3. C_4 估计（∇×U 项）
\begin{align}\label{ineq:barC4}
    C_4 \le  C \|\nabla^\perp\cdot  U\|_{Y^{1,m}}^{2} + C \|\widehat{\omega}\|_{Y^{1,m}}^{2}.
\end{align}
In summary, this proposition can be proven by using \eqref{ineq:omega}-\eqref{ineq:barC4}.\end{proof}

%Appendix
\section{Appendix}
This appendix is devoted to prove Proposition \ref{prop:uniform bounds}. Similar to Section~\ref{section_operator}, we define:
\begin{equation}\label{def:inner derivative}
    \widehat{\partial}^\alpha = \partial_t^{\alpha_1} \partial_x^{\alpha_2} (\delta z)^{\alpha_3}\partial_z^{\alpha_3}.
\end{equation}
Then we have
\begin{equation}\label{inner exchange operator}
    [\partial_z, \widehat{\partial}^\alpha] = \delta \alpha_3 \widehat{\partial}^{\alpha-(0,0,1)}\partial_z, 
\end{equation}
and for $s,l,m \in \mathbb{N}$, the inner conormal Sobolev space is defined by
\begin{equation}\label{def:Zlms}
    Z^{l,m}_s \left(\mathbb{R}_+^2 \right) = \left\{ u \;\Bigg|\; \| u \|_{Z^{l,m}_s}^2 = \sum_{\alpha_1 \leq l, |\alpha| \leq m} \| (1+z)^s \widehat{\partial}^\alpha u\|_2^2 < \infty \right\}.
\end{equation}
Then by the definitions \eqref{def:conormal sobolev} and \eqref{def:sobolev} we have
\begin{equation}\label{neq:embed y}
    \begin{aligned}
        \| h(t,x,y) \|_{Y^{l,m}}^2 &= \sum_{\alpha_1 \le l,|\alpha| \le m} \int_{R_+^2} \left| \partial_t^l \partial_x^{\alpha_1} \psi^{\alpha_3} \partial_y^{\alpha_3} h \right|^2  dxdy \\
        &\leq \sum_{\alpha_1 \le l,|\alpha| \le m} \left\|  \psi^{\alpha_3} \right\|_{L_y^\infty}^2 \int_{R_+^2} \left| \partial_t^l \partial_x^{\alpha_1} \partial_y^{\alpha_3} h \right|^2  dxdy \\
        &\le C\| h \|_{\overline{H}^{l,m}}^2, 
    \end{aligned}
\end{equation}
    and 
\begin{equation}\label{neq:embed z}
    \begin{aligned}
        \| h(t,x,z) \|_{Y^{l,m}}^2 &= \sum_{\alpha_1 \le l,|\alpha| \le m} \int_{R_+^2} \left| \partial_t^l \partial_x^{\alpha_1} \frac{\psi^{\alpha_3}}{(\delta z)^{\alpha_3}} \frac{(\delta z)^{\alpha_3}}{\varepsilon^{\alpha_3}} \partial_z^{\alpha_3} h \right|^2  dxdy \\
        &\le \sum_{\alpha_1 \le l,|\alpha| \le m} \left\| \frac{\psi^{\alpha_3}}{(\delta y)^{\alpha_3}} \right\|_{L_y^\infty}^2 \int_{R_+^2} \left| \partial_t^l \partial_x^{\alpha_1} (\delta z)^{\alpha_3} \partial_z^{\alpha_3} h \right|^2  dxd(\varepsilon z) \\
        &\le C\varepsilon \| h \|_{Z_0^{l,m}}^2,
    \end{aligned}
\end{equation}
which indicates that we need to compute the Sobolev norm of the outer solution and the conormal Sobolev norm of the inner solution, respectively.

In this section, we denote 
\begin{equation}\label{def:coupled norm}
    \begin{aligned}
        &\|(n,c,u)\|_{Y^{l,m}}^2 = \|n\|_{Y^{l,m}}^2 + \|c\|_{Y^{l,m}}^2 + \|u\|_{Y^{l,m}}^2,\\
        &\|(n,c,u)\|_{Y^{l,m}_\infty}^2 = \|n\|_{Y^{l,m}_\infty}^2 + \|c\|_{Y^{l,m}_\infty}^2 + \|u\|_{Y^{l,m}_\infty}^2,\\
        &\|(n,c,u)\|_{Z^{l,m}_s}^2 = \|n\|_{Z^{l,m}_s}^2 + \|c\|_{Z^{l,m}_s}^2 + \|u\|_{Z^{l,m}_s}^2.\\
    \end{aligned}
\end{equation}

\subsection{Local well-posedness of solutions to the equations \eqref{eq:e0} in $\overline{H}^{l,m}$.}\label{sec:loocal well-posed} Recall the space $\overline{H}^{l,m}$  defined in   \eqref{def:sobolev}.
Assume that the $m$-th order compatibility conditions hold:
\begin{equation*}
    (\mathcal{I}^m)
    \left\{
    \begin{aligned}
        &\partial^{\tau+1}_t n^{e,0}(x,y,0) + \partial^{\tau}_t({u}^{0} \cdot \nabla n^{e,0})(x,y,0) - \partial^{\tau}_t\Delta n^{e,0}(x,y,0) = -\partial^{\tau}_t\nabla \cdot (n^{e,0} \nabla c^{e,0})(x,y,0), \\ 
        &\partial^{\tau+1}_t c^{e,0}(x,y,0) + \partial^{\tau}_t({u^{e,0}} \cdot \nabla c^{e,0})(x,y,0)  = -\partial^{\tau}_t(c^{e,0}n^{e,0})(x,y,0), \\ 
        &\partial^{\tau+1}_t {u^{e,0}}(x,y,0) + \partial^{\tau}_t({u^{e,0}} \cdot \nabla {u^{e,0}})(x,y,0) + \partial^{\tau}_t\nabla p^{e,0}(x,y,0) = -\binom{0}{\partial^{\tau}_tn^{e,0}(x,y,0)},\\
        &0 \le 2\tau\le m.
    \end{aligned}
    \right.
\end{equation*}

Then the local well-posed results of the system \eqref {eq:e0} are stated as follows:
\begin{proposition}\label{proposition_e0}
    Assume that $(n_{in},c_{in},u_{in})\in \left(H^{m+1}\times H^{m+1}\times H^{m+1}\right) (3\leq m \in \mathbb{N})$ fulfills the compatibility conditions $(\mathcal{I}^{m-1})$. Then \eqref{eq:e0} admits a unique solution $(n^{e,0},c^{e,0},u^{e,0},\nabla p^{e,0})$, whose maximal time of existence is denoted by \(0<T_{e}\leq\infty\), satisfying for any \(0<T<T_{e}\)
    \begin{align*}
        (n^{0},c^{0},u^{0}) \in C\left([0,T];H^{m}\times H^{m+1}\times H^{m+1}\right)\cap L^2\left(0,T;H^{m+1}\times H^{m+1}\times H^{m+1}\right), 
    \end{align*}
    \begin{align*}
        \nabla p^{e,0} \in L^{\infty}\left(0,T;H^{m}\right),
    \end{align*}
    and
    \begin{align*}
    % &\partial_t^j 
    n^{e,0} &\in C\left([0,T]; \overline{H}^{k,m-2k}\right) 
    \cap L^2\left(0,T; \overline{H}^{k,m+1-2k}\right), \quad 2k\leq m,\\
    % &\partial_t^j 
    c^{e,0},u^{e,0} &\in C\left([0,T]; \overline{H}^{k,m+1-2k}\right) \cap L^2\left(0,T; \overline{H}^{k,m+1-2k}\right), \quad 2<2k\leq m,\\
    \partial_tc^{e,0},\partial_tu^{e,0} &\in C\left([0,T]; H^{m}\right) \cap L^2\left(0,T; H^{m}\right).
    \end{align*}
\end{proposition}

Before proving Proposition \ref{proposition_e0}, we give the following Korn's type lemma and Kato-Ponce commutator estimates, which are frequently used in the proof.
\begin{lemma}[\cite{tan2023global}]\label{lem:korn}
    {Let $u = (u_1, u_2)$ be a vector field on the upper half-plane $\mathbb{R}_+^2 = \{ (x,y) \in \mathbb{R}^2 : y>0 \}$, satisfying the vanishing property at $\infty$, 
    $\mathrm{div}\,u = 0$
    with the boundary condition $u_2|_{y=0}=0$. Denote $\omega =\nabla\times\,u = \partial_x u_2 - \partial_y u_1$, then for any integer $k \geq 0$, there exists a constant $C = C(k) > 0$ depending only on $k$, such that
    \begin{equation*}
            \begin{split}{}
    \left\| \nabla^{k+1} u \right\|_{L^2(\mathbb{R}_+^2)} \leq C \left\| \nabla^k \omega \right\|_{L^2(\mathbb{R}_+^2)}.
    \end{split}
        \end{equation*}	
        }
\end{lemma}
    
\begin{lemma}[\cite{KP1988,WW2014}]\label{lem:commutator}
    For $\alpha, \beta \in \mathbb{N}^3$, there hold
    \begin{align*}
        \left\| [D^\alpha, f] D^\beta g \right\|_{L^2} \leq C \left(
        \sum_{|\gamma| = |\alpha| + |\beta|} \| D^\gamma f \|_{L^2} \| g \|_{L^\infty}
        + \sum_{|\gamma| = |\alpha| + |\beta| - 1} \| \nabla f \|_{L^\infty} \| D^\gamma g \|_{L^2}
        \right),
    \end{align*}
    and
 \begin{align*}
        \left\| D^\alpha( f g) \right\|_{L^2} \leq C \left(
        \sum_{|\gamma| = |\alpha| } \| D^\gamma f \|_{L^2} \| g \|_{L^\infty}
        + \sum_{|\gamma| = |\alpha| } \| f \|_{L^\infty} \| D^\gamma g \|_{L^2}
        \right).
    \end{align*}
    
\end{lemma}

% We consider the local-in-time well-posedness of the following inviscid limit system in the half-plane $\mathbb{R}_+^2 = \{ (x,y) \in \mathbb{R}^2 \mid y > 0 \}$:
% \begin{equation}\label{eq:limit}
%     \left\{
%     \begin{aligned}
%         &\partial_t n^{e,0} + u^{e,0} \cdot \nabla n^{e,0} - \Delta n^{e,0} = -\nabla \cdot (n^{e,0} \nabla c^{e,0}), \\
%         &\partial_t c^{e,0} + u^{e,0} \cdot \nabla c^{e,0} = -c^{e,0}n^{e,0}, \\
%         &\partial_t u^{e,0} + u^{e,0} \cdot \nabla u^{e,0} + \nabla P^{e,0} = n^{e,0} e_2, \\
%         &\nabla \cdot u^{e,0} = 0, \\
%         &\partial_y n^{e,0} = n^{e,0} \partial_y c^{e,0},~u_2^0 = 0, \quad y=0,\\
%         &(u^{e,0},c^{e,0},n^{e,0})=(u_{in},c_{in},n_{in}), \quad t=0.
%     \end{aligned}
%     \right.
% \end{equation}

\begin{proof}[Proof of Proposition \ref{proposition_e0}]
    
To overcome the {loss of derivatives} caused by the chemotaxis term $\nabla \cdot (n^{e,0} \nabla c^{e,0})$ and the {boundary commutator issues} in the half-plane, we assume $m \geq 3$ (which ensures $H^m(\mathbb{R}_+^2) \hookrightarrow W^{1,\infty}(\mathbb{R}_+^2)$) and define the high-order total energy functional:
\begin{equation}
    \begin{aligned}
        \mathcal{E}_m(t) =& \sum_{2k\le m+1} \|\partial_t^k u^{e,0}\|_2^2 + \|\omega^{e,0}(t)\|_{H^{m}}^2 + \|c^{e,0}(t)\|_{H^{m+1}}^2 \\
        &+ |||\nabla n^{e,0}(t)|||_{m-1} + |||\partial_t n^{e,0}(t)|||_{m-1} + \| n^{e,0} \|_{2},
    \end{aligned}
\end{equation}
where $|||n^{e,0}|||_m$ is the anisotropic energy norm involving only time and tangential spatial derivatives:
\begin{equation}
    |||n^{e,0}(t)|||_{m}^2 = \sum_{2\alpha_1 + \alpha_2 \le m} \|\partial_t^{\alpha_1} \partial_x^{\alpha_2} n^{e,0}(t)\|_{L^2(\mathbb{R}_+^2)}^2.
\end{equation}
We will establish a closed \textit{a priori} estimate for $\mathcal{E}_m(t)$.

{\bf Step 1: Estimate for the velocity $u^{e,0}$.}
To avoid estimating the pressure $p^0$, we take the curl of the Euler equation. Let $\omega^{e,0} = \nabla^\perp \cdot u^{e,0} = \partial_x u_2^0 - \partial_y u_1^0$, which satisfies:
\begin{equation}
    \partial_t \omega^{e,0} + u^{e,0} \cdot \nabla \omega^{e,0} = \partial_x n^{e,0}.
\end{equation}
Applying $\nabla^l$ to the vorticity equation and taking the $L^2$ inner product with $\nabla^l \omega^{e,0}$ ($0\le l\leq m$), the boundary integral of the convective term vanishes since $u_2^0|_{y=0} = 0$,
where $\sigma > 0$ is a small constant to be determined later.
Using Lemma \ref{lem:commutator}, we obtain:
\begin{equation*}
    \begin{aligned}
        \frac{1}{2}\frac{d}{dt}\|\omega^{e,0}\|_{H^m}^2 &\le C\|\nabla u^{e,0}\|_{L^\infty} \|\omega^{e,0}\|_{H^m}^2 + C\| u^{e,0}\|_{H^m} \|\nabla \omega^{e,0}\|_{L^\infty}^2+ \| \partial_x n^{e,0}\|_{H^{m}} \| \omega^{e,0}\|_{H^m} \\
        &\le C\|u^{e,0}\|_{H^{m+1}} \|\omega^{e,0}\|_{H^m}^2 + \sigma \|n^{e,0}\|_{H^{m+1}}^2 + C_\sigma \|\omega^{e,0}\|_{H^m}^2.
    \end{aligned}
\end{equation*}
% By the standard elliptic Div-Curl estimate with the boundary condition $u_2^0|_{y=0}=0$, we have $\|u^{e,0}\|_{H^{m+1}} \le C(\|u^{e,0}\|_{L^2} + \|\omega^{e,0}\|_{H^m})$, thus 
By Lemma \ref{lem:korn} we have
\begin{equation}\label{est:omega}
    \begin{aligned}
        \frac{1}{2}\frac{d}{dt}\|\omega^{e,0}\|_{H^m}^2 & \le 2\sigma \|n^{e,0}\|_{H^{m+1}}^2 + C_\sigma P(\mathcal{E}_m(t)),
    \end{aligned}
\end{equation}
where $P(\cdot)$ denotes a generic polynomial and $\sigma > 0$ is a small constant to be determined later. 

For derivatives with respect to $t$ only, we apply $\partial_t^k$ to the velocity equation and taking the $L^2$ inner product with $\partial^k_t u^{e,0}$, we obtain
\begin{equation}\label{est:u}
\begin{aligned}
    &\frac{1}{2}\frac{d}{dt} \sum_{2k\le m+1}\|\partial^k_t u^{e,0}\|_2^2 \\
    &\le \sum_{2k\le m+1} |\langle \partial^k_t (u^{e,0}\cdot \nabla u^{e,0}), \partial_t^k u^{e,0} \rangle| + \sum_{2k\le m+1} |\langle \partial_t^k n, \partial^k_t u_2\rangle| \\
    &\le \sum_{2k\le m+1} \|\partial_t^k u^{e,0}\|_{H^{m+1-2k}}^2 \sum_{2k\le m+1} \|\partial^k_t u^{e,0}\|_2 
    + (|||\partial_t n^{e,0}|||_{m-1} + \|n^{e,0}\|_2)\sum_{2k\le m+1} \|\partial_t^k u^{e,0}_2\|_2 \\
    & \le P \left(\mathcal{E}_m(t), \sum_{2k\le m+1} \|\partial_t^k u^{e,0}\|_{H^{m+1-2k}}^2 \right).
\end{aligned}
\end{equation}

{\bf Step 2: $H^{m+1}$ estimates for the chemical concentration $c^{e,0}$.}
Applying $\nabla^{l}$ to the transport equation of $c^{e,0}$ and taking the $L^2$ inner product with $\nabla^{l}c^{e,0}(0\le l\le m+1)$, we get:
\begin{equation}
    \begin{aligned}
        \frac{1}{2}\frac{d}{dt}\|c^{e,0}\|_{H^{m+1}}^2 \le& C\|\nabla u^{e,0}\|_{L^\infty} \|c^{e,0}\|_{H^{m+1}}^2 + C\|\nabla c^{e,0}\|_{L^\infty} \|u^{e,0}\|_{H^{m+1}} \|c^{e,0}\|_{H^{m+1}} \\
        &+C \| c^{e,0}\|_{L^\infty} \|n^{e,0}\|_{H^{m}} + \| c^{e,0}\|_{H^{m+1}} \|n^{e,0}\|_{L^\infty},
    \end{aligned}
\end{equation}
which implies that
%Applying Young's inequality to extract the highest order term of $n^{e,0}$, we obtain:
\begin{equation}\label{est:c}
    \frac{d}{dt}\|c^{e,0}\|_{H^{m+1}}^2 \le \sigma \|n^{e,0}\|_{H^{m+1}}^2 + C_\sigma P(\mathcal{E}_m(t)).
\end{equation}

{\bf Step 3: Anisotropic Estimate for $n^{e,0}$ and Perfect Boundary Cancellation.}
In the half-plane, normal derivatives do not commute with the boundary trace. Therefore, we define the tangential-time derivative operator $\partial_\Gamma^\alpha = \partial_t^{\alpha_1} \partial_x^{\alpha_2}$.
% with $2\alpha_1 + \alpha_2 \le m$.
We rewrite the equation for $n^{e,0}$ in the divergence form of the total flux:
\begin{equation}
    \partial_t n^{e,0} + u^{e,0} \cdot \nabla n^{e,0} - \nabla \cdot (\nabla n^{e,0} - n^{e,0} \nabla c^{e,0}) = 0.
\end{equation}
Applying $\partial_t\partial_\Gamma^\alpha$ to the equation, multiplying by $\partial_t\partial_\Gamma^\alpha n^{e,0}$, integrating over $\mathbb{R}_+^2$, and summing by $2\alpha_1 + \alpha_2 \le m-1$, we integrate the total flux term by parts (noting the outward normal vector is $-\vec{e}_2$):
\begin{align}\label{eq:n0 estimate}
    &-\int_{\mathbb{R}_+^2} \nabla \cdot \partial_t \partial_\Gamma^\alpha (\nabla n^{e,0} - n^{e,0} \nabla c^{e,0}) \partial_t \partial_\Gamma^\alpha n^{e,0} \,dxdy \nonumber\\
    &= \int_{\mathbb{R}_+^2} \partial_t \partial_\Gamma^\alpha (\nabla n^{e,0} - n^{e,0} \nabla c^{e,0}) \cdot \nabla \partial_t\partial_\Gamma^\alpha n^{e,0} \,dxdy 
    + \int_{y=0} \partial_t\partial_\Gamma^\alpha \underbrace{\left( \partial_y n^{e,0} - n^{e,0} \partial_y c^{e,0} \right)}_{\text{Total normal flux}} \partial_t\partial_\Gamma^\alpha n^{e,0} \,dx.
\end{align}
Note that the physical boundary condition $\partial_y n^{e,0} = n^{e,0} \partial_y c^{e,0}$, which implies the boundary integral of \eqref{eq:n0 estimate} identically vanishes.
% \begin{equation}
%     \int_{y=0} \partial_\Gamma^\alpha \left( \partial_y n^{e,0} - n^{e,0} \partial_y c^{e,0} \right) \partial_\Gamma^\alpha n^{e,0} \,dx = \int_{y=0} \partial_\Gamma^\alpha (0) \cdot \partial_\Gamma^\alpha n^{e,0} \,dx \equiv 0.
% \end{equation}
% Thus, no trace theorems or rough bounds are needed. Bounding the remaining interior terms using Cauchy-Schwarz, Moser estimates, and 
    \begin{align*}
        &\frac{1}{2}\frac{d}{dt} |||\partial_t n^{e,0}|||_{m-1}^2 + |||\nabla \partial_t n^{e,0}|||_{m-1}^2  \\
        \le& \sum_{|\alpha| \le m-1} \left| \left\langle \partial^\alpha_\Gamma \left( \partial_t u^{e,0} \cdot \nabla n^{e,0} \right), \partial^\alpha_\Gamma \partial_t n^{e,0} \right\rangle \right|
        + \sum_{|\alpha| \le m-1} \left| \left\langle \partial^\alpha_\Gamma \left( u^{e,0} \cdot \nabla \partial_t n^{e,0} \right), \partial^\alpha_\Gamma \partial_t n^{e,0} \right\rangle \right| \\
        &  + \sum_{|\alpha| \le m-1} \left| \left\langle \partial^\alpha_\Gamma \left( \partial_t n^{e,0} \nabla c^{e,0} \right), \partial^\alpha_\Gamma \partial_t \nabla n^{e,0} \right\rangle \right| \\
        & + \sum_{|\alpha| \le m-1} \left| \left\langle \partial^\alpha_\Gamma \left[ n^{e,0} \left( \nabla u^{e,0} \cdot \nabla c^{e,0} + (u^{e,0} \cdot \nabla )\nabla c^{e,0} + \nabla c^{e,0} n^{e,0} + c^{e,0} \nabla n^{e,0} \right) \right], \partial^\alpha_\Gamma \partial_t \nabla n^{e,0} \right\rangle \right|. 
    \end{align*}
By the divergence-free condition, we obtain 
    \begin{align*}\label{neq:nt}
        &\frac{1}{2}\frac{d}{dt} |||\partial_t n^{e,0}|||_{m-1}^2 + |||\nabla \partial_t n^{e,0}|||_{m-1}^2\nonumber \\
        \le & C\sum_{2k \le m+1} \|\partial_t^k u^{e,0}\|_{H^{m+1-2k}}
        \left( \sum_{l\le m} |||\partial_y^l n^{e,0}|||_{m-l} + |||\partial_t n^{e,0}|||_{m-1}\right)
        ||| \partial_t n^{e,0} |||_{m-1} \nonumber
        \\
        & + C\left( ||| \partial_t n^{e,0} |||_{m-1} + \sum_{l\le m} ||| \partial_y^l n^{e,0} |||_{m-l} \right)
        \sum_{2k \le m} \left\| \partial^k_t c^{e,0} \right\|_{H^{m-2k}}
        \sum_{|\alpha| \le m-1} \left\| \partial^\alpha_\Gamma \partial_t \nabla n^{e,0} \right\|_2 \nonumber\\
        & + C\sum_{l\le m-1} |||\partial_y^l n^{e,0}|||_{m-1-l}
        \sum_{2k\le m} \|\partial_t^k u^{e,0}\|_{H^{m-2k}}
        \sum_{2k \le m+1} \|\partial_t^k c^{e,0}\|_{H^{m+1-2k}}
        \sum_{|\alpha| \le m-1} \left\| \partial^\alpha_\Gamma \partial_t \nabla n^{e,0} \right\|_2\nonumber\\
        & + C\sum_{l\le m} |||\partial_y^l n^{e,0}|||_{m-l}^2
        \sum_{2k \le m+1} \|\partial_t^k c^{e,0}\|_{H^{m+1-2k}}
        \sum_{|\alpha| \le m-1} \left\| \partial^\alpha_\Gamma \partial_t \nabla n^{e,0} \right\|_2\nonumber\\
%%%%
        % \le& \sigma |||\partial_t \nabla n^{e,0}|||_{m-1}^2 + C_\sigma P\big( \mathcal{E}_m(t), \sum_{l \le m} ||| \partial_y^l n^{e,0} |||_{m-l}, \sum_{2k \le m+1}\|\partial_t^k c^{e,0}\|_{H^{m+1-2k}}, \sum_{2k \le m+1}\|\partial_t^k u^{e,0}\|_{H^{m+1-2k}}\big).\nonumber
    \end{align*}
    which implies 
    \ben\label{neq:nt}
    &&\frac{d}{dt} |||\partial_t n^{e,0}|||_{m-1}^2 + |||\nabla \partial_t n^{e,0}|||_{m-1}^2
   \nonumber \\&\leq & C P\big( \mathcal{E}_m(t), \sum_{l \le m} ||| \partial_y^l n^{e,0} |||_{m-l}, \sum_{2k \le m+1}\|\partial_t^k c^{e,0}\|_{H^{m+1-2k}}, \sum_{2k \le m+1}\|\partial_t^k u^{e,0}\|_{H^{m+1-2k}}\big)\een
In addition, applying $\partial_\Gamma^\alpha$ to the equation, multiplying by $\partial_t\partial_\Gamma^\alpha n^{e,0}$, then by a similar process, we can get:
\begin{equation}\label{neq:nabla n}
    \begin{aligned}
        &\frac{1}{2}\frac{d}{dt} |||\nabla n^{e,0}|||_{m-1}^2 + |||\partial_t n^{e,0}|||_{m-1}^2 \\
        \leq& \sum_{|\alpha|\leq m-1} \left| \left\langle \partial^\alpha_\Gamma (u^{e,0} \cdot \nabla n^{e,0}), \partial^\alpha_\Gamma \partial_t n^{e,0} \right\rangle \right| + \sum_{|\alpha|\leq m-1} \left| \left\langle \partial^\alpha_\Gamma (n^{e,0} \nabla c^{e,0}), \partial^\alpha_\Gamma \partial_t \nabla n^{e,0} \right\rangle \right|\\
        \le& \sum_{2k\le m-1} || \partial_t^k u^{e,0} ||_{H^{m-1-2k}} \sum_{l\le m}||| \partial^l_y n^{e,0} |||_{m-l}  ||| \partial_t n^{e,0} |||_{m-1}\\
       \qquad &+ \sum_{l\le m-1}||| \partial_y^l n^{e,0} |||_{{m-1-l}} \sum_{2k\le m}|| \partial_t^k c^{e,0} ||_{H^{m-2k}} \sum_{|\alpha|\leq m-1} ||| \partial^\alpha \partial_t \nabla n^{e,0} |||_{2}\\
        \le& \sigma |||\partial_t \nabla n^{e,0}|||_{m-1}^2
        +\frac12|||\partial_t n^{e,0}|||_{m-1}^2\\
        &\qquad+C_\sigma P\big( \sum_{l \le m}||| \partial_y^l n^{e,0} |||_{m-l}, \sum_{2k \le m}\|\partial_t^k c^{e,0}\|_{H^{m-2k}}, \sum_{2k \le m}\|\partial_t^k u^{e,0}\|_{H^{m-2k}}\big).
    \end{aligned}
\end{equation}
Finally, multiplying $n^{e,0}$ on the both sides of equation, we have:
\begin{equation}\label{neq:n}
    \begin{aligned}
        &\frac{1}{2}\frac{d}{dt} || n^{e,0}||_{2}^2 + \frac{1}{2} ||\nabla n^{e,0}||_{2}^2 
        \le C P\big( \mathcal{E}_m(t) \big).
    \end{aligned}
\end{equation}

Combining \eqref{neq:nt}, \eqref{neq:nabla n} and \eqref{neq:n}, we have
\begin{equation}\label{neq:n summing}
    \begin{aligned}
        &\frac{d}{dt}  \left( |||\partial_t n^{e,0}|||_{m-1}^2 + |||\nabla n^{e,0}|||_{m-1}^2 + || n^{e,0}||_{2}^2 \right) 
        + |||\nabla \partial_t n^{e,0}|||_{m-1}^2
         \\
        \le&  C P\left( \mathcal{E}_m(t), \sum_{l \le m} ||| \partial_y^l n^{e,0} |||_{m-l}, \sum_{2k \le m+1}\|\partial_t^k c^{e,0}\|_{H^{m+1-2k}}, \sum_{2k \le m+1}\|\partial_t^k u^{e,0}\|_{H^{m+1-2k}}\right).
    \end{aligned}
\end{equation}

{\bf Step 4: Elliptic lifting for normal derivatives.}
To recover the full isotropic norm $\|n^{e,0}\|_{H^{m+1}}$, we utilize the equation itself as an elliptic equation in the normal direction $y$:
\begin{equation}\label{eq:partial_y^2 n}
    \partial_y^2 n^{e,0} = \partial_t n^{e,0} + u^{e,0} \cdot \nabla n^{e,0} - \partial_x^2 n^{e,0} + \nabla \cdot (n^{e,0} \nabla c^{e,0}).
\end{equation}
To control $\partial_y^k n^{e,0}$, we only need to control terms on the right-hand side which involve at most $\partial_y^{k-1} n^{e,0}$, tangential derivatives $\partial_x$, and the {time derivative $\partial_t$}. The inclusion of time derivatives in our anisotropic norm $|||n^{e,0}|||_m$ ensures that terms like $\partial_t n^{e,0} \in H^{m-2}$ are well-controlled.

To be precise, 
let $\mathcal{M}(l)$ denote the claim:
\begin{align}\label{induction n}
    |||\partial_y^l n^{e,0}|||_{m-l+1}^2 
        \le& P \left(\mathcal{E}_m(t), \sum_{2k \le m}\|\partial_t^k c^{e,0}\|_{H^{m+1-2k}}, \sum_{2k \le m}\|\partial_t^k u^{e,0}\|_{H^{m+1-2k}} \right) + |||\nabla n^{e,0}|||_m^2,
\end{align}
which is related to $l$. Then by noting that
\ben\label{eq:n-norm}
    \|n^{e,0}\|_{H^{m+1}}^2 \le \sum_{l\le m+1} |||\partial_y^l n^{e,0}|||_{m+1-l},
\een
it suffices to show that $\mathcal{M}(l)$ holds for all $0\le l \le m+1$ for the estimate of $ \|n^{e,0}\|_{H^{m+1}}^2$ in \eqref{est:c}. 
We prove \eqref{induction n} by mathematical induction.

    (\textbf{I} Basis Step.) For $l=0$ and $l=1$, we have
    \begin{align*}
        |||n^{e,0}|||_{m+1} + |||\partial_y n^{e,0}|||_{m} 
        \le \|n^{e,0}\|_2 + |||\nabla n^{e,0}|||_{m}  + |||\partial_t n^{e,0} |||_{m-1}.
    \end{align*}
    % and 
    % \begin{align*}
    %     |||n^{e,0}|||_{m} + |||\partial_y n^{e,0}|||_{m-1} 
    %     \le& \|n^{e,0}\|_2 + |||\nabla n^{e,0}|||_{m-1}  + |||\partial_t n^{e,0} |||_{m-2}.
    % \end{align*}
    
    (\textbf{II} Inductive Step.) Assume $\mathcal{M}(l)$ holds for all $1\le l \le l_0 \le m$, then for $l=l_0+1$, by \eqref{eq:partial_y^2 n} we have
    \begin{align*}
        |||\partial_y^{l_0+1} n^{e,0}|||_{m-l_0} =& |||\partial_y^{l_0-1} \partial_y^2 n^{e,0}|||_{m-l_0} \\
        \le&  |||\partial_y^{l_0-1} \partial_t n^{e,0}|||_{m-l_0} 
        + |||\partial_y^{l_0-1} \partial_x^2 n^{e,0}|||_{m-l_0}\\
        &+ |||\partial_y^{l_0-1} (u^{e,0}\cdot \nabla n^{e,0})|||_{m-l_0}
        + |||\partial_y^{l_0-1} \nabla \cdot (n\nabla c)|||_{m-l_0} \\
        \le& |||\partial_y^{l_0-1} n^{e,0}|||_{m-l_0+2} 
        + \sum_{2k\le m+1} \| \partial_t^k u^{e,0}\|_{H^{m+1-2k}} \sum_{l\le l_0} |||\partial_y^l n^{e,0}|||_{m-l} \\
        &+ \sum_{2k\le m+1} \| \partial_t^k c^{e,0}\|_{H^{m+1-2k}} \sum_{l\le l_0} |||\partial_y^l n^{e,0}|||_{m-l},
    \end{align*}
    % as well as
    % \begin{align*}
    %     |||\partial_y^{l_0+1} n^{e,0}|||_{m-l_0-1} 
    %     \le& |||\partial_y^{l_0-1} n^{e,0}|||_{m-l_0+1} 
    %     + \sum_{2k\le m} \| \partial_t^k u^{e,0}\|_{H^{m}} \sum_{l\le l_0} |||\partial_y^l n^{e,0}|||_{m-l-1} \\
    %     &+ \sum_{2k\le m} \| \partial_t^k c^{e,0}\|_{H^{m}} \sum_{l\le l_0} |||\partial_y^l n^{e,0}|||_{m-l-1}.
    % \end{align*}
    Thus, by the inductive hypothesis, the inductive step is proved.
%\end{proof}

Moreover, for $c$ and $u$, we denote the claim $\mathcal{N}(k)$ by 
    \begin{align}\label{induction cu}
        &\|\partial_t^k c^{e,0}\|_{H^{m+1-2k}}
        + \|\partial_t^k u^{e,0}\|_{H^{m+1-2k}} 
        \le P(\mathcal{E}_m(t)),
    \end{align}
    which is related to $k$. It is suffices to show that $\mathcal{N}(k)$ holds for all $0\le 2k \le m+1$, and we use mathematical induction again.

    (\textbf{I} Basis Step.) For $k=0$, by Lemma \ref{lem:korn} we have
    \begin{align*}
        \|c^{e,0}\|_{H^{m+1}} + \|u^{e,0}\|_{H^{m+1}} \le \mathcal{E}_m(t).
    \end{align*}

    (\textbf{II} Inductive Step.) Assume $\mathcal{N}(K)$ holds for all $k \le k_0 \le \lfloor \frac{m+1}2 \rfloor -1$, then for $k=k_0+1$, by equations \eqref{eq:e0}$_2$ and \eqref{eq:e0}$_3$ we have
    \begin{align*}
        \|\partial_t^{k_0+1} c^{e,0}\|_{H^{m-2(k_0+1)}} 
        =& \|\partial_t^{k_0} \partial_t c^{e,0}\|_{H^{m-2(k_0+1)}} \\
        \le& \|\partial_t^{k_0} (u^{e,0}\cdot \nabla c^{e,0})\|_{H^{m-2(k_0+1)}}
        + \|\partial_t^{k_0} (c^{e,0}n^{e,0}) \|_{H^{m-2(k_0+1)}} \\
        \le& \sum_{k\le k_0} \|\partial_t^{k} u^{e,0}\|_{H^{m-2k}} \sum_{k\le k_0} \|\partial_t^{k} c^{e,0}\|_{H^{m-2k-1}} \\
        &+ \sum_{k \le k_0}\|\partial_t^{k_0} c^{e,0}\|_{H^{m-2k-1}} \sum_{l\le m-1}|||\partial_y^{l} n^{e,0}|||_{H^{m-l-1}}, \\
        \|\partial_t^{k_0+1} u^{e,0}\|_{H^{m-2(k_0+1)}} 
        \le& \|\partial_t^{k_0+1} u^{e,0}\|_2
        +\|\partial_t^{k_0+1} \nabla u^{e,0}\|_{H^{m-2k_0-3}} \\
        \le& \|\partial_t^{k_0+1} u^{e,0}\|_2
        +\|\partial_t^{k_0} \partial_t \omega^{e,0}\|_{H^{m-2k_0-3}} \\
        \le& \|\partial_t^{k_0+1} u^{e,0}\|_2
        +\|\partial_t^{k_0} (u^{e,0}\cdot \nabla \omega^{e,0})\|_{H^{m-2k_0-3}}
        +\|\partial_t^{k_0} \partial_x n^{e,0}\|_{H^{m-2k_0-3}} \\
        \le&\|\partial_t^{k_0+1} u^{e,0}\|_2
        + \sum_{k\le k_0}  \|\partial_t^{k} u^{e,0}\|_{H^{m-2k-1}} \sum_{k \le k_0} \|\partial_t^{k} \omega^{e,0} \|_{H^{m-2k-2}}\\
        &+\sum_{l\le m-2} |||\partial_y^{l} n^{e,0} |||_{H^{m-l-2}} .
    \end{align*}  
    Thus, by the inductive hypothesis and \eqref{induction n}, the inductive step is proved.

Combining \eqref{neq:nt} to \eqref{neq:n}, \eqref{induction n} and \eqref{induction cu}, there exists a dissipation 
\begin{equation}\label{est:n}
    \begin{aligned}
        \frac{d}{dt} \left( |||\nabla n^{e,0}|||_{m-1}^2 + |||\partial_t n^{e,0}|||_{m-1}^2 + \|n\|_2^2 \right) +  |||\partial_t \nabla n^{e,0}|||_{m-1}^2 
        \le P(\mathcal{E}_m(t)).
    \end{aligned}
\end{equation}

{\bf Step 5: Energy closure and conclusion.}
Summing the estimates \eqref{est:omega}, \eqref{est:u}, \eqref{est:c}, and \eqref{est:n}, we have:
\begin{equation}
    \frac{d}{dt} \mathcal{E}_m(t)+C(1-\sigma)\|n^{e,0}\|_{H^{m+1}}^2 \le C P(\mathcal{E}_m(t)).
\end{equation}
Since the initial data $(u_{in}, c_{in}, n_{in}) \in H^{m+1} \times H^{m+1} \times H^m$, we have $\mathcal{E}_m(0) < \infty$. By the standard ODE comparison principle, there exists a local time $T > 0$ such that $\sup_{t \in [0,T]} \mathcal{E}_m(t) \le 2\mathcal{E}_m(0) < \infty$. 
The existence of the solution follows from constructing a successive approximation scheme (e.g., Friedrichs mollifiers) and passing to the limit using the uniform bounds derived above and the Aubin-Lions compactness lemma. The uniqueness follows from standard $L^2$ stability estimates for the difference of two solutions $(\delta u, \delta c, \delta n)$, utilizing the Lipschitz continuity of the flow (guaranteed by $m\geq 3$). 

Thus, the system is locally well-posed in the specified Sobolev spaces. 
\end{proof}

\subsection{Well-posedness of boundary layers.}
The purpose of this section is to obtain some necessary a priori estimates of $n^{b,1}, c^{b,1}, (u_1^{b,1})_p$ and $n^{b,2}, c^{b,2}, u_1^{b,2}$. The main results are as follows:
% Though Lemma \ref{lem:tool-1} to Lemma \ref{lem:tool-3}, we have the following conclusions.
\begin{proposition}\label{proposition_b,1}
    Let $m \ge 2, s \ge 2 \in \mathbb{N}$. Assume that $(n^{e,0}, c^{e,0}, u^{e,0}) \in \overline{H}^{3,m+3}$ defined in \eqref{def:sobolev} are the same as in Proposition~\ref{proposition_e0}, then there exists $T^{b,1} \geq 0$, such that for all $t \in [0,T^{b,1}]$ and $\delta$ small enough, the systems \eqref{eq:u_1^b1} and \eqref{eq:c^b1,n^b1} admit a unique solution $(n^{b,1},c^{b,1},(u^{b,1})_p)$, which satisfy
    \begin{equation}\label{neq:cu^b1}
        \begin{aligned}
        % \|  n^{b,1} \|_{L_t^\infty Z^{2,{m}}_s}^2&+ 
        &\|  (c^{b,1},u_1^{b,1},\partial_z c^{b,1},\partial_z u_1^{b,1}) \|_{L_t^\infty Z^{2,{m}}_s}^2
      %  + \|  u_1^{b,1} \|_{L_t^\infty Z^{2,{m}}_{s}}^2
    %     &+  \|  \partial_z c^{b,1} \|_{L_t^2 Z^{2,{m}}_s}^2
    %     + \| \partial_z u_1^{b,1} \|_{L_t^2 Z^{2,{m}}_{s}}^2+\|  \partial_z n^{b,1} \|_{L_t^2 Z^{2,{m}}_s}^2  \leq C,
    % \end{align*}
    % and 
    % \begin{align*}
        % \| \partial_z n^{b,1} \|_{L_t^\infty Z^{2,{m}}_s}^2&
       % + \| \partial_z c^{b,1} \|_{L_t^\infty Z^{2,{m}}_s}^2
        %+ \| \partial_z u_1^{b,1} \|_{L_t^\infty Z^{2,{m}}_s}^2 \\
        %&+\| \partial_z^2 c^{b,1} \|_{L_t^2 Z^{2,{m}}_s}^2
        %+ \| \partial_z^2 u_1^{b,1} \|_{L_t^2 Z^{2,{m}}_s}^2
        % +\| \partial_z^2 n^{b,1} \|_{L_t^2 Z^{2,{m}}_s}^2 
        +\| (\partial_z^2 c^{b,1},\partial_z^2 u_1^{b,1}) \|_{L_t^2 Z^{2,{m}}_s}^2\leq C,
    \end{aligned}
    \end{equation}
and
\begin{equation}\label{neq:n^b1,u_2^b2}
    \begin{aligned}
        &\| (n^{b,1},\partial_z n^{b,1}) \|_{L_t^\infty Z_s^{2,m}}^2 +\| \partial_z^2 n^{b,1} \|_{L_t^2 Z_s^{2,m}}^2
        %+ \| \partial_z n^{b,1} \|_{L_t^\infty Z_s^{2,m}}^2 +\| \partial_z^2 n^{b,1} \|_{L_t^2 Z_s^{2,m}}^2  \\
        % \le C \| n^{e,0} \|_{\overline{H}^{2,m+2}} \| c^{b,1} \|_{Z_s^{2,m}}, \\
        +\| (u_2^{b,2},\partial_z u_2^{b,2},\partial_z^2 u_2^{b,2} ) \|_{L_t^\infty Z_{s-2}^{2,m-1}}^2 
        %+ \| \partial_z u_2^{b,2} \|_{L_t^\infty Z_{s-2}^{2,m-1}}^2 + \| \partial_z^2 u_2^{b,2} \|_{L_t^\infty Z_{s-2}^{2,m-1}}^2 
        \le C,~s\geq2,
        % = \left\| \int_z^\infty \partial_x u_1^{b,1}d\zeta \right\|_{Z_{s-2}^{2,m}} \le C \| u_1^{b,1} \|_{Z_{s}^{2,m}}~(s\geq2).
    \end{aligned}
\end{equation}
where $Z_s^{2,m}$ is defined in \eqref{def:Zlms}.
\end{proposition}

\begin{proposition}\label{proposition_b,2}
    Let $m,s \in \mathbb{N}$. Assume that $(n^{e,0}, c^{e,0}, u^{e,0}) \in \overline{H}^{3,m+4}$ are the same as in Proposition~\ref{proposition_e0} and $n^{b,1}, c^{b,1}, u_1^{b,1} \in Z^{2,m+1}_{s+3}$ are the same as in Propositon~\ref{proposition_b,1}, then there exists $T^{b,2} \geq 0$, such that for all $t \in [0,T^{b,2}]$ and $\delta$ small enough, the systems \eqref{eq:u_1^b2}
    and \eqref{eq:c^b2,n^b2} admit a unique solution $(n^{b,2},c^{b,2},u^{b,2}_1)$, which satisfy
    \begin{align*}
        % \|  n^{b,2} \|_{L_t^\infty Z^{2,{m}}_s}^2&+ 
        &\|  (c^{b,2},u_1^{b,2},\partial_z c^{b,2},\partial_z u_1^{b,2}) \|_{L_t^\infty Z^{2,{m}}_s}^2
        %+ \|  u_1^{b,2} \|_{L_t^\infty Z^{2,{m}}_{s}}^2 
        % +  \|  \partial_z c^{b,2} \|_{L_t^2 Z^{2,{m}}_s}^2
        % + \| \partial_z u_1^{b,2} \|_{L_t^2 Z^{2,{m}}_{s}}^2 \\
        % +\|  \partial_z n^{b,2} \|_{L_t^2 Z^{2,{m}}_s}^2  \leq C,
    % \end{align*}
    % and 
    % \begin{align*}
        % \| \partial_z n^{b,2} \|_{L_t^\infty Z^{2,{m}}_s}^2&+ 
        %+\| \partial_z c^{b,2} \|_{L_t^\infty Z^{2,{m}}_s}^2
        %+ \| \partial_z u_1^{b,2} \|_{L_t^\infty Z^{2,{m}}_s}^2 \\
        +\| (\partial_z^2 c^{b,2},\partial_z^2 u_1^{b,2}) \|_{L_t^2 Z^{2,{m}}_s}^2
        %+ \| \partial_z^2 u_1^{b,2} \|_{L_t^2 Z^{2,{m}}_s}^2 
        \leq C,
    \end{align*}
    and 
    \begin{align*}
        \| n^{b,2} \|_{L_t^\infty Z^{2,{m}}_s}^2 
        + \| \partial_z n^{b,2} \|_{L_t^\infty Z^{2,{m}}_s}^2 
        +\| \partial_z^2 n^{b,2} \|_{L_t^2 Z^{2,{m}}_s}^2 \le C.
    \end{align*}
    \end{proposition}

Before proving Proposition \ref{proposition_b,1} and Proposition \ref{proposition_b,2}, we give the following lemmas, which are frequently used in
the proof.

\begin{lemma}\label{lem:tool-1}
    Let
    $$A=\sum_{\alpha_1 \leq l, |\alpha| \leq m} 
    \Bigl| \Bigl\langle \widehat{\partial}^\alpha \left( (u^{E,0})_p \cdot \widehat{\nabla} G \right), 
    (1+z)^{2s} \widehat{\partial}^\alpha G \Bigr\rangle \Bigr|.
    $$
Then there holds
% \begin{equation}
% \begin{aligned}
% A
% &\leq C_\delta \Biggl( \sum_{\substack{\alpha_1 \leq 1 , \alpha_3 = 0 \\|\alpha| \leq m}} 
% \bigl\| \widehat{\partial}^\alpha \overline{u_1^{e,0}} \bigr\|_{L_x^\infty}^2+ \sum_{\substack{\alpha_1 \leq 1 , \alpha_3 = 0 \\ |\alpha| \leq m}} 
% \bigl\| \widehat{\partial}^\alpha \overline{\partial_y u_2^{e,0}}   \bigr\|_{L_x^\infty}^2 \Biggr)^{\!\frac 1 2}\times \\
% &\qquad\sum_{\alpha_1 \leq 1, |\alpha| \leq m} 
% \bigl\| (1+z)^s \widehat{\partial}^\alpha G \bigr\|_2^2.
% \end{aligned}
% \end{equation}
\begin{equation}\label{eq:uE0}
\begin{aligned}
A
&\leq C_\delta \Biggl( 
\left\| u_1^{e,0} \right\|_{\overline{H}^{l,m+2}}+ 
\left\|\partial_y u_2^{e,0} \right\|_{\overline{H}^{l,m+2}} \Biggr) \|  G \bigr\|_{Z^{l,m}_s}^2.
\end{aligned}
\end{equation}
\end{lemma}

\begin{lemma}\label{lem:tool-2}
   Let
    $$
    B=\sum_{\alpha_1 \leq l, |\alpha| \leq m} 
    \Bigl| \Bigl\langle \widehat{\partial}^\alpha (H G), (1+z)^{2s} \widehat{\partial}^\alpha Q \Bigr\rangle \Bigr|.
    $$
    Then the following hold:

    (a) when \(H = H(t,x)\) we have
    %     \begin{equation}
    % \begin{aligned}
    % B&\leq C \Biggl( \sum_{\substack{\alpha_1 \leq 1 , \alpha_3 = 0 \\ |\alpha| \leq m}} 
    % \| \widehat{\partial}^\alpha H \|_{L_x^\infty}^2 \Biggr)^{\!\frac 1 2}  \Biggl( \sum_{\alpha_1 \leq 1, |\alpha| \leq m} 
    % \| (1+z)^s \widehat{\partial}^\alpha G \|_2^2 \Biggr)^{\!\frac 1 2}\times \\
    % &\quad  \Biggl( \sum_{\alpha_1 \leq 1, |\alpha| \leq m} 
    % \| (1+z)^s \widehat{\partial}^\alpha Q \|_2^2 \Biggr)^{\!\frac 1 2}.
    % \end{aligned}
    % \end{equation}
    \begin{equation}\label{eq:tool-21}
    \begin{aligned}
        B&\leq C \sum_{\alpha_1 \leq l,\alpha_3=0, |\alpha| \leq m}
           \|\widehat{\partial}^{\alpha} H\|_{L_x^\infty} \|G\|_{Z^{l,m}_s}\|Q\|_{Z^{l,m}_s},
        \end{aligned}
    \end{equation}
    
    (b) when $H=H(t,x,z)$ we have 
    %     \begin{equation}
    % \begin{aligned}
    % B&\leq C \Biggl( \sum_{\alpha_1 \leq 1, |\alpha| \leq m+1} 
    % \| \widehat{\partial}^\alpha H \|_2^2 \Biggr)^{\!1/4}  \Biggl( \sum_{\alpha_1 \leq 1, |\alpha| \leq m+1} 
    % \| \widehat{\partial}^\alpha \partial_z H \|_2^2 \Biggr)^{\!1/4}\times  \\
    % &\quad \Biggl( \sum_{\alpha_1 \leq 1, |\alpha| \leq m} 
    % \| (1+z)^s \widehat{\partial}^\alpha G \|_2^2 \Biggr)^{\!\frac 1 2} \Biggl( \sum_{\alpha_1 \leq 1, |\alpha| \leq m} 
    % \| (1+z)^s \widehat{\partial}^\alpha Q \|_2^2 \Biggr)^{\!\frac 1 2}.
    % \end{aligned}
    % \end{equation}
    \begin{equation}\label{eq:tool-22}
        \begin{aligned}
            B&\leq C \|H\|_{Z_0^{l,m+1}}^{\frac 1 2}\|\partial_z H\|_{Z_0^{l,m+1}}^{\frac 1 2} \|G\|_{Z^{l,m}_s}\|Q\|_{Z^{l,m}_s}.
        \end{aligned}
    \end{equation}
\end{lemma}

\begin{lemma}\label{lem:tool-3}
   Letting
    $$
    D=-\sum_{\alpha_1 \leq l, |\alpha| \leq m} 
    \Bigl\langle \widehat{\partial}^\alpha \partial_z^2 H, (1+z)^{2s} \widehat{\partial}^\alpha H \Bigr\rangle,
    $$
and  $H\partial_z H|_{z=0} = 0$, we have
\begin{equation}\label{eq:partial_z^2}
    \begin{aligned}
        D\geq (1 - C\sigma - C\delta) \|\partial_zH\|_{Z^{l,m}_s}^2- \frac C\sigma \|H\|_{Z^{l,m}_{s-1}},
    \end{aligned}
\end{equation}
where $\delta,\sigma>0.$
% \begin{equation}
% \begin{aligned}
% -\sum_{\alpha_1 \leq 1, |\alpha| \leq m} 
% &\Bigl\langle \widehat{\partial}^\alpha \partial_z^2 H, (1+z)^{2s} \widehat{\partial}^\alpha H \Bigr\rangle\\
% &\geq (1 - C\sigma - C\delta) 
% \sum_{\alpha_1 \leq 1, |\alpha| \leq m} 
% \| (1+z)^s \widehat{\partial}^\alpha \partial_z H \|_2^2 \\
% &\quad - \frac{C}{\sigma} 
% \sum_{\alpha_1 \leq 1, |\alpha| \leq m} 
% \| (1+z)^{s-1} \widehat{\partial}^\alpha H \|_2^2.
% \end{aligned}
% \end{equation}
\end{lemma}

\begin{proof}[Proof of Proposition \ref{proposition_b,1}]
    We  briefly sketch the proof.

   {\bf Step I. The estimates of $c^{b,1}$ and $u_1^{b,1}$.} For $c^{b,1}$, applying $\widehat{\partial}^\alpha$ to the equation \eqref{eq:c^b1,n^b1}, multiplying by $(1+z)^{2s}\widehat{\partial}^\alpha c^{b,1}$, integrating over $\mathbb{R}_+^2$, and summing by $|\alpha| \le m$, we have
    \begin{align*}
        &\sum_{\alpha_1 \le 2, |\alpha| \le m} \bigl\langle \widehat{\partial}^\alpha \partial_t c^{b,1}, (1+z)^{2s}\,\widehat{\partial}^\alpha c^{b,1} \bigr\rangle
        - \sum_{\alpha_1 \le 2, |\alpha| \le m} \bigl\langle \widehat{\partial}^\alpha \partial_z^2 c^{b,1}, (1+z)^{2s}\,\widehat{\partial}^\alpha c^{b,1} \bigr\rangle \\
        & \le \Bigl| \sum_{\alpha_1 \le 2, |\alpha| \le m} \bigl\langle \widehat{\partial}^\alpha \bigl( u_1^{b,1} \partial_x c^{E,0} \bigr), (1+z)^{2s}\,\widehat{\partial}^\alpha c^{b,1} \bigr\rangle \Bigr| 
        + \Bigl| \sum_{\alpha_1 \le 2, |\alpha| \le m} \bigl\langle \widehat{\partial}^\alpha \bigl( (u^{E,0})_p \cdot \widehat{\nabla} c^{b,1} \bigr), (1+z)^{2s}\,\widehat{\partial}^\alpha c^{b,1} \bigr\rangle \Bigr| \\
        & \quad + \Bigl| \sum_{\alpha_1 \le 2, |\alpha| \le m} \bigl\langle \widehat{\partial}^\alpha \bigl( c^{b,1} n^{E,0} + c^{E,0} n^{b,1} \bigr), (1+z)^{2s}\,\widehat{\partial}^\alpha c^{b,1} \bigr\rangle \Bigr|.
    \end{align*}
    For $u_1^{b,1}$, applying $\widehat{\partial}^\alpha$ to the equation \eqref{eq:u_1^b1}$_1$, multiplying by $(1+z)^{2s}\widehat{\partial}^\alpha u_1^{b,1}$, integrating over $\mathbb{R}_+^2$, and summing by $|\alpha| \le m$, we have
    \begin{align*}
        &\sum_{\alpha_1 \le 2, |\alpha| \le m} \bigl\langle \widehat{\partial}^\alpha \partial_t u_1^{b,1}, (1+z)^{2s}\,\widehat{\partial}^\alpha u_1^{b,1} \bigr\rangle
        - \sum_{\alpha_1 \le 2, |\alpha| \le m} \bigl\langle \widehat{\partial}^\alpha \partial_z^2 u_1^{b,1}, (1+z)^{2s}\,\widehat{\partial}^\alpha u_1^{b,1} \bigr\rangle \\
        \le& \Bigl| \sum_{\alpha_1 \le 2, |\alpha| \le m} \bigl\langle \widehat{\partial}^\alpha \bigl( (u^{E,0})_p \cdot \widehat{\nabla} u_1^{b,1} \bigr), (1+z)^{2s}\,\widehat{\partial}^\alpha u_1^{b,1} \bigr\rangle \Bigr|
        + \Bigl| \sum_{\alpha_1 \le 2, |\alpha| \le m} \bigl\langle \widehat{\partial}^\alpha \bigl( u_1^{b,1} \partial_x u_1^{E,0} \bigr), (1+z)^{2s}\,\widehat{\partial}^\alpha u_1^{b,1} \bigr\rangle \Bigr|.
    \end{align*}
    Using \eqref{eq:uE0}, one could  estimate the terms of $(u^{E,0})_p \cdot \widehat{\nabla} c^{b,1}$ and $(u^{E,0})_p \cdot \widehat{\nabla} u_1^{b,1}$. The bound of the terms $u_1^{b,1} \partial_x c^{E,0}$, $u_1^{b,1} \partial_x u_1^{E,0}$ and $c^{b,1} n^{E,0}, c^{E,0} n^{b,1}$ follows from  \eqref{eq:tool-21}. Moreover, by \eqref{eq:partial_z^2}, we get the estimates of  $\partial_z^2 c^{b,1}, \partial_z^2 u_1^{b,1}$. In summary, we have
    \begin{equation}
        \begin{aligned}
            &\frac 1 2 \frac{d}{dt} \| c^{b,1} \|_{Z_s^{2,m}}^2 + \frac 1 2 \frac{d}{dt} \| u_1^{b,1} \|_{Z_s^{2,m}}^2 + (1-C\delta)  \| \partial_z u_1^{b,1} \|_{Z_s^{2,m}}^2 + (1-C\delta)  \| \partial_z c^{b,1} \|_{Z_s^{2,m}}^2 \\
            &\le C_\delta\left( 1 + \| n^{e,0} \|_{\overline{H}^{2,m+2}}^2 + \| c^{e,0} \|_{\overline{H}^{2,m+3}}^2 + \| u^{e,0} \|_{\overline{H}^{2,m+3}}^2 \right) \|  c^{b,1} \|_{Z_s^{2,m}}^2 \\
            &\qquad + C_\delta \left( 1 + \| u^{e,0} \|_{\overline{H}^{2,m+3}} \right)  \|  u_1^{b,1} \|_{Z_s^{2,m}}^2+ \| c^{e,0} \|_{\overline{H}^{2,m+2}}^2+\| u^{e,0} \|_{\overline{H}^{2,m+2}}^2.
        \end{aligned}
    \end{equation}
    
     {\bf Step II. The estimates of $\partial_zc^{b,1}$ and $\partial_z u_1^{b,1}$.}  For $\partial_z c^{b,1},\partial_z u_1^{b,1}$, firstly applying boundary corrections \(\widehat{\partial_z c^{b,1}}=\partial_z c^{b,1}+e^{-z}\overline{\partial_y c^{e,0}}\) and \(\widehat{\partial_z u_1^{b,1}}=\partial_z u_1^{b,1}+e^{-z}\overline{\partial_y u_1^{e,0}}\) to the equations \eqref{eq:c^b1,n^b1} and \eqref{eq:u_1^b1}$_1$, we have
    \begin{align*}
        &\partial_t \widehat{\partial_z c^{b,1}} - \partial_z^2\bigl(\widehat{\partial_z c^{b,1}}\bigr) + (u^{E,0})_p \cdot \widehat{\nabla}\bigl(\widehat{\partial_z c^{b,1}}\bigr) + \overline{\partial_y u_2^{e,0}}\bigl(\widehat{\partial_z c^{b,1}}\bigr) + \widehat{\partial_z u_1^{b,1}} \overline{\partial_x c^{e,0}} + \bigl(\overline{c^{e,0}} + 1\bigr) \overline{n^{e,0}} \bigl(\widehat{\partial_z c^{b,1}}\bigr) \nonumber \\
        &= \left[\partial_t - \partial_z^2 + (u^{E,0})_p \cdot \nabla + \overline{\partial_y u_2^{e,0}}  \nonumber + \bigl(\overline{c^{e,0}} + 1\bigr) \overline{n^{e,0}}\right] \bigl(e^{-z} \overline{\partial_y c^{e,0}}\bigr) + e^{-z} \overline{\partial_y u_1^{e,0}} \overline{\partial_x c^{e,0}},
         \end{align*}
       and
        \begin{align*} &\partial_t \widehat{\partial_z u_1^{b,1}}
        - \partial_z^2 \bigl(\widehat{\partial_z u_1^{b,1}}\bigr)
        + (u^{E,0})_p \cdot \widehat{\nabla} \bigl(\widehat{\partial_z u_1^{b,1}}\bigr)
        + \overline{\partial_y u_2^{e,0}} \bigl(\widehat{\partial_z u_1^{b,1}}\bigr)
        + \bigl(\widehat{\partial_z u_1^{b,1}}\bigr) \overline{\partial_x u_1^{e,0}}
        \\
        &= \left[\partial_t 
        - \partial_z^2 
        + (u^{E,0})_p \cdot \widehat{\nabla} 
        + \overline{\partial_y u_2^{e,0}} 
        +  \overline{\partial_x u_1^{e,0}}\right] \bigl(e^{-z} \overline{\partial_y u_1^{e,0}}\bigr).
    \end{align*}
    Applying $\widehat{\partial}^\alpha$ to the above equation , multiplying by $ (1+z)^{2s}\widehat{\partial}^\alpha \widehat{\partial_z c^{b,1}}, (1+z)^{2s}\widehat{\partial}^\alpha \widehat{\partial_z u_1^{b,1}} $ respectively, then by Lemma \ref{lem:tool-1} to Lemma \ref{lem:tool-3}, we have
    \begin{equation}
        \begin{aligned}
            &\frac 1 2 \frac{d}{dt} \left\| \widehat{\partial_z c^{b,1}} \right\|_{Z_s^{2,m}}^2 + \frac 1 2 \frac{d}{dt}  \left\| \widehat{\partial_z u_1^{b,1}} \right\|_{Z_s^{2,m}}^2 + (1-C\delta)  \left\| \widehat{\partial_z^{2} u_1^{b,1}} \right\|_{Z_s^{2,m}}^2 + (1-C\delta)  \left\|\widehat{\partial_z^{2} c^{b,1} } \right\|_{Z_s^{2,m}}^2 \\
            &\le C_\delta\left( 1 + \| n^{e,0} \|_{\overline{H}^{2,m+2}}^2 + \| c^{e,0} \|_{\overline{H}^{2,m+3}}^2 + \| u^{e,0} \|_{\overline{H}^{2,m+3}}^2 \right) \left\|\widehat{\partial_z c^{b,1}}  \right\|_{Z_s^{2,m}}^2 \\
            &\qquad + C_\delta \left( 1 + \| u^{e,0} \|_{\overline{H}^{2,m+3}}+ \| c^{e,0} \|_{\overline{H}^{2,m+3}}^2 \right)  \left\| \widehat{\partial_z u_1^{b,1}} \right\|_{Z_s^{2,m}}^2\\
            &\qquad+C\left(1+\| n^{e,0} \|_{\overline{H}^{2,m+2}}^2 + \| c^{e,0} \|_{\overline{H}^{3,m+3}}^2 + \| u^{e,0} \|_{\overline{H}^{3,m+3}}^2\right)^2.
        \end{aligned}
    \end{equation}
    By noting that
    \begin{align*}
        \sup_t \| \partial_z c^{b,1} \|_{Z_s^{2,m}}^2&\leq \sup_t \left\| \widehat{\partial_z c^{b,1}} \right\|_{Z_s^{2,m}}^2+\sup_t \| c^{e,0} \|_{\overline{H}^{2,m+2}}^2,\\
        \sup_t \| \partial_z u_1^{b,1} \|_{Z_s^{2,m}}^2&\leq \sup_t \left\| \widehat{\partial_z u_1^{b,1}} \right\|_{Z_s^{2,m}}^2+\sup_t \| u_1^{e,0} \|_{\overline{H}^{2,m+2}}^2,\\
       \| \partial_z^2 c^{b,1} \|_{L^2_tZ_s^{2,m}}^2&\leq  \left\| \widehat{\partial_z^2 c^{b,1}} \right\|_{L^2_tZ_s^{2,m}}^2+ \| c^{e,0} \|_{L^2_t\overline{H}^{2,m+3}}^2,\\
        \| \partial_z^2 u_1^{b,1} \|_{L^2_tZ_s^{2,m}}^2&\leq \left\| \widehat{\partial_z^2 u_1^{b,1}} \right\|_{L^2_tZ_s^{2,m}}^2+\| u_1^{e,0} \|_{L^2_t\overline{H}^{2,m+3}}^2,
    \end{align*}
    the inequality \eqref{neq:cu^b1}  is proved.

    {\bf Step III. The estimates of $n^{b,1}$ and $u_2^{b,2}$.}  For $n^{b,1}$ and $u_2^{b,2}$, by \eqref{eq:c^b1,n^b1}$_2$ and \eqref{eq:u_1^b1}$_2$, we have
    \begin{align*}
        &\sum_{i=0}^2 \| \partial_z^i n^{b,1} \|_{Z_s^{2,m}}^2 \le C \| n^{e,0} \|_{\overline{H}^{2,m+2}} \| \partial_z^i c^{b,1} \|_{Z_s^{2,m}},
    \end{align*}
    and
    % \begin{align*}
    %     \| u_2^{b,2} \|_{Z_{s-2}^{2,m-1}} 
    %     =& \left\| \int_z^\infty \partial_x u_1^{b,1}d\zeta \right\|_{Z_{s-2}^{2,m-1}} 
    %     \le C \| u_1^{b,1} \|_{Z_{s}^{2,m}},\\
    %     \sum_{i =1}^2 \| \partial_z^i u_2^{b,2}\|_{Z_{s-2}^{2,m-1}} 
    %     =& \sum_{i =0}^1 \left\| \partial_z^i \partial_x u_1^{b,1} \right\|_{Z_{s-2}^{2,m-1}} 
    %     \le C \sum_{i=0}^1 \| \partial_z^i u_1^{b,1} \|_{Z_{s}^{2,m}}.
    % \end{align*}
     \begin{align*}
        \| u_2^{b,2} \|_{Z_{s-2}^{2,m-1}} 
        =& \left\| \int_z^\infty \partial_x u_1^{b,1}d\zeta \right\|_{Z_{s-2}^{2,m-1}} 
        \le C \| u_1^{b,1} \|_{Z_{s}^{2,m}},\\
        \| \partial_z^i u_2^{b,2}\|_{Z_{s-2}^{2,m-1}} 
        =& \left\| \partial_z^{i-1} \partial_x u_1^{b,1} \right\|_{Z_{s-2}^{2,m-1}} 
        \le C\| \partial_z^{i-1} u_1^{b,1} \|_{Z_{s}^{2,m}}，
    \end{align*}
    which $i=1,2$. 
    The proof is  complete.
\end{proof}

\begin{proof}[Proof of Proposition \ref{proposition_b,2}]
    The proof of Proposition \ref{proposition_b,2} is similar to that of Proposition \ref{proposition_b,1}, and we only present the energy inequality.
    
    (a) For \(n^{b,2}\), it follows that
\begin{align*}
    \| n^{b,2} \|_{ Z^{1,{m}}_s}^2&\leq C \| n^{e,0} \|^2_{\overline{H}^{1,m+2}}\|  c^{b,2} \|_{Z_s^{1,m}}^2+C\| c^{e,0} \|^2_{\overline{H}^{1,m+3}}\|  n^{b,1} \|_{Z_{s+2}^{1,m}}^2\\
    &\quad+C\left( \| n^{e,0} \|^2_{\overline{H}^{1,m+3}}+\|  n^{b,1} \|_{Z_0^{1,m+1}}^2+\|  \partial_zn^{b,1} \|_{Z_0^{1,m+1}}^2 \right)\| \partial_z c^{b,1} \|_{Z_{s+3}^{1,m}}^2.
\end{align*}

(b) For \(\partial_z n^{b,2}\), we have
\begin{align*}
    \| \partial_zn^{b,2} \|_{ Z^{1,{m}}_s}^2&\leq C \| n^{e,0} \|^2_{\overline{H}^{1,m+2}}\| \partial_z c^{b,2} \|_{Z_s^{1,m}}^2+C\| c^{e,0} \|^2_{\overline{H}^{1,m+3}}\|  n^{b,1} \|_{Z_{s}^{1,m}}^2\\
    &\quad+C\left( \| n^{e,0} \|^2_{\overline{H}^{1,m+3}}+\|  n^{b,1} \|_{Z_0^{1,m+1}}^2+\|  \partial_zn^{b,1} \|_{Z_0^{1,m+1}}^2 \right)\| \partial_z c^{b,1} \|_{Z_{s+1}^{1,m}}^2.
\end{align*}

(c) For \(\partial_z^2 n^{b,2}\), it shows that
\begin{align*}
    \|\partial_z^2 n^{b,2} \|_{ Z^{1,{m}}_s}^2&\leq C \| n^{e,0} \|^2_{\overline{H}^{1,m+3}}\| \partial_z^2 c^{b,2} \|_{Z_{s+1}^{1,m}}^2+C\| c^{e,0} \|^2_{\overline{H}^{1,m+3}}\| \partial_z n^{b,1} \|_{Z_{s}^{1,m}}^2\\
    &\quad+C\left( \| n^{e,0} \|^2_{\overline{H}^{1,m+3}}+\| \partial_z n^{b,1} \|_{Z_0^{1,m+1}}^2+\|  \partial^2_zn^{b,1} \|_{Z_0^{1,m+1}}^2 \right)\| \partial_z c^{b,1} \|_{Z_{s}^{1,m}}^2\\
    &\quad+C\left( \| n^{b,1} \|_{Z_0^{1,m+1}}^2+\|  \partial_zn^{b,1} \|_{Z_0^{1,m+1}}^2 \right)\| \partial_z^2 c^{b,1} \|_{Z_{s}^{1,m}}^2.
\end{align*}
    
    (d) For the \(c^{b,2}\) and \(u_1^{b,2}\), we derive
\begin{align*}
    &\frac 1 2 \frac{d}{dt}  \left(\|  c^{b,2} \|_{Z_s^{1,m}}^2+\|  u_1^{b,2} \|_{Z_s^{1,m}}^2 \right)+ (1-C\sigma -C\delta)  \left(\|\partial_z c^{b,2} \|_{Z_s^{1,m}}^2+\|\partial_z u_1^{b,2} \|_{Z_s^{1,m}}^2\right) \\
    &\leq C_\delta \left( 1+\| u^{e,0} \|^2_{\overline{H}^{1,m+3}}+ \| c^{e,0} \|^2_{\overline{H}^{1,m+2}} + \| n^{e,0} \|^2_{\overline{H}^{1,m+2}} \right)^2 \|  c^{b,2} \|_{Z_s^{1,m}}^2\\
    &\quad+C_\delta \left( 1+\| u^{e,0} \|^2_{\overline{H}^{1,m+3}} + \| c^{e,0} \|^2_{\overline{H}^{1,m+3}} \right)\|  u_1^{b,2} \|_{Z_s^{1,m}}^2+g_1,
\end{align*}
where 
\begin{align*}
    g_1=& C_\delta \left( \| n^{e,0} \|^2_{\overline{H}^{1,m+3}}+\| u^{e,0} \|^2_{\overline{H}^{1,m+4}}+\|(u_1^{b,1}, n^{b,1})\|_{Z_{s+2}^{1,m+1}}^2+\|(\partial_z u_1^{b,1} ,\partial_z n^{b,1})\|_{Z_{0}^{1,m+1}}^2\right)\\
    &\qquad \times\left(1+\| c^{e,0} \|^2_{\overline{H}^{1,m+4}} +\|c^{b,1} \|_{Z_{s+1}^{1,m+1}}^2+\|\partial_z c^{b,1} \|_{Z_{s+3}^{1,m+1}}^2\right)^2\\
    &\quad +C \left( 1+\| u^{e,0} \|^2_{\overline{H}^{1,m+4}}+\|u_1^{b,1} \|_{Z_{2}^{1,m+1}}^2+\|\partial_z u_1^{b,1} \|_{Z_{0}^{1,m+1}}^2\right)\|(u_1^{b,1},\partial_z u_1^{b,1},n^{b,1}) \|_{Z_{s+2}^{1,m+1}}^2\\
    &\quad +C\|(\partial_z^2 c^{b,1} ,\partial_z^2 u_1^{b,1})\|_{Z_{0}^{1,m+1}}^2\|u_1^{b,1} \|_{Z_{s+2}^{1,m+1}}^2.
\end{align*}

(e) For the \(\partial_z c^{b,2}\) and \(\partial_z u_1^{b,2}\), we have the following energy inequality:
\begin{align*}
    &\frac 1 2 \frac{d}{dt}  \left(\|  \partial_z c^{b,2} \|_{Z_s^{1,m}}^2+\|  \partial_z  u_1^{b,2} \|_{Z_s^{1,m}}^2 \right)+ (1-C\sigma -C\delta)  \left(\|\partial_z^2 c^{b,2} \|_{Z_s^{1,m}}^2+\|\partial_z^2 u_1^{b,2} \|_{Z_s^{1,m}}^2\right) \\
    &\leq C_\delta \left( 1+\| u^{e,0} \|^2_{\overline{H}^{1,m+3}}+ \| c^{e,0} \|^2_{\overline{H}^{1,m+3}} + \| n^{e,0} \|^2_{\overline{H}^{1,m+2}} \right)^2 \|  \partial_z  c^{b,2} \|_{Z_s^{1,m}}^2\\
    &\quad+C_\delta \left( 1+\| u^{e,0} \|^2_{\overline{H}^{1,m+3}} + \| c^{e,0} \|^2_{\overline{H}^{1,m+3}} \right)\|  \partial_z  u_1^{b,2} \|_{Z_s^{1,m}}^2+g_2,
\end{align*}
where 
\begin{align*}
    g_2=& C_\delta \bigg( \| n^{e,0} \|^2_{\overline{H}^{1,m+3}}+\| u^{e,0} \|^2_{\overline{H}^{1,m+4}}+\|(u_1^{b,1}, n^{b,1}) \|_{Z_{s}^{1,m+1}}^2+\|(\partial_z u_1^{b,1} ,\partial_z n^{b,1})\|_{Z_{s+1}^{1,m+1}}^2\bigg) \\
    &\times \left(1+\| c^{e,0} \|^2_{\overline{H}^{1,m+4}} +\|c^{b,1} \|_{Z_{s}^{1,m+1}}^2+\|\partial_z c^{b,1} \|_{Z_{s+1}^{1,m+1}}^2\right)^2\\
    &\quad +C \left( \| u^{e,0} \|^2_{\overline{H}^{1,m+4}}+\|(\partial_z u_1^{b,1},u_1^{b,1}) \|_{Z_{0}^{1,m+1}}^2\right)\left(\|u_1^{b,1} \|_{Z_{s}^{1,m+1}}^2+\|\partial_z u_1^{b,1} \|_{Z_{s+1}^{1,m+1}}^2\right)\\
    &\quad +C \left( \| u^{e,0} \|^2_{\overline{H}^{1,m+4}}+\|c^{b,1} \|_{Z_{s}^{1,m+1}}^2+\| u_1^{b,1} \|_{Z_{s+2}^{1,m+1}}^2\right)\|\partial_z^2 u_1^{b,1} \|_{Z_{s+2}^{1,m+1}}^2\\
    &\quad +C_\delta  \left( \|u^{e,0} \|_{\overline{H}^{1,m+4}}^2+\|u_1^{b,1} \|_{Z_{s+2}^{1,m+1}}^2\right)\|\partial_z^2 c^{b,1} \|_{Z_{s+2}^{1,m+1}}^2.
\end{align*}
\end{proof}

\begin{proof}[Proof of Lemma \ref{lem:tool-1}]
   We divide it into two parts:
    \begin{align*}
        &\sum_{\alpha_1 \leq l, |\alpha| \leq m} \left| \langle \widehat{\partial}^\alpha ((u^{E,0})_p \cdot \widehat{\nabla} G), (1+z)^{2s} \widehat{\partial}^\alpha G \rangle \right| \\
        \leq& \sum_{\alpha_1 \leq l, |\alpha| \leq m} \left| \langle (u^{E,0})_p \cdot \widehat{\partial}^\alpha \widehat{\nabla} G, (1+z)^{2s} \widehat{\partial}^\alpha G \rangle \right| \\
        &+ C \sum_{\alpha_1 \leq l, |\alpha| \leq m} \sum_{\beta \le \alpha, |\beta|\le m-1} 
           \left| \langle \widehat{\partial}^{\alpha-\beta} (u^{E,0})_p \cdot \widehat{\partial}^\beta \widehat{\nabla} G, (1+z)^{2s} \widehat{\partial}^\alpha G \rangle \right|\\
       =:& K_1+K_2.
    \end{align*}
    For \(K_1\),  swapping \(\partial_z\) and \(\widehat{\partial}^\alpha\), it follows that 
    \begin{align*}
        &\sum_{\alpha_1 \leq l, |\alpha| \leq m} \left| \langle (u^{E,0})_p \cdot \widehat{\partial}^\alpha \widehat{\nabla} G, (1+z)^{2s} \widehat{\partial}^\alpha G \rangle \right| \\
        \leq& \sum_{\alpha_1 \leq l, |\alpha| \leq m} \left| \langle (u^{E,0})_p \cdot \widehat{\nabla} \widehat{\partial}^\alpha G, (1+z)^{2s} \widehat{\partial}^\alpha G \rangle \right| + \sum_{\alpha_1 \leq l, |\alpha| \leq m} \left| \langle u_2^{E,1} [\widehat{\partial}^\alpha, \partial_z] G, (1+z)^{2s} \widehat{\partial}^\alpha G \rangle \right| ,
        % &= \sum_{\alpha_1 \leq 1, |\alpha| \leq m} \left| \int_{\mathbb{R}_+^2} (1+z)^{2s} (u^{E,0})_p \cdot \widehat{\nabla} (\widehat{\partial}^\alpha G)^2 dxdy \right| \\
        % &\quad + \delta \sum_{\alpha_1 \leq 1, |\alpha| \leq m} \left| \langle z \overline{\partial_y u_2^{e,0}} \widehat{\partial}^{\alpha-(0,0,1)} \partial_z G, (1+z)^{2s} \widehat{\partial}^\alpha G \rangle \right| \\
        % &\leq 2s \sum_{\alpha_1 \leq 1, |\alpha| \leq m} \left| \int_{\mathbb{R}_+^2} (1+z)^{2s-1} z \overline{\partial_y u_2^{e,0}} (\widehat{\partial}^\alpha G)^2 dxdy \right| \\
        % &\quad + \sum_{\alpha_1 \leq 1, |\alpha| \leq m} \left| \langle \overline{\partial_y u_2^{e,0}} \widehat{\partial}^\alpha G, (1+z)^{2s} \widehat{\partial}^\alpha G \rangle \right| \\
        % &\leq C \|\overline{\partial_y u_2^{e,0}}\|_{L_x^\infty} \sum_{\alpha_1 \leq 1, |\alpha| \leq m} \|(1+z)^s \widehat{\partial}^\alpha G\|_2^2
    \end{align*}
    and by \eqref{inner exchange operator}, \eqref{def:Ej}, \eqref{ncu^e1=0} and integration by parts, this becomes
    \begin{align}\label{eq:K1}
        % \sum_{\alpha_1 \leq 1, |\alpha| \leq m} &\left| \langle (u^{E,0})_p \cdot \widehat{\partial}^\alpha \widehat{\nabla} G, (1+z)^{2s} \widehat{\partial}^\alpha G \rangle \right| \\
        % &\leq \sum_{\alpha_1 \leq 1, |\alpha| \leq m} \left| \langle (u^{E,0})_p \cdot \widehat{\nabla} \widehat{\partial}^\alpha G, (1+z)^{2s} \widehat{\partial}^\alpha G \rangle \right| \\
        % &\quad + \sum_{\alpha_1 \leq 1, |\alpha| \leq m} \left| \langle u_2^{E,1} [\widehat{\partial}^\alpha, \partial_z] G, (1+z)^{2s} \widehat{\partial}^\alpha G \rangle \right| \\
        \leq& \sum_{\alpha_1 \leq l, |\alpha| \leq m} \left| \int_{\mathbb{R}_+^2} (1+z)^{2s} (u^{E,0})_p \cdot \widehat{\nabla} (\widehat{\partial}^\alpha G)^2 dxdy \right|\nonumber \\
        &+ C\delta \sum_{\alpha_1 \leq l, |\alpha| \leq m} \left| \langle z \overline{\partial_y u_2^{e,0}} \widehat{\partial}^{\alpha-(0,0,1)} \partial_z G, (1+z)^{2s} \widehat{\partial}^\alpha G \rangle \right| \nonumber\\
        \leq& 2s \sum_{\alpha_1 \leq l, |\alpha| \leq m} \left| \int_{\mathbb{R}_+^2} (1+z)^{2s-1} z \overline{\partial_y u_2^{e,0}} (\widehat{\partial}^\alpha G)^2 dxdy \right| 
        + \sum_{\alpha_1 \leq l, |\alpha| \leq m} \left| \langle \overline{\partial_y u_2^{e,0}} \widehat{\partial}^\alpha G, (1+z)^{2s} \widehat{\partial}^\alpha G \rangle \right| \nonumber\\
        \leq& C \left\|\overline{\partial_y u_2^{e,0}}\right\|_{L_x^\infty} \sum_{\alpha_1 \leq l, |\alpha| \leq m} \|(1+z)^s \widehat{\partial}^\alpha G\|_2^2.
    \end{align}
    For \(K_2\), 
    % first \((u^{E,0})_p \cdot \widehat{\partial}^\beta \widehat{\nabla}\) is decomposed into \(u^{E,0}_1 \widehat{\partial}^\beta \partial_x\) and \(u^{E,1}_2 \widehat{\partial}^\beta \partial_z\), 
    we have
    \begin{align*}
        &\sum_{\alpha_1 \leq l, |\alpha| \leq m} \sum_{\beta \le \alpha, |\beta|\le m-1} 
           \left| \langle \widehat{\partial}^{\alpha-\beta} (u^{E,0})_p \cdot \widehat{\partial}^\beta \widehat{\nabla} G, (1+z)^{2s} \widehat{\partial}^\alpha G \rangle \right| \\
        \leq& C \sum_{\alpha_1 \leq l, |\alpha| \leq m} \sum_{\beta \le \alpha, |\beta|\le m-1} 
           \left| \langle \widehat{\partial}^{\alpha-\beta} u_1^{E,0} \widehat{\partial}^\beta \partial_x G, (1+z)^{2s} \widehat{\partial}^\alpha G \rangle \right| \\
        &+ C \sum_{\alpha_1 \leq l, |\alpha| \leq m} \sum_{\beta \le \alpha, |\beta|\le m-1} 
           \left| \langle \widehat{\partial}^{\alpha-\beta} u_2^{E,1} \widehat{\partial}^\beta \partial_z G, (1+z)^{2s} \widehat{\partial}^\alpha G \rangle \right|,
    \end{align*}
    then by \eqref{def:Ej} and \eqref{def:inner derivative}, it can be controlled by
    \begin{align*}
        % &\leq C \sum_{\alpha_1 \leq 1, |\alpha| \leq m} \sum_{\beta \le \alpha, |\beta|\le m-1} 
        %    \left| \langle \widehat{\partial}^{\alpha-\beta} \overline{u_1^{e,0}} \widehat{\partial}^\beta \partial_x G, (1+z)^{2s} \widehat{\partial}^\alpha G \rangle \right| \\
        % &\quad + C \sum_{\alpha_1 \leq 1, |\alpha| \leq m} \sum_{\beta \le \alpha, |\beta|\le m-1} 
        %    \left| \langle z \widehat{\partial}^{\alpha-\beta} \overline{\partial_y u_2^{e,0}} \widehat{\partial}^\beta \partial_z G, (1+z)^{2s} \widehat{\partial}^\alpha G \rangle \right| \\
        % &\quad + \sum_{\alpha_1 \leq 1, |\alpha| \leq m} \sum_{\beta \le \alpha, |\beta|\le m-1} 
        %    \left| \langle [\widehat{\partial}^{\alpha-\beta}, z] \overline{\partial_y u_2^{e,0}} \widehat{\partial}^\beta \partial_z G, (1+z)^{2s} \widehat{\partial}^\alpha G \rangle \right| \\
        &\leq C \sum_{\alpha_1 \leq l, |\alpha| \leq m} \sum_{\beta \le \alpha, |\beta|\le m-1} 
           \left| \langle \widehat{\partial}^{\alpha-\beta} \overline{u_1^{e,0}} \widehat{\partial}^{\beta+(0,1,0)} G, (1+z)^{2s} \widehat{\partial}^{\alpha} G \rangle \right| \\
        &\quad + \frac{C}{\delta} \sum_{\alpha_1 \leq l, |\alpha| \leq m} \sum_{\beta \le \alpha, |\beta|\le m-1} 
           \left| \langle \widehat{\partial}^{\alpha-\beta} \overline{\partial_y u_2^{e,0}} \widehat{\partial}^{\beta+(0,0,1)} G, (1+z)^{2s} \widehat{\partial}^{\alpha} G \rangle \right| \\
        &\quad + \sum_{\alpha_1 \leq l, |\alpha| \leq m} \sum_{\beta \le \alpha, |\beta|\le m-1} 
           \left| \langle \widehat{\partial}^{\alpha-\beta-(0,0,1)} \overline{\partial_y u_2^{e,0}} \widehat{\partial}^{\beta+(0,0,1)} G, (1+z)^{2s} \widehat{\partial}^{\alpha} G \rangle \right|.
    \end{align*}
    Finally, by H\"older inequality and the embedding property, we have
    % \begin{align*}
    % &\leq C \left( \sum_{\substack{\alpha_1 \leq 1 , \alpha_3 = 0 \\ |\alpha| \leq m}} \|\widehat{\partial}^{\alpha} \overline{u_1^{e,0}}\|_{L_x^\infty}^2 \right)^{\frac{1}{2}} 
    %    \sum_{\alpha_1 \leq 1, |\alpha| \leq m} \|(1+z)^s \widehat{\partial}^{\alpha} G\|_2^2 \\
    % &\quad + \frac{C}{\delta} \left( \sum_{\substack{\alpha_1 \leq 1 , \alpha_3 = 0 \\ |\alpha| \leq m}} \|\widehat{\partial}^{\alpha} \overline{\partial_y u_2^{e,0}}\|_{L_x^\infty}^2 \right)^{\frac{1}{2}} 
    %    \sum_{\alpha_1 \leq 1, |\alpha| \leq m} \|(1+z)^s \widehat{\partial}^{\alpha} G\|_2^2
    % \end{align}
    \begin{align}\label{eq:K2}
        &K_2\leq C \left(\left\|u_1^{e,0}\right\|_{\overline{H}^{l,m+2}}+\left\|\partial_y u_2^{e,0}\right\|_{\overline{H}^{l,m+2}}\right)
          \|G\|_{Z^{l,m}_s}^2.
    \end{align}
    Combining \eqref{eq:K1} and \eqref{eq:K2}, this lemma can be proven.
\end{proof}

\begin{proof}[Proof of Lemma \ref{lem:tool-2}]
    (a) By the Leibniz rule and H\"older inequality, it follows that
    \begin{align*}
        &\sum_{\alpha_1 \leq l, |\alpha| \leq m} \left| \langle \widehat{\partial}^\alpha (HG), (1+z)^{2s} \widehat{\partial}^\alpha Q \rangle \right| \\
        % &\leq C \sum_{\alpha_1 \leq 1, |\alpha| \leq m} \sum_{\beta \leq \alpha} \left| \langle \widehat{\partial}^{\alpha-\beta} H \widehat{\partial}^\beta G, (1+z)^{2s} \widehat{\partial}^\alpha Q \rangle \right| \\
        &\leq C \sum_{\alpha_1 \leq l, |\alpha| \leq m} \sum_{\beta \leq \alpha} 
           \|\widehat{\partial}^{\alpha-\beta} H\|_{L_x^\infty} \|(1+z)^s \widehat{\partial}^\beta G\|_2 \|(1+z)^s \widehat{\partial}^\alpha Q\|_2 \\
        &\leq C \left( \sum_{\alpha_1 \leq l, |\alpha| \leq m} \sum_{\beta \leq \alpha} 
           \|\widehat{\partial}^{\alpha-\beta} H\|_{L_x^\infty}^2 \|(1+z)^s \widehat{\partial}^\beta G\|_2^2 \right)^{\frac{1}{2}}
           \left( \sum_{\alpha_1 \leq l, |\alpha| \leq m} \|(1+z)^s \widehat{\partial}^\alpha Q\|_2^2 \right)^{\frac{1}{2}} \\
           % &\leq C \|H\|_{\overline{H}^{l,m+1}}\|G\|_{Z^{l,m}_s}\|Q\|_{Z^{l,m}_s}\\
        &\leq C  \sum_{\alpha_1 \leq l,\alpha_3=0, |\alpha| \leq m} 
           \|\widehat{\partial}^{\alpha} H\|_{L_x^\infty} \|G\|_{Z^{l,m}_s}\|Q\|_{Z^{l,m}_s}.
           % &\leq C \left( \sum_{\substack{\alpha_1 \leq 1 , \alpha_3 = 0 \\ |\alpha| \leq m}} 
           % \|\widehat{\partial}^{\alpha} H\|_{L_x^\infty}^2 \right)^{\frac{1}{2}}
           % \left( \sum_{\alpha_1 \leq 1, |\alpha| \leq m} \|(1+z)^s \widehat{\partial}^\alpha G\|_2^2 \right)^{\frac{1}{2}}\times\\
           % &\qquad\left( \sum_{\alpha_1 \leq 1, |\alpha| \leq m} \|(1+z)^s \widehat{\partial}^\alpha Q\|_2^2 \right)^{\frac{1}{2}}
    \end{align*}

    (b) By Lemma \ref{lem:embedding}, it shows that
    \begin{align*}
        &\sum_{\alpha_1 \leq l, |\alpha| \leq m} \left| \langle \widehat{\partial}^\alpha (HG), (1+z)^{2s} \widehat{\partial}^\alpha Q \rangle \right| \\
        % &\qquad\leq C \sum_{\alpha_1 \leq 1, |\alpha| \leq m} \sum_{\beta \leq \alpha} \left| \langle \widehat{\partial}^{\alpha-\beta} H \widehat{\partial}^\beta G, (1+z)^{2s} \widehat{\partial}^\alpha Q \rangle \right| \\
        % &\qquad\leq C \sum_{\alpha_1 \leq 1, |\alpha| \leq m} \sum_{\beta \leq \alpha} 
        %    \|\widehat{\partial}^{\alpha-\beta} H\|_{L^\infty} \|(1+z)^s \widehat{\partial}^\beta G\|_2 \|(1+z)^s \widehat{\partial}^\alpha Q\|_2 \\
        \leq& C \left( \sum_{\alpha_1 \leq l, |\alpha| \leq m} \sum_{\beta \leq \alpha} 
           \|\widehat{\partial}^{\alpha-\beta} H\|_{L^\infty}^2 \|(1+z)^s \widehat{\partial}^\beta G\|_2^2 \right)^{\frac{1}{2}}
           \left( \sum_{\alpha_1 \leq l, |\alpha| \leq m} \|(1+z)^s \widehat{\partial}^\alpha Q\|_2^2 \right)^{\frac{1}{2}} \\
        \leq& C \|H\|_{Z_0^{l,m+1}}^{\frac 1 2}\|\partial_z H\|_{Z_0^{l,m+1}}^{\frac 1 2}\|G\|_{Z^{l,m}_s}\|Q\|_{Z^{l,m}_s}.
    \end{align*}
\end{proof}

\begin{proof}[Proof of Lemma \ref{lem:tool-3}]
    Firstly, by integration by parts we have
    \begin{align*}
        &- \sum_{\alpha_1 \leq l, |\alpha| \leq m} \langle \widehat{\partial}^\alpha \partial_z^2 H, (1+z)^{2s} \widehat{\partial}^\alpha H \rangle \\
        % =& - \sum_{\alpha_1 \leq 1, |\alpha| \leq m} \langle \partial_z \widehat{\partial}^\alpha \partial_z H, (1+z)^{2s} \widehat{\partial}^\alpha H \rangle  - \sum_{\alpha_1 \leq 1, |\alpha| \leq m} \langle [\widehat{\partial}^\alpha, \partial_z] \partial_z H, (1+z)^{2s} \widehat{\partial}^\alpha H \rangle \\
        =& \sum_{\alpha_1 \leq l, |\alpha| \leq m} \langle \widehat{\partial}^\alpha \partial_z H, \partial_z \{ (1+z)^{2s} \widehat{\partial}^\alpha H \}\rangle - \sum_{\alpha_1 \leq l, |\alpha| \leq m} \langle [\widehat{\partial}^\alpha, \partial_z] \partial_z H, (1+z)^{2s} \widehat{\partial}^\alpha H \rangle \\
        % &=2s\sum_{\alpha_1 \leq 1, |\alpha| \leq m} \langle \widehat{\partial}^\alpha \partial_z H, (1+z)^{2s-1} \widehat{\partial}^\alpha H \rangle+\sum_{\alpha_1 \leq 1, |\alpha| \leq m} \langle \widehat{\partial}^\alpha \partial_z H, (1+z)^{2s}\partial_z \widehat{\partial}^\alpha H \rangle \\
        % &\quad - \sum_{\alpha_1 \leq 1, |\alpha| \leq m} \langle [\widehat{\partial}^\alpha, \partial_z] \partial_z H, (1+z)^{2s} \widehat{\partial}^\alpha H \rangle\\
        =& 2s\sum_{\alpha_1 \leq l, |\alpha| \leq m} \langle \widehat{\partial}^\alpha \partial_z H, (1+z)^{2s-1} \widehat{\partial}^\alpha H \rangle+\sum_{\alpha_1 \leq l, |\alpha| \leq m} \langle \widehat{\partial}^\alpha \partial_z H, (1+z)^{2s}\widehat{\partial}^\alpha \partial_z H \rangle \\
        &- \sum_{\alpha_1 \leq l, |\alpha| \leq m} \langle [\widehat{\partial}^\alpha, \partial_z] \partial_z H, (1+z)^{2s} \widehat{\partial}^\alpha H \rangle+\sum_{\alpha_1 \leq l, |\alpha| \leq m} \langle \widehat{\partial}^\alpha \partial_z H, (1+z)^{2s}[\partial_z, \widehat{\partial}^\alpha] H \rangle.
    \end{align*}
    Finally, by \eqref{inner exchange operator} and H\"older inequality, we derive
    \begin{align*}
        D&\geq \sum_{\alpha_1 \leq l, |\alpha| \leq m} \|(1+z)^s \widehat{\partial}^\alpha \partial_z H\|_2^2  - 2s \sum_{\alpha_1 \leq l, |\alpha| \leq m} |\langle \widehat{\partial}^\alpha \partial_z H, (1+z)^{2s-1} \widehat{\partial}^\alpha H \rangle| \\
        &\qquad - C\delta \sum_{\alpha_1 \leq l, |\alpha| \leq m} |\langle \widehat{\partial}^\alpha \partial_z H, (1+z)^{2s} \widehat{\partial}^{\alpha-(0,0,1)} \partial_z H \rangle| \\
        &\geq \|\partial_z H\|_{Z^{l,m}_s}^2 - C \left( \sigma  \|\partial_z H\|_{Z^{l,m}_s}^2 
              + \frac{1}{\sigma} \| H\|_{Z^{l,m}_{s-1}}^2 \right)  - C \delta \left( \|\partial_z H\|_{Z^{l,m}_s}^2
              + \|\partial_z H\|_{Z^{l,m}_{s-1}}^2 \right) \\
        &\geq (1 - C\sigma - C\delta) \|\partial_z H\|_{Z^{l,m}_s}^2  - \frac{C}{\sigma} \|H\|_{Z^{l,m}_{s-1}}^2.
    \end{align*}
\end{proof}

% From the above propositions, we can establish Proposition \ref{prop:uniform bounds}.

%%%%%%%%%%%%%%%%%%%%%%%%%%%%%%%%%%%%%%%%%%%%%%%%%%%%%%%%%%%%%%%%%%%

%%%%%%%%%%%%%%%%%%%%%%%%%%%%%%%%%%%%%%%%%%%%%%%%%%%%%%%%%%%%%%%%%
%%%%%%%%%%%%%%%%%%%%%%%%%%%%%%%%%%%%%%%%%%%%%%%%%%%%%%%%%%%%%%%%%
\subsection{Proof of Proposition \ref{prop:uniform bounds}}\label{proof:uniform bounds}
In this section, we establish the uniform boundedness in the conormal Sobolev spaces of the approximate solutions and remainders $N,K,U$.

First, concluding the results of Proposition~\ref{proposition_e0}, Proposition~\ref{proposition_b,1}, and Proposition~\ref{proposition_b,2}, we obtain the following estimates immediately.
\begin{proposition}
    Assume that $(n^{e,0},c^{e,0},u^{e,0})$ are the same as in Proposition~\ref{proposition_e0}, $(n^{b,1},c^{b,1},u^{b,1})$ are the same as in Proposition~\ref{proposition_b,1}, and $(n^{b,2},c^{b,2},u^{b,2})$ are the same as in Proposition~\ref{proposition_b,2}. Then we have
    
    (a)
    \begin{align*}
        &\|n^a\|_{Y^{1,3}_\infty}^2 + \|c^a\|_{Y^{1,4}_\infty}^2 + \|u^a\|_{Y^{1,4}_\infty}^2 \\
        \leq
        &C\Bigl(\|n^{e,0}\|_{Y^{1,3}_\infty}^2+\|(c^{e,0},u^{e,0})\|_{Y^{1,4}_\infty}^2 + \|(n^{b,1},n^{b,2},\partial_z n^{b,1},\partial_z n^{b,2})\|_{Z^{1,4}_0}^2\\
        &+\|(c^{b,1},c^{b,2},\partial_z c^{b,1},\partial_z c^{b,2},u_1^{b,1},\partial_z u_1^{b,1},u_2^{b,2},\partial_z u_2^{b,2})\|_{Z^{1,5}_0}^2
        \Bigr),
    \end{align*}
    
    (b)
    \begin{align*}
        &\|\partial_y n^a\|_{Y^{1,2}_\infty}^2 + \|\partial_y c^a\|_{Y^{1,3}_\infty}^2 + \|\partial_y u^a\|_{Y^{1,3}_\infty}^2 \\
        \leq
        &C\Bigl(\|\partial_y n^{e,0}\|_{Y^{1,2}_\infty}^2+\|(\partial_y c^{e,0},\partial_y u^{e,0})\|_{Y^{1,3}_\infty}^2+\|(\partial_z n^{b,1},\partial_z n^{b,2},\partial_z^2 n^{b,1},\partial_z^2 n^{b,2})\|_{Z^{1,3}_0}^2\\
        &+\|(\partial_z c^{b,j},\partial_z^2 c^{b,j},\partial_z u_1^{b,1},\partial_z^2 u_1^{b,1},\partial_z u_2^{b,2},\partial_z^2 u_2^{b,2})\|_{Z^{1,4}_0}^2\Bigr),
    \end{align*}
    
    (c)
    \begin{align*}
        &\|\partial_y^2 c^a\|_{Y^{1,2}}^2 + \|\partial_y^2 u^a\|_{Y^{1,2}}^2 
        \leq C\varepsilon^{-\frac{1}{2}}\Bigl(\|(\partial_y^2 c^{e,0},\partial_y^2 u^{e,0})\|_{Y^{1,2}}^2+\|(\partial_z^2 c^{b,1},\partial_z^2 c^{b,2},\partial_z^2 u_1^{b,1},\partial_z^2 u_2^{b,2})\|_{Z^{1,2}}^2\Bigr).
    \end{align*}
\end{proposition}

\begin{proof}
    For (a), by \eqref{def:soluntion^a} and embedding theory, we have
    \begin{align*}
        &\|n^a\|_{Y^{1,3}_\infty}^2 + \|c^a\|_{Y^{1,4}_\infty}^2 + \|u^a\|_{Y^{1,4}_\infty}^2 \\
        \leq& \Big( \|(n^{e,0},\varepsilon n^{b,1},\varepsilon^2 n^{b,2)}\|_{Y_\infty^{1,3}}^2 
        + \|(c^{e,0},\varepsilon c^{b,1}, \varepsilon^2 c^{b,2})\|_{Y_\infty^{1,4}}^2 
        + \|(u^{e,0},\varepsilon u_1^{b,1}, \varepsilon^2 u_2^{b,2})\|_{Y_\infty^{1,4}}^2 \Big) \\
        \le& \Big( \|n^{e,0}\|_{Y_\infty^{1,3}}^2 
        + \|c^{e,0}\|_{Y_\infty^{1,4}}^2 
        + \|u^{e,0}\|_{Y_\infty^{1,4}}^2 \Big) \\
        &+ C\varepsilon^2 \Big(  \|n^{b,1}\|_{Z_0^{1,4}}\|\partial_z n^{b,1}\|_{Z_0^{1,4}} + \|c^{b,1}\|_{Z_0^{1,5}}\|\partial_z c^{b,1}\|_{Z_0^{1,5}} + \|u_1^{b,1}\|_{Z_0^{1,5}}\|\partial_z u_1^{b,1}\|_{Z_0^{1,5}}\Big) \\
        &+C\varepsilon^4 \Big(  \|n^{b,2}\|_{Z_0^{1,4}}\|\partial_z n^{b,2}\|_{Z_0^{1,4}} + \|c^{b,2}\|_{Z_0^{1,5}}\|\partial_z c^{b,2}\|_{Z_0^{1,5}} + \|u_2^{b,2}\|_{Z_0^{1,5}}\|\partial_z u_2^{b,2}\|_{Z_0^{1,5}}\Big),
    \end{align*}
    where we have used
    \begin{align*}
        \|h(t,x,z)\|_{Y^{l,m}_\infty}^2 =& \sum_{\alpha_1 \le l, |\alpha| \le m} \| \partial^\alpha h\|_{\infty}^2 \\
        \le& \sum_{\alpha_1 \le l, |\alpha| \le m} \left\| \frac{\psi}{\delta y}\right\|^{2\alpha_3}_\infty \| \widehat{\partial}^\alpha h\|_{\infty}^2 \\
        \le& C\sum_{\alpha_1 \le l,|\alpha|\le m+1} \|\widehat{\partial}^\alpha h\|_{L_{xz}^2} \|\partial_z\widehat{\partial}^\alpha h\|_{L_{xz}^2},
    \end{align*}
    then the proof of (a) is completed. 
    
    For (b), note that
    \begin{align*}
        &\partial_y n^a =\partial_y n^{e,0} + \partial_z n^{b,1} +\varepsilon \partial_z^2 n^{b,2}, \\
        &\partial_y c^a =\partial_y c^{e,0} + \partial_z c^{b,1} +\varepsilon \partial_z^2 c^{b,2}, \\
        &\partial_y u^a =\partial_y u^{e,0} + \binom{\partial_z u_1^{b,1}}{\varepsilon \partial_z^2 u_2^{b,2}}, 
    \end{align*}
    then the proof is similar to (a).
    
    For (c), note that
    \begin{align*}
        &\|\partial_y^2 c^a\|_{Y^{1,2}}^2 + \|\partial_y^2 u^a\|_{Y^{1,2}}^2 \\
        \le & \| (\partial_y^2 c^{e,0}, \partial_y^2 u^{e,0})\|_{Y^{l,2}}^2 
        + \| (\varepsilon^{-1} \partial_z^2 c^{b,1}, \varepsilon^{-1} \partial_z^2 u_1^{b,1})\|_{Y^{l,2}}^2 
        +\| (\partial_z^2 c^{b,1}, \partial_z^2 u_2^{b,2})\|_{Y^{l,2}}^2,
    \end{align*}
    then by \eqref{neq:embed z}, this becomes
    \begin{align*}
        \le & \| (\partial_y^2 c^{e,0}, \partial_y^2 u^{e,0})\|_{Y^{l,2}}^2 
        + C\varepsilon \| (\varepsilon^{-1} \partial_z^2 c^{b,1}, \varepsilon^{-1} \partial_z^2 u_1^{b,1})\|_{Z_0^{l,2}}^2 
        + C\varepsilon\| (\partial_z^2 c^{b,1}, \partial_z^2 u_2^{b,2})\|_{Z_0^{l,2}}^2,
    \end{align*}
    and the proof is complete. 
\end{proof}

Second, for the remainders $N,K,U$ there hold the following inequalities.
\begin{proposition}\label{NKU}
    Assume that $(n^{e,0},c^{e,0},u^{e,0})$ are the same as in Proposition~\ref{proposition_e0}, $(n^{b,1},c^{b,1},u^{b,1})$ are the same as in Proposition~\ref{proposition_b,1}, and $(n^{b,2},c^{b,2},u^{b,2})$ are the same as in Proposition~\ref{proposition_b,2}. Then we have
    \begin{align*}
        \|N\|^2_{Y^{1,m}} 
        &\leq \varepsilon^3 \|(n^{b,1},n^{b,2})\|_{Z^{2,m+1}_{0}}^2+\varepsilon^3 \Big(1 + \|c^{e,0}\|^2_{\overline{H}^{1,m+4}} + \|u^{e,0}\|^2_{\overline{H}^{1,m+3}} + \big\| u_1^{b,1}\big\|^2_{Z^{1,m+2}_2} \\
        &\quad + \|(c^{b,1},c^{b,2},\partial_z u_1^{b,1})\|^2_{Z^{1,m+2}_0} + \|\partial_z c^{b,1}\|^2_{Z^{1,m+1}_1} + \|\partial_z c^{b,2}\|^2_{Z^{1,m}_0} \Big)\\
        &\quad  \times \Big(\|n^{e,0}\|^2_{\overline{H}^{1,m+4}} + \|(n^{b,1},n^{b,2})\|^2_{Z^{1,m+2}_0} + \|(\partial_z n^{b,1},\partial_z n^{b,2})\|^2_{Z^{1,m+2}_1}  \Big) \\
        % &\quad +\varepsilon^5 \left(1+\|n^{e,0}\|^2_{\overline{H}^{1,m+2}} + \|u_2^{e,0}\|^2_{\overline{H}^{1,m+3}} + \|u_1^{e,0}\|^2_{\overline{H}^{1,m+2}} + \|c^{e,0}\|^2_{\overline{H}^{1,m+3}} \right)^2\times\\
        % &\qquad \left( \| u_1^{b,1}\|^2_{Z^{1,m}_0} + \|c^{b,1}\|^2_{Z^{1,m+1}_0} + \|\partial_z c^{b,1}\|^2_{Z^{1,m+1}_1} \right)\\
        &\quad + \varepsilon^3 \Big(\|n^{e,0}\|^2_{\overline{H}^{1,m+4}} + \|\partial_z c^{b,2}\|^2_{Z^{1,m}_0} + \|(n^{b,2},\partial_z n^{b,2},n^{b,1},\partial_z n^{b,1})\|^2_{Z^{1,m+1}_0} \Big) \\
        &\quad \times \Big(\|\partial_z^2 c^{b,1}\|^2_{Z^{1,m+1}_2}  + \|\partial_z^2 c^{b,2}\|^2_{Z^{1,m}_1}+ \|(\partial_z^2 n^{b,1},\partial_z^2 n^{b,2})\|^2_{Z^{1,m+1}_0}\Big),
    \end{align*}
    \begin{align*}
        \|K\|^2_{Y^{1}} 
        &\leq \varepsilon^5 \Big(1 + \|(u_1^{e,0},n^{e,0})\|^2_{\overline{H}^{1,4}} \Big) \| (c^{b,1},c^{b,2})\|^2_{Z^{1,3}_1} + \varepsilon^5 \|u_2^{e,0}\|^2_{\overline{H}^{1,5}} \| (\partial_z c^{b,1},\partial_z c^{b,2})\|^2_{Z^{1,0}_2} \\
        &\quad + \varepsilon^4 \|c^{e,0}\|^2_{\overline{H}^{1,5}} \Big(1 + \| (n^{b,1},n^{b,2})\|^2_{Z^{1,0}_1} + \|  u_1^{b,1}\|^2_{Z^{1,2}_3} + \|u_1^{b,2}\|^2_{Z_{0}^{1,0}}\Big) \\
        &\quad + \varepsilon^7\|(u_1^{b,1},\partial_z u_1^{b,1},n^{b,1},\partial_z n^{b,1},n^{b,2},\partial_z n^{b,2})\|^2_{Z^{1,2}_0}
       \|c^{b,2}\|^2_{Z^{1,2}_0}  \\
        &\quad + \varepsilon^5 \| u_1^{b,1}\|^2_{Z^{1,3}_2} \|(\partial_z c^{b,2},\partial_z c^{b,1})\|^2_{Z_0^{1,0}}+ \varepsilon^7 \|(c^{b,1},\partial_z c^{b,1})\|^2_{Z^{1,2}_0} \|n^{b,2}\|^2_{Z_0^{1,0}},
    \end{align*}
    % \begin{align*}
    %     \|K\|^2_{Y^{1}} 
    %     &\leq \varepsilon^5 \Big(1 + \|u_1^{e,0}\|^2_{\overline{H}^{1,4}} + \|n^{e,0}\|^2_{\overline{H}^{1,4}}\Big) \Big(\| c^{b,1}\|^2_{Z^{1,2}_1} + \| c^{b,2}\|^2_{Z^{1,2}_1} + \|c^{b,1}\|^2_{Z^{1,3}_0} + \|c^{b,2}\|^2_{Z^{1,3}_0}\Big) \\
    %     &\quad + \varepsilon^5 \|u_2^{e,0}\|^2_{\overline{H}^{1,5}} \Big(\| \partial_z c^{b,1}\|^2_{Z^{1,0}_2} + \| \partial_z c^{b,2}\|^2_{Z^{1,0}_2}\Big) \\
    %     &\quad + \varepsilon^4 \|c^{e,0}\|^2_{\overline{H}^{1,5}} \Big(1 + \| n^{b,1}\|^2_{Z^{1,0}_1} + \| n^{b,2}\|^2_{Z^{1,0}_1} + \|  u_1^{b,1}\|^2_{Z^{1,2}_3} + \|u_1^{b,2}\|^2_{Z_{0}^{1,0}}\Big) \\
    %     &\quad + \varepsilon^7 \Big(\|u_1^{b,1}\|^2_{Z^{1,2}_0} + \|\partial_z u_1^{b,1}\|^2_{Z^{1,2}_0} + \|n^{b,1}\|^2_{Z^{1,2}_0} + \|\partial_z n^{b,1}\|^2_{Z^{1,2}_0} + \|n^{b,2}\|^2_{Z^{1,2}_0} + \|\partial_z n^{b,2}\|^2_{Z^{1,2}_0}\Big)  \\
    %     &\quad \times \|c^{b,2}\|^2_{Z^{1,2}_0} + \varepsilon^5 \| u_1^{b,1}\|^2_{Z^{1,3}_2} \Big(\|\partial_z c^{b,2}\|^2_{Z_0^{1,0}} + \|\partial_z c^{b,1}\|^2_{Z_0^{1,0}}\Big) \\
    %     &\quad + \varepsilon^7 \Big(\|c^{b,1}\|^2_{Z^{1,2}_0} + \|\partial_z c^{b,1}\|^2_{Z^{1,2}_0}\Big) \|n^{b,2}\|^2_{Z_0^{1,0}},
    % \end{align*}
    \begin{align*}
        \|\partial_x K\|^2_{\overline{H}^{1,m}} 
        &\leq \varepsilon^5 \Big(1 + \|(u_1^{e,0},n^{e,0})\|^2_{\overline{H}^{1,m+4}} \Big) \|( c^{b,1},c^{b,2})\|^2_{Z^{1,m+3}_1}+ \varepsilon^5 \|u_2^{e,0}\|^2_{\overline{H}^{1,m+5}} \| (\partial_z c^{b,1},\partial_z c^{b,2})\|^2_{Z^{1,m+1}_2}\\
        &\quad + \varepsilon^4 \|c^{e,0}\|^2_{\overline{H}^{1,m+5}} \Big(1 + \| (n^{b,1},n^{b,2})\|^2_{Z^{1,m+1}_1}+ \|  u_1^{b,1}\|^2_{Z^{1,m+2}_3} + \|u_1^{b,2}\|^2_{Z^{1,m+1}_0}\Big) \\
        &\quad + \varepsilon^7 \|(u_1^{b,1},\partial_z u_1^{b,1},n^{b,1},\partial_z n^{b,1},n^{b,2},\partial_z n^{b,2})\|^2_{Z^{1,m+2}_0}  \|c^{b,2}\|^2_{Z^{1,m+2}_0}  \\
        &\quad + \varepsilon^5 \| u_1^{b,1}\|^2_{Z^{1,m+3}_2} \|(\partial_z c^{b,1},\partial_z c^{b,2})\|^2_{Z^{1,m+1}_0}+ \varepsilon^7 \|(c^{b,1},\partial_z c^{b,1})\|^2_{Z^{1,m+2}_0} \|n^{b,2}\|^2_{Z^{1,m+1}_0},
    \end{align*}
    % \begin{align*}
    %     \|\partial_x K\|^2_{\overline{H}^{1,m}} 
    %     &\leq \varepsilon^5 \Big(1 + \|u_1^{e,0}\|^2_{\overline{H}^{1,m+4}} + \|n^{e,0}\|^2_{\overline{H}^{1,m+4}}\Big) \Big(\| c^{b,1}\|^2_{Z^{1,m+2}_1} + \| c^{b,2}\|^2_{Z^{1,m+2}_1} \\
    %     &\quad + \|c^{b,1}\|^2_{Z^{1,m+3}_0} + \|c^{b,2}\|^2_{Z^{1,m+3}_0}\Big) + \varepsilon^5 \|u_2^{e,0}\|^2_{\overline{H}^{1,m+5}} \Big(\| \partial_z c^{b,1}\|^2_{Z^{1,m+1}_2} + \| \partial_z c^{b,2}\|^2_{Z^{1,m+1}_2}\Big) \\
    %     &\quad + \varepsilon^4 \|c^{e,0}\|^2_{\overline{H}^{1,m+5}} \Big(1 + \| n^{b,1}\|^2_{Z^{1,m+1}_1} + \| n^{b,2}\|^2_{Z^{1,m+1}_1} + \|  u_1^{b,1}\|^2_{Z^{1,m+2}_3} + \|u_1^{b,2}\|^2_{Z^{1,m+1}_0}\Big) \\
    %     &\quad + \varepsilon^7 \Big(\|u_1^{b,1}\|^2_{Z^{1,m+2}_0} + \|\partial_z u_1^{b,1}\|^2_{Z^{1,m+2}_0} + \|n^{b,1}\|^2_{Z^{1,m+2}_0} + \|\partial_z n^{b,1}\|^2_{Z^{1,m+2}_0}  + \|n^{b,2}\|^2_{Z^{1,m+2}_0}\\
    %     &\quad + \|\partial_z n^{b,2}\|^2_{Z^{1,m+2}_0}\Big) \|c^{b,2}\|^2_{Z^{1,m+2}_0}  + \varepsilon^5 \| u_1^{b,1}\|^2_{Z^{1,m+3}_2} \Big(\|\partial_z c^{b,2}\|^2_{Z^{1,m+1}_0} + \|\partial_z c^{b,1}\|^2_{Z^{1,m+1}_0}\Big) \\
    %     &\quad + \varepsilon^7 \Big(\|c^{b,1}\|^2_{Z^{1,m+2}_0} + \|\partial_z c^{b,1}\|^2_{Z^{1,m+2}_0}\Big) \|n^{b,2}\|^2_{Z^{1,m+1}_0},
    % \end{align*}
    \begin{align*}
        \|\partial_y K\|^2_{Y^{1,m}} 
        &\leq \varepsilon^3 \Big(1 +  \|u^{e,0}\|^2_{\overline{H}^{1,m+4}} + \|n^{e,0}\|^2_{\overline{H}^{1,m+3}}\Big) \Big(\|(c^{b,1},c^{b,2})\|^2_{Z^{1,m+1}_0} + \|(\partial_z c^{b,1},\partial_z c^{b,2})\|^2_{Z^{1,m+2}_1} \Big) \\
        &\quad + \varepsilon^3 \|c^{e,0}\|^2_{\overline{H}^{1,m+4}} \Big(1 + \|(n^{b,1},n^{b,2},\partial_z u_1^{b,2})\|^2_{Z^{1,m}_0} + \|(\partial_z n^{b,1},\partial_z n^{b,2},\partial_z u_1^{b,1})\|^2_{Z^{1,m}_1}\\
        &\quad+ \| u_1^{b,1}\|^2_{Z^{1,m+1}_2} \Big) + \varepsilon^5 \|(\partial_z^2 u_1^{b,1},\partial_z^2 n^{b,1})\|^2_{Z^{1,m+1}_0} \|c^{b,2}\|^2_{Z^{1,m+1}_0}\\
        &\quad+ \varepsilon^5 \|(u_1^{b,1},\partial_z u_1^{b,1},n^{b,1},\partial_z n^{b,1},n^{b,2},\partial_z n^{b,2})\|^2_{Z^{1,m+1}_0}
       \|(c^{b,2},\partial_z c^{b,2},c^{b,1},\partial_z c^{b,1})\|^2_{Z^{1,m+1}_0} \\
        &\quad + \varepsilon^3 \Big(\| u_1^{b,1}\|^2_{Z^{1,m+2}_2}+ \|n^{b,2}\|^2_{Z^{1,m}_0} + \|u_2^{e,0}\|^2_{\overline{H}^{1,m+4}}\Big) \| (\partial_z^2 c^{b,1},\partial_z^2 c^{b,2})\|^2_{Z^{1,m+1}_2} ,
    \end{align*}
    \begin{align*}
        \|U_1\|^2_{Y^{1}} 
        &\leq \varepsilon^4 \|u_1^{e,0}\|^2_{\overline{H}^{1,5}} \Big(1 + \| u_1^{b,1}\|^2_{Z^{1,2}_2} \Big) 
        + \varepsilon^5 \|u_2^{e,0}\|^2_{\overline{H}^{1,5}} \| \partial_z u_1^{b,1}\|^2_{Z^{1,1}_2} 
        + \varepsilon^5 \| n^{b,1}\|^2_{Z^{1,2}_2} \\
        &\quad + \varepsilon^5 \|u_1^{b,1}\|^2_{Z^{1,3}_0} \Big(1 + \|\partial_z u_1^{b,1}\|^2_{Z^{1,3}_0}\Big) 
        + \varepsilon^5 \|u_1^{b,1}\|^2_{Z^{1,2}_2}\left( \|\partial_z u_1^{b,1}\|^2_{Z^{1,2}_0}+ \|\partial_z^2 u_1^{b,1}\|^2_{Z^{1,2}_0}\right),
    \end{align*}
    \begin{align*}
        \|U_2\|^2_{Y^{1}} 
        &\leq \varepsilon^5 \big\| u_1^{b,1}\big\|^2_{Z^{2,3}_2} +\varepsilon^5 \|n^{b,2}\|^2_{Y^{1}}+ \varepsilon^4\|u^{e,0}\|^2_{\overline{H}^{1,5}}+\varepsilon^5\|\partial_z u_1^{b,1} \|_{Z^{1,2}_0}
         \\
        &\quad + \varepsilon^4 \Big(1 + \|u_1^{b,1}\|^2_{Z^{1,3}_0} + \|\partial_z u_1^{b,1}\|^2_{Z^{1,2}_0} + \|u^{e,0}\|^2_{\overline{H}^{1,5}} \Big) \big\| u_1^{b,1}\big\|^2_{Z^{1,4}_2},
    \end{align*}
    \begin{align*}
        \| \partial_y U_1 \|_{Y^{1,m}}^2
        &\leq \varepsilon^3 \|u^{e,0}\|^2_{\overline{H}^{1,m+4}} \Big(1 + \|\partial_z u_1^{b,1}\|^2_{Z^{1,m+1}_1}+ \big\| u_1^{b,1}\big\|^2_{Z^{1,m+1}_2}\Big) + \varepsilon^3 \|n^{b,1}\|^2_{Z^{1,m+1}_0} \\
        &\quad + \varepsilon^3 \Big(1 + \|\partial_z u_1^{b,1}\|^2_{Z^{1,m+1}_0}\Big) \|(u_1^{b,1},\partial_z u_1^{b,1})\|^2_{Z^{1,m+2}_0} \\
        &\quad + \varepsilon^3 \left(\|u_2^{e,0}\|^2_{\overline{H}^{1,m+4}} +\big\| u_1^{b,1}\big\|^2_{Z^{1,m+2}_2}\right)\|\partial_z^2 u_1^{b,1}\|^2_{Z^{1,m+1}_2},
    \end{align*}
    % \begin{align*}
    %     \| \partial_y U_1 \|_{Y^{1,m}}^2
    %     &\leq \varepsilon^3 \|u_1^{e,0}\|^2_{\overline{H}^{1,m+4}} \Big(1 + \|\partial_z u_1^{b,1}\|^2_{Z^{1,m+1}_1} + \|u_1^{b,1}\|^2_{Z^{1,m+1}_0} + \big\| u_1^{b,1}\big\|^2_{Z^{1,m+1}_2}\Big) \\
    %     &\quad + \varepsilon^3 \|u_2^{e,0}\|^2_{\overline{H}^{1,m+4}} \|\partial_z u_1^{b,1}\|^2_{Z^{1,m}_1} + \varepsilon^3 \|n^{b,1}\|^2_{Z^{1,m+1}_0} \\
    %     &\quad + \varepsilon^3 \Big(1 + \|\partial_z u_1^{b,1}\|^2_{Z^{1,m+1}_0}\Big) \Big(\|u_1^{b,1}\|^2_{Z^{1,m+2}_0} + \|\partial_z u_1^{b,1}\|^2_{Z^{1,m+2}_0}\Big) \\
    %     &\quad + \varepsilon^3 \|u_2^{e,0}\|^2_{\overline{H}^{1,m+4}} \| \partial_z^2 u_1^{b,1}\|^2_{Z^{1,m}_2} + \varepsilon^3 \Big(\big\| u_1^{b,1}\big\|^2_{Z^{1,m+2}_2} + \|u_1^{b,1}\|^2_{Z^{1,m+2}_0}\Big) \|\partial_z^2 u_1^{b,1}\|^2_{Z^{1,m+1}_0},
    % \end{align*}
and
    % \begin{align*}
    %     \|\partial_x U_2\|^2_{Y^{1,m}} 
    %     &\leq \varepsilon^5 \big\| u_1^{b,1}\big\|^2_{Z^{2,m+3}_2} 
    %     + \varepsilon^4 \|u_2^{e,0}\|^2_{\overline{H}^{1,m+5}} \Big(1 + \big\| u_1^{b,1}\big\|^2_{Z^{1,m+2}_2}\Big) \\
    %     &\quad + \varepsilon^5 \Big(1 + \|u_1^{e,0}\|^2_{\overline{H}^{1,m+5}} + \|u_1^{b,1}\|^2_{Z^{1,m+2}_0} + \|\partial_z u_1^{b,1}\|^2_{Z^{1,m+2}_0} + \|u_1^{b,1}\|^2_{Z^{1,m+3}_0}\Big) \big\| u_1^{b,1}\big\|^2_{Z^{1,m+4}_2} \\
    %     &\quad + \varepsilon^5 \|n^{b,2}\|^2_{Z^{1,m+1}_0} 
    %     + \varepsilon^5 \|\partial_z u_1^{b,1}\|^2_{Z^{1,m+2}_0}.
    % \end{align*}
\begin{align*}
        \|\partial_x U_2\|^2_{Y^{1,m}} 
        &\leq \varepsilon^5 \big\| u_1^{b,1}\big\|^2_{Z^{2,m+3}_2} 
        + \varepsilon^4 \|u_2^{e,0}\|^2_{\overline{H}^{1,m+5}} \Big(1 + \big\| u_1^{b,1}\big\|^2_{Z^{1,m+2}_2}\Big) \\
        &\quad + \varepsilon^5 \Big(1 + \|u_1^{e,0}\|^2_{\overline{H}^{1,m+5}} + \|\partial_z u_1^{b,1}\|^2_{Z^{1,m+2}_0} + \|u_1^{b,1}\|^2_{Z^{1,m+3}_0}\Big) \big\| u_1^{b,1}\big\|^2_{Z^{1,m+4}_2} \\
        &\quad + \varepsilon^5 \|n^{b,2}\|^2_{Z^{1,m+1}_0} 
        + \varepsilon^5 \|\partial_z u_1^{b,1}\|^2_{Z^{1,m+2}_0}.
    \end{align*}
\end{proposition}

\begin{proposition}
    Assume that $(n^{e,0},c^{e,0},u^{e,0})$ are the same as in Proposition~\ref{proposition_e0}, $(n^{b,1},c^{b,1},u^{b,1})$ are the same as in Proposition~\ref{proposition_b,1}, and $(n^{b,2},c^{b,2},u^{b,2})$ are the same as in Proposition~\ref{proposition_b,2}. Then we have
    \begin{align*}
        \| f e^{-y} \|_{Y^{2,{5}}} &\le C\|u_1^{b,1}\|_{Z^{2,6}_1},
    \end{align*}
    and 
     \begin{align*}
        \| F e^{-y} \|_{Y^{2,2}}&\leq C\|u_1^{b,1}\|_{Z^{2,2}_1}.
    \end{align*}
\end{proposition}

\begin{proof} By direct computations, we have
    \begin{align*}
        \| f e^{-y} \|_{Y^{2,{5}}}^2&=\sum_{|\alpha|\leq5,|\alpha_1|\leq 2}\int_{\mathbb{R}^2_+} \left|\partial^\alpha \left(\int_0^\infty \partial_x u_1^{b,1}dz e^{-y}\right)\right|^2dxdy\\
        &\leq \sum_{|\alpha|\leq5,|\alpha_1|\leq 2}\int_{\mathbb{R}^2_+} \left(1+z\right)^2 \left(\partial^\alpha \partial_x u_1^{b,1}\right)^2 dxdz\leq C\|u_1^{b,1}\|_{Z^{2,6}_1}^2,
    \end{align*}
    and 
     \begin{align*}
        \| F e^{-y} \|_{Y^{2,2}}^2&=\sum_{|\alpha|\leq2,|\alpha_1|\leq 2}\int_{\mathbb{R}^2_+} \left|\partial^\alpha \left(\int_0^\infty u_1^{b,1}dz e^{-y}\right)\right|^2dxdy\\
        &\leq \sum_{|\alpha|\leq2,|\alpha_1|\leq 2}\int_{\mathbb{R}^2_+} \left(1+z\right)^2 \left(\partial^\alpha u_1^{b,1}\right)^2 dxdz\leq\|u_1^{b,1}\|_{Z^{2,2}_1}^2,
         \end{align*}
         where the definition of  $f$ in \eqref{eq:f-u2b2}. The proof is complete.
\end{proof}

{\bf Proof of Proposition \ref{prop:uniform bounds}.}
\begin{proof} It follows from the above three propositions, together with integration in $t$, that this proof is complete. \end{proof}

\noindent {\bf Acknowledgment:}\,
The authors would like to thank Professor Zhi-An Wang for some helpful communications.
 W. Wang was supported by National Key R\&D Program of China (No. 2023YFA1009200) and NSFC under grant 12471219 and 12071054.

\medskip
\noindent\textbf{Data Availability Statement:}
Data sharing is not applicable to this article as no datasets were generated or analyzed during the current study.

\noindent\textbf{Conflict of Interest:}
The authors declare that they have no conflict of interest.

\bibliographystyle{plain}
\bibliography{refs}

\end{document}